\documentclass[reqno,oneside,12pt]{amsart}
\usepackage[utf8]{inputenc}
\usepackage{amsmath}
\usepackage{amssymb,amsthm, amscd}
\usepackage[dvipsnames]{xcolor}
\usepackage{mathtools}
\usepackage{graphicx}
\usepackage{enumitem}
\usepackage{bm}
\usepackage{tikz-cd}
\usepackage{tikz}
\usetikzlibrary{calc,intersections,through,backgrounds}
\usetikzlibrary{positioning}
\usetikzlibrary{decorations.pathmorphing, decorations.shapes}
\usetikzlibrary{decorations.pathreplacing,calligraphy}
\usepackage{todonotes}

\newcommand\mapsfrom{\mathrel{\reflectbox{\ensuremath{\mapsto}}}}

\numberwithin{equation}{section}
\usepackage{hyperref}
\hypersetup{
linkcolor = {black},
urlcolor = {blue},
citecolor = {black}
}

\theoremstyle{plain}

\newtheorem{theorem}{Theorem}[section]
\newtheorem{lemma}[theorem]{Lemma}
\newtheorem{prop}[theorem]{Proposition}
\newtheorem{cor}[theorem]{Corollary}

\newtheorem*{theoremIntro}{Theorem}

\theoremstyle{definition}

\newtheorem{definition}[theorem]{Definition}
\newtheorem{example}[theorem]{Example}

\newtheorem*{assumption}{Assumption}
\newtheorem*{notation}{Notation}
\newtheorem*{convention}{Convention}

\theoremstyle{remark}

\newtheorem{remark}[theorem]{Remark}

\begin{document}

% \begin{textblock*}{100mm}(.75\textwidth,-2cm)
% Draft: \today
% \end{textblock*}

\title{Chevalley--Monk formulas for bow varieties}
\author{Till Wehrhan}
\address{Max-Planck Institute for Mathematics, Vivatgasse 7, 53111 Bonn, Germany}
\email{wehrhan@mpim-bonn.mpg.de}

\begin{abstract}
We prove a formula for the multiplication of equivariant first Chern classes of tautological bundles of type A bow varieties with respect to the stable envelope basis. This formula naturally generalizes the classical Chevalley--Monk formula and can be formulated in terms of  creating crossings of skein type diagrams that label the stable envelope basis.
\end{abstract}

\maketitle

\tableofcontents

\section{Introduction}

By introducing the theory of \textit{stable envelopes}, Maulik and Okounkov provided in~\cite{maulik2019quantum} a way to assign Hopf algebras (called \textit{quantum groups}) to a large family of symplectic varieties with appropriate torus action.
Given such a variety $X$ then stable envelopes provide families of bases of the localized torus equivariant cohomology of $X$ which are uniquely characterized by stability conditions which are similar to the stability conditions from equivariant Schubert calculus. The base change matrices of these bases give solutions of the Yang--Baxter equation and the corresponding quantum group is defined via the usual FRT-construction~\cite{faddeev1988quantization}. By construction, this quantum group naturally acts on the equivariant cohomology of $X$ and contains all operators of multiplication and (more generally) quantum multiplication with equivariant cohomology classes.

The main focus of the work~\cite{maulik2019quantum} lies on the case where $X$ is a Nakajima quiver variety, which includes the special case where $X$ is the cotangent bundle of a partial flag variety, the Hilbert schemes of points in the complex plane or (more generally) a moduli spaces of framed torsion sheaves. In their study, the first Chern classes of tautological bundles are of great significance as they are well behaved with respect to the stable envelopes basis. These operators can be described in terms of well-known Hamiltonians of integrable models in specific examples. For instance, in the case of Hilbert schemes of points in the complex plane, the operator of multiplication with the first Chern class of the tautological bundle coincides with the  Calogero--Moser--Sutherland Hamiltonian (see e.g.~\cite{costello2003hilbert}), whereas the operator of quantum multiplication yields a deformation of this Hamiltonian, see~\cite{okounkov2010quantum}. In the case of cotangent bundles of partial flag varieties, the operators of multiplication  and quantum multiplication with first Chern classes of tautological bundles are identified with the dynamical Hamiltonians for the XXX model for general linear groups as defined by Tarasov and Varchenko, see~\cite{tarasov2000difference}, \cite{tarasov2005dynamical}, \cite{tarasov2014hypergeometric} and \cite{rimanyi2015partial}.

In this article, we state and prove a formula for the multiplication of first Chern classes of tautological bundles with respect to the stable envelope basis for bow varieties. These varieties are a very general class of varieties which again are defined in terms of quiver representations and some Hamiltonian reductions construction extending the class of type A Nakajima quiver varieties.

Bow varieties were defined by Nakajima and Takayama in~\cite{nakajima2017cherkis}. They form a rich family of smooth symplectic varieties which fits into the framework of Maulik and Okounkov's theory. They have their origins in theoretical physics, where they appear as certain instanton moduli spaces, see~\cite{cherkis2009moduli}, \cite{cherkis2010instantons} and \cite{cherkis2011instantons}. In particular, they play a remarkable role in the context of 3d mirror symmetry which is a phenomenon from theoretical physics which relates different field theories, see in particular~\cite{nakajima2017cherkis}.

By construction, bow varieties are naturally endowed with a family of tautological vector bundles. They also admit a torus action which scales the symplectic form and has only finite fixed point locus. The torus fixed points of bow varieties were classified by Nakajima in~\cite[Theorem~A.5]{nakajima2021geometric} in terms of Maya diagrams or equivalenty partitions. In~\cite{rimanyi2020bow}, Rim\'anyi and Shou give an equivalent classification of torus fixed points in terms of skein type diagrams called \textit{tie diagrams} which naturally extend the torus fixed point combinatorics of partial flag varieties:
\[
\{\textup{Tie diagrams}\}\xleftrightarrow{\phantom{x}1:1\phantom{x}} \{\textup{Torus fixed points}\},\quad D\mapsto x_D.
\]

Our main result is then an \textit{explicit} formula for the multiplication of the tautological bundles (see Theorem~\ref{thm:CMAntidominantChamber} and Theorem~\ref{thm:CMGeneralCase}):
\begin{theoremIntro}[Multiplication formula]
Let $\mathcal C(\mathcal D)$ be a bow variety and $(\mathrm{Stab}(x_D))_D$ a fixed choice of stable basis and $\xi_i$ a tautological bundle. Then, we have 
\begin{equation}\label{eq:IntroductionCMBowVarieties}
c_1(\xi_i)\cup \mathrm{Stab}(x_D)= \iota_{x_D}^\ast(c_1(\xi_i))\cdot \mathrm{Stab}(x_D) + \sum_{D'\in \mathrm{SM}_{D,i}} \operatorname{sgn}(D,D')h\cdot\mathrm{Stab}(x_{D'}),
\end{equation}
where $h$ is the equivariant parameter corresponding to the scaling of the symplectic form, $\mathrm{SM}_{D,i}$ a certain set of tie diagrams which are obtained from $D$ by resolving one crossing and $\operatorname{sgn}(D,D')$ a explicitly computable sign.
\end{theoremIntro}

The formula holds in the localized torus equivariant cohomology of $\mathcal C(\mathcal D)$. The localization coefficients of $c_1(\xi_i)$ which appear in \eqref{eq:IntroductionCMBowVarieties} can be explicitly determined using the formula from~\cite[Theorem~4.10]{rimanyi2020bow}.

The formula \eqref{eq:IntroductionCMBowVarieties} generalizes the classical Chevalley--Monk formula from Schubert calculus as follows:
if our bow variety is the cotangent bundle of the full flag variety $F_n$ then the torus fixed points are labeled by elements in the symmetric group $S_n$ and the tautological bundles $\xi_i$ correspond to the universal quotient bundles $\mathcal Q_1,\ldots,\mathcal Q_{n-1}$.
In this special case, the formula \eqref{eq:IntroductionCMBowVarieties} was already proved by Su in~\cite[Theorem~3.7]{su2016equivariant} and can be reformulated as
\begin{equation}\label{eq:IntroductionCMCotangenBundles}
c_1(\mathcal Q_i)\cup \mathrm{Stab}(x_w)= \iota_{x_w}^\ast(c_1(\mathcal Q_i))\cdot \mathrm{Stab}(x_w) + \sum_{\substack{j\le i <k \\ \ell(w t_{k,j})<\ell(w)}} h\cdot\mathrm{Stab}(x_{w t_{k,j}}),
\end{equation}
where $w\in S_n$, $l$ is the Bruhat length and $t_{k,j}\in S_n$ the transposition switching $k$ and $j$. Since basis elements can be interpreted as one-parameter deformations of Schubert classes, a well-known limit argument (see e.g.~\cite[Section~9]{aluffi2017shadows}) proves that~\eqref{eq:IntroductionCMCotangenBundles} implies the classical Chevalley--Monk formula in the singular cohomology of $F_n$ (\cite{chevalley1994decompositions}, \cite{monk1959geometry}):
\begin{equation*}
c_1(\mathcal Q_i)\cup \mathfrak S_w= \sum_{\substack{j\le i <k \\ \ell(w t_{k,j})=\ell(w)+1}} \mathfrak S_{w t_{k,j}}, 
\end{equation*}
where $\mathfrak S_w$ denotes the Schubert class corresponding to the permutation $w$.

For the proof of \eqref{eq:IntroductionCMBowVarieties}, we employ the equivariant resolution theorem of Botta and Rim\'anyi \cite[Theorem~6.13]{botta2023mirror} that provides a way to compute localization coefficients of stable envelopes of bow varieties in terms of localization coefficients of stable envelopes of cotangent bundles of partial flag varieties. Then applying the localization formula from \cite{su2017stable} allows us to express these localization coefficients in terms of symmetric group combinatorics which gives a good control over them.

\medskip
\noindent
\textbf{Acknowledgements.} First, I want to thank my supervisor C. Stroppel for stimulating discussions and continuous support. Moreover, I am grateful to R. Rim\'anyi and T. Botta for sharing knowledge about bow varieties and further interesting mathematical discussions.
I thank the Max--Planck Institute for Mathematics Bonn for financial support.

\section{Bow varieties}

In this section, we recall the construction of bow varieties from~\cite{nakajima2017cherkis} and state some of their important properties following~\cite{nakajima2017cherkis} and~\cite{rimanyi2020bow}.

\begin{convention}
If not stated otherwise, all varieties and vector spaces are over $\mathbb C$. Given a variety $Y$ and a smooth point $y\in Y$ we denote by $T_yY$ the tangent space of $Y$ at $y$. 
\end{convention}

\subsection{Brane diagrams}\label{subsection:BraneDiagrams}

For the construction of bow varieties, we use the language of brane diagrams introduced by Rim\'anyi and Shou which we now briefly recall.

A \textit{brane diagram} is an object like this:
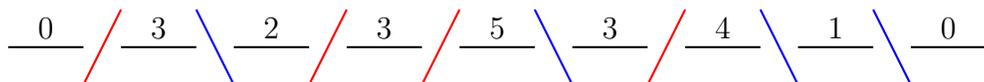
\[
\begin{tikzpicture}
[scale=.5]
\draw[thick] (0,0)--(2,0);
\draw[thick] (3,0)--(5,0);
\draw[thick] (6,0)--(8,0);
\draw[thick] (9,0)--(11,0);
\draw[thick] (12,0)--(14,0);
\draw[thick] (15,0)--(17,0);
\draw[thick] (18,0)--(20,0);
\draw[thick] (21,0)--(23,0);
\draw[thick] (24,0)--(26,0);

%red branes
\draw [thick,red] (2,-1) --(3,1); 
\draw [thick,red] (8,-1) --(9,1); 
\draw [thick,red] (11,-1) --(12,1); 
\draw [thick,red] (18,1) --(17,-1);

%blue branes
\draw [thick,blue] (5,1) --(6,-1); 
\draw [thick,blue] (15,-1) --(14,1);  
\draw [thick,blue] (21,-1) --(20,1); 
\draw [thick,blue] (24,-1) --(23,1);  

%numbers
\node at (1,0.5){$0$};
\node at (4,0.5){$3$};
\node at (7,0.5){$2$};
\node at (10,0.5){$3$};
\node at (13,0.5){$5$};
\node at (16,0.5){$3$};
\node at (19,0.5){$4$};
\node at (22,0.5){$1$};
\node at (25,0.5){$0$};
\end{tikzpicture}
\]
That is, a brane diagram is a finite sequence of black horizontal lines drawn from left to right. Between each consecutive pair of horizontal lines there is either a blue SE-NW line \textcolor{blue}{$\backslash$} or a red SW-NE line \textcolor{red}{$\slash$}. Each horizontal line $X$ is labeled by a non-negative integer $d_X$ where we demand that the first and the last horizontal line is labeled by $0$. 

\begin{remark}
In~\cite{rimanyi2020bow},
the black horizontal lines are, motivated by string theory, called D3 branes, the blue lines D5 branes and the red lines NS5 branes. 
Since brane diagrams appear for us only as purely combinatorial objects, we will not use this naming, but refer to their color instead.
\end{remark}

Given a brane diagram $\mathcal D$, we denote by $\mathrm h(\mathcal D)$, $\mathrm b(\mathcal D)$ and $\mathrm r(\mathcal D)$ the subset of black, blue and red lines. For lines $Y_1$, $Y_2$ in $\mathcal D$ write $Y_1 \triangleleft Y_2$ if $Y_1$ is to the left of $Y_2$. We denote the number of red lines  by $M$ and the individual red lines by $V_1,\ldots,V_M$ where we number the lines from \textit{right to left}.  Likewise, let $N$ be the number of blue lines and the blue lines are denoted by  $U_1,\ldots,U_N$ numbered from \textit{left to right}. So in the above brane diagrams the colored lines are labeled like this:
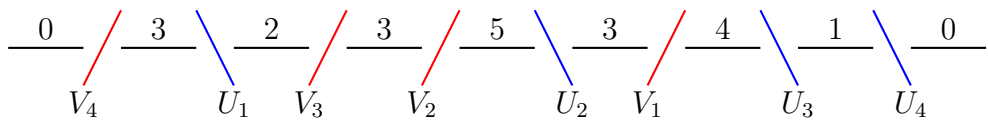
\[
\begin{tikzpicture}
[scale=.5]
\draw[thick] (0,0)--(2,0);
\draw[thick] (3,0)--(5,0);
\draw[thick] (6,0)--(8,0);
\draw[thick] (9,0)--(11,0);
\draw[thick] (12,0)--(14,0);
\draw[thick] (15,0)--(17,0);
\draw[thick] (18,0)--(20,0);
\draw[thick] (21,0)--(23,0);
\draw[thick] (24,0)--(26,0);

%red branes
\draw [thick,red] (2,-1) --(3,1); 
\draw [thick,red] (8,-1) --(9,1); 
\draw [thick,red] (11,-1) --(12,1); 
\draw [thick,red] (18,1) --(17,-1);

%V_i
\node at (2,-1.5) {$V_4$};
\node at (8,-1.5) {$V_3$};
\node at (11,-1.5) {$V_2$};
\node at (17,-1.5) {$V_1$};

%blue branes
\draw [thick,blue] (5,1) --(6,-1); 
\draw [thick,blue] (15,-1) --(14,1);  
\draw [thick,blue] (21,-1) --(20,1); 
\draw [thick,blue] (24,-1) --(23,1);

\node at (6,-1.5) {$U_1$};
\node at (15,-1.5) {$U_2$};
\node at (21,-1.5) {$U_3$};
\node at (24,-1.5) {$U_4$};  

%numbers
\node at (1,0.5){$0$};
\node at (4,0.5){$3$};
\node at (7,0.5){$2$};
\node at (10,0.5){$3$};
\node at (13,0.5){$5$};
\node at (16,0.5){$3$};
\node at (19,0.5){$4$};
\node at (22,0.5){$1$};
\node at (25,0.5){$0$};
\end{tikzpicture}
\]
The black lines are denoted by $X_1,\ldots,X_{M+N+1}$ numbered from \textit{left to right}. So in the above example $(d_{X_1},\ldots,d_{X_9})=(0,3,2,3,5,3,4,1,0)$.
Note that our choice of labeling of red lines differs from~\cite{rimanyi2020bow}.

The \textit{separatedness degree of $\mathcal D$} is defined as
\begin{equation}\label{eq:DefinitionSeparatednessDegree}
\mathrm{sdeg}(\mathcal D)=|\{(U,V)\in \mathrm{b}(\mathcal D)\times\mathrm{r}(\mathcal D) \mid U\triangleleft V \}|.
\end{equation}
We call $\mathcal D$ \textit{separated} if $\mathrm{sdeg}(\mathcal D)=0$. So $\mathcal D$ is separated if and only if it has the shape
$\textcolor{red}{\slash}\textcolor{red}{\slash}\cdots \textcolor{red}{\slash}\textcolor{blue}{\backslash}\textcolor{blue}{\backslash}\cdots \textcolor{blue}{\backslash}$. 

Bow varieties associated to separated brane diagrams have many convenient properties, as we will discuss in Section~\ref{section:StabSpearatedCase} and Section~\ref{section:EquivariantResolutionTheorem}. 
There are explicit moves on brane diagrams (called \textit{Hanany--Witten transition}) that reduce the separatedness degree by $1$ which we discuss in Subsection~\ref{subsection:HW}.

\subsection{Construction of bow varieties}\label{subsection:ConstructionBow}

Let $\mathcal D$ be a brane diagram. The construction of its associated bow variety proceeds through several steps. At first, we assign to each blue line $U\in \mathrm{b}(\mathcal D)$ a smooth affine variety $\mathrm{tri}_U$ called \textit{the triangle part (of $U$)} as follows:
we define
\[
\mathbb M_U\coloneqq \operatorname{Hom}({\mathbb C}^{d_{U^+}},{\mathbb C}^{d_{U^-}})\oplus \operatorname{End}({\mathbb C}^{d_{U^+}})\oplus \operatorname{End}({\mathbb C}^{d_{U^-}})\oplus\operatorname{Hom}({\mathbb C}^{d_{U^+}},\mathbb C)\oplus \operatorname{Hom}(\mathbb C,{\mathbb C}^{d_{U^-}})
\]
and denote the elements of $\mathbb M_U$ by tuples $(A_U,B_U^+,B_U^-,a_U,b_U)$ according to the diagram: 
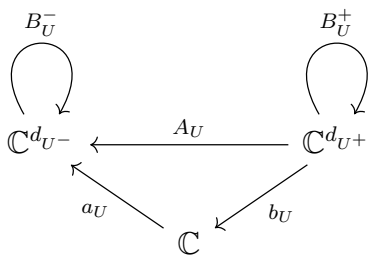
\[
\begin{tikzcd}
{\mathbb C}^{d_{U^-}}\arrow[out=120,in=60,loop,"B_U^-"]&&{\mathbb C}^{d_{U^+}}\arrow[out=120,in=60,loop,"B_U^+"]\arrow[ll, "A_U", swap]\arrow[dl,"b_U"]\\
&\mathbb C\arrow[ul,"a_U"]
\end{tikzcd}
\]
The \textit{triangle part} $\mathrm{tri}_U$ is then defined as 
$
\mathrm{tri}_U=\{x\in\mu^{-1}(0)\mid \textup{$x$ satisfies \ref{item:S1},\;\ref{item:S2}}\},
$
where the map $\mu\colon\mathbb M_U\rightarrow \operatorname{Hom}({\mathbb C}^{d_{U^+}},{\mathbb C}^{d_{U^-}})$ is given by
\begin{equation}\label{eq:nonMomentMap}
(A_U,B_U^+,B_U^-,a_U,b_U)\mapsto B_U^-A_U-A_UB_U^++a_Ub_U
\end{equation}
and the conditions \ref{item:S1} and \ref{item:S2} are defined as
\begin{enumerate}[label=(S\arabic*)]
\item\label{item:S1}  If $S\subset \mathbb C^{d_{U^+}}$ is a subspace with $B_U^+(S)\subset S,A_U(S)=0$ and $b_U(S)=0$ then $S=0$.
\item \label{item:S2} If $T\subset \mathbb C^{d_{U^-}}$ is a subspace with $B_U^-(T)\subset T$ and $\mathrm{Im}(A_U)+\mathrm{Im}(a_U)\subset T$ then $T=\mathbb C^{d_{U^-}}$.
\end{enumerate}
The following was shown in~\cite[Proposition~2.20]{takayama2016nahm}:

\begin{prop} The triangle part $\mathrm{tri}_U$ is a smooth and affine open subvariety of $\mu^{-1}(0)$.
\end{prop}

The next step is to define the \textit{affine brane variety} $\widetilde{\mathcal M}(\mathcal D)$ as follows:
\[
\widetilde{\mathcal M}(\mathcal D)\coloneqq \Big(\prod_{V\in \mathrm r(\mathcal D)}\operatorname{Hom}(\mathbb C^{d_{V^+}},\mathbb C^{d_{V^-}}) \times \operatorname{Hom}(\mathbb C^{d_{V^-}},\mathbb C^{d_{V^+}})\Big)\times \Big( \prod_{U\in \mathrm b(\mathcal D)} \mathrm{tri}_U\Big).
\]
Denote the elements of $\operatorname{Hom}(\mathbb C^{d_{V^+}},\mathbb C^{d_{V^-}}) \times \operatorname{Hom}(\mathbb C^{d_{V^-}},\mathbb C^{d_{V^+}})$ by $(C_V,D_V)$ and consequently elements of $\widetilde{\mathcal M}(\mathcal D)$ by tuples
\[
((A_U,B_U^+,B_U^-,a_U,b_U)_U,(C_V,D_V)_V).
\]
There is a natural algebraic action of
$
\mathcal G\coloneqq \prod_{X\in\mathrm h(\mathcal D)} \operatorname{GL}_{d_X}
$
on $\widetilde{\mathcal M}(\mathcal D)$ given by
\begin{align*}
(g_X)_X.&((A_U,B_U^+,B_U^-,a_U,b_U)_U,(C_V,D_V)_V) \\
&= ((g_{U^-}A_Ug_{U^+}^{-1},g_{U^+}B_U^+g_{U^+}^{-1},g_{U^-}B_U^-g_{U^-}^{-1},g_{U^-}a_U,b_Ug_{U^+}^{-1})_U,(g_{V^-}C_Vg_{V^+}^{-1},g_{V^+}D_Vg_{V^-}^{-1})_V).
\end{align*}
As shown in~\cite{nakajima2017cherkis} that $\widetilde{\mathcal M}(\mathcal D)$ admits a $\mathcal G$-equivariant symplectic structure with moment map
\[
m\colon\widetilde{\mathcal M}(\mathcal D)\xrightarrow{\phantom{xxx}} \prod_{X\in\mathrm{h}(X)} \operatorname{End}(\mathbb C^{d_X}),
\]
given for $x=((A_U,B_U^+,B_U^-,a_U,b_U)_U,(C_V,D_V)_V)$ by
\[
m(x)_X=
\begin{cases}
			   B_{X^+}^- - B_{X^-}^+ &\textup{if $X^+,X^-$ are both blue,} \\
			D_{X^-}C_{X^-}-C_{X^+}D_{X^+} &\textup{if $X^+,X^-$ are both red,}\\
			D_{X^-}C_{X^-}+ B_{X^+}^- &\textup{if $X^+$ is blue and $X^-$ is red,}\\
			-C_{X^+}D_{X^+}-B_{X^-}^+ &\textup{if $X^+$ is red and $X^-$ is blue.}
\end{cases}
\]

\begin{definition} Fix the character $\chi\colon\mathcal G\rightarrow\mathbb C^\ast, (g_X)_X\mapsto \prod_{X\in M_{\mathcal D}}\operatorname{det}(g_X)$, where $M_{\mathcal D}$ denote the set of black lines $X\in\mathrm{h}(\mathcal D)$ such that $X^-$ is red.
Then, the \textit{bow variety associated to $\mathcal D$} is defined as the G.I.T. quotient
\[
\mathcal C(\mathcal D)\coloneqq m^{-1}(0)/\!/_\chi \mathcal G.
\]
\end{definition}

The following properties of $\mathcal C(\mathcal D)$ were shown in \cite[Section~2]{nakajima2017cherkis}:
\begin{theorem}\label{thm:PropertiesOfBowVarieties} The following holds:
\begin{enumerate}[label=(\roman*)]
\item The $\chi$-stable locus $m^{-1}(0)^{\mathrm{s}}$ equals the $\chi$-semistable locus $m^{-1}(0)^{\mathrm{ss}}$.
\item $\mathcal C(\mathcal D)$ is a smooth and symplectic variety.
\item\label{item:PrincipalBundleProperty} The projection $m^{-1}(0)^{\mathrm{s}}\rightarrow \mathcal C(\mathcal D)$ is a principal $\mathcal G$-bundle in the Zariski topology.
\end{enumerate}
\end{theorem}

\begin{remark} The family of bow varieties contains type A Nakajima quiver varieties which can be realized as bow varieties corresponding to so called cobalanced brane diagrams, see~\cite[Section~2.6]{nakajima2017cherkis} for more details. In Section~\ref{subsec:TFlagAsBow}, we consider this realization in the special case of cotangent bundles of flag varieties.
\end{remark}

\subsection{Torus actions}\label{subsection:TorusAction} As described in~\cite[Section~6.9.3]{nakajima2017cherkis}, $\mathcal C(\mathcal D)$ is equipped with two torus actions. One action of $\mathbb A=(\mathbb C^\ast)^N$ and one of $\mathbb C^\ast_h=\mathbb C^\ast$. The $\mathbb A$-action is a direct consequence of the construction which we call the \textit{obvious action}. This action leaves the symplectic form on $\mathcal C(\mathcal D)$ invariant. The $\mathbb C_h^\ast$-action induces a scaling of the symplectic form and hence we refer to it as the \textit{scaling action}. In our exposition of these actions, we follow the conventions from~\cite[Section~3.1]{rimanyi2020bow}.
Denote the elements of $\mathbb A=(\mathbb C^\ast)^N$ by $(t_1,\ldots,t_N)$ or $(t_U)_{U\in \mathrm b(\mathcal D)}$ or just $(t_U)_U$. The obvious action is given by% $\mathbb A$-acts on $\mathcal C(\mathcal D)$ via
\[
(t_U)_{U}.[(A_U,B_U^+,B_U^-,a_U,b_U)_U,(C_V,D_V)_V]=
[(A_U,B_U^+,B_U^-,a_Ut_U^{-1},t_Ub_U)_U,(C_V,D_V)_V].
\]
We denote the elements of $\mathbb C^\ast_h$ by $h$. Then, the scaling action is given by
\[
h.[(A_U,B_U^+,B_U^-,a_U,b_U)_U,(C_V,D_V)_V]=[(A_U,hB_U^+,hB_U^-,a_U,hb_U)_U,(hC_V,D_V)_V].
\]
The $\mathbb T$-equivariant cohomology algebra of $\mathcal C(\mathcal D)$ for the large torus $\mathbb T\coloneqq \mathbb A\times \mathbb C_h^\ast$ is one of the main actors of this article. Our main tool in the study of this algebra is the localization principle which allows a powerful interplay of local and global data. To apply the localization principle appropriately, we briefly recall the classification of $\mathbb T$-fixed points of bow varieties as presented in~\cite[Section~4]{rimanyi2020bow}, in the upcoming subsections. For this, we first  define certain combinatorial objects that can be assigned to brane diagrams.

\subsection{Tie diagrams}
Given a pair of colored lines $(Y_1,Y_2)$ in $\mathcal D$ with $Y_1<Y_2$ then we say that a black line $X$ is \textit{covered} by $(Y_1,Y_2)$ if $Y_1$ is to the left of $X$ and $Y_2$ is to the right of $X$. 

\begin{definition}\label{definition:TieDiagrams}
\textit{A tie data with underlying brane diagram $\mathcal D$} is the data of $\mathcal D$ together with a set $D$ of pairs of colored lines of $\mathcal D$ such that the following holds:
\begin{itemize}
\item If $(Y_1,Y_2)\in D$ then $Y_1\triangleleft Y_2$.
\item If $(Y_1,Y_2)\in D$ then either $Y_1$ is blue and $Y_2$ is red or $Y_1$ is red and $Y_2$ is blue.
\item For all black lines $X$ of $\mathcal D$, the number of pairs in $D$ covering $X$ is equal to $d_X$.
\end{itemize}
\end{definition}

Usually, we work with a fixed brane diagram $\mathcal D$, so we refer to a tie data just by the set $D$. The set of all tie data corresponding to $\mathcal D$ is denoted by $\mathrm{Tie}(\mathcal D)$. 

A tie data $D$ can be visualized as follows: attach to the brane diagram $\mathcal D$ dotted curves (that are called \textit{ties}) following the conventions:
\begin{itemize}
\item For each pair $(Y_1,Y_2)\in D$ with $Y_1$ blue and $Y_2$ red, we draw a dotted curve below the diagram $\mathcal D$. 
\item For each pair $(Y_1,Y_2)\in D$ with $Y_1$ red and $Y_2$ blue, we draw a dotted curve above the diagram $\mathcal D$. 
\end{itemize}
The resulting diagram is then called \textit{tie diagram of $D$}. In the following, we also refer to the elements of $\mathrm{Tie}(\mathcal D)$ as tie diagrams.

\begin{example}\label{example:tiediagrams}
Let $\mathcal D$ be the brane diagram $
0\textcolor{blue}{\backslash} 
2\textcolor{red}{\slash} 
3\textcolor{blue}{\backslash} 
4\textcolor{blue}{\backslash} 
4\textcolor{red}{\slash}
3\textcolor{blue}{\backslash}
2\textcolor{red}{\slash}
0$
and $D$ the tie data
\[
D=\{(U_1,V_3),(U_1,V_2),(V_3,U_3),(V_3,U_4),(U_2,V_1),(U_3,V_1)\}.
\]
Then, the visualization of $D$ is given as follows:
\[
\begin{tikzpicture}[scale=.4]
\foreach \i in {0,...,7}
{
\draw[thick] (3*\i,0) -- (3*\i+2,0);
}

\draw[thick, red] (15,1) -- (14,-1);
\draw[thick, red] (6,1) -- (5,-1);
\draw[thick, red] (21,1) -- (20,-1);

\draw[thick, blue] (2,1) -- (3,-1);
\draw[thick, blue] (8,1) -- (9,-1);
\draw[thick, blue] (11,1) -- (12,-1);
\draw[thick, blue] (17,1) -- (18,-1);

\node at (1,0.5) {$0$};
\node at (4,0.5) {$2$};
\node at (7,0.5) {$3$};
\node at (10,0.5) {$4$};
\node at (13,0.5) {$4$};
\node at (16,0.5) {$3$};
\node at (19,0.5) {$2$};
\node at (22,0.5) {$0$};

\draw[dashed] (3,-1) to[out=-45,in=-135] (5,-1);
\draw[dashed] (3,-1) to[out=-45,in=-135] (14,-1);
\draw[dashed] (6,1) to[out=45,in=135] (11,1);
\draw[dashed] (6,1) to[out=45,in=135] (17,1);
\draw[dashed] (9,-1) to[out=-45,in=-135] (20,-1);
\draw[dashed] (12,-1) to[out=-45,in=-135] (20,-1);

\end{tikzpicture}
\]
\end{example}

\subsection{Torus fixed points}\label{subsection:TorusFixedPoints}

In \cite[Section~4]{rimanyi2020bow}, Rim\'anyi and Shou give an explicit construction which assigns a $\mathbb T$-fixed point $x_D\in\mathcal C(\mathcal D)$ to each tie diagram $D\in\mathrm{Tie}(\mathcal D)$. Based on this construction, they proved that there is a bijection
\begin{equation}\label{eq:RSTFixedPoints}
\mathrm{Tie}(\mathcal D)\xrightarrow{\phantom{x}\sim\phantom{x}} \mathcal C(\mathcal D)^{\mathbb T},\quad D\mapsto x_D
\end{equation}
which classifies the $\mathbb T$-fixed point of $\mathcal C(\mathcal D)$. We will usually identify a tie diagram $D$ with its corresponding $\mathbb T$-fixed point $x_D$.

The classification result \eqref{eq:RSTFixedPoints} can be strengthened in the following way: for a cocharacter 
\[
\sigma\colon\mathbb C^\ast\rightarrow \mathbb A,\quad t\mapsto (\sigma_U(t))_U,
\]
we denote by $\mathcal C(\mathcal D)^\sigma$ the corresponding $\mathbb C^\ast$-fixed locus. If $\sigma$ is \textit{generic}, that means if $\sigma_U\ne\sigma_{U'}$ for all $U$, $U'\in\mathrm{b}(\mathcal D)$, then $\mathcal C(\mathcal D)^{\mathbb T}=\mathcal C(\mathcal D)^{\sigma}$, see e.g.~\cite[Theorem~4.14]{expository2023orthogonality}.  

The bijection \eqref{eq:RSTFixedPoints} implies that $\mathcal C(\mathcal D)$ admits a $\mathbb T$-fixed point if and only if $\mathrm{Tie}(\mathcal D)\ne \emptyset$. We call a brane diagram $\mathcal D$ \textit{admissible} if $\mathrm{Tie}(\mathcal D)\ne \emptyset$. As the theory of stable envelopes is based on the localization principle in equivariant cohomology, we make the following assumption:

\begin{assumption} For the remainder of this article, we assume that $\mathcal D$ is admissible.
\end{assumption}

\begin{remark} The $\mathbb T$-fixed points of bow varieties were first classified in \cite[Theorem~A.5]{nakajima2021geometric}. For more details on the relation between this classification and the classification from \cite{rimanyi2020bow}, see \cite[Appendix~B]{shou2021bow}.
\end{remark}

\subsection{Binary contingency tables}\label{subsec:BCT} We continue with giving an equivalent definition of tie diagrams in terms of matrices with entries in $\{0,1\}$ which satisfies convenient compatibilities as we will discuss in Subsection~\ref{subsection:HW}.

Given a brane diagram $\mathcal D$, we first assign the following invariants to $\mathcal D$:
\[
r_i(\mathcal D)\coloneqq d_{V_i^+}-d_{V_i^-}+|\{U\in\mathrm{b}(\mathcal D)\mid U\triangleleft V_i\}|,\quad
c_j(\mathcal D)\coloneqq d_{U_j^-}-d_{U_j^+}+|\{V\in\mathrm{r}(\mathcal D)\mid U_j\triangleleft V\}|,
\]
where $i\in\{1,\ldots,M\} ,j\in\{1,\ldots,N\}$. In addition, we set
\[
R_l(\mathcal D)\coloneqq \sum_{i=1}^l r_i,\quad C_l(\mathcal D)\coloneqq \sum_{j=1}^lc_j.
\]
As $\mathcal D$ is usually a fixed brane diagram, we just denote $r_i(\mathcal D)$, $c_j(\mathcal D)$, $R_{i}(\mathcal D)$ and $C_j(\mathcal D)$ by $r_i$, $c_j$, $R_i$ and $C_j$. 
The vectors $\textbf r=(r_1,\ldots,r_M)$ and $\textbf c=(c_1,\ldots,c_N)$ are called \textit{margin vectors}.
The following was proved in~\cite[Lemma~2.3]{rimanyi2020bow}:
\begin{lemma} All $r_i(\mathcal D)$ and $c_j(\mathcal D)$ are non-negative and we have $R_M=C_N$.
\end{lemma} 

Let $\mathrm{bct}(\mathcal D)$ denote the set of all $M\times N$ matrices $B$ with entries in $\{0,1\}$ satisfying following row and column sum conditions:
\begin{itemize}
\item $\sum_{i=1}^M B_{i,j}=c_j$, for all $j\in\{1,\ldots,N\}$,
\item $\sum_{j=1}^N B_{i,j}=r_i$, for all $i\in\{1,\ldots,M\}$.
\end{itemize}
The elements of $\mathrm{bct}(\mathcal D)$ are called \textit{binary contingency tables of $\mathcal D$}.

The importance of binary contingency tables is the following bijection, see e.g.~\cite[Proposition~2.2.8]{shou2021bow}
\begin{equation}\label{eq:TieDiagramsVSBCTs}
\mathrm{Tie}(\mathcal D)\xleftrightarrow{\phantom{x}1:1\phantom{x}} \mathrm{bct}(\mathcal D),\quad  D\mapsto M(D),\quad D_B\mapsfrom B.
\end{equation}
The bijection is given as follows: if $D\in\mathrm{Tie}(\mathcal D)$ then the corresponding binary contingency table $M(D)$ is defined as 
\[
M(D)_{i,j}=\begin{cases}
1 &\textup{if $V_i\triangleleft U_j$ and $(V_i,U_j)\in D$,} \\
1 &\textup{if $U_j\triangleleft V_i$ and $(U_j,V_i)\notin D$,}\\
0 &\textup{if $V_i\triangleleft U_j$ and $(V_i,U_j)\notin D$,}\\
0 &\textup{if $U_j\triangleleft V_i$ and $(U_j,V_i)\in D$.}
\end{cases}
\]
Conversely, if we are given $B\in \mathrm{bct}(\mathcal D)$, we obtain a tie diagram $D_B$ as follows
\[
D_B=D_B'\cup D_B'',
\]
where 
\[
D_B'=\{(V_i,U_j)\mid V_i\triangleleft U_j,\;B_{i,j}=1\},\quad  D_B''=\{(U_j,V_i)\mid U_j\triangleleft V_i,\;B_{i,j}=0\}.
\] 
%It is straightforward to see that these constructions are inverse to each other which gives the bijection~\eqref{eq:TieDiagramsVSBCTs}.

Next, we describe the separating line of a binary contingency table $B$ which is a useful tool for illustrating the associated tie diagram $D_B$. For this, we draw the matrix $B$ into a coordinate system where the entry $B_{i,j}$ is put into the square box with side length $1$ and south-west corner at $(M-i,j-1)$. Then, we define points $p_0,\ldots,p_{M+N}$ in this coordinate system via $p_0=(0,0)$ and
\[
p_i=
\begin{cases}
 p_{i-1}+(1,0) &\textup{if $X_i^-$ is blue,}\\
p_{i-1}+(0,1) &\textup{if $X_i^-$ is red.}
\end{cases}
\]
The \textit{separating line $S_B$ of $B$} is then obtained by connecting each $p_i$ with $p_{i+1}$ by a straight line. Using $S_B$ we can easily illustrate $D_B$ using the following rules:
\begin{itemize}
\item For each $(i,j)$ such that $B_{i,j}=1$ and the entry $B_{i,j}$ lies below $S_D$ draw a dotted curve connecting $V_i$ and $U_j$.
\item For each $(i,j)$ such that $B_{i,j}=0$ and the entry $B_{i,j}$ lies above $S_D$ draw a dotted curve connecting $V_i$ and $U_j$.
\end{itemize}

\begin{example}
Let $\mathcal D$ and $D$ be as in Example~\ref{example:tiediagrams}. Then, the corresponding binary contingency table $M(D)$ with separating line is given as follows:
\[
\begin{tikzpicture}[scale=.4]
\draw[ultra thin] (0,0) -- (4,0);
\draw[ultra thin]  (0,1) -- (4,1);
\draw[ultra thin]  (0,2) -- (4,2);
\draw[ultra thin]  (0,3) -- (4,3);
\draw[ultra thin]  (0,0) -- (0,3);
\draw[ultra thin]  (1,0) -- (1,3);
\draw[ultra thin]  (2,0) -- (2,3);
\draw[ultra thin]  (3,0) -- (3,3);
\draw[ultra thin]  (4,0) -- (4,3);
\draw[ultra thick] (0,0) -- (1,0) -- (1,1) -- (3,1) -- (3,2) -- (4,2)-- (4,3);  

\node at (-0.7,0.6) {\tiny $V_3$};
\node at (-0.7,1.6) {\tiny $V_2$};
\node at (-0.7,2.6) {\tiny $V_1$};

\node at (0.5,3.7) {\tiny $U_1$};
\node at (1.5,3.7) {\tiny $U_2$};
\node at (2.5,3.7) {\tiny $U_3$};
\node at (3.5,3.7) {\tiny $U_4$};

\node at (0.5,0.5) {\tiny $0$};
\node at (1.5,0.5) {\tiny $0$};
\node at (2.5,0.5) {\tiny $1$};
\node at (3.5,0.5) {\tiny $1$};

\node at (0.5,1.5) {\tiny $0$};
\node at (1.5,1.5) {\tiny $1$};
\node at (2.5,1.5) {\tiny $1$};
\node at (3.5,1.5) {\tiny $0$};

\node at (0.5,2.5) {\tiny $1$};
\node at (1.5,2.5) {\tiny $0$};
\node at (2.5,2.5) {\tiny $0$};
\node at (3.5,2.5) {\tiny $1$};

\end{tikzpicture}
\]
\end{example}

\subsection{Localization in equivariant cohomology} 
We now recall the localization principle in torus equivariant cohomology which is an astonishing feature of this cohomology theory that provides an interesting connection between local and global data. For more details on this subject see e.g.~\cite{hsiang1975cohomology},~\cite{tomdieck1987transformation},~\cite{allday1993cohomological},~\cite{goresky1998equivariant} ,~\cite{anderson2012introduction} and~\cite{anderson2023equivariant}.

In the framework of bow varieties we are in the preferable situation of finitely many torus fixed points which are classified by tie diagrams respectively binary contingency tables as we discussed in the previous subsections.

We work with a fixed brane diagram $\mathcal D$ and let $\mathbb T=\mathbb A\times\mathbb C_h^\ast$ be the torus from Subsection~\ref{subsection:TorusAction}. Let $H_{\mathbb T}^\ast$ denote the $\mathbb T$-equivariant cohomology functor with coefficients in $\mathbb Q$. The $\mathbb T$-equivariant cohomology of a point $H^{\ast}_{\mathbb T}(\operatorname{pt})$ is isomorphic to the polynomial algebra $\mathbb Q[t_1,\ldots,t_N,h]$, where the parameters $t_1,\ldots,t_N$ correspond to $\mathbb A$ and the parameter $h$ corresponds to the factor $\mathbb C_h^\ast$.

Given a $\mathbb T$-fixed point $p\in\mathcal C(\mathcal D)^{\mathbb T}$ we denote by $\iota_p^\ast\colon H_{\mathbb T}^\ast(\mathcal C(\mathcal D))\rightarrow H_{\mathbb T}^\ast(\{p\})$ the corresponding restriction morphism in equivariant cohomology. If $\gamma\in H_{\mathbb T}^\ast(\mathcal C(\mathcal D))$ then we call $\iota_p^\ast(\gamma)$ the \textit{localization coefficient of $\gamma$ at $p$}.

The equivariant localization theorem, see e.g.~\cite[Theorem~7.1.1]{anderson2023equivariant}, states that the restriction $\iota^\ast\colon H_{\mathbb T}^\ast(\mathcal C(\mathcal D))\rightarrow H_{\mathbb T}^\ast(\mathcal C(\mathcal D)^{\mathbb T})$ induces an isomorphism
\[
H_{\mathbb T}^\ast(\mathcal C(\mathcal D))_{\mathrm{loc}}\xrightarrow{\phantom{x}\sim\phantom{x}} H_{\mathbb T}^\ast(\mathcal C(\mathcal D)^{\mathbb T})_{\mathrm{loc}},
\]
where $H_{\mathbb T}^\ast(\mathcal C(\mathcal D))_{\mathrm{loc}}$ resp. $H_{\mathbb T}^\ast(\mathcal C(\mathcal D)^{\mathbb T})_{\mathrm{loc}}$ is the localization at the multiplicative set generated by
\[
\{ a_1t_1+\dots +a_Nt_N+bh\mid a_1,\ldots,a_N,b\in\mathbb Q \}
\subset H_{\mathbb T}^\ast(\operatorname{pt}).
\]

%\begin{remark}
%The localization theorem can be applied to $\mathcal C(\mathcal D)$, as $\mathcal C(\mathcal D)$ is smooth.
%\end{remark}

\subsection{Tautological bundles}\label{subsection:TautologicalBundles} Bow varieties come with a family of \textit{tautological bundles} which are $\mathbb T$-equivariant. As we discuss below, their first Chern classes generate the localized equivariant cohomology ring $H_{\mathbb T}^\ast(\mathcal C(\mathcal D)^{\mathbb T})_{\mathrm{loc}}$, see Corollary~\ref{cor:FirstChernGenerators}. Since the restrictions of tautological bundles fit well into the framework of tie diagrams, they form a preferred choice of generators for this cohomology ring.

Given a black line $X$ in $\mathcal D$, the \textit{tautological bundle of $X$} is defined as the geometric quotient 
\begin{equation}\label{eq:DefinitionTautological}
\xi_X\coloneqq \mathbb C^{d_X}\times m^{-1}(0)^{\mathrm s}/\mathcal G.
\end{equation}
Here, $\mathcal G$ and $m\colon \widetilde{\mathcal D}\rightarrow \prod_{X'\in \mathrm{h}(\mathcal D)}\operatorname{End}(\mathbb C^{d_X'})$ are defined as in Subsection~\ref{subsection:ConstructionBow}, $\mathcal G$ acts diagonally on the product, where $\mathcal G$ acts on  $\mathbb C^{d_X}$ via
$(g_{X'})_{X'}.v=g_Xv$.
Since the quotient morphism $m^{-1}(0)^{\mathrm s}\rightarrow \mathcal C(\mathcal D)$ is a principal $\mathcal G$-bundle (in the Zariski topology), we conclude that $\xi_X$ is a vector bundle (in the Zariski topology) over $\mathcal C(\mathcal D)$ of rank $d_X$.
If $X=X_i$, we also denote $\xi_X$ by $\xi_{i}$.

The torus $\mathbb T$ acts on the first factor of $\mathbb C^{d_X}\times m^{-1}(0)^{\mathrm s}$ which induces a $\mathbb T$-action on $\xi_X$ giving $\xi_X$ the structure of a $\mathbb T$-equivariant vector bundle over $\mathcal C(\mathcal D)$.

The restrictions of tautological bundles to the $\mathbb T$-fixed points of $\mathcal C(\mathcal D)$ can  explicitly be computed via the following formula, see~\cite[Theorem~4.10]{rimanyi2020bow}:
\begin{equation}\label{eq:RestrictionTautologicalBundles}
\iota_D^\ast(\xi_i)=\bigoplus_{U\in \mathrm{b}(\mathcal D)}\bigoplus_{k=0}^{d_{D,U,X_i}-1} h^{c_{D,U,j}-d_{D,U,U^-}+1+k}\mathbb C_U.
\end{equation}
Here, $\mathbb C_U$ is the one-dimensional $\mathbb A$-module given by the character
$
\mathbb A\rightarrow\mathbb C^\ast
$ that projects to the factor with label $U$ and $h^i\mathbb C_U$ is the $\mathbb T$-module by further letting $\mathbb C_h^\ast $ act on $\mathbb C_U$ via $h.v=h^iv$ for $v\in\mathbb C_U$. The non-negative integer $d_{D,U,X}$ is defined for every black line $X$ as
$d_{D,U,X}=|\{\textup{$V\in \mathrm r(\mathcal D)$}\mid (V, U)\in D\}| $ for $X\triangleleft U$ and as
$d_{D,U,X}=|\{\textup{$V\in \mathrm r(\mathcal D)$}\mid (U,V)\in D\}|$ for $U\triangleleft X$.
The $c_{D,U,j}$ are given by 
$
c_{D,U,j} \coloneqq d_{D,U,U^+}- d_{D,U,X_j}
$
for all $j$ with $U \triangleleft X_j$. For $j$ with $X_j \triangleleft U$, $c_{D,U,j}$ is defined recursively as
\[
c_{D, U,j}\coloneqq \begin{cases}
c_{D, U,{j}} &\textup{if $X_j^+$ is blue},\\
c_{D, U,{j}} &\textup{if $X_j^+$ is red and $d_{D, U,X_j}+1=d_{D, U,X_{j+1}}$},\\
c_{D, U,{j}}-1 &\textup{if $X_j^+$ is red and $d_{D, U,X_j}=d_{D,U,X_{j+1}}$}.
\end{cases}
\]

\begin{example}
Consider the brane diagram $0\textcolor{red}{\slash} 2
\textcolor{red}{\slash} 3
\textcolor{red}{\slash} 4
\textcolor{red}{\slash} 5
\textcolor{blue}{\backslash} 2
\textcolor{red}{\slash}3
\textcolor{blue}{\backslash}1
\textcolor{red}{\slash}0$ with tie diagram $D$:
\[
\begin{tikzpicture}[scale=0.4]

\foreach \i in {0,...,8}
{
\draw[thick] (3*\i,0) -- (3*\i+2,0);
}

\draw[thick, red] (3,1) -- (2,-1);
\draw[thick, red] (6,1) -- (5,-1);
\draw[thick, red] (9,1) -- (8,-1);
\draw[thick, red] (12,1) -- (11,-1);
\draw[thick, red] (18,1) -- (17,-1);
\draw[thick, red] (24,1) -- (23,-1);

\draw[thick, blue] (14,1) -- (15,-1);
\draw[thick, blue] (20,1) -- (21,-1);

\draw[dashed] (3,1) to[out=45,in=135] (14,1);
\draw[dashed] (3,1) to[out=45,in=135] (20,1);
\draw[dashed] (6,1) to[out=45,in=135] (14,1);
\draw[dashed] (9,1) to[out=45,in=135] (14,1);
\draw[dashed] (12,1) to[out=45,in=135] (20,1);
\draw[dashed] (18,1) to[out=45,in=135] (20,1);

\draw[dashed] (15,-1) to[out=-45,in=-135] (17,-1);
\draw[dashed] (15,-1) to[out=-45,in=-135] (23,-1);

\node at (1,0.5) {$0$};
\node at (4,0.5) {$2$};
\node at (7,0.5) {$3$};
\node at (10,0.5) {$4$};
\node at (13,0.5) {$5$};
\node at (16,0.5) {$2$};
\node at (19,0.5) {$3$};
\node at (22,0.5) {$1$};
\node at (25,0.5) {$0$};

\end{tikzpicture}
\]
For a blue line $U$ and a black line $X_j$, the number $d_{D, U,X_j}$ is the number of ties starting in $U$ and covering $X_j$ from above or below. For instance, there are three ties starting in $U_1$ and covering $X_4$ which gives $d_{D, U_1,X_4}=3$. The other numbers $d_{D, U,X_j}$ are recorded in the following table:
\begin{center}
\begin{tabular}{| c | c | c | c | c | c | c | c |}
  \hline			
  $j$ & $2$ & $3$ & $4$ & $5$ & $6$ & $7$ & $8$ \\ \hline
  $d_{D, U_1,X_j}$ & $1$ & $2$ & $3$ & $3$ & $2$ & $1$ & $1$  \\
  \hline 
  $d_{D, U_2,X_j}$ & $1$ & $1$ & $1$ & $2$ & $2$ & $3$ & $0$  \\
  \hline  
\end{tabular}
\end{center}
The resulting indices $c_{D, U,j}$ are then given as follows:
\begin{center}
\begin{tabular}{| c | c | c | c | c | c | c | c |}
  \hline			
  $j$ & $2$ & $3$ & $4$ & $5$ & $6$ & $7$ & $8$ \\ \hline
  $c_{D, U_1,j}$ & $-1$ & $-1$ & $-1$ & $0$ & $0$ & $1$ & $1$  \\
  \hline
  $c_{D, U_2,j}$ & $-2$ & $-1$ & $0$ & $0$ & $0$ & $0$ & $0$  \\
  \hline  
\end{tabular}
\end{center}
Inserting in \eqref{eq:RestrictionTautologicalBundles} then gives the restrictions of the tautological bundles to the $\mathbb T$-fixed point $D$. For example, since $d_{D, U_1,X_4}=3$ and $c_{D, U_1,4}=-1$, the $U_1$-contribution in \eqref{eq:RestrictionTautologicalBundles} equals $h^{-3}\mathbb C_{U_1}\oplus h^{-2}\mathbb C_{U_1}\oplus h^{-1}\mathbb C_{U_1}$. Likewise, as $d_{D, U_1,X_4}=1$ and $c_{D, U_1,4}=-1$, the $U_2$ contribution in \eqref{eq:RestrictionTautologicalBundles} is $h^{-1}\mathbb C_{U_2}$. Consequently, $\iota_D^\ast(\xi_4)\cong  h^{-3}\mathbb C_{U_1} \oplus h^{-2}\mathbb C_{U_1} \oplus h^{-1}\mathbb C_{U_1} \oplus h^{-2}\mathbb C_{U_2}$. The other restrictions $\iota_D^\ast(\xi_j)$ are given as follows:
\begin{center}
\begin{tabular}{| c | c |}
\hline
$j$ & $\iota_D^\ast(\xi_j)$ \\ \hline \hline
$2$ & $h^{-3}\mathbb C_{U_1} \oplus h^{-4}\mathbb C_{U_2}$
\\ \hline
$3$ & $h^{-3}\mathbb C_{U_1} \oplus h^{-2}\mathbb C_{U_1} \oplus h^{-3}\mathbb C_{U_2}$
\\ \hline
$4$ & $h^{-3}\mathbb C_{U_1} \oplus h^{-2}\mathbb C_{U_1} \oplus h^{-1}\mathbb C_{U_1} \oplus h^{-2}\mathbb C_{U_2}$
\\ \hline
$5$ & $h^{-2}\mathbb C_{U_1} \oplus h^{-1}\mathbb C_{U_1} \oplus \mathbb C_{U_1} \oplus h^{-2} \mathbb C_{U_2}\oplus h^{-1}\mathbb C_{U_2}$
\\ \hline
$6$ & $h^{-2}\mathbb C_{U_1} \oplus h^{-1}\mathbb C_{U_1} \oplus h^{-2}\mathbb C_{U_2}\oplus h^{-1}\mathbb C_{U_2}$
\\ \hline
$7$ & $h^{-1}\mathbb C_{U_1}\oplus h^{-2}\mathbb C_{U_2}\oplus h^{-1}\mathbb C_{U_2}\oplus \mathbb C_{U_2}$
\\ \hline
$8$ & $h^{-1}\mathbb C_{U_1}$
\\ \hline
\end{tabular}
\end{center}
\end{example}

Using~\eqref{eq:RestrictionTautologicalBundles}, one can easily prove the following result:

\begin{cor}\label{cor:TautologicalRestrictionDeterminesTieDiagrams} Let $D$, $D'\in\mathrm{Tie}(\mathcal D)$. Then, $D=D'$ if and only if $\iota_D^{\ast}(\xi_X)=\iota_{D'}^{\ast}(\xi_X)$ for all $X\in\mathrm{h}(\mathcal D)$.
\end{cor}
\begin{proof} If $\iota_D^{\ast}(\xi_X)=\iota_{D'}^{\ast}(\xi_X)$ for all $X\in\mathrm{h}(\mathcal D)$ then $d_{D,U,X}=d_{D',U,X}$ for all $U\in\mathrm{b}(\mathcal D)$ and $X\in\mathrm{h}(\mathcal D)$ by~\eqref{eq:RestrictionTautologicalBundles}. This is equivalent to $D=D'$.
%The other implication is trivial.
\end{proof}

We conclude that the first Chern classes of the tautological bundles generate the localized equivariant cohomology $H_{\mathbb T}^\ast(\mathcal C(\mathcal D))_{\mathrm{loc}}$:

\begin{cor}\label{cor:FirstChernGenerators} We have that $c_1(\xi_1),\ldots,c_1(\xi_{M+N+1})$ are $H_{\mathbb T}^\ast(\operatorname{pt})_{\mathrm{loc}}$-algebra generators of $H_{\mathbb T}^\ast(\mathcal C(\mathcal D))_{\mathrm{loc}}$.
\end{cor}
\begin{proof} Let $A\subset H_{\mathbb T}^\ast(\mathcal C(\mathcal D))_{\mathrm{loc}}$ be the  $H_{\mathbb T}^\ast(\operatorname{pt})_{\mathrm{loc}}$-algebra generated by $c_1(\xi_1),\ldots,c_1(\xi_{M+N+1})$. By Corollary~\ref{cor:TautologicalRestrictionDeterminesTieDiagrams}, we have for all $D,D'\in\mathrm{Tie}(\mathcal D)$ that $D=D'$ if and only if
$\iota_D^\ast(c_1(\xi_X))=\iota_{D'}^\ast(c_1(\xi_X))$ for all $X\in \mathrm{h}(\mathcal D)$. Thus, the Chinese remainder theorem implies that for all $D\in \mathrm{Tie}(\mathcal D)$ there exists an element $f_D\in A$ such that $f_D\in \bigcap_{D'\ne D}\mathrm{ker}(\iota_{D'}^\ast)$ and $f_D\equiv 1\;\mathrm{mod} \; \mathrm{ker}(\iota_{D}^\ast)$. Thus, the localization theorem implies $A= H_{\mathbb T}^\ast(\mathcal C(\mathcal D))_{\mathrm{loc}}$.
\end{proof}

Suppose now that $\mathcal D$ is separated. As shown in e.g.~\cite[Proposition~3.4]{botta2023mirror}, the $\mathbb T$-equivariant vector bundles $\xi_{U^-}$ are topologically trivial. The corresponding $\mathbb T$-characters given as follows:

\begin{cor}\label{cor:PropertiesTautologicalBundlesSeparated} 
Let $\mathcal D$ be separated. Then, we have
\[
\xi_{U_j^-}= \bigoplus_{k=j}^N \bigoplus_{i=0}^{c_k-1} h^{-i}\mathbb C_{U_k},
\quad
\textit{for $U_j\in \mathrm{b}(\mathcal D)$ and all $D\in\mathrm{Tie}(\mathcal D)$.}
\]
\end{cor}

\begin{proof} Following the definition of the $d_{D,U,X}$ and $c_{D,U,X}$, the separatedness condition yields
$d_{D,U_k,U_j}=0$ if $k<j$ and $d_{D,U_k,U_j}=c_k$ if $k\geq j$.
We also deduce $c_{D,U_k,M+j}=-c_k$ for $k\geq j$. Thus, Corollary~\ref{cor:PropertiesTautologicalBundlesSeparated} follows from \eqref{eq:RestrictionTautologicalBundles}.
\end{proof}

\subsection{Hanany\textbf{--}Witten transition}\label{subsection:HW}

We consider now important classes of isomorphisms, so-called \textit{Hanany--Witten isomorphisms}, corresponding to certain moves on brane diagrams: given brane diagrams $\mathcal D$ and $\tilde{\mathcal D}$. Then, we say that $\tilde{\mathcal D}$ is obtained from $\mathcal D$ via \textit{Hanany--Witten transition} if $\tilde{\mathcal D}$ differs from $\mathcal D$ by performing a local move of the form
\[
\begin{tikzpicture}
[scale=.5]
\draw[thick] (0,0)--(2,0);
\draw[thick] (3,0)--(5,0);
\draw[thick] (6,0)--(8,0);

\draw [thick,red] (5,-1) --(6,1);
\draw [thick,blue] (3,-1) --(2,1);

\node at (5,-1.5) {$V_i$};
\node at (3,-1.5) {$U_j$};

\draw[-to] decorate[decoration=zigzag] {(10,0) -- (12,0)};

\draw[thick] (14,0)--(16,0);
\draw[thick] (17,0)--(19,0);
\draw[thick] (20,0)--(22,0);

\draw [thick,red] (16,-1) --(17,1);
\draw [thick,blue] (20,-1) --(19,1);

\node at (16,-1.5) {$V_i$};
\node at (20,-1.5) {$U_j$};

\node at (1,0.5){$d_1$};
\node at (4,0.5){$d_2$};
\node at (7,0.5){$d_3$};
\node at (15,0.5){$d_1$};
\node at (18,0.5){$\tilde{d}_2$};
\node at (21,0.5){$d_3$};
\end{tikzpicture}
\]
where $d_1+d_3+1=d_2+\tilde d_2$. Clearly, Hanany--Witten transitions reduce the separatedness degree (as defined in \eqref{eq:DefinitionSeparatednessDegree}) by $1$. It is straightforward to show that from any (admissible) brane diagram we can obtain a separated brane diagram via a finite number of Hanany--Witten transitions, see e.g.~\cite[Proposition~4.12]{expository2023orthogonality}.

The following proposition, see~\cite[Proposition~7.1]{nakajima2017cherkis} and \cite[Theorem~3.9]{rimanyi2020bow}, characterizes the isomorphism corresponding to a Hanany--Witten transition as well as the interplay of tautological bundles under this isomorphism:
\begin{prop}\label{prop:HWTransition}
Assume $\tilde{\mathcal D}$ is obtained from ${\mathcal D}$ via a Hanany--Witten transition where the blue line $U_j$ is exchanged with the red line $V_i$. Let $X_k$ be the black line in $\mathcal D$  with $X_k^-=U_j$ and $X_k^+=V_i$.
Then, there exists a $\rho_j$-equivariant isomorphism of symplectic varieties \begin{equation}\label{eq:HWIsomorphism}
\Phi:\mathcal C(\mathcal D)\xrightarrow{\phantom{x}\sim\phantom{x}} \mathcal C(\tilde{\mathcal D}),
\end{equation}
where $\rho_j$ is the algebraic group automorphism 
\[
\rho_j\colon \mathbb T\rightarrow \mathbb T,\quad (t_1,\ldots,t_N,h)\mapsto (t_1,\ldots, t_{j-1},ht_j,t_{j+1},\ldots,t_N,h).
\]
Furthermore, the following holds:
\begin{enumerate}[label=(\roman*)]
\item We have $\mathbb T$-equivariant isomorphisms of vector bundles $\xi_l\cong \Phi^\ast \tilde \xi_l$ for $l\ne k$.
\item There is a short exact sequence of $\mathbb T$-equivariant vector bundles
\begin{equation}\label{eq:HWExactSequence}
0\rightarrow \xi_{k} \rightarrow \xi_{k-1}\oplus \xi_{k+1} \oplus h\mathbb C_{U_i}\rightarrow \Phi^\ast \tilde \xi_k \rightarrow 0.
\end{equation}
\end{enumerate}
Here, the $\tilde \xi_l$ denote the tautological bundles on $\mathcal C(\tilde {\mathcal D})$ and 
$\Phi^\ast \tilde \xi_l$ is the $\mathbb T$-equivariant pull-back of $\tilde \xi_l$ via $\Phi$.
\end{prop}

The fixed point matching under Hanany--Witten transition is described in~\cite[Section~4.7]{rimanyi2020bow} as follows:
\begin{prop}\label{prop:HWFixedPointMatching}
With the assumptions of Proposition~\ref{prop:HWTransition}, let
$
\phi\colon \mathcal C(\mathcal D)^{\mathbb T}\xrightarrow{\sim} \mathcal C(\tilde{\mathcal D})^{\mathbb T}
$
denote the bijection induced by the Hanany--Witten isomorphism $\Phi$ from \eqref{eq:HWIsomorphism}. Then, we have
$M(D)=M(\phi(D))$ for all $D\in\mathrm{Tie}(\mathcal D)$.% which uniquely determines $\phi$.
\end{prop}
The statement of tautological bundles from Proposition~\ref{prop:HWTransition} directly gives the following relation of equivariant first Chern classes:

\begin{cor}\label{cor:MatchingOfTautologicalsHW} 
Under the assumptions of Proposition~\ref{prop:HWTransition}
let $\Phi^\ast\colon H_{\mathbb T}^\ast(\mathcal C(\tilde{\mathcal D})) \xrightarrow\sim H_{\mathbb T}^\ast(\mathcal C({\mathcal D}))$ be the induced isomorphism of rings. Then, we have
\begin{enumerate}[label=(\roman*)]
\item\label{item:MatchingOfTautologicalsHW1}
$\Phi^\ast(c_1(\tilde\xi_l))=c_1(\xi_l)$ for $l\ne k$,
\item\label{item:MatchingOfTautologicalsHW2}
$\Phi^\ast(c_1(\tilde \xi_k))=c_1(\xi_{k+1})+c_1(\xi_{k-1})+h+t_j-c_1(\xi_k)$.
\end{enumerate}
\end{cor}

For a $\mathbb T$-equivariant vector bundle $\mathcal V$ over $\mathcal C(\mathcal D)$, we denote 
by $c_i(\mathcal V)\in H^{2i}_{\mathbb T}(\mathcal C(\mathcal D))$ its $i$-th $\mathbb T$-equivariant Chern class. By Corollary~\ref{cor:MatchingOfTautologicalsHW}, the localization coefficients of first Chern classes of tautological bundles satisfy the following matching properties:

\begin{cor}\label{cor:MatchingOfEqMultTautHW} With the same notation as in Corollary~\ref{cor:MatchingOfTautologicalsHW}  we have
\[
\varphi_j(\iota_{\phi(D)}^\ast (c_1( \tilde \xi_l)))=\iota_D^\ast (c_1(\xi_l)),\quad\textit{for $l\ne k$}
\]
and
\[
\varphi_j(\iota_{\phi(D)}^\ast (c_1( \tilde \xi_k)))=\iota_D^\ast(c_1(\xi_{k+1})+c_1(\xi_{k-1})-c_1(\xi_k))+h+t_j,
\]
for all $D\in\mathrm{Tie}(\mathcal D)$,
where $\varphi_j\colon \mathbb Q[t_1,\ldots,t_N,h]\xrightarrow\sim \mathbb Q[t_1,\ldots,t_N,h]$ is the $\mathbb Q[h]$-algebra automorphism given by $t_j\mapsto t_j+h$ and $t_i\mapsto t_i$ for $i\ne j$. 
\end{cor}

\section{Cotangent bundles of partial flag varieties as bow varieties}\label{subsec:TFlagAsBow}

For natural numbers
$
0<d_1<d_2<\ldots <d_m<n
$
let $F(d_1,\ldots,d_m;n)$ denote the partial flag variety parameterizing inclusions of $\mathbb C$-linear subspaces 
\[
\{0\}\subset E_1 \subset E_2 \subset \dots \subset E_m\subset \mathbb C^n
\] 
with $\operatorname{dim}(E_i)=d_i$ for $i=1,\ldots,m$.

It is well-known, see \cite[Theorem~7.3]{nakajima1994instantons}, that the cotangent bundle $T^\ast F(d_1,\ldots,d_m;n)$ is isomorphic to the Nakajima quiver variety corresponding to the framed quiver
\[
\begin{tikzcd}
\overset{1}\bullet  & \overset{2}\bullet\arrow[l] & \arrow[l]\dots & \overset{m}\bullet\arrow[l]\\
&&& \square \arrow[u]
\end{tikzcd}
\]
with dimension vector $(n-d_m,\ldots,n-d_1)$, framing vector $(0,\ldots,0,n)$ and character
\[
\theta\colon \prod_{i=1}^m\mathrm{GL}_{n-d_i}(\mathbb C)\xrightarrow{\phantom{xxx}} \mathbb C^\ast,\quad (g_1,\ldots,g_m)\mapsto \prod_{i=1}^m\operatorname{det}(g_i).
\]
Thus, by \cite[Theorem~2.15]{nakajima2017cherkis}, we can realize $T^\ast F(d_1,\ldots,d_m;n)$ as a bow variety. In this section, we will explicitly describe this realization. We also characterize the induced correspondence between the torus fixed point combinatorics of these varieties.

\subsection{Realization via parabolic subgroups} Let $d_0=0$, $d_{m+1}=n$, $E_0=0$, $E_{m+1}=\mathbb C^n$ and $\delta_i=d_i-d_{i-1}$ for $i=1,\ldots, m+1$. Let $G=\mathrm{GL}_n(\mathbb C)$ and $P\subset G$ be the parabolic subgroup of block matrices of the shape
\[
\begin{pmatrix}
P_{1,1} & P_{1,2} & \dots & P_{1,m+1} \\
& P_{2,2} & \dots & P_{2,m+1} \\
&&\ddots&\vdots\\
&&&P_{m+1,m+1}
\end{pmatrix},
\]
where each $P_{i,j}$ is a $\delta_i\times \delta_j$ matrix. It is well-known that the geometric quotient $G/P$ exists and we have an isomorphism of varieties $G/P\xrightarrow{\sim}F(d_1,\ldots,d_m;n)$ given by
\[
[g]\mapsto (\{0\}\subset \langle g_1,\ldots,g_{d_1}\rangle \subset\dots\subset \langle g_1,\ldots,g_{d_m}\rangle \subset \mathbb C^n).
\]
Here, $g_i$ denotes the $i$-th column vector of $g$ for $i=1,\ldots,n$. 

Let $\mathfrak g=\mathfrak{gl}_n(\mathbb C)$ be the Lie algebra of $G$ and $\mathfrak p\subset \mathfrak g$ be the Lie-subalgebra corresponding to $P$. We denote by $\mathfrak p^\perp$ the annihilator of $\mathfrak p$ with respect to the trace pairing. That is, $\mathfrak p^\perp$ is the Lie subalgebra of $\mathfrak g$ consisting of block matrices of the form
\[
\begin{pmatrix} 
0 & P_{1,2} & P_{1,3} & \dots & P_{1,m+1} \\
& 0 & P_{2,3} & \dots & P_{2,m+1} \\
& & \ddots & \ddots & \vdots \\
& & & 0 & P_{m,m+1}\\
&&&&0
\end{pmatrix},
\quad P_{i,j}\in \mathrm{Mat}_{\delta_i,\delta_j}(\mathbb C).
\]
Here, again $P_{i,j}$ is a $\delta_i\times \delta_j$ matrix. The parabolic subgroup $P$ acts algebraically on $\mathfrak p^\perp$ via conjugation. It is well-known that the cotangent bundle $T^\ast F(d_1,\ldots,d_m;n)$ is isomorphic as algebraic variety to the geometric quotient $(G\times\mathfrak p^\perp)/P$, see e.g.~\cite[Lemma~1.4.9]{chriss1997representation}. Hence, the points of $T^\ast F(d_1,\ldots,d_m;n)$ can be identified with pairs $(\mathcal F, f)$ where
\[
\mathcal F= (\{0\}\subset E_1  \subset\dots\subset E_m \subset \mathbb C^n)
\]
is a point in $F(d_1,\ldots,d_m;n)$ and $f\in\operatorname{End}(\mathbb C^n)$ such that
$f(E_i)\subset E_{i-1}$ for $i=1,\ldots,m+1$.

\subsection{Bow variety realization}
Let $\tilde{\mathcal D}(d_1,\ldots,d_m;n)$ be the brane diagram:
\[
\begin{tikzpicture}
[scale=.5]
\draw[thick] (0,0)--(2,0);
\draw[thick] (3,0)--(5,0);
\draw[thick] (6,0)--(8,0);
\draw[dotted] (9,0)--(12,0);
\draw[thick] (13,0)--(15,0);
\draw[thick] (16,0)--(18,0);
\draw[thick] (19,0)--(21,0);
\draw[dotted] (22,0)--(25,0);
\draw[thick] (26,0)--(28,0);
\draw[thick] (29,0)--(31,0);
%\draw[thick] (32,0)--(34,0);

\draw [thick,red] (2,-1) --(3,1); 
\draw [thick,red] (5,-1) --(6,1); 
\draw [thick,red] (8,-1) --(9,1); 
\draw [thick,red] (12,-1) --(13,1);
\draw [thick,red] (15,-1) --(16,1);
\draw [thick,red] (28,-1) --(29,1);

\node at (2,-1.5){$V_{m+1}$};
\node at (5,-1.5){$V_m$};
\node at (8,-1.5){$V_{m-1}$};
\node at (12,-1.5){$V_{3}$};
\node at (15,-1.5){$V_2$};
\node at (28,-1.5){$V_{1}$};

\node at (19,-1.5){$U_{1}$};
\node at (22,-1.5){$U_{2}$};
%\node[blue] at (26,-1.5){$U_{n-1}$};
\node at (26,-1.5){$U_{n}$};

%blue branes 
\draw [thick,blue] (19,-1) --(18,1); 
\draw [thick,blue] (22,-1) --(21,1);  
\draw [thick,blue] (26,-1) --(25,1); 
%\draw [thick,blue] (29,-1) --(28,1);  

%numbers
\node at (1,0.5){$0$};
\node at (4,0.5){$d_m'$};
\node at (7,0.5){$d_{m-1}'$};
\node at (14,0.5){$d_{2}'$};
\node at (17,0.5){$d_1'$};
\node at (20,0.5){$d_1'$};
\node at (27,0.5){$d_1'$};
%\node at (30,0.5){$d_1'$};
\node at (30,0.5){$0$};
\end{tikzpicture}
\]
where $d_i'=n-d_{i}$ for $i=1,\ldots,m$.
We denote elements of $\mathcal C(\tilde{\mathcal D}(d_1,\ldots,d_m;n))$ according to the diagram
\[
\begin{tikzpicture}[scale=.5]
\node at (1,0){$0$};
\node at (4,0){$\mathbb C^{d_m'}$};
\node at (7,0){$\mathbb C^{d_{m-1}'}$};
\draw[-to] (1.3,-0.5) to[out=-45, in=-135] (3.4,-0.5); %D
\node at ( 2.5,-1.4) {$\begin{smallmatrix}{D_{m+1}}\end{smallmatrix} $};
\draw[-to] (4.6,-0.5) to[out=-45, in=-135] (6.3,-0.5); %D
\node at ( 5.5,-1.4) {$\begin{smallmatrix}{D_{m}}\end{smallmatrix} $};
\draw[-to] (7.6,-0.5) to[out=-45, in=-135] (9.4,-0.5); %D
\node at ( 8.5,-1.4) {$\begin{smallmatrix}{D_{m-1}}\end{smallmatrix} $};
\draw[to-] (1.3,0.5) to[out=45, in=135] (3.4,0.5); %C
\node at ( 2.5,1.4) {$\begin{smallmatrix}{C_{m+1}}\end{smallmatrix} $};
\draw[to-] (4.6,0.5) to[out=45, in=135] (6.3,0.5); %C
\node at ( 5.5,1.4) {$\begin{smallmatrix}{C_{m}}\end{smallmatrix} $};
\draw[to-] (7.6,0.5) to[out=45, in=135] (9.4,0.5); %C
\node at ( 8.5,1.4) {$\begin{smallmatrix}{C_{m-1}}\end{smallmatrix} $};
\draw[dotted] (10,0) -- (11,0);
\draw[-to] (11.6,-0.5) to[out=-45, in=-135] (13.4,-0.5); %D
\node at ( 12.5,-1.4) {$\begin{smallmatrix}{D_{3}}\end{smallmatrix} $};
\node at ( 15.5,-1.4) {$\begin{smallmatrix}{D_{2}}\end{smallmatrix} $};
\draw[to-] (11.6,0.5) to[out=45, in=135] (13.4,0.5); %C
\node at ( 12.5,1.4) {$\begin{smallmatrix}{C_{3}}\end{smallmatrix} $};
\node at ( 15.5,1.4) {$\begin{smallmatrix}{C_{2}}\end{smallmatrix} $};
\node at (14,0){$\mathbb C^{d_{2}'}$};
\draw[-to] (14.6,-0.5) to[out=-45, in=-135] (16.4,-0.5); %D
\draw[to-] (14.6,0.5) to[out=45, in=135] (16.4,0.5); %C
\node at (17,0){$\mathbb C^{d_{1}'}$};
\node at (18.5,-2) {$\mathbb C$};
\draw[-to] (18,-1.5) -- (17.5,-0.5); %a
\node at ( 17.2,-0.9) {$\begin{smallmatrix}{a_1}\end{smallmatrix} $};
\draw[-to] (19.3,0) -- (17.7,0); %A
\node at ( 18.5,0.4) {$\begin{smallmatrix}{A_1}\end{smallmatrix} $};
\node at (20,0){$\mathbb C^{d_{1}'}$};
\node at (21.5,-2) {$\mathbb C$};
\draw[-to] (22.5,-0.5) -- (22,-1.5); %b
\node at ( 22.7,-1.35) {$\begin{smallmatrix}b_2 \end{smallmatrix} $};
\draw[to-] (19,-1.5) -- (19.5,-0.5); %b
\node at ( 19.7,-1.35) {$\begin{smallmatrix}{b_1}\end{smallmatrix} $};
\draw[-to] (21,-1.5) -- (20.5,-0.5); %a
\node at ( 20.2,-0.9) {$\begin{smallmatrix}{a_2}\end{smallmatrix} $};
\draw[-to] (22.3,0) -- (20.7,0); %A
\node at ( 21.5,0.4) {$\begin{smallmatrix}{A_2}\end{smallmatrix} $};
\draw[dotted] (22.8,0) -- (24.2,0);
\draw[-to] (26.3,0) -- (24.7,0); %A
\node at ( 25.5,0.4) {$\begin{smallmatrix}{A_n}\end{smallmatrix} $};
\node at (25.5,-2) {$\mathbb C$};
\draw[-to] (25,-1.5) -- (24.5,-0.5); %a
\node at ( 24.2,-0.9) {$\begin{smallmatrix}{a_n}\end{smallmatrix} $};
\draw[-to] (26.5,-0.5) -- (26,-1.5); %b
\node at ( 26.7,-1.35) {$\begin{smallmatrix}{b_n}\end{smallmatrix} $};
\node at (27,0){$\mathbb C^{d_{1}'}$};
\draw[-to] (27.5,-0.5) to[out=-45, in=-135] (29.7,-0.5); %D
\node at ( 28.5,-1.4) {$\begin{smallmatrix}{D_{1}}\end{smallmatrix} $};
\draw[to-] (27.5,0.5) to[out=45, in=135] (29.7,0.5); %C
\node at ( 28.5,1.4) {$\begin{smallmatrix}{C_{1}}\end{smallmatrix} $};
\node at (30,0){$0$};
\draw[-to] (16.75,0.55) to[out=110, in=70, looseness=8] (17.25,0.55);
\node at (17,2.1) {$\begin{smallmatrix} B_1^- \end{smallmatrix}$};
\draw[-to] (19.75,0.55) to[out=110, in=70, looseness=8] (20.25,0.55);
\node at (20,2.1) {$\begin{smallmatrix} B_1^+,\,B_2^- \end{smallmatrix}$};
\draw[-to] (26.75,0.55) to[out=110, in=70, looseness=8] (27.25,0.55);
\node at (27,2.1) {$\begin{smallmatrix} B_n^+\end{smallmatrix}$};
\end{tikzpicture}
\]

Given $x=[(A_{i},B_{i}^+,B_{i}^-,a_{i},b_{i})_i;(C_{j},D_{j})_j]\in\mathcal C(\tilde{\mathcal D}(d_1,\ldots,d_m;n))$. By~\cite[Lemma~2.18]{takayama2016nahm}, the linear operators $A_1,\ldots,A_n$ are isomorphisms. Hence, we can define the operators
\[
a\colon \mathbb C^n\xrightarrow{\phantom{xxx}} \mathbb C^{d_1'},\quad b\colon \mathbb C^{d_1'}\xrightarrow{\phantom{xxx}} \mathbb C^{n}
\]
via the matrices
\[
a=\begin{pmatrix} a_1 & A_1a_2 & \dots & A_1\cdots A_{n-1} a_n\end{pmatrix},\quad b=\begin{pmatrix} b_1 A_1^{-1} \\
b_2A_{2}^{-1} A_1^{-1} \\
\vdots\\
b_{n-1} A_{n-1}^{-1}\ldots A_1^{-1}\\
b_nA_n^{-1}\dots A_1^{-1} \end{pmatrix}.
\]
The stability criterion~\cite[Proposition~2.8]{nakajima2017cherkis} implies that
\[
\mathcal F_x\coloneqq (\{0\}\subset \mathrm{ker}(a)  \subset \mathrm{ker}(C_1a) \subset \ldots \subset \mathrm{ker}(C_1\dots C_m a) \subset \mathbb C^n) 
\]
defines a point on $F(d_1,\ldots,d_m;n)$ which is independent of choice of representative for $x$. By construction, the linear operator $f_x\in \operatorname{End}(\mathbb C^n), f_x\coloneqq ba$ is also independent of the choice of an representative for $x$. The isomorphism 
\begin{equation}\label{eq:DefinitionIsoFlagBow}
\Psi\colon \mathcal C(\tilde{\mathcal D}(d_1,\ldots,d_m;n))\xrightarrow{\phantom{x}\sim\phantom{x}} T^\ast F(d_1,\ldots,d_m;n)
\end{equation}
from \cite[Theorem~2.15]{nakajima2017cherkis} is then given by \[
\Psi(x)=(\mathcal F_x,f_x), \quad\textup{for all $x\in \mathcal C(\tilde{\mathcal D}(d_1,\ldots,d_m;n))$.}
\]

\subsection{Torus action and tautological bundles} The $\mathbb T=(\mathbb A\times \mathbb C_h^\ast)$-action on $\mathcal C(\mathcal D)$ from Subsection~\ref{subsection:TorusAction} induces the following $\mathbb T$-action on $T^\ast F(d_1,\ldots,d_m;n)$:
\[
t.(\mathcal F,f)= (d(t)(\mathcal F), d(t) f d(t)^{-1} ),\quad h.(\mathcal F,f)=(\mathcal F,hf),
\]
where $t=(t_1,\ldots,t_n)\in \mathbb A$, $(\mathcal F,f)\in T^\ast F(d_1,\ldots,d_m;n)$ and $d(t)$ is the diagonal operator such that $d(t)(e_i)=t_ie_i$ for $i=1,\ldots n$, where $e_1,\ldots,e_n$ denote the standard basis vectors of $\mathbb C^n$.

For $i\in\{1,\ldots,m\}$, let  
\[
\mathcal S_i=\{((\{0\}\subset E_1 \subset E_2 \subset \dots \subset E_m\subset \mathbb C^n),v)\mid v\in E_i  \}\subset F(d_1,\ldots,d_m;n)\times\mathbb C^n
\] 
be the corresponding tautological bundle and $\mathcal Q_i=( F(d_1,\ldots,d_m;n)\times \mathbb C^n)/\mathcal S_i$ the quotient bundle.
By abuse of language, we also denote the pullbacks of $\mathcal S_i$ and $\mathcal Q_i$ to $T^\ast F(d_1,\ldots,d_m;n)$ by $\mathcal S_i$ and $\mathcal Q_i$. Both, $\mathcal S_i$ and $\mathcal Q_i$ are $\mathbb T$-equivariant vector bundles over  $T^\ast F(d_1,\ldots,d_m;n)$, where the $\mathbb T$-action is induced by the $\mathbb T$-action on $T^\ast F(d_1,\ldots,d_m;n)\times\mathbb C^n$ where $\mathbb A$ acts on $\mathbb C^n$ via the standard action and $\mathbb C_h^\ast$ acts trivially on $\mathbb C^n$.
 
From the construction of $\Psi$, it follows that $\Psi^{\ast}\mathcal Q_i=h^{i-1} \xi_{m-i+1}$, so up to a scaling factor the tautological bundles on the bow variety $\mathcal C(\tilde{\mathcal D}(d_1,\ldots,d_m;n))$ correspond to the quotient bundles on $T^\ast F(d_1,\ldots,d_m;n)$.

\subsection{Fixed point matching} In this subsection, we describe the bijection of $\mathbb T$-fixed points 
\begin{equation}\label{eq:FPMatchingCotangentVSBow}
T^\ast F(d_1,\ldots,d_m;n)^{\mathbb T}\xleftrightarrow{\phantom{x}1:1\phantom{x}} \mathcal C(\tilde{\mathcal D}(d_1,\ldots,d_m;n))^{\mathbb T}.
\end{equation}
that is induced by the isomorphism $\Psi$ from \eqref{eq:DefinitionIsoFlagBow}. 

Let $S_n$ the symmetric group on $n$ letters. We usually denote permutations $w\in S_n$ in one line notation $w=w(1)w(2)\ldots w(n)$.

For a permutation $w\in S_n$, we define the flag
\[
\mathcal F_w\coloneqq (\{0\} \subset \langle e_{w(1)},\ldots,e_{w(d_1)} \rangle \subset \dots \subset \langle e_{w(1)},\ldots,e_{w(d_m)} \rangle \subset \mathbb C^n).
\]
It is well-known that we have a bijection
\[
S_n/S_{\bm\delta} \xrightarrow{\phantom{x}\sim\phantom{x}} (T^\ast F(d_1,\ldots,d_m;n))^{\mathbb T},\quad wS_{\bm\delta}\mapsto (\mathcal F_w,0).
\]
where $S_{\bm \delta}=S_{\delta_1}\times\cdots\times S_{\delta_{m+1}}\subset S_n$ is the Young subgroup corresponding to $\bm \delta=(\delta_1,\ldots,\delta_{m+1})$.

Following the construction of $\Psi$ reveals that the preimage $\Psi^{-1}(\mathcal F_w,0)$ corresponds to the diagram is represented by 
\[
[(A_{w,i},B_{w,i}^+,B_{w,i}^-,a_{w,i},b_{w,i})_i;(C_{w,j},D_{w,j})_j],
\]
where $A_{w,i}=\operatorname{id}_{\mathbb C^{d_1'}}$ for all $i$ and
\begin{align*}
a_{w,i}&\colon\mathbb C\rightarrow \mathbb C^{d_1'},\quad 1\mapsto 0,\quad \textup{for $i\in\{w(1),\ldots,w(d_1)\}$,}\\
a_{w,i}&\colon \mathbb C\rightarrow \mathbb C^{d_1'},\quad 1\mapsto e_{w^{-1}(i)-d_1},\quad \textup{for $i\in\{w(d_1+1),\ldots,w(n)\}$,}\\
C_{w,j}&\colon \mathbb C^{d_{j-1}'} \rightarrow \mathbb C^{d_{j}'} \quad
C_{w,j}(e_i)=\begin{cases}
0 &\textup{if $i=1,\ldots,d_j$,}\\
e_{i-d_j'} &\textup{if $i=d_j+1,\ldots,d_j'$,}
\end{cases}
\quad
\textup{for $j=2,\ldots,m+1$.}
\end{align*}
The remaining operators vanish.

Following the explicit construction of $\mathbb T$-fixed points of bow varieties from \cite[Section~4]{rimanyi2020bow} gives $\Psi^{-1}(\mathcal F_w,0)=x_{\tilde D_w}$, where $\tilde D_w= \tilde D_w'\cup \tilde D_w''$. Here, $\tilde D_w'$ is the set of all pairs $(V_{i},U_j)$ with $i\in\{2,3,\ldots,m\},j\in\{1,\ldots, n\}$ and there exists $l\in\{d_{i-1}+1,\ldots,d_{i}\}$ such that $w(l)=j$. The set $\tilde D_w''$ is defined as the set of all pairs $(U_j,V_{1})$ with $ j\in\{1,\ldots, n\}$ and there exists $l\in \{d_1+1,\ldots,n\}$ such that $w(l)=j$.

A straight-forward check shows that $\tilde D_w$ only depends on the coset $wS_{\bm\delta}$. Hence, we denote $\tilde D_w$ also by $\tilde D_{wS_{\bm \delta}}$. Thus,\eqref{eq:FPMatchingCotangentVSBow} corresponds to the following combinatorial bijection:
\begin{equation}\label{eq:FPMatchingCotangentVSBowCombinatoric}
S_n/S_{\bm \delta}\xrightarrow{\phantom{x}\sim\phantom{x}} \mathrm{Tie}(\tilde{\mathcal D}(d_1,\ldots,d_m;n)),\quad 
wS_{\bm \delta}\mapsto D_{wS_{\bm \delta}}.
\end{equation}

\begin{example}\label{example:FixedPointMatching1} Let $m=3,d_1=2,d_2=4,d_3=5$ and $n=6$. Then, the brane diagram $\tilde{\mathcal D}(2,4,5;6)$ equals $
0
\textcolor{red}{\slash}
1
\textcolor{red}{\slash}
2
\textcolor{red}{\slash}
4
\textcolor{blue}{\backslash}
4
\textcolor{blue}{\backslash}
4
\textcolor{blue}{\backslash}
4
\textcolor{blue}{\backslash}
4
\textcolor{blue}{\backslash}
4
\textcolor{blue}{\backslash}
4
\textcolor{red}{\slash}
0
$
. Let $w\in S_6$ be the permutation $w=253614$. To construct $\tilde D_w$, note that since $d_1=2$, we have that there are no ties in $\tilde D_w$ which are connected to $U_2$ and $U_5$. As $d_2=4$, the blue lines $U_3$, $U_6$ are both connected to $V_1$ and $V_2$. Likewise, since $d_3=5$, the blue line $U_1$ is connected to $V_1$ and $V_3$. Finally, $n=6$ implies that there are ties between $U_4$ and $V_1$, $V_4$. Hence,  $\tilde D_w$ is illustrated as follows:
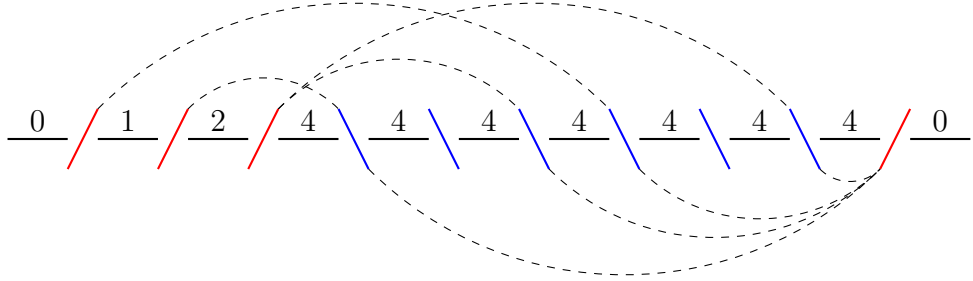
\[
\begin{tikzpicture}
[scale=.4]
\draw[thick] (0,0)--(2,0);
\draw[thick] (3,0)--(5,0);
\draw[thick] (6,0)--(8,0);
\draw[thick] (9,0)--(11,0);
\draw[thick] (12,0)--(14,0);
\draw[thick] (15,0)--(17,0);
\draw[thick] (18,0)--(20,0);
\draw[thick] (21,0)--(23,0);
\draw[thick] (24,0)--(26,0);
\draw[thick] (27,0)--(29,0);
\draw[thick] (30,0)--(32,0);

\draw [thick,red] (2,-1) --(3,1); 
\draw [thick,red] (5,-1) --(6,1); 
\draw [thick,red] (8,-1) --(9,1);
\draw [thick,red] (29,-1) --(30,1);

\draw [thick,blue] (12,-1) --(11,1); 
\draw [thick,blue] (15,-1) --(14,1); 
\draw [thick,blue] (18,-1) --(17,1); 
\draw [thick,blue] (21,-1) --(20,1); 
\draw [thick,blue] (24,-1) --(23,1); 
\draw [thick,blue] (27,-1) --(26,1);

%numbers
\node at (1,0.6){$0$};
\node at (4,0.6){$1$};
\node at (7,0.6){$2$};
\node at (10,0.6){$4$};
\node at (13,0.6){$4$};
\node at (16,0.6){$4$};
\node at (19,0.6){$4$};
\node at (22,0.6){$4$};
\node at (25,0.6){$4$};
\node at (28,0.6){$4$};
\node at (31,0.6){$0$};

\draw[dashed] (3,1) to[out=45,in=135] (20,1);
\draw[dashed] (6,1) to[out=45,in=135] (11,1);
\draw[dashed] (9,1) to[out=45,in=135] (17,1);
\draw[dashed] (9,1) to[out=45,in=135] (26,1);

\draw[dashed] (12,-1) to[out=315,in=225] (29,-1);
\draw[dashed] (18,-1) to[out=315,in=225] (29,-1);
\draw[dashed] (21,-1) to[out=315,in=225] (29,-1);
\draw[dashed] (27,-1) to[out=315,in=225] (29,-1);
\end{tikzpicture}
\]
\end{example}

\subsection{Transition to separated brane diagram}\label{subsection:DepBraneDiagPartialFlag}
In the later course of this article, we will mostly work with bow varieties corresponding to separated brane diagrams. We define the brane diagram $\mathcal D(d_1,\ldots,d_m;n)$ as follows:
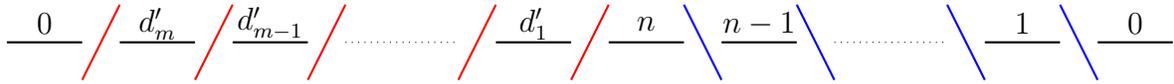
\[
\begin{tikzpicture}
[scale=.5]
\draw[thick] (0,0)--(2,0);
\draw[thick] (3,0)--(5,0);
\draw[thick] (6,0)--(8,0);
\draw[dotted] (9,0)--(12,0);
\draw[thick] (13,0)--(15,0);
\draw[thick] (16,0)--(18,0);
\draw[thick] (19,0)--(21,0);
\draw[dotted] (22,0)--(25,0);
\draw[thick] (26,0)--(28,0);
\draw[thick] (29,0)--(31,0);

\draw [thick,red] (2,-1) --(3,1); 
\draw [thick,red] (5,-1) --(6,1); 
\draw [thick,red] (8,-1) --(9,1); 
\draw [thick,red] (12,-1) --(13,1);
\draw [thick,red] (15,-1) --(16,1);

%\node[red] at (2,-1.5){$V_{m+1}$};
%\node[red] at (5,-1.5){$V_m$};
%\node[red] at (8,-1.5){$V_{m-1}$};
%\node[red] at (12,-1.5){$V_{2}$};
%\node[red] at (15,-1.5){$V_{1}$};
%
%\node[blue] at (19,-1.5){$U_{1}$};
%\node[blue] at (22,-1.5){$U_{2}$};
%\node[blue] at (26,-1.5){$U_{n-1}$};
%\node[blue] at (29,-1.5){$U_{n}$};

%blue branes 
\draw [thick,blue] (19,-1) --(18,1); 
\draw [thick,blue] (22,-1) --(21,1);  
\draw [thick,blue] (26,-1) --(25,1); 
\draw [thick,blue] (29,-1) --(28,1);  

%numbers
\node at (1,0.5){$0$};
\node at (4,0.5){$d_m'$};
\node at (7,0.5){$d_{m-1}'$};
\node at (14,0.5){$d_{1}'$};
\node at (17,0.5){$n$};
\node at (20,0.5){$n-1$};
\node at (27,0.5){$1$};
\node at (30,0.5){$0$};
\end{tikzpicture}
\]
Note that $\mathcal D(d_1,\ldots,d_m;n)$ is obtained from $\tilde{\mathcal D}(d_1,\ldots,d_m;n)$ via Hanany--Witten transitions by moving $V_{1}$ to the left of $U_1,\ldots,U_n$. Let 
\begin{equation}\label{eq:DefinitionHWIsoFlagSeparated}
\Phi\colon \mathcal C(\mathcal D(d_1,\ldots,d_m;n))\xrightarrow{\phantom{x}\sim\phantom{x}} \mathcal C(\tilde{\mathcal D}(d_1,\ldots,d_m;n))
\end{equation}
be the corresponding Hanany--Witten isomorphism. Then, by Proposition~\ref{prop:HWFixedPointMatching}, the $\mathbb T$-fixed point matching under $\Phi$ can be characterized as follows: for $w\in S_n$, let $D_w$ be the tie diagram attached to $\mathcal D(d_1,\ldots,d_m;n)$ consisting of all pairs $(V_i,U_j)$ with $i\in{1,\ldots,m+1}$, $j\in\{1,\ldots,n\}$ and there exists $l\in\{d_{i-1}+1,\ldots,d_{i}\}$ such that $w(l)=j$. Again, the $D_w$ only depends on the coset $wS_{\bm\delta}$. Employing Proposition~\ref{prop:HWFixedPointMatching} gives that for all $w\in S_n$, $\Phi$ maps $x_{\tilde D_{wS_{\bm \delta}}}$ to $x_{D_{wS_{\bm \delta}}}$. Thus, by~\eqref{eq:FPMatchingCotangentVSBowCombinatoric}, we have a bijection
\[
S_n/S_{\bm \delta} \xrightarrow{\phantom{x}\sim\phantom{x}} \mathrm{Tie}(\mathcal D(d_1,\ldots,d_m;n)),
\quad
wS_{\bm \delta}\mapsto D_{wS_{\bm \delta}}.
\]

\begin{example} As in Example~\ref{example:FixedPointMatching1}, we choose $m=3,d_1=2,d_2=4,d_3=5$ and $n=6$. Thus, $\mathcal D(2,4,5;6)$ is given by $0\textcolor{red}{\slash}1\textcolor{red}{\slash}2\textcolor{red}{\slash}4\textcolor{red}{\slash}6\textcolor{blue}{\backslash}5\textcolor{blue}{\backslash}4\textcolor{blue}{\backslash}3\textcolor{blue}{\backslash}2\textcolor{blue}{\backslash}1\textcolor{blue}{\backslash}0$. Again, choose $w=253614$. For the construction of $D_w$, note that since $d_1=2$, the blue lines $U_2$ and $U_5$ are connected in $D_w$ to $V_1$. As $d_2=4$, the blue lines $U_3$ and $U_6$ are connected to $V_2$. Likewise, $d_3=5$ gives that there is a tie between $U_1$ and $V_3$. Finally, $n=6$ implies that $U_4$ is connected to $V_4$. Therefore, we can illustrate $D_w$ as follows:
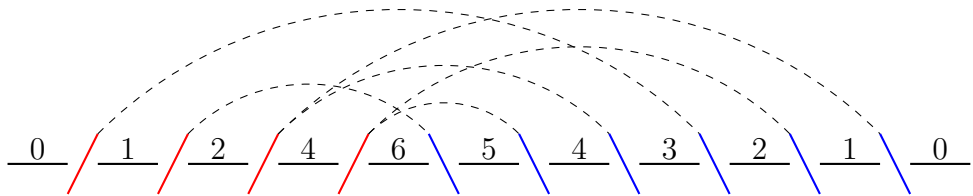
\[
\begin{tikzpicture}
[scale=.4]
\draw[thick] (0,0)--(2,0);
\draw[thick] (3,0)--(5,0);
\draw[thick] (6,0)--(8,0);
\draw[thick] (9,0)--(11,0);
\draw[thick] (12,0)--(14,0);
\draw[thick] (15,0)--(17,0);
\draw[thick] (18,0)--(20,0);
\draw[thick] (21,0)--(23,0);
\draw[thick] (24,0)--(26,0);
\draw[thick] (27,0)--(29,0);
\draw[thick] (30,0)--(32,0);

\draw [thick,red] (2,-1) --(3,1); 
\draw [thick,red] (5,-1) --(6,1); 
\draw [thick,red] (8,-1) --(9,1);
\draw [thick,red] (11,-1) --(12,1);

\draw [thick,blue] (15,-1) --(14,1); 
\draw [thick,blue] (18,-1) --(17,1); 
\draw [thick,blue] (21,-1) --(20,1); 
\draw [thick,blue] (24,-1) --(23,1); 
\draw [thick,blue] (27,-1) --(26,1);
\draw [thick,blue] (30,-1) --(29,1);

%numbers
\node at (1,0.5){$0$};
\node at (4,0.5){$1$};
\node at (7,0.5){$2$};
\node at (10,0.5){$4$};
\node at (13,0.5){$6$};
\node at (16,0.5){$5$};
\node at (19,0.5){$4$};
\node at (22,0.5){$3$};
\node at (25,0.5){$2$};
\node at (28,0.5){$1$};
\node at (31,0.5){$0$};

\draw[dashed] (3,1) to[out=45,in=135] (23,1);
\draw[dashed] (6,1) to[out=45,in=135] (14,1);
\draw[dashed] (9,1) to[out=45,in=135] (20,1);
\draw[dashed] (9,1) to[out=45,in=135] (29,1);
\draw[dashed] (12,1) to[out=45,in=135] (26,1);
\draw[dashed] (12,1) to[out=45,in=135] (17,1);
\end{tikzpicture}
\]
\end{example}
\section{Stable envelopes}\label{section:StableEnvelopes}

Stable envelopes are families of equivariant cohomology classes introduced by Maulik and Okounkov in~\cite{maulik2019quantum}. They exist for a large class of symplectic varieties with torus action including Nakajima quiver varieties and more generally bow varieties. Their definition involves stability conditions which are similar to the stability conditions of equivariant Schubert classes, see e.g.~\cite{knutson2003puzzles}, \cite{gorbounov2020yang}.

In this section, we recall the definition of stable envelopes in the framework of bow varieties and some of their properties. 

\subsection{Attracting cells}\label{subsection:AttractionCells}

The theory of stable envelopes is based on the theory of attracting cells. We briefly recall important features of this theory.

\begin{notation}
Let $V$ be a finite dimensional representation of $\mathbb C^\ast$ and \[
V_a=\{v\in V\mid t.v=t^av\textup{ for all }t\in\mathbb C^\ast\},\quad\textup{for $a\in\mathbb Z$.} 
\]
We set $V^+\coloneqq \bigoplus_{a\geq 1}V_a$ and $V^-\coloneqq\bigoplus_{a\leq -1}V_a$. 
\end{notation}

As in Subsection~\ref{subsection:TorusFixedPoints}, let $\sigma$ be a generic cocharacter of $\mathbb A$.
The \textit{attracting cell (with respect to $\sigma$)} is defined as
\[
\operatorname{Attr}_\sigma(p)=\{x\in\mathcal C(\mathcal D)\mid\lim_{t\to0}\sigma(t).x=p\}.
\]
Using the classical Bia\l{}nicky-Birula theorem, one can show that $\operatorname{Attr}_{\sigma}(p)$ is an affine and locally closed $\mathbb T$-invariant subvariety of $\mathcal C(\mathcal D)$ which is $\mathbb T$-equivariantly isomorphic to $T_p\mathcal C(\mathcal D)_\sigma^+$, see \cite[Lemma~3.2.4]{maulik2019quantum} and also the exposition in \cite[Proposition~5.1]{expository2023orthogonality}. Here, $T_p\mathcal C(\mathcal D)_\sigma$ is the $\mathbb C^\ast$-representation obtained from the $\mathbb A$-representation $T_p\mathcal C(\mathcal D)$ via $\sigma$.

Attracting cells admit the following independence property: let $\Lambda$ be the cocharacter lattice of $\mathbb A$ and set $\Lambda_{\mathbb R}\coloneqq \Lambda\otimes_{\mathbb Z}\mathbb R$. We have the usual root hyperplanes of $\mathfrak{gl}_N(\mathbb C)$ in $\Lambda_{\mathbb R}$:
	\[
	H_{i,j}\coloneqq \{(t_1,\ldots,t_N)\mid t_i=t_j\} \subset \Lambda_{\mathbb R},\quad \textup{for $1\le i,j\le N$ with $i\ne j$.}
	\]
	The \textit{chambers} of $\mathfrak{gl}_N(\mathbb C)$ are the connected components of 
	\[
	\Lambda_{\mathbb R} \setminus\bigg( \bigcup_{\substack{1\le i,j\le N\\ i\ne j}} H_{i,j}\bigg)
	\]
	We have the \textit{dominant chamber} and the \textit{antidominant chamber}
	\[
	\mathfrak C_+= \{(t_1,\ldots,t_N)\mid t_1>t_2>\dots>t_N\},\quad
	\mathfrak C_-= \{(t_1,\ldots,t_N)\mid t_1<t_2<\dots<t_N\}.
	\]
	
	As usual, reflecting along the root hyperplanes gives an action of the symmetric group $S_N$ on $\Lambda_{\mathbb R}$. It coincides with the action which permutes the coordinates and provides a well-known bijection from Lie theory
	\[
	\{\textup{Chambers}\}\xleftrightarrow{\phantom{x}1:1\phantom{x}} S_N,
	\]
	 where we assign to a permutation $w\in S_N$ the chamber $w.\mathfrak C_+$.
	 
     Attracting cells are constant on chambers, that is if $\sigma$, $\tau$ lie in the same chamber $\mathfrak C$, we have $\operatorname{Attr}_\sigma(p)=\operatorname{Attr}_\tau(p)$, see e.g.~\cite[Proposition~5.4]{expository2023orthogonality}. Hence, we also just write $\operatorname{Attr}_{\mathfrak C}(p)$. In addition, $T_p\mathcal C(\mathcal D)_\sigma^\pm$ also just depends on $\mathfrak C$. Therefore, we also just write $T_p\mathcal C(\mathcal D)_{\mathfrak C}^\pm$.

There is a partial order $\preceq=\preceq_{\mathfrak C}$ on $\mathcal C(\mathcal D)^{\mathbb T}$ defined as the transitive closure of the relation
\[
q\in \overline{\operatorname{Attr}_{\mathfrak C}(p)} \Rightarrow q\preceq p,
\]
where $ \overline{\operatorname{Attr}_{\mathfrak C}(p)}$ denotes the Zariski closure of $\operatorname{Attr}_{\mathfrak C}(p)$ in $\mathcal C(\mathcal D)$. For any $p\in\mathcal C(\mathcal D)^{\mathbb T}$, we set
\[
\operatorname{Attr}_{\mathfrak C}^f(p)\coloneqq \bigsqcup_{q\preceq p}\operatorname{Attr}_{\mathfrak C}(q).
\]
As in the case of Schubert varieties, one can show that $\operatorname{Attr}_{\mathfrak C}^f(p)$ is a closed subvariety of $\mathcal C(\mathcal D)$ which is called the \textit{full attracting cell} of $p$, see \cite[Lemma~3.2.7]{maulik2019quantum} or the exposition in \cite[Proposition~5.8]{expository2023orthogonality}.

The \textit{opposite chamber} of $\mathfrak C$ is defined as
	\[
		\mathfrak C^{\operatorname{op}}\coloneqq \{a\in \mathfrak a_{\mathbb R} \mid -a \in \mathfrak C\}.
	\]
It is a general that fact that $\preceq_{\mathfrak C^{\operatorname{op}}}$ is the opposite order of $\preceq_{\mathfrak C}$, see e.g.~\cite[Proposition~5.10]{expository2023orthogonality}.

%The partial order corresponding to $\mathfrak C_+$ is denoted by $\preceq_+$ and the partial order corresponding to $\mathfrak C_-$ by $\preceq_-$. Its a general fact that $\preceq_+$ is the opposite order of $\preceq_-$.

\subsection{Definition of stable envelopes}

Let $d=\operatorname{dim}(\mathcal C(\mathcal D))$ be the dimension of $\mathcal C(\mathcal D)$ as complex variety. 

\textit{Stable envelopes} are maps 
	\[
	\mathcal C(\mathcal D)^{\mathbb T} \xrightarrow{\mathrm{Stab}_{\mathfrak C}}  H_{\mathbb T}^{d}(\mathcal C(\mathcal D)),
	\]
	depending on a choice of chamber $\mathfrak C$ which are uniquely characterized by the following stability conditions, see \cite[Theorem~3.3.4]{maulik2019quantum}:
	
	\begin{theorem}\label{thm:existenceStableEnvelopes} 
	There exist a unique family $(\mathrm{Stab}_{\mathfrak C}(p))_{p\in\mathcal C(\mathcal D)^{\mathbb T}}$ in $H_{\mathbb T}^{d}(\mathcal C(\mathcal D))$ satisfying the following conditions:
	\begin{enumerate}[label=(Stab\arabic*), leftmargin=1.75cm]
		\item\label{item:Normalization} We have $\iota_p^\ast(\mathrm{Stab}_{\mathfrak C}(p))=e_{\mathbb T}(T_p\mathcal C(\mathcal D)_{\mathfrak C}^-)$, for all $p\in\mathcal C(\mathcal D)^{\mathbb T}$.
		\item\label{item:Support} We have that $\mathrm{Stab}_{\mathfrak C}(p)$ is supported on $\operatorname{Attr}_{\mathfrak C}^f(p)$, for all $p\in\mathcal C(\mathcal D)^{\mathbb T}$.
		\item\label{item:Smallness} Given $p$, $q\in\mathcal C(\mathcal D)^{\mathbb T}$ with $q\prec p$. Then, $\iota_q^\ast(\mathrm{Stab}_{\mathfrak C}(p))$ is divisible by $h$.
	\end{enumerate}
	Here, $e_{\mathbb T}$ denotes the $\mathbb T$-equivariant Euler class.
	\end{theorem}
	
	The condition \ref{item:Normalization} is called \textit{normalization condition}, \ref{item:Support} is called \textit{support condition} and \ref{item:Smallness} is called \textit{smallness condition}.
	
	By the normalization and support condition, we obtain that $(\mathrm{Stab}_{\mathfrak C}(p))_{p\in\mathcal C(\mathcal D)^{\mathbb T}}$ is a basis of the localized equivariant cohomology ring $H_{\mathbb T}^\ast(\mathcal C(\mathcal D))_{\mathrm{loc}}$.
	The basis $(\mathrm{Stab}_{\mathfrak C}(p))_{p\in\mathcal C(\mathcal D)^{\mathbb T}}$ is called the \textit{stable basis (corresponding to $\mathfrak C$)}. We refer to the equivariant cohomology classes $\mathrm{Stab}_{\mathfrak C}(p)\in H^d_{\mathbb T}(\mathcal C(\mathcal D))$ as \textit{stable basis elements}.
	
	\begin{remark} In \cite{maulik2019quantum} the definition of stable envelopes slightly differs from our definition given in Theorem~\ref{thm:existenceStableEnvelopes}. Their version of the normalization condition \ref{item:Normalization} involves a certain choice of signs depending on a choice of polarization bundle of the bow variety. It was shown in \cite[Section~4]{shou2021bow} that polarization bundles always exist for bow varieties. For simplicity, we chose all signs in the normalization condition \ref{item:Normalization} to be $+1$.
	\end{remark}
	
	In the next subsections, we state convenient properties of stable bases.

\subsection{Orthogonality} As shown in \cite[Theorem~4.4.1]{maulik2019quantum} (see also \cite[Theorem~8.1]{expository2023orthogonality}), stable bases satisfy the following useful 
orthogonality property which is analogous to the orthogonality properties of Schubert bases:

\begin{theorem}[Orthogonality]\label{thm:OrthogonalityOfStableEnvelopes}
For all $p$, $q\in\mathcal C(\mathcal D)^{\mathbb T}$ we have
\[
(\mathrm{Stab}_{\mathfrak C}(p),\mathrm{Stab}_{\mathfrak C^{\mathrm{op}}}(q))_{\mathrm{virt}}=\begin{cases}
1&\textit{if $p=q$,}\\
0&\textit{if $p\ne q$,} 
\end{cases}
\]
where 
\[
(.,.)_{\mathrm{virt}}:H_{\mathbb T}^\ast(\mathcal C(\mathcal D))\times 
H_{\mathbb T}^\ast(\mathcal C(\mathcal D))\rightarrow S^{-1}H_{\mathbb T}^\ast(\operatorname{pt}),\quad 
(\alpha,\beta)_{\mathrm{virt}}=\sum_{p\in \mathcal C(\mathcal D)^{\mathbb T}}\frac{\iota_p^\ast(\alpha\cup \beta)}{e_{\mathbb T}(T_p\mathcal C(\mathcal D))}
\]
is the virtual intersection pairing with
\begin{equation}\label{eq:DefinitionS}
S=\{t_i-t_j+mh\mid i,j\in\{1,\ldots,N\},m\in\mathbb Z\}.
\end{equation}
\end{theorem}

\begin{remark} It is well-known that the tangent weights of fixed points of $\mathcal C(\mathcal D)$ are contained in $S$, see e.g.~\cite[Corollary~4.15]{expository2023orthogonality}. Thus, the virtual intersection form takes indeed values in $S^{-1}H_{\mathbb T}^\ast(\operatorname{pt})$.
\end{remark}

The following practical property of the stable basis on matrix coefficients of multiplication operators of equivariant cohomology classes with respect to the stable basis
can be found in \cite[Proposition~8.2]{expository2023orthogonality}:

\begin{prop}[Polynomiality]\label{prop:MatrixCoefficients}
For all $\gamma\in H_{\mathbb T}^\ast(\mathcal C(\mathcal D))$ and all  $p$, $q\in\mathcal C(\mathcal D)^{\mathbb T}$ we have
\[
(\gamma\cup\mathrm{Stab}_{\mathfrak C}(p),\mathrm{Stab}_{\mathfrak C^{\mathrm{op}}}(q))_{\mathrm{virt}}\in H_{\mathbb T}^{\ast}(\operatorname{pt}).
\]
\end{prop}

\subsection{Matching under Hanany--Witten transition}

Suppose $\tilde{\mathcal D}$ is obtained from $\mathcal D$ via Hanany--Witten transition and let $\Phi\colon\mathcal C(\mathcal D)\xrightarrow\sim\mathcal C(\tilde{\mathcal D})$ be the corresponding Hanany--Witten isomorphism and $\phi\colon \mathrm{Tie}(\mathcal D)\xrightarrow\sim \mathrm{Tie}(\tilde{\mathcal D})$ be the induced bijection.

Stable bases are compatible with Hanany--Witten transition:

\begin{prop}\label{prop:MatchingStableEnvelopesHW} Let $\Phi^\ast\colon H_{\mathbb T}^{\ast}(\mathcal C(\tilde{\mathcal D}))\xrightarrow\sim H_{\mathbb T}^{\ast}(\mathcal C(\mathcal D))$ be the induced isomorphism. Then, we have for all $D\in\mathrm{Tie}(\mathcal D)$
\[
\Phi^\ast(\mathrm{Stab}_{\mathfrak C}(\phi(D)))=\mathrm{Stab}_{\mathfrak C}(D).
\]
\end{prop}
\begin{proof}
The smallness condition is immediate from $\Phi^\ast(h)=h$. As $\Phi$ is $\mathbb A$-equivariant, we have
\begin{equation}\label{eq:ProofOFHWStableEnvelopeMatching}
\Phi^{-1}(\operatorname{Attr}_{\mathfrak C}(\phi(D)))=\operatorname{Attr}_{\mathfrak C}(D),\quad\textup{for all $D\in\mathrm{Tie}(\mathcal D)$}
\end{equation}
which implies the support condition. In addition,~\eqref{eq:ProofOFHWStableEnvelopeMatching} gives
\[
e_{\mathbb T}(T_D\mathcal C(\mathcal D)_{\mathfrak C}^{-})
= 
\iota_D^\ast(\overline{\operatorname{Attr}}_{\mathfrak C}(D)
=\iota_D^\ast (\Phi^\ast(\overline{\operatorname{Attr}}_{\mathfrak C}(\phi(D)))
\]
which proves the normalization condition.
\end{proof}

\section{Stable bases in the separated case}\label{section:StabSpearatedCase}
In this section, we recall some convenient properties of stable basis elements of bow varieties corresponding to separated brane diagrams from~\cite{botta2023mirror}. In this reference, the authors work in the framework of elliptic cohomology. As explained in e.g.~\cite{wehrhanphd}, the same results also hold in torus equivariant cohomology.

\begin{assumption}
For this section, we assume that $\mathcal D$ is a fixed separated brane diagram.
\end{assumption}

\subsection{Forgetting chargeless lines}\label{subsection:ForgettingChargelessLines}

We call a colored line $Y$ in $\mathcal D$ \textit{chargeless} if $d_{Y^+}=d_{Y^-}$. That is, the horizontal lines to the left and to the right of $Y$ are labeled by the same number. If $Y$ is not chargeless then $Y$ is called \textit{essential}. If all colored lines of $\mathcal D$ are essential we also call $\mathcal D$ \textit{essential}. 

Let $\mathrm{ess}(\mathcal D)$ be the brane diagram obtained from $\mathcal D$ by forgetting all chargeless lines. That is for all chargeless lines in $\mathcal D$ we perform the local moves:
\[
\begin{tikzpicture}[scale=.4]
\draw[thick] (0,0)--(2,0);
\draw[thick] (3,0)--(5,0);

\draw[thick,red] (3,1)--(2,-1);

\node at (1,0.6) {$d$};
\node at (4,0.6) {$d$};

\draw[-to] decorate [decoration=zigzag] {(6,0) -- (8,0)};
\draw[thick] (9,0)--(11,0);
\node at (10,0.6) {$d$};

\node at (15,0) {\textup{respectively}};

\def \s{5};

\draw[thick] (14+\s,0)--(16+\s,0);
\draw[thick] (17+\s,0)--(19+\s,0);

\draw[thick,blue] (17+\s,-1)--(16+\s,1);

\node at (15+\s,0.6) {$d$};
\node at (18+\s,0.6) {$d$};

\draw[-to] decorate [decoration=zigzag] {(20+\s,0) -- (22+\s,0)};
\draw[thick] (23+\s,0)--(25+\s,0);
\node at (24+\s,0.6) {$d$};
\end{tikzpicture}
\]

For instance if $\mathcal D=0
\textcolor{red}{\slash}2
\textcolor{red}{\slash}2
\textcolor{red}{\slash}4
\textcolor{red}{\slash}5
\textcolor{blue}{\backslash}5
\textcolor{blue}{\backslash}4
\textcolor{blue}{\backslash}2
\textcolor{blue}{\backslash}0
$
then $\mathrm{ess}(\mathcal D)=0\textcolor{red}{\slash}2
\textcolor{red}{\slash}4
\textcolor{red}{\slash}5
\textcolor{blue}{\backslash}4
\textcolor{blue}{\backslash}2
\textcolor{blue}{\backslash}0$.

Note that in a tie diagram $D\in \mathrm{Tie}(\mathcal D)$ there is no tie which is connected to a chargeless line. Thus, forgetting the chargeless lines gives a bijection
\[
f\colon \mathrm{Tie}(\mathcal D)\xrightarrow{\phantom{x}\sim\phantom{x}} \mathrm{Tie}(\mathrm{ess}(\mathcal D)).
\]
We also view $\mathcal C(\mathrm{ess}(\mathcal D))$ as $\mathbb T$-variety, where the components of $\mathbb A$ corresponding to chargeless lines act trivially on $\mathcal C(\mathrm{ess}(\mathcal D))$.

The localization coefficients of stable basis elements of $\mathcal C(\mathcal D)$ and $\mathcal C(\mathrm{ess}(\mathcal D))$ are closely connected, see \cite[Section~5.10]{botta2023mirror} or \cite[Proposition~8.44]{wehrhanphd}:
%\cite[Proposition~7.28]{wehrhanphd}:

\begin{prop}\label{prop:EqMultiplicitiesEssentialBraneDiagram}
The following holds:
\begin{enumerate}[label=(\roman*)]
\item There is a $\mathbb T$-equivariant closed immersion $\iota\colon\mathcal C( \mathrm{ess}(\mathcal D))\hookrightarrow \mathcal C(\mathcal D)$.
\item The $\mathbb T$-equivariant cohomology class of the normal bundle of $\iota$ equals $[N_\iota]$, where
\[
N_\iota =\bigoplus_{U\in\mathrm{b}'(\mathcal D)}\Big(\operatorname{Hom}(\mathbb C_U,\xi_{U^-})\oplus h\operatorname{Hom}(\xi_{U^+},\mathbb C_U)\Big)
\]
and $\mathrm{b}'(\mathcal D)$ denotes the set of chargeless blue lines in $\mathcal D$.
\item We have
\[
\iota_{D'}^\ast \mathrm{Stab}_{\mathfrak C}(D)=e_{\mathbb T}(N_{\iota,\mathfrak C}^-) \iota_{f(D')}^\ast(\mathrm{Stab}_{\mathfrak C}(f(D))),
\]
for all $D$, $D'\in \mathrm{Tie}(\mathcal D')$, where $N_{\iota,\mathfrak C}^-$ is the negative part of $N_\iota$ with respect to $\mathfrak C$.
\end{enumerate}
\end{prop}

It follows from Corollary~\ref{cor:PropertiesTautologicalBundlesSeparated} that  $N_\iota$ is trivial with character
\begin{equation}\label{eq:NiotaCharacter}
N_\iota = \Big( \bigoplus_{U_j\in \mathrm{b}'(\mathcal D)} \bigoplus_{k=j+1}^{N} \bigoplus_{l=0}^{c_k-1} h^{-l} \mathbb C_{U_k}\otimes \mathbb C_{U_j}^\vee  \Big)
\oplus
\Big( \bigoplus_{U_j\in \mathrm{b}'(\mathcal D)} \bigoplus_{k=j+1}^{N} \bigoplus_{l=0}^{c_k} h^{l+1} \mathbb C_{U_j}\otimes \mathbb C_{U_k}^\vee  \Big),
\end{equation}
where the left summand corresponds to $\operatorname{Hom}(\mathbb C_U,\xi_{U^-})$ whereas the right summand corresponds to $h\operatorname{Hom}(\xi_{U^+},\mathbb C_U)$. In the above formula, $\mathbb C_{U_k}^\vee$ denotes the dual bundle of $\mathbb C_{U_k}$.

If we choose $\mathfrak C=\mathfrak C_-$ to be the antidominant chamber then \eqref{eq:NiotaCharacter} implies 
\[
N_{\iota,\mathfrak C_-}^+= \bigoplus_{U\in\mathrm{b}'(\mathcal D)}\operatorname{Hom}(\mathbb C_U,\xi_{U^-}),\quad N_{\iota,\mathfrak C_-}^- = \bigoplus_{U\in\mathrm{b}'(\mathcal D)} h\operatorname{Hom}(\xi_{U^+},\mathbb C_U).
\]
Thus, we have
\[
e_{\mathbb T}(N_{\iota,\mathfrak C_-}^+)= \prod_{U_j\in \mathrm{b}'(\mathcal D)} \prod_{k=j+1}^{N} \prod_{l=0}^{c_k-1}(t_k-t_j-lh),\quad
e_{\mathbb T}(N_{\iota, \mathfrak C_-}^-)= \prod_{U_j\in \mathrm{b}'(\mathcal D)} \prod_{k=j+1}^{N} \prod_{l=0}^{c_k-1}(t_j-t_k+(l+1)h)
\]
which determines the normalization factor in Proposition~\ref{prop:EqMultiplicitiesEssentialBraneDiagram} for the antidominant chamber.

\subsection{Compatibility with the symmetric group action}
In this subsection, we assume that $\mathcal D$ is essential as defined in the previous subsection. 

Given a tie diagram $D\in\mathrm{Tie}(\mathcal D)$ and two blue lines $U$, $U'\in\mathrm{b}(\mathcal D)$. Then, swapping the blue lines $U$, $U'$ with their connected ties gives a new tie diagram over a brane diagram that possibly differs from $\mathcal D$:
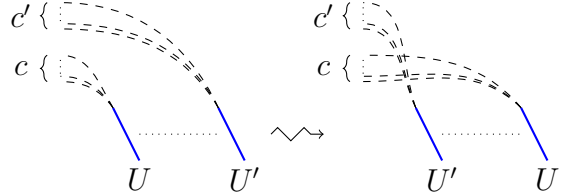
\begin{center}
\begin{tikzpicture}[scale=.35]
\draw[blue, thick] (1,-1)--(0,1);
\draw[blue, thick] (5,-1)--(4,1);
\draw[dotted] (1,0) -- (4,0);
\node at (1,-1.7) {$U$};
\node at (5,-1.7) {$U'$};

\draw[dashed] (0,1) to[out=125, in=0] (-2,2);
\draw[dashed] (0,1) to[out=123, in=0] (-2,2.2);
\draw[dotted] (-2,2.3) -- (-2,2.9);
\draw[dashed] (0,1) to[out=120, in=0] (-2,3);

\draw[decorate, decoration = {brace}] (-2.5,2) -- (-2.5,3);
\node at (-3.5,2.5) {$c$};

\draw[dashed] (4,1) to[out=120, in=0] (-2,4);
\draw[dashed] (4,1) to[out=117, in=0] (-2,4.2);
\draw[dotted] (-2,4.3) -- (-2,4.9);
\draw[dashed] (4,1) to[out=115, in=0] (-2,5);

\draw[decorate, decoration = {brace}] (-2.5,4) -- (-2.5,5);
\node at (-3.5,4.5) {$c'$};

\draw[-to] decorate [decoration=zigzag] {(6,0) -- (8,0)};

\def\s{11.5};

\draw[dashed] (4+\s,1) to[out=135, in=0] (-2+\s,2);
\draw[dashed] (4+\s,1) to[out=133, in=0] (-2+\s,2.2);
\draw[dotted] (-2+\s,2.3) -- (-2+\s,2.9);
\draw[dashed] (4+\s,1) to[out=130, in=0] (-2+\s,3);

\draw[blue, thick] (1+\s,-1)--(0+\s,1);
\draw[blue, thick] (5+\s,-1)--(4+\s,1);
\draw[dotted] (1+\s,0) -- (4+\s,0);

\node at (1+\s,-1.7) {$U'$};
\node at (5+\s,-1.7) {$U$};

\draw[decorate, decoration = {brace}] (-2.5+\s,2) -- (-2.5+\s,3);
\node at (-3.5+\s,2.5) {$c$};

\draw[dashed] (0+\s,1) to[out=115, in=0] (-2+\s,4);
\draw[dashed] (0+\s,1) to[out=113, in=0] (-2+\s,4.2);
\draw[dotted] (-2+\s,4.3) -- (-2+\s,4.9);
\draw[dashed] (0+\s,1) to[out=110, in=0] (-2+\s,5);

\draw[decorate, decoration = {brace}] (-2.5+\s,4) -- (-2.5+\s,5);
\node at (-3.5+\s,4.5) {$c'$};

\end{tikzpicture}
\end{center}
This gives $S_N$-actions on the sets
\[
B_N\coloneqq \{\textup{Speatated brane diagrams $\mathcal D$}\mid |\mathrm{b}(U)|=N\}\quad 
\textup{and}\quad
\bigcup_{\mathcal D\in B_N}\mathrm{Tie}(\mathcal D).
\]
Given a permutation $w\in S_N$ then $w.\mathcal D$ is the separated brane diagram with $M$ red lines, $N$ blue lines and the numbers on the horizontal lines are given as
\[
d_{X_i}(w.\mathcal D)=d_{X_i}(\mathcal D) 
,\quad i=1,\ldots,M+1,\quad 
d_{X_{M+j}}(w.\mathcal D)=\sum_{l=j}^{N} c_{w^{-1}(l)}(\mathcal D),\quad
j=1,\ldots,N+1.
\]
By construction, $\textbf r(w.\mathcal D)=\textbf r(\mathcal D)$ and $\textbf c(w.\mathcal D)=(c_{w^{-1}(1)}(\mathcal D),\ldots,c_{w^{-1}(N)}(\mathcal D))$, where $\textbf r, \textbf c$ are the margin vectors from Subsection~\ref{subsection:BraneDiagrams}. Likewise, if $D\in\mathrm{Tie}(\mathcal D)$ we define $w.D\in \mathrm{Tie}(w.\mathcal D)$ via
\[
w.D=\bigcup_{(V_i,U_j)\in D}\{(V_i,U_{w(j)})\}.
\]
Pictorially, the action is given by moving each blue line $U_i$ with its attached ties to the position of $U_{w(i)}$. 

\begin{example} Consider the following tie diagram $D$ with underlying brane diagram $\mathcal D$:
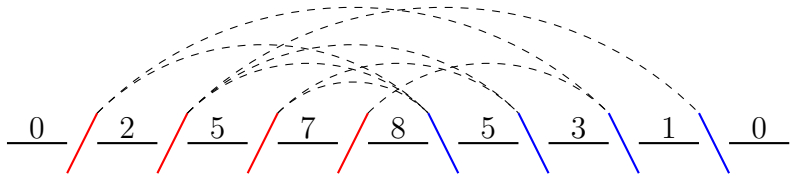
\begin{center}
\begin{tikzpicture}[scale=.4]
\foreach \i in {0,...,8}
{
\draw[thick] (3*\i,0) -- (3*\i+2,0);
}
\foreach \i in {1,...,4}
{
\draw[thick, red] (3*\i,1) -- (3*\i-1,-1);
}
\foreach \i in {5,...,8}
{
\draw[thick, blue] (3*\i,-1) -- (3*\i-1,1);
}
\draw[dashed] (3,1) to[out=45,in=135] (3*5-1,1);
\draw[dashed] (6,1) to[out=45,in=135] (3*5-1,1);
\draw[dashed] (9,1) to[out=45,in=135] (3*5-1,1);

\draw[dashed] (6,1) to[out=45,in=135] (3*6-1,1);
\draw[dashed] (9,1) to[out=45,in=135] (3*6-1,1);

\draw[dashed] (3,1) to[out=45,in=135] (3*7-1,1);
\draw[dashed] (12,1) to[out=45,in=135] (3*7-1,1);

\draw[dashed] (6,1) to[out=45,in=135] (3*8-1,1);

\node at (1,0.5) {$0$};
\node at (4,0.5) {$2$};
\node at (7,0.5) {$5$};
\node at (10,0.5) {$7$};
\node at (13,0.5) {$8$};
\node at (16,0.5) {$5$};
\node at (19,0.5) {$3$};
\node at (22,0.5) {$1$};
\node at (25,0.5) {$0$};
\end{tikzpicture}
\end{center}
Let $w=3142\in S_4$. To obtain the tie diagram $w.D$, we permute the blue lines with the attached ties according to $w$, i.e. the blue line $U_1$ is moved with its three attached ties to the position of $U_3$ and so on. The respective labels of the horizontal lines of $w.\mathcal D$ can then be easily determined by counting the number of ties above the horizontal lines:
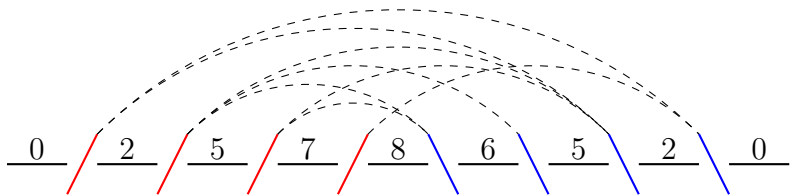
\begin{center}
\begin{tikzpicture}[scale=.4]
\foreach \i in {0,...,8}
{
\draw[thick] (3*\i,0) -- (3*\i+2,0);
}
\foreach \i in {1,...,4}
{
\draw[thick, red] (3*\i,1) -- (3*\i-1,-1);
}
\foreach \i in {5,...,8}
{
\draw[thick, blue] (3*\i,-1) -- (3*\i-1,1);
}
\draw[dashed] (3,1) to[out=45,in=135] (3*7-1,1);
\draw[dashed] (6,1) to[out=45,in=135] (3*7-1,1);
\draw[dashed] (9,1) to[out=45,in=135] (3*7-1,1);

\draw[dashed] (6,1) to[out=45,in=135] (3*5-1,1);
\draw[dashed] (9,1) to[out=45,in=135] (3*5-1,1);

\draw[dashed] (3,1) to[out=45,in=135] (3*8-1,1);
\draw[dashed] (12,1) to[out=45,in=135] (3*8-1,1);

\draw[dashed] (6,1) to[out=45,in=135] (3*6-1,1);

\node at (1,0.5) {$0$};
\node at (4,0.5) {$2$};
\node at (7,0.5) {$5$};
\node at (10,0.5) {$7$};
\node at (13,0.5) {$8$};
\node at (16,0.5) {$6$};
\node at (19,0.5) {$5$};
\node at (22,0.5) {$2$};
\node at (25,0.5) {$0$};
\end{tikzpicture}
\end{center}
\end{example}

Stable basis elements satisfy the following compatibility relation with this $S_N$-action, see \cite[Proposition~6.18]{botta2023mirror} and \cite[Theroem~9.20]{wehrhanphd}:

\begin{theorem}\label{thm:ComparisonOfChambers}
For $D$, $D'\in\mathrm{Tie}(\mathcal D)$ and $w\in S_N$, we have
\[
e_{\mathbb T}(N^-_{\mathcal D,\mathfrak C})\iota_{D'}^\ast(\mathrm{Stab}_{\mathfrak C}(D))
=
w^{-1}.( e_{\mathbb T}(N^-_{w.\mathcal D,w.\mathfrak C}) \iota_{w.D'}^\ast(\mathrm{Stab}_{w.\mathfrak C}(w.D))).
\]
Here, $w\in S_N$ acts on $\mathbb Q[t_1,\ldots,t_N,h]$ via $w.h=h$ and $w.t_i=t_{w(i)}$ for $i=1,\ldots,N$ and $N_{\mathcal D,\mathfrak C}^-$ (resp.~ $N_{w.\mathcal D,w.\mathfrak C}^-$) is the negative part of the trivial $\mathbb T$-equivariant bundle $N_{\mathcal D}$ (resp.~$N_{w.\mathcal D}$) defined in the remark below.
\end{theorem}

\begin{remark} 
Set \begin{equation}\label{eq:NormalBundleResolutionDefinition}
N_{\mathcal D}\coloneqq
\Big( \bigoplus_{j=1}^N \bigoplus_{l=1}^{c_j-1} h^{l}(\xi_{U_j^+}\otimes \mathbb C_{U_j}^\vee) \Big) 
\oplus
\Big( \bigoplus_{j=1}^N \bigoplus_{l=1}^{c_j-1} h^{1-l}(\mathbb C_{U_j}\otimes \xi_{U_j^+}^\vee) \Big). 
\end{equation}
As explained in \cite[Corollary~6.6]{botta2023mirror}, the bow variety $\mathcal C(\mathcal D)$ can be $\mathbb T$-equivariantly embedded into the partial flag variety $T^\ast F(R_1,\ldots,R_M;n)$. As shown in \cite[Corollary~9.53]{wehrhanphd}, the $\mathbb T$-equivariant K-theory class of the normal bundle of this embedding is a sum of classes of trivial bundles. Its positive (resp.~negative) part with respect to any choice of chamber equals the positive (resp.~negative) part of $[N_{\mathcal D}]$.
\end{remark}

Recall from Corollary~\ref{cor:PropertiesTautologicalBundlesSeparated} that for all $U_j\in\mathrm{b}(\mathcal D)$, we have an isomorphism of $\mathbb T$-equivariant vector bundles $
\xi_{U_j^+} \cong \bigoplus_{i=j+1}^N  \bigoplus_{l=0}^{c_i-1} h^{-l} \mathbb C_{U_i}$. Thus, the positive (resp.~negative) part of $N_{\mathcal D}$ with respect to a choice of chamber $\mathfrak C$ can be easily read off from the definition. For instance, if $\mathfrak C$ equals the antidominant chamber $\mathfrak C_-$ then $N_{\mathcal D,\mathfrak C_-}^\pm$ are given as follows:

\begin{prop}
We have
\[
N_{\mathcal D,\mathfrak C_-}^+= \bigoplus_{j=1}^N \bigoplus_{i=j+1}^N \bigoplus_{l=1}^{c_j-1} \bigoplus_{k=0}^{c_i-1} h^{l-k} (\mathbb C_{U_i}\otimes \mathbb C_{U_j}^\vee),
\quad
N_{\mathcal D,\mathfrak C_-}^-=
\bigoplus_{j=1}^N \bigoplus_{i=j+1}^N \bigoplus_{l=1}^{c_j-1} \bigoplus_{k=0}^{c_i-1}
h^{k-l+1}
(\mathbb C_{U_j}\otimes \mathbb C_{U_i}^\vee).
\]
\end{prop}
%In general, the Euler class $e_{\mathbb T}(N_{\mathcal D}^-)$ can be computed directly via Corollary~\ref{cor:PropertiesTautologicalBundlesSeparated}. For example, if $\mathcal C$ is the antidominant chamber, the following holds:
%
%\begin{prop}\label{prop:AntidominantNormalizationFactor} If $\mathfrak C=\mathfrak C_-$ then
%\[
%e_{\mathbb T}(N_{\mathcal D}^-)=
%\prod_{1\le i < j \le N}\Big(\prod_{r=0}^{c_i-1}\prod_{s=0}^{c_j-1} (t_i-t_j+(s-r+1)h).
%\]
%\end{prop}
\begin{proof} By Corollary~\ref{cor:PropertiesTautologicalBundlesSeparated}, all weights of $\xi_{U_j^-}\otimes \mathbb C_{U_j}^\vee$ are non-negative. Hence, only the second summand of \eqref{eq:NormalBundleResolutionDefinition} contributes to $N_{\mathcal D}^-$. Then, inserting $
\xi_{U_j^+} \cong \bigoplus_{i=j+1}^N  \bigoplus_{l=0}^{c_i-1} h^{-l} \mathbb C_{U_i}$ proves the proposition.
%The proposition then follows from 
%\[
%(\mathbb C_{U_j}\otimes \xi_{U_j^-}^\vee)_{\mathfrak C_-}^- = \bigoplus_{i=j+1}^N \bigoplus_{l=1}^{c_i} h^l(\mathbb C_{U_j}\otimes \mathbb C_{U_i}^\vee),\quad \textup{for $j=1,\ldots,N$,}
%\]
%which is immediate from Corollary~\ref{cor:PropertiesTautologicalBundlesSeparated}.
\end{proof}

Because of Theorem~\ref{thm:ComparisonOfChambers}, it is sometimes more convenient to work with the following normalized version of stable basis elements:

\begin{notation} With the above notation, we set
\begin{equation}\label{equation:DefinitionStabTilde} 
\widetilde{\mathrm{Stab}}_{\mathfrak C}(D)\coloneqq e_{\mathbb T}(N_{\mathcal D,\mathfrak C}^-)\mathrm{Stab}_{\mathfrak C}(D),\quad\textup{for all $D\in\mathrm{Tie}(\mathcal D)$.}
\end{equation}
\end{notation}

In the special case $\mathcal D=\mathcal D(d_1,\ldots,d_m;n)$ is as in Subsection~\ref{subsection:DepBraneDiagPartialFlag}, we have $c_1=\ldots=c_n=1$. Thus, by definition, $N_{\mathcal D}$ is the trivial bundle of rank zero which yields $\widetilde{\mathrm{Stab}}_{\mathfrak C}(D)=\mathrm{Stab}_{\mathfrak C}(D)$ in this case.

\begin{remark} In \cite[Proposition~6.18]{botta2023mirror}, Rim\'anyi and Botta prove a version Theorem~\ref{thm:ComparisonOfChambers} in the framework of elliptic cohomology with a different normalization factor.
\end{remark}

\section{Stable bases of cotangent bundles of flag varieties}

In this section, we recall the localization formula for stable basis elements of cotangent bundles of flag varieties from \cite[Theorem~1.1]{su2017restriction}, see also \cite{su2017stable}. Then, we give an equivalent reformulation of this formula using the language of diagrammatic calculus of symmetric groups.

\subsection{Reminders on symmetric groups} 

We denote the simple transpositions of $S_n$ by $s_1,\ldots,s_{n-1}$, where $s_i=(i,i+1)$.
Every permutation can be written as $
w=\sigma_1\cdots\sigma_{r},
$ where all $\sigma_i$ are simple transpositions. If $r$ is as small as possible, we call the expression $\sigma_1\cdots\sigma_{r}$ for $w$ \textit{reduced} and we call $r$ the \textit{length} of $w$ and denote it by $\ell(w)$.

It is well-known that $\ell(w)$ is equal to the number of inversions of $w$:
\begin{equation}\label{eq:LengthVSInversions}
\ell(w)= |\mathrm{Inv}(w)|,
\quad
\mathrm{Inv}(w)=\{(i,j)\mid 1\le i <j\le n, w(i)>w(j)\}.
\end{equation}

By definition, a permutation $w$ is larger than a permutation $w'$ in the \textit{Bruhat order} if some (not necessarily a consecutive) subword of a reduced expression for $w$ is a reduced word for $w'$. 
It is a well-known fact that if $w$ dominates $w'$ in the Bruhat order then every reduced expression for $w$ admits a subword which is a reduced expression for $w'$, see e.g.~\cite[Theorem~5.10]{humphreys1992reflection}.

Let $R^+=\{t_i-t_j\mid 1\le i<j\le n \}\subset \mathbb Q[t_1,\ldots,t_n]$ be the set of positive roots and $R^-=\{t_i-t_j\mid 1\le j<i\le n \}\subset  \mathbb Q[t_1,\ldots,t_n]$ the set of negative roots. By \eqref{eq:LengthVSInversions}, we have
\begin{equation}\label{eq:BruhatPositiveToNegative}
\ell(w)=|\{\alpha\in R^+ \mid w.\alpha\in R^-\}|.
\end{equation}
The set on the right hand side of \eqref{eq:BruhatPositiveToNegative} can also be characterized as follows: for $s=s_i$ we denote by $\alpha_{s}=t_i-t_{i+1}$ the corresponding simple root. Given a reduced expression $w=\sigma_1\cdots\sigma_{\ell(w)}$, we set
\begin{equation}\label{eq:DefinitionBeta}
\beta_i\coloneqq \sigma_1\cdots\sigma_{i-1}(\alpha_{\sigma_i}),\quad i=1,\ldots,\ell(w).
\end{equation}
Then, by e.g.~\cite[Section 5.6]{humphreys1992reflection}, we have
\begin{equation}\label{eq:SetOfBetasIndependentFromReducedExpression}
\{\beta_1,\ldots,\beta_{\ell(w)}\}= \{\alpha\in R^+ \mid w^{-1}.\alpha\in R^-\}.
\end{equation}

\begin{example}\label{example:Beta}
Let $n=5$ and $w=35412$. Then, $\ell(w)=7$ and $w=s_4s_2s_1s_3s_2s_4s_3\eqqcolon\sigma_1\cdots\sigma_7$ is a reduced expression of $w$. The corresponding $\beta_i$ are recorded in the following table:
\begin{center}
\begin{tabular}{|c||c|c|c|c|c|c|c|}
\hline
$i$ & $1$ & $2$ & $3$ & $4$ & $5$ & $6$ & $7$\\
\hline
$\beta_i$ & $t_4-t_5$ & $t_2-t_3$ & $t_1-t_3$ & $t_2-t_5$ & $t_1-t_5$ & $t_2-t_4$ & $t_1-t_4$\\
\hline
\end{tabular}
\end{center}
\end{example}

\subsection{Diagrammatics of permutations}\label{subsection:DiagrammaticsOfPermutations}

We illustrate permutations in the common way using string diagrams as follows: a \textit{strand} is a smooth embedding $\lambda\colon [0,1]\rightarrow \mathbb R^2$.

\begin{definition} 
Let $w\in S_n$ be a permutation. A collection $\lambda_1,\ldots,\lambda_n$ of $n$ strands is called a \textit{diagram of $w$} if the following holds:
\begin{enumerate}[label=(\roman*)]
\item $\lambda_i(0)=(i,0)$ and $\lambda_i(1)=(w(i),1)$ for all $i=1,\ldots,n$,
\item every two strands intersect only in finitely many points and all of these intersections are transversal,
\item there are no triple or even higher intersections among $\lambda_1,\ldots,\lambda_n$. 
\end{enumerate}
A diagram is called \textit{reduced} if the number of intersections among $\lambda_1,\ldots,\lambda_n$ is equal to $\ell(w)$.
\end{definition}

We call the intersection of two strands a \textit{crossing}. Given a diagram $d_w$ of a permutation $w$, we denote by $K(d_w)$ its set of crossings. If the second coordinates of all crossings of $d_w$ are pairwise distinct, we denote the crossings of $d_w$ by $\kappa_1,\ldots,\kappa_{\ell(w)}$ where $\kappa_1$ is the crossing with the highest second coordinate, $\kappa_2$ is the crossing with the second highest second coordinate and so on.

Given a crossing $\kappa\in K(d_w)$ we refer to the local move
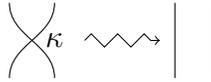
\begin{center}
\begin{tikzpicture}
\draw (0.6,-1) to[out=90,in=-90] (0,0);
\draw (0,-1) to[out=90,in=-90] (0.6,0);
\node at (0.6,-0.5) {$\kappa$};
\draw [-to] decorate [decoration=zigzag] {(1,-0.5) -- (2,-0.5)};
\draw (2.2,-1) to[out=90,in=-90] (2.2,0);
\draw (2.7,-1) to[out=90,in=-90] (2.7,0);
\end{tikzpicture}
\end{center}
as \textit{resolving} of $\kappa$.

By definition, $w$ is larger than a permutation $w'\in S_n$ in the Bruhat order if and only if we can obtain a diagram for $w'$ by resolving crossing from $d_w$.

If $d_w$ is a reduced diagram such that all crossings of $d_w$ have pairwise distinct second coordinate then , by viewing $d_w$ as stacking of diagrams corresponding to simple transpositions, we can easily read off a reduced expression for $w$ from $d_w$.

\begin{example}\label{example:ReducedDiagrams} Let $w\in S_5$ be as in Example~\ref{example:Beta}. Then, a reduced diagram $d_w$ of $w$ is given by
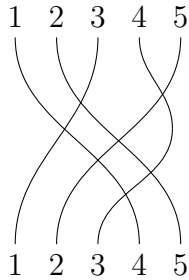
\[
\begin{tikzpicture}[scale=.55]
\coordinate (B1) at (0,-5);
\coordinate (B2) at (1,-5);
\coordinate (B3) at (2,-5);
\coordinate (B4) at (3,-5);
\coordinate (B5) at (4,-5);

\node[below=0 of B1] {$1$};
\node[below=0 of B2] {$2$};
\node[below=0 of B3] {$3$};
\node[below=0 of B4] {$4$};
\node[below=0 of B5] {$5$};

\coordinate (T1) at (0,0);
\coordinate (T2) at (1,0);
\coordinate (T3) at (2,0);
\coordinate (T4) at (3,0);
\coordinate (T5) at (4,0);

\node[above=0 of T1] {$1$};
\node[above=0 of T2] {$2$};
\node[above=0 of T3] {$3$};
\node[above=0 of T4] {$4$};
\node[above=0 of T5] {$5$};

\draw (B1) to[out=90,in=-90] (T3);
\draw (B2) to[out=90,in=-90] (T5);
\draw (B3) to[out=90,in=-90] (3.8,-2.2);
\draw (3.8,-2.2) to[out=90,in=-90] (T4);
\draw (B4) to[out=90,in=-90] (T1);
\draw (B5) to[out=90,in=-90] (T2);
\end{tikzpicture}
\]
The diagram corresponds to the reduced expression $s_4s_2s_1s_3s_2s_4s_3$.
\end{example}

For given $w\in S_n$ with reduced diagram $d_w$, define a function
\[
\mathrm{wt}\colon K(d_w)\xrightarrow{\phantom{xxx}} \mathbb Q[t_1,\ldots,t_n]
\]
as follows: let $\kappa$ be a crossing between the strands $\lambda$ and $\lambda'$. Let $j$ resp. $j'$ be the endpoints of $\lambda$ resp. $\lambda'$. Assuming $j<j'$, we set
\[
\mathrm{wt}(\kappa)\coloneqq t_{j}-t_{j'}.
\]
We call $\mathrm{wt}(\kappa)$ the \textit{weight} of $\kappa$.

The next proposition gives that the weights of crossings correspond exactly to the $\beta_i$ from the previous subsection:

\begin{prop}\label{prop:WeightVSBeta} Suppose all crossings of $d_w$ have pairwise distinct second coordinate and let $w=\sigma_1\cdots\sigma_{\ell(w)}$ be the reduced expression corresponding to $d_w$. Then, we have
\[
\mathrm{wt}(\kappa_i)=\beta_{i},\quad \textit{for all }i=1,\ldots,\ell(w).
\]
\end{prop}
\begin{proof} For given $i\in\{1,\ldots,\ell(w)\}$, set $w_1\coloneqq \sigma_1\cdots\sigma_{i-1}$, $w_2\coloneqq \sigma_{i+1}\cdots\sigma_{\ell(w)}$ and $\sigma_i=(j,j+1)$. After applying a homotopy, we can view $d_w$ as a stacking of a reduced diagram of $w_2$, a reduced diagram of $\sigma_i$ and a reduced diagram of $w_1$:
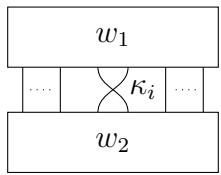
\[
\begin{tikzpicture}
\node[rectangle, draw, text = black,minimum width=2.8cm,minimum height=0.8cm] (w1) at (0,0) {$w_1$};
\draw (-0.2,-0.4) to[out=-90,in=90] (0.2,-1);
\draw (0.2,-0.4) to[out=-90,in=90] (-0.2,-1);

\draw (1.2,-1) -- (1.2,-0.4);
\draw (-1.2,-1) -- (-1.2,-0.4);
\draw (-0.7,-1) -- (-0.7,-0.4);
\draw (0.7,-1) -- (0.7,-0.4);

\node[rectangle, draw, text = black,minimum width=2.8cm,minimum height=0.8cm] (w2) at (0,-1.4) {$w_2$};
\node at (0.4,-0.7) {$\kappa_i$};
\draw[dotted] (-1.1,-0.7)--(-0.8,-0.7);
\draw[dotted] (1.1,-0.7)--(0.8,-0.7);
\end{tikzpicture}
\]
Thus, we have $
\mathrm{wt}(\kappa_i)=t_{w_1(j)}-t_{w_1(j+1)}=w_1.\alpha_{\sigma_i}=\beta_i
$ which completes the proof.
\end{proof}

\subsection{Localization formula}
Let $F=F(1,2,\ldots,n-1;n)$ be the full flag variety of $\mathbb C^n$ endowed with the $\mathbb T$-action from Subsection~\ref{subsec:TFlagAsBow}. For $w\in S_n$, we also denote the $\mathbb T$-fixed point $(\mathcal F_{w},0)$ just by $w$.

The localization formula from \cite[Theorem~1.1]{su2017restriction} determines the $\mathbb T$-localization coefficients of the stable basis elements of $T^\ast F$ with respect to the antidominant chamber $\mathfrak C_-$.

\begin{theorem}[Localization formula]\label{thm:SuLocalizationFormula}
Let $w\in S_n$ and $w=\sigma_1\sigma_2\cdots \sigma_{\ell(w)}$ a reduced expression. Then, for all $w'\in S_n$, we have
\[
\iota_{w}^\ast (\mathrm{Stab}_{\mathfrak C_-}(w')) = \Big( \prod_{(i,j)\in L_w} (t_i-t_j+h) \Big)
\Big(
\sum_{\substack{1\le i_1 < \dots < i_k\le \ell(w)\\ w'=\sigma_{i_1}\cdots\sigma_{i_k}}} h^{\ell(w)-k}\prod_{j=1}^k \beta_{i_j}
\Big),
\]
where the $\beta_i$ are defined as in \eqref{eq:DefinitionBeta} and
\[ 
L_w=R^+ \setminus \{\alpha\in R^+\mid w^{-1}.\alpha\in R^-\}= \{\alpha\in R^+\mid \textit{ $\alpha \ne \beta_l$ for all $l$}\}.
\]
\end{theorem}

\begin{example}\label{example:SuLocalizationFormula}
Let $n=5$ and consider the permutations
$w=35412$ and $w'=23415$. To compute the localization coefficient $\iota_w^\ast(\mathrm{Stab}_{\mathfrak C_-}(w'))$, we choose, as in Example~\ref{example:Beta}, $w=s_4s_2s_1s_3s_2s_4s_3$ as reduced expression for $w$.
Checking all possible subwords of this expression for $w$ yields that there are only two subwords that give $w'$, namely $\sigma_1\sigma_3\sigma_5\sigma_6\sigma_7$ and $\sigma_3\sigma_5\sigma_7$. 
We already computed the $\beta_i$ in Example~\ref{example:Beta}. Our computation implies $L_w=\{(t_1-t_2),(t_3-t_4),(t_3-t_5)\}$. Hence, Theorem~\ref{thm:SuLocalizationFormula} yields
\begin{equation}\label{eq:ExampleStableEnvelopeViaSusFormula}
\iota_w^\ast(\mathrm{Stab}_{\mathfrak C_-}(w'))=(t_1-t_2+h)(t_3-t_4+h)(t_3-t_5+h)\cdot h^2(\beta_1\beta_6+h^2)\beta_3\beta_5\beta_7.
\end{equation}
\end{example}

\subsection{Diagrammatic localization formula} 
Employing the diagrammatics related to permutations from Subsection~\ref{subsection:DiagrammaticsOfPermutations} leads to the following diagrammatic version of Theorem~\ref{thm:SuLocalizationFormula}:

\begin{prop}[Diagrammatic localization formula]\label{prop:SuLocalizationFormulaDiagrammatic}
Let $w\in S_n$ and $d_w$ a reduced diagram of $w$. Then, for all $w'\in S_n$, we have
\[
\iota_{w}^\ast (\mathrm{Stab}_{\mathfrak C_-}(w')) = \Big( \prod_{\alpha\in L_w'} (\alpha+h) \Big)
\Big(
\sum_{K'\in K_{d_w,w'}} h^{|K'|}\prod_{\kappa\in K(d_w)\setminus K'}\mathrm{wt}(\kappa)
\Big),
\]
where $ K_{d_w,w'}$ is the set of all subsets $K'\subset K(d_w)$ such that resolving all crossings of $K'$ from $d_w$ gives a diagram for $w'$ and 
\[ 
L_w'=\{\alpha\in R^+ \mid \textit{$\alpha \ne \mathrm{wt}(\kappa)$ for all $\kappa\in K(d_w)$}\}.
\]
\end{prop}

\begin{proof} We may assume without loss of generality that the second coordinates of all crossings in $d_w$ are pairwise distinct. Let $1\le i_1<\cdots <i_k\le \ell(w)$ and let $d'$ be the diagram obtained from $d_w$ by resolving all crossings $\kappa_i$ with $i\ne i_1,\ldots,i_k$. By viewing $d_w$ as a stacking of diagrams corresponding to simple transpositions we deduce that $w'=\sigma_{i_1}\cdots\sigma_{i_k}$ if and only if $d'$ is a diagram for $w'$.
Thus, Proposition~\ref{prop:WeightVSBeta} implies
\[
\sum_{K'\in K_{d_w,w'}} h^{|K'|}\prod_{\kappa\in K(d_w)\setminus K'}\mathrm{wt}(\kappa) = \sum_{\substack{1\le i_1 < \dots < i_k\le \ell(w)\\ w'=\sigma_{i_1}\cdots\sigma_{i_k}}} h^{\ell(w)-k}\prod_{j=1}^k \beta_{i_j}.
\]
In addition, Proposition~\ref{prop:WeightVSBeta} also gives $L_w=L_w'$ which completes the proof.
\end{proof}

\begin{example} Let $w$ and $w'$ be as in Example~\ref{example:SuLocalizationFormula} and let $d_w$ be as in Example~\ref{example:ReducedDiagrams}. To compute $\iota_w^\ast(\mathrm{Stab}_{\mathfrak C_-}(w'))$, note that there are just two possibilities to obtain a diagram for $w'$ by resolving crossings from $d_w$. One is given by resolving the crossings $\kappa_2$ and $\kappa_4$. The second one is given by resolving the crossings $\kappa_1$, $\kappa_2$, $\kappa_4$ and $\kappa_6$, in pictures:
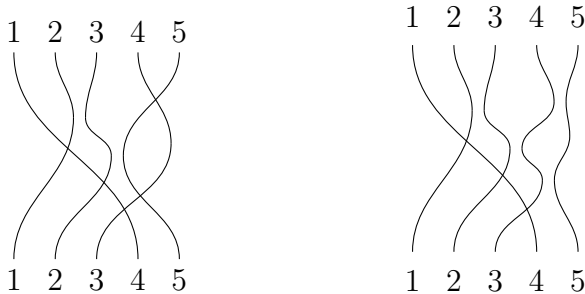
\begin{center}
\begin{tikzpicture}[scale=.55]
\coordinate (B1) at (0,-5);
\coordinate (B2) at (1,-5);
\coordinate (B3) at (2,-5);
\coordinate (B4) at (3,-5);
\coordinate (B5) at (4,-5);

\node[below=0 of B1] {$1$};
\node[below=0 of B2] {$2$};
\node[below=0 of B3] {$3$};
\node[below=0 of B4] {$4$};
\node[below=0 of B5] {$5$};

\coordinate (T1) at (0,0);
\coordinate (T2) at (1,0);
\coordinate (T3) at (2,0);
\coordinate (T4) at (3,0);
\coordinate (T5) at (4,0);

\node[above=0 of T1] {$1$};
\node[above=0 of T2] {$2$};
\node[above=0 of T3] {$3$};
\node[above=0 of T4] {$4$};
\node[above=0 of T5] {$5$};

\draw[draw=none, name path=line1] (B1) to[out=90,in=-90] (T3);
\draw[draw=none, name path=line2] (B2) to[out=90,in=-90] (T5);
\draw (B3) to[out=90,in=-90] (3.8,-2.2);
\draw (3.8,-2.2) to[out=90,in=-90] (T4);
\draw (B4) to[out=90,in=-90] (T1);
\draw[draw=none,name path=line5] (B5) to[out=90,in=-90] (T2);

%\coordinate (E) at (intersection of line1 and line2);
\path[name intersections={of=line1 and line5,by=E}];
\coordinate[left=0.08 of E] (E15l); 
\coordinate[right=0.08 of E] (E15r); 

\path [name intersections={of=line2 and line5,by=E25}];
\coordinate[left=0.08 of E25] (E25l); 
\coordinate[right=0.08 of E25] (E25r);

\draw (B2) to[out=90,in=-90] (E25l);
\draw (B5) to[out=90,in=-90] (E25r);
\draw (E25l) to[out=90,in=-90] (E15r);
\draw (E15r) to[out=90,in=-90] (T3);
\draw (E25r) to[out=90,in=-90] (T5);

\draw (B1) to[out=90,in=-90] (E15l);
\draw (E15l) to[out=90,in=-90] (T2);

\end{tikzpicture}
\qquad
\qquad 
\qquad
\begin{tikzpicture}[scale=.55]
%[scale=.5]
\coordinate (B1) at (0,-5);
\coordinate (B2) at (1,-5);
\coordinate (B3) at (2,-5);
\coordinate (B4) at (3,-5);
\coordinate (B5) at (4,-5);

\node[below=0.1 of B1] {$1$};
\node[below=0.1 of B2] {$2$};
\node[below=0.1 of B3] {$3$};
\node[below=0.1 of B4] {$4$};
\node[below=0.1 of B5] {$5$};

\coordinate (T1) at (0,0);
\coordinate (T2) at (1,0);
\coordinate (T3) at (2,0);
\coordinate (T4) at (3,0);
\coordinate (T5) at (4,0);

\node[above=0.1 of T1] {$1$};
\node[above=0.1 of T2] {$2$};
\node[above=0.1 of T3] {$3$};
\node[above=0.1 of T4] {$4$};
\node[above=0.1 of T5] {$5$};

\draw[draw=none, name path=line1] (B1) to[out=90,in=-90] (T3);
\draw[draw=none, name path=line2] (B2) to[out=90,in=-90] (T5);
\draw[draw=none, name path=line3a] (B3) to[out=90,in=-90] (3.8,-2.2);
\draw[draw=none, name path=line3b] (3.8,-2.2) to[out=90,in=-90] (T4);
\draw (B4) to[out=90,in=-90] (T1);
\draw[draw=none,name path=line5] (B5) to[out=90,in=-90] (T2);

%\coordinate (E) at (intersection of line1 and line2);
\path[name intersections={of=line1 and line5,by=E}];
\coordinate[left=0.08 of E] (E15l); 
\coordinate[right=0.08 of E] (E15r); 

\path[name intersections={of=line3a and line5,by=E35}];
\coordinate[left=0.08 of E35] (E35l); 
\coordinate[right=0.08 of E35] (E35r);

\path[name intersections={of=line2 and line3b,by=E23}];
\coordinate[left=0.08 of E23] (E23l); 
\coordinate[right=0.08 of E23] (E23r);

\path [name intersections={of=line2 and line5,by=E25}];
\coordinate[left=0.08 of E25] (E25l); 
\coordinate[right=0.08 of E25] (E25r);

\draw (B5) to[out=90,in=-90] (E35r);
\draw (E35r) to[out=90,in=-90] (3.8,-2.2);
\draw (3.8,-2.2) to[out=90,in=-90] (E23r);
\draw (E23r) to[out=90,in=-90] (T5);

\draw (B3) to[out=90,in=-90] (E35l);
\draw (E35l) to[out=90,in=-90] (E25r);
\draw (E25r) to[out=90,in=-90] (E23l);
\draw (E23l) to[out=90,in=-90] (T4);

\draw (B2) to[out=90,in=-90] (E25l);
\draw (E25l) to[out=90,in=-90] (E15r);
\draw (E15r) to[out=90,in=-90] (T3);

\draw (B1) to[out=90,in=-90] (E15l);
\draw (E15l) to[out=90,in=-90] (T2);
\end{tikzpicture}
\end{center}
The diagram on the left corresponds to the subword $\sigma_1\sigma_3\sigma_5\sigma_6\sigma_7$ and the diagram on the right hand side to $\sigma_3\sigma_5\sigma_7$. The first diagram contributes the summand
\[
(t_1-t_2+h)(t_3-t_4+h)(t_3-t_5+h)h^2\mathrm{wt}(\kappa_1)\mathrm{wt}(\kappa_3)\mathrm{wt}(\kappa_5)\mathrm{wt}(\kappa_6)\mathrm{wt}(\kappa_7),
\]
whereas the second diagram contributes the summand
\[
(t_1-t_2+h)(t_3-t_4+h)(t_3-t_5+h)h^4\mathrm{wt}(\kappa_3)\mathrm{wt}(\kappa_5)\mathrm{wt}(\kappa_7).
\]
It follows that
\begin{align*}
\iota_w^\ast(&\mathrm{Stab}_{\mathfrak C_-}(w'))=\\ 
&(t_1-t_2+h)(t_3-t_4+h)(t_3-t_5+h)\cdot h^2(\mathrm{wt}(\kappa_1)\mathrm{wt}(\kappa_6)+h^2)\mathrm{wt}(\kappa_3)\mathrm{wt}(\kappa_5)\mathrm{wt}(\kappa_7)
\end{align*}
which coincides with the computation \eqref{eq:ExampleStableEnvelopeViaSusFormula} from Example~\ref{example:SuLocalizationFormula}.
\end{example}

\subsection{Cotangent bundles of partial flag varieties} 
Let $F=F(d_1,\ldots,d_m;n)$ be a partial flag variety and $\bm{\delta}=(\delta_1,\ldots,\delta_{m+1})$ as in Subsection~\ref{subsec:TFlagAsBow}. As before, for a given $w\in S_n$, we also denote the $\mathbb T$-fixed point $(\mathcal F_{wS_{\bm \delta}},0)$ by $wS_{\bm \delta}$.

It was proved in \cite[Corollary~4.3]{su2017restriction} that the localization coefficients of the stable basis elements of $T^\ast F$ can be computed via localization coefficients of stable basis elements of $T^\ast F(1,2,\ldots,n-1;n)$ as follows: 

\begin{prop}\label{prop:PartialViaFullFlag}
For all $w$, $w'\in S_n$, we have
\[
\iota_{wS_{\bm \delta}}^\ast (\mathrm{Stab}_{\mathfrak C_-}(w'S_{\bm \delta})) = \sum_{z\in wS_{\bm \delta}} \frac{(-1)^{\ell(w'S_{\bm \delta})+\ell(w')} \iota_z^\ast (\mathrm{Stab_{\mathfrak C_-}(w')})}{\prod_{\alpha\in R_{\bm \delta}} (z.\alpha)},
\]
where $\ell(wS_{\bm \delta})$ is the length of the shortest coset representative of $wS_{\bm \delta}$ and 
\[
R_{\bm \delta}^+ =\{t_i-t_j \mid \textit{there exist $l\in\{1,\ldots,r\}$ with $d_1+\cdots+d_{l-1}\le i<j\le d_1+\cdots+d_{l}$}\}.
\]
\end{prop}

\begin{example}\label{example:StabsOfCotangentPartialFlag} Let $\bm \delta=(2,2,1)$ and 
$w=25143$, $w'=52314$.
Note that $w$ is the minimal element of $wS_{\bm \delta}$. Using that $\ell(w)=5$ and $\ell(w')=6$, we deduce that $z=w(s_1\times s_1\times \operatorname{id})$ is the only element in $wS_{\bm \delta}$ that dominates $w'$ in the Bruhat order. Hence, Proposition~\ref{prop:PartialViaFullFlag} gives
\begin{equation}\label{eq:EquivariantMultiplicityParialExample}
\iota_{wS_{\bm \delta}}^\ast (\mathrm{Stab}_{\mathfrak C_-}(w'S_{\bm \delta}))
=
\frac{\iota_z^\ast(\mathrm{Stab}_{\mathfrak C_-}(w'))}{(t_{z(1)}-t_{z(2)})(t_{z(3)}-t_{z(4)})}\\
=
\frac{\iota_z^\ast(\mathrm{Stab}_{\mathfrak C_-}(w'))}{(t_5-t_2)(t_4-t_1)}.
\end{equation}
To compute $\iota_z^\ast(\mathrm{Stab}_{\mathfrak C_-}(w'))$ we use Proposition~\ref{prop:SuLocalizationFormulaDiagrammatic}. The following figure shows a reduced diagram $d_z$ for $z$. Since $\ell(z)=7$ and $\ell(w')=6$, there is only one possibility to obtain a diagram $d_{w'}$ for $w'$ from $d_z$ by resolving crossings:
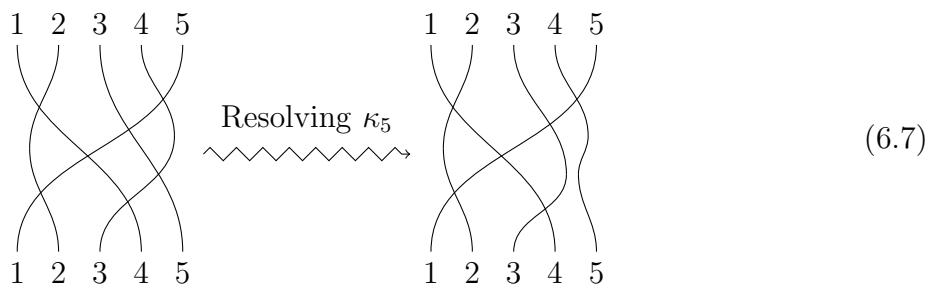
\begin{equation}\label{eq:ExampleUniqueCrossingPossibility}
\begin{tikzpicture}[baseline=(current  bounding  box.center), scale=.55]
\coordinate (B1) at (0,-5);
\coordinate (B2) at (1,-5);
\coordinate (B3) at (2,-5);
\coordinate (B4) at (3,-5);
\coordinate (B5) at (4,-5);

\node[below=0 of B1] {$1$};
\node[below=0 of B2] {$2$};
\node[below=0 of B3] {$3$};
\node[below=0 of B4] {$4$};
\node[below=0 of B5] {$5$};

\coordinate (T1) at (0,0);
\coordinate (T2) at (1,0);
\coordinate (T3) at (2,0);
\coordinate (T4) at (3,0);
\coordinate (T5) at (4,0);

\node[above=0 of T1] {$1$};
\node[above=0 of T2] {$2$};
\node[above=0 of T3] {$3$};
\node[above=0 of T4] {$4$};
\node[above=0 of T5] {$5$};

\draw (B1) to[out=90,in=-90] (T5);
\draw (B2) to[out=90,in=-90] (0.3,-2.5);
\draw (0.3,-2.5) to[out=90,in=-90] (T2);
\draw (B3) to[out=90,in=-90] (3.8,-2.2);
\draw (3.8,-2.2) to[out=90,in=-90] (T4);
\draw (B4) to[out=90,in=-90] (T1);
\draw (B5) to[out=90,in=-90] (T3);

\def\s{10};

\coordinate (BB1) at (0+\s,-5);
\coordinate (BB2) at (1+\s,-5);
\coordinate (BB3) at (2+\s,-5);
\coordinate (BB4) at (3+\s,-5);
\coordinate (BB5) at (4+\s,-5);

\node[below=0 of BB1] {$1$};
\node[below=0 of BB2] {$2$};
\node[below=0 of BB3] {$3$};
\node[below=0 of BB4] {$4$};
\node[below=0 of BB5] {$5$};

\coordinate (TT1) at (0+\s,0);
\coordinate (TT2) at (1+\s,0);
\coordinate (TT3) at (2+\s,0);
\coordinate (TT4) at (3+\s,0);
\coordinate (TT5) at (4+\s,0);

\node[above=0 of TT1] {$1$};
\node[above=0 of TT2] {$2$};
\node[above=0 of TT3] {$3$};
\node[above=0 of TT4] {$4$};
\node[above=0 of TT5] {$5$};

\draw (BB1) to[out=90,in=-90] (TT5);
\draw (BB2) to[out=90,in=-90] (0.3+\s,-2.5);
\draw (0.3+\s,-2.5) to[out=90,in=-90] (TT2);
\draw[draw=none, name path=line3a] (BB3) to[out=90,in=-90] (3.8+\s,-2.2);
\draw (3.8+\s,-2.2) to[out=90,in=-90] (TT4);
\draw (BB4) to[out=90,in=-90] (TT1);
\draw[draw=none, name path=line5] (BB5) to[out=90,in=-90] (TT3);

\path[name intersections={of=line3a and line5,by=E}];
\coordinate[left=0.08 of E] (El); 
\coordinate[right=0.08 of E] (Er); 

\draw (BB5) to[out=90,in=-90] (Er);
\draw (Er) to[out=90,in=-90] (3.8+\s,-2.2);

\draw (BB3) to[out=90,in=-90] (El);
\draw (El) to[out=90,in=-90] (TT3);

\node at (7,-1.8) {Resolving $\kappa_5$};
\draw[-to] decorate[decoration=zigzag] {(4.5,-2.64) -- (9.5,-2.64)};
\end{tikzpicture}
\end{equation}
We record the weights of the crossings of $d_z$ in the following table:
\begin{center}
\begin{tabular}{|c||c|c|c|c|c|c|c|}
\hline
$i$ & $1$ & $2$ & $3$ & $4$ & $5$ & $6$ & $7$\\
\hline
$\mathrm{wt}(\kappa_i)$ & $t_4-t_5$ & $t_1-t_2$ & $t_3-t_5$ & $t_1-t_5$ & $t_3-t_4$ & $t_2-t_5$ & $t_1-t_4$\\
\hline
\end{tabular}
\end{center}
Thus, we have
\[
\iota_{z}^\ast (\mathrm{Stab}_{\mathfrak C_-}(w'))= (t_1-t_3+h)(t_2-t_3+h)(t_2-t_4+h) h\cdot \prod_{i\ne 5} \mathrm{wt}(\kappa_i) 
\]
which implies
\begin{equation}\label{eq:ExampleStableEnvelopesPartialFlag}
\eqref{eq:EquivariantMultiplicityParialExample}= (t_1-t_3+h)(t_2-t_3+h)(t_2-t_4+h) h(t_4-t_5)(t_1-t_2)(t_3-t_5)(t_1-t_5).
\end{equation}
\end{example}

\subsection{Transition invariance} 
Let $\mathcal D(d_1,\ldots,d_m;n)$ be the separated brane diagram from Subsection~\ref{subsec:TFlagAsBow} and $\Phi\circ\Psi\colon F(d_1,\ldots,d_m;n)\xrightarrow\sim \mathcal C(\mathcal D(d_1,\ldots,d_m;n))$ be the isomorphism from \eqref{eq:DefinitionIsoFlagBow} and \eqref{eq:DefinitionHWIsoFlagSeparated}. By Proposition~\ref{prop:HWTransition}, this isomorphism is $\rho$-equivariant, where
\[
\rho\colon\mathbb T\xrightarrow{\phantom{x}\sim\phantom{x}} \mathbb T, \quad (t_1,\ldots,t_n,h)\mapsto (t_1h,\ldots,t_nh,h).
\]
Thus, by Proposition~\ref{prop:MatchingStableEnvelopesHW}, we have
\begin{equation}\label{eq:EquivariantMultiplicitiesHWTwist}
\iota_{D_{wS_{\bm \delta}}}^\ast( \mathrm{Stab}_{\mathfrak C_-}(D_{w'S_{\bm \delta}})) =
\varphi(\iota_{wS_{\bm \delta}}^\ast(\mathrm{Stab}_{\mathfrak C_-}(w'S_{\bm \delta}))),
\quad
\textup{for all $w$, $w'\in S_n$.}
\end{equation}
Here $\varphi\colon\mathbb Q[t_1,\ldots,t_n,h]\xrightarrow\sim \mathbb Q[t_1,\ldots,t_n,h]$ is the $\mathbb Q[h]$-algebra automorphism given by $t_i\mapsto t_i-h$ for all $i$.

The localization formula implies that the localization coefficients of stable basis elements of $\mathcal C(\mathcal D(d_1,\ldots,d_m;n))$ are $\varphi$-invariant:

\begin{prop}\label{prop:StableEnvelopesPartialFlagVarietyInvariantUnderHW}
We have for all $w$, $w'\in S_n$
\[
\iota_{D_{wS_{\bm \delta}}}^\ast( \mathrm{Stab}_{\mathfrak C_-}(D_{w'S_{\bm \delta}})) =
\iota_{wS_{\bm \delta}}^\ast(\mathrm{Stab}_{\mathfrak C_-}(w'S_{\bm \delta})).
\]
\end{prop}
\begin{proof}
Since we have $\varphi(t_i-t_j+mh)=t_i-t_j+mh$ for all $i,j\in\{1,\ldots,n\}$ and $m\in\mathbb Z$, Theorem~\ref{thm:SuLocalizationFormula} and Proposition~\ref{prop:PartialViaFullFlag} imply
\[
\varphi(\iota_{wS_{\bm \delta}}^\ast(\mathrm{Stab}_{\mathfrak C_-}(w'S_{\bm \delta})))=
\iota_{wS_{\bm \delta}}^\ast(\mathrm{Stab}_{\mathfrak C_-}(w'S_{\bm \delta})).
\]
Hence, the proposition follows from \eqref{eq:EquivariantMultiplicitiesHWTwist}.
\end{proof}

\section{Symmetric group calculus for bow varieties}\label{section:TieDiagramsAndDoubleCosets}

Let $\mathcal D$ be a brane diagram and recall the margin vectors $\textbf r$ and $\textbf c$ from Subsection~\ref{subsec:BCT}. Let $n=\sum_{i=1}^Mr_i =\sum_{j=1}^Nc_j$. We denote by
\[
S_{\textbf c}\coloneqq S_{c_1}\times \ldots\times S_{c_N}\subset S_n\quad\textup{and}\quad S_{\textbf r}\coloneqq S_{r_1}\times \ldots\times S_{r_M}\subset S_n
\]
the corresponding Young subgroups.

In this section, we describe a correspondence between the binary contingency tables of $\mathcal D$ and a special class of $(S_{\textbf c},S_{\textbf r})$-double cosets which we call \textit{fully separated}, see Definition~\ref{defi:FullySeparated}. As we will discuss in Subsection~\ref{subsection:UniquenessFullySeparatedness}, permutations that belong to fully separated double cosets satisfy strong uniqueness properties which distinguish fully separated double cosets from general double cosets.

\subsection{Fully separated double cosets}\label{subsection:TieDiagramsAndDoubleCosets}

The usual assignment of a $(S_{\textbf c},S_{\textbf r})$-double coset to a matrix leads to the following well-known bijection, see e.g.~\cite[Theorem~1.3.10]{james1981representation}:

\begin{theorem}\label{thm:BijectionDoubleCosetsMatrices}
Let $\Xi(\textbf r,\textbf c)$ be the set of all $M\times N$-matrices $A$ with entries in $\mathbb Z_{\geq0}$ satisfying
\[
\sum_{l=1}^N A_{i,l}=r_i,\quad \sum_{l=1}^M A_{l,j}=c_j,\quad \textit{for all $i,j$.}
\]
Then, the map $Z\colon S_n\rightarrow \Xi(\textbf c,\textbf r)$ given by
\[
Z(w)_{i,j}=|w(\{R_{i-1}+1,\ldots,R_i\})\cap \{C_{j-1}+1,\ldots,C_j\}|
\]
induces a bijection 
\begin{equation}\label{eq:DoubleCosetBijectionMatrices}
\bar Z\colon S_{\textbf c}\backslash S_n/S_{\textbf r}
\xrightarrow{\phantom{x}\sim\phantom{x}} 
\Xi(\textbf r,\textbf c),\quad S_{\textbf c}wS_{\textbf r}\mapsto Z(w).
\end{equation}
\end{theorem}

By definition, the elements of $\mathrm{bct}(\mathcal D)$ are exactly the matrices in $\Xi(\textbf r,\textbf c)$ with all entries contained in $\{0,1\}$. The following notion characterizes the double cosets that correspond to $\mathrm{bct}(\mathcal D)$ under $\bar Z$:

\begin{definition}\label{defi:FullySeparated}
A permutation $w\in S_n$ is called \textit{fully separated (with respect to $(\textbf r,\textbf c)$)} if
\[
|w(\{R_{i-1}+1,\ldots,R_i\})\cap \{C_{j-1}+1,\ldots,C_j\}|\le 1,
\]
for all $i\in\{1,\ldots,M\}$, $j\in\{1,\ldots,N\}$.
\end{definition} 

If $w$ is fully separated then so is every element in $S_{\textbf c}wS_{\textbf r}$. Hence, we call a double coset $S_{\textbf c}wS_{\textbf r}$ fully separated if all its elements are fully separated. Likewise, we call a left $S_{\textbf r}$-coset (resp. right $S_\textbf c$-coset) fully separated if all its elements are fully separated.

Clearly, a permutation $w$ is fully separated if and only if $Z(w)$ is contained in $\mathrm{bct}(\mathcal D)$. Thus, we have the following corollary:

\begin{cor}\label{cor:DoubleCosetsFSep} The bijection $\bar Z$ from \eqref{eq:DoubleCosetBijectionMatrices} restricts to a bijection
\[
\mathrm{fsep_{\textbf c,\textbf r}}
\xrightarrow{\phantom{x}\sim \phantom{x}}
\mathrm{bct}(\mathcal D),
\]
where $\mathrm{fsep_{\textbf c,\textbf r}}$ denotes the set of fully separated $(S_{\textbf c},S_{\textbf r})$-double cosets.
\end{cor}

\begin{example} 
Let $n=5$ and $\textbf r=(2,2,1), \textbf c=(1,2,2)$.
The permutations $w_1=14253$ and $w_2=14235$ get assigned to the matrices
\[
Z(w_1)=\begin{pmatrix}
1&0&1\\
0&1&1\\
0&1&0
\end{pmatrix},\quad
Z(w_2)=
\begin{pmatrix}
1&0&1\\
0&2&0\\
0&0&1
\end{pmatrix}.
\]
Thus $w_1$ is fully separated, as all entries of $Z(w_1)$ are contained in $\{0,1\}$.
On the other hand, $Z(w_2)$ admits an entry equal to $2$, so $w_2$ is not fully separated.
\end{example}

In diagrammatic language the fully separatedness condition can be reformulated as follows: given $w\in S_N$ and a diagram $d_w$ for $w$. Then, $w$ is fully separated if and only if for all $i\in\{1,\ldots, M\},j\in\{1,\ldots, N\}$, there exists at most one strand with source in $\{R_{i-1}+1,\ldots,R_i\}$ and target in $\{C_{j-1}+1,\ldots,C_j\}$.

\begin{remark} In \cite{james1981representation} the fully separatedness condition is called \textit{trivial intersection property}.
\end{remark}

\subsection{Shortest double coset representatives}\label{subsection:ConstructionShortestRepresentative}

It is well-known, see e.g.~\cite[Section~5.12]{humphreys1992reflection}, that each left coset $wS_{\textbf r}$ (resp.~right coset $S_{\textbf c}w$) contains a unique representative of shortest Bruhat length $w_l$ (resp.~$w_r$) which is uniquely determined by the condition
$
w_l(R_{i-1}+1)<\ldots <w_l(R_{i})
$ for all $i$
(resp. 
$
w_r^{-1}(C_{j-1}+1)<\ldots <w_r^{-1}(C_{j})
$ for all $j$). Likewise, each double coset $S_{\textbf c}wS_{\textbf r}$ contains a unique representative of shortest Bruhat length $w_d$ which is uniquely characterized by the conditions
\[
w_d(R_{i-1}+1)<\ldots <w_d(R_{i})\quad\textup{and} \quad w_d^{-1}(C_{j-1}+1)<\ldots <w_d^{-1}(C_{j}),\quad\textup{for all $i,j$.}
\]

In the following, we describe the shortest representative of $(S_{\textbf c},S_{\textbf r})$-double cosets corresponding to binary contingency tables. We begin with a hopefully intuitive example:

\begin{example}\label{example:ShortestCosetRepresentative}
Let $n=10$, $\textbf r=(3,2,2,3)$, $\textbf c=(2,3,2,1,2)$ and 
\[
A=\begin{pmatrix}
1&1&0&0&1\\
0&0&1&0&1\\
1&1&0&0&0\\
0&1&1&1&0
\end{pmatrix}.
\]
We draw a diagram for the shortest double coset representative of ${\bar Z}^{-1}(A)$ following the next steps: at first, we define functions
\[
F_{A,i}\colon\{1,\ldots, r_i \}\xrightarrow{\phantom{xxx}} \{1,\ldots ,N\},\quad i=1,\ldots,M,
\]
where $F_{A,i}(l)$ is the column index of the $l$-th $1$-entry in the $i$-th row of $A$.
For instance, $F_{A,1}\colon\{1,2,3\}\rightarrow \{1,\ldots,5\}$ is given by $F_{A,1}(1)=1$, $F_{A,1}(2)=2$ and
$F_{A,1}(3)=5$. 

We start drawing our diagram by drawing strands $\lambda_l$ starting in $l=1,\ldots, r_1$ and ending in $C_{F_{A,1}(1)-1}+1,\ldots,C_{F_{A,1}(r_1)-1}+1$. Then, we draw strands $\lambda_{r_1+l}$ starting in $r_1+1,\ldots,r_1+r_2$ and the endpoint of $\lambda_{r_1+l}$ is the smallest element of
$\{C_{F_{A,2}(l)-1}+1,\ldots C_{F_{A,2}}(l)\}$ that is not already the endpoint of a strand. 
Continuing this procedure leads to the following permutation diagram:

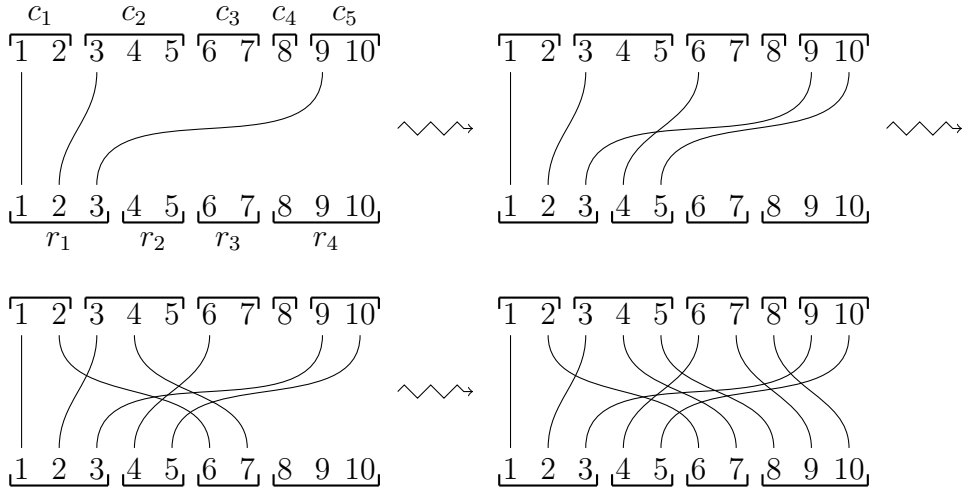
\begin{center}
\begin{tikzpicture}[scale=.5]
\coordinate (B1) at (0,-3);
\coordinate (B2) at (1,-3);
\coordinate (B3) at (2,-3);
\coordinate (B4) at (3,-3);
\coordinate (B5) at (4,-3);
\coordinate (B6) at (5,-3);
\coordinate (B7) at (6,-3);
\coordinate (B8) at (7,-3);
\coordinate (B9) at (8,-3);
\coordinate (B10) at (9,-3);

\node[below=0 of B1] {$1$};
\node[below=0 of B2] {$2$};
\node[below=0 of B3] {$3$};
\node[below=0 of B4] {$4$};
\node[below=0 of B5] {$5$};
\node[below=0 of B6] {$6$};
\node[below=0 of B7] {$7$};
\node[below=0 of B8] {$8$};
\node[below=0 of B9] {$9$};
\node[below=0 of B10] {$10$};

\coordinate (T1) at (0,0);
\coordinate (T2) at (1,0);
\coordinate (T3) at (2,0);
\coordinate (T4) at (3,0);
\coordinate (T5) at (4,0);
\coordinate (T6) at (5,0);
\coordinate (T7) at (6,0);
\coordinate (T8) at (7,0);
\coordinate (T9) at (8,0);
\coordinate (T10) at (9,0);

\node[above=0 of T1] {$1$};
\node[above=0 of T2] {$2$};
\node[above=0 of T3] {$3$};
\node[above=0 of T4] {$4$};
\node[above=0 of T5] {$5$};
\node[above=0 of T6] {$6$};
\node[above=0 of T7] {$7$};
\node[above=0 of T8] {$8$};
\node[above=0 of T9] {$9$};
\node[above=0 of T10] {$10$};

\draw (B1) to[out=90,in=-90] (T1);
\draw (B2) to[out=90,in=-90] (T3);
\draw (B3) to[out=90,in=-90] (T9);

\draw[black, thick] (-0.3,1) -- (1.3,1);
\draw[black, thick] (-0.3,1) -- (-0.3,0.7);
\draw[black, thick] (1.3,1) -- (1.3,0.7);

\draw[black, thick] (1.7,1) -- (4.3,1);
\draw[black, thick] (1.7,1) -- (1.7,0.7);
\draw[black, thick] (4.3,1) -- (4.3,0.7);

\draw[black, thick] (4.7,1) -- (6.3,1);
\draw[black, thick] (4.7,1) -- (4.7,0.7);
\draw[black, thick] (6.3,1) -- (6.3,0.7);

\draw[black, thick] (6.7,1) -- (7.3,1);
\draw[black, thick] (6.7,1) -- (6.7,0.7);
\draw[black, thick] (7.3,1) -- (7.3,0.7);

\draw[black, thick] (7.7,1) -- (9.5,1);
\draw[black, thick] (7.7,1) -- (7.7,0.7);
\draw[black, thick] (9.5,1) -- (9.5,0.7);

\node at (0.5,1.5) {$c_1$};
\node at (3,1.5) {$c_2$};
\node at (5.5,1.5) {$c_3$};
\node at (7,1.5) {$c_4$};
\node at (8.6,1.5) {$c_5$};

\draw[black, thick] (-0.3,-4) -- (2.3,-4);
\draw[black, thick] (-0.3,-4) -- (-0.3,-3.7);
\draw[black, thick] (2.3,-4) -- (2.3,-3.7);

\draw[black, thick] (2.7,-4) -- (4.3,-4);
\draw[black, thick] (2.7,-4) -- (2.7,-3.7);
\draw[black, thick] (4.3,-4) -- (4.3,-3.7);

\draw[black, thick] (4.7,-4) -- (6.3,-4);
\draw[black, thick] (4.7,-4) -- (4.7,-3.7);
\draw[black, thick] (6.3,-4) -- (6.3,-3.7);

\draw[black, thick] (6.7,-4) -- (9.5,-4);
\draw[black, thick] (6.7,-4) -- (6.7,-3.7);
\draw[black, thick] (9.5,-4) -- (9.5,-3.7);

\node at (1,-4.5) {$r_1$};
\node at (3.5,-4.5) {$r_2$};
\node at (5.5,-4.5) {$r_3$};
\node at (8.1,-4.5) {$r_4$};

\draw[-to] decorate[decoration=zigzag] {(10,-1.5) -- (12,-1.5)};

\def \s {13};

\coordinate (B1) at (0+\s,-3);
\coordinate (B2) at (1+\s,-3);
\coordinate (B3) at (2+\s,-3);
\coordinate (B4) at (3+\s,-3);
\coordinate (B5) at (4+\s,-3);
\coordinate (B6) at (5+\s,-3);
\coordinate (B7) at (6+\s,-3);
\coordinate (B8) at (7+\s,-3);
\coordinate (B9) at (8+\s,-3);
\coordinate (B10) at (9+\s,-3);

\node[below=0 of B1] {$1$};
\node[below=0 of B2] {$2$};
\node[below=0 of B3] {$3$};
\node[below=0 of B4] {$4$};
\node[below=0 of B5] {$5$};
\node[below=0 of B6] {$6$};
\node[below=0 of B7] {$7$};
\node[below=0 of B8] {$8$};
\node[below=0 of B9] {$9$};
\node[below=0 of B10] {$10$};

\coordinate (T1) at (0+\s,0);
\coordinate (T2) at (1+\s,0);
\coordinate (T3) at (2+\s,0);
\coordinate (T4) at (3+\s,0);
\coordinate (T5) at (4+\s,0);
\coordinate (T6) at (5+\s,0);
\coordinate (T7) at (6+\s,0);
\coordinate (T8) at (7+\s,0);
\coordinate (T9) at (8+\s,0);
\coordinate (T10) at (9+\s,0);

\node[above=0 of T1] {$1$};
\node[above=0 of T2] {$2$};
\node[above=0 of T3] {$3$};
\node[above=0 of T4] {$4$};
\node[above=0 of T5] {$5$};
\node[above=0 of T6] {$6$};
\node[above=0 of T7] {$7$};
\node[above=0 of T8] {$8$};
\node[above=0 of T9] {$9$};
\node[above=0 of T10] {$10$};

\draw (B1) to[out=90,in=-90] (T1);
\draw (B2) to[out=90,in=-90] (T3);
\draw (B3) to[out=90,in=-90] (T9);
\draw (B4) to[out=90,in=-90] (T6);
\draw (B5) to[out=90,in=-90] (T10);

\draw[black, thick] (-0.3+\s,1) -- (1.3+\s,1);
\draw[black, thick] (-0.3+\s,1) -- (-0.3+\s,0.7);
\draw[black, thick] (1.3+\s,1) -- (1.3+\s,0.7);

\draw[black, thick] (1.7+\s,1) -- (4.3+\s,1);
\draw[black, thick] (1.7+\s,1) -- (1.7+\s,0.7);
\draw[black, thick] (4.3+\s,1) -- (4.3+\s,0.7);

\draw[black, thick] (4.7+\s,1) -- (6.3+\s,1);
\draw[black, thick] (4.7+\s,1) -- (4.7+\s,0.7);
\draw[black, thick] (6.3+\s,1) -- (6.3+\s,0.7);

\draw[black, thick] (6.7+\s,1) -- (7.3+\s,1);
\draw[black, thick] (6.7+\s,1) -- (6.7+\s,0.7);
\draw[black, thick] (7.3+\s,1) -- (7.3+\s,0.7);

\draw[black, thick] (7.7+\s,1) -- (9.5+\s,1);
\draw[black, thick] (7.7+\s,1) -- (7.7+\s,0.7);
\draw[black, thick] (9.5+\s,1) -- (9.5+\s,0.7);

\draw[black, thick] (-0.3+\s,-4) -- (2.3+\s,-4);
\draw[black, thick] (-0.3+\s,-4) -- (-0.3+\s,-3.7);
\draw[black, thick] (2.3+\s,-4) -- (2.3+\s,-3.7);

\draw[black, thick] (2.7+\s,-4) -- (4.3+\s,-4);
\draw[black, thick] (2.7+\s,-4) -- (2.7+\s,-3.7);
\draw[black, thick] (4.3+\s,-4) -- (4.3+\s,-3.7);

\draw[black, thick] (4.7+\s,-4) -- (6.3+\s,-4);
\draw[black, thick] (4.7+\s,-4) -- (4.7+\s,-3.7);
\draw[black, thick] (6.3+\s,-4) -- (6.3+\s,-3.7);

\draw[black, thick] (6.7+\s,-4) -- (9.5+\s,-4);
\draw[black, thick] (6.7+\s,-4) -- (6.7+\s,-3.7);
\draw[black, thick] (9.5+\s,-4) -- (9.5+\s,-3.7);

\draw[-to] decorate[decoration=zigzag] {(10+\s,-1.5) -- (12+\s,-1.5)};

\def \s {0};
\def \t {7};

\coordinate (B1) at (0+\s,-3-\t);
\coordinate (B2) at (1+\s,-3-\t);
\coordinate (B3) at (2+\s,-3-\t);
\coordinate (B4) at (3+\s,-3-\t);
\coordinate (B5) at (4+\s,-3-\t);
\coordinate (B6) at (5+\s,-3-\t);
\coordinate (B7) at (6+\s,-3-\t);
\coordinate (B8) at (7+\s,-3-\t);
\coordinate (B9) at (8+\s,-3-\t);
\coordinate (B10) at (9+\s,-3-\t);

\node[below=0 of B1] {$1$};
\node[below=0 of B2] {$2$};
\node[below=0 of B3] {$3$};
\node[below=0 of B4] {$4$};
\node[below=0 of B5] {$5$};
\node[below=0 of B6] {$6$};
\node[below=0 of B7] {$7$};
\node[below=0 of B8] {$8$};
\node[below=0 of B9] {$9$};
\node[below=0 of B10] {$10$};

\coordinate (T1) at (0+\s,0-\t);
\coordinate (T2) at (1+\s,0-\t);
\coordinate (T3) at (2+\s,0-\t);
\coordinate (T4) at (3+\s,0-\t);
\coordinate (T5) at (4+\s,0-\t);
\coordinate (T6) at (5+\s,0-\t);
\coordinate (T7) at (6+\s,0-\t);
\coordinate (T8) at (7+\s,0-\t);
\coordinate (T9) at (8+\s,0-\t);
\coordinate (T10) at (9+\s,0-\t);

\node[above=0 of T1] {$1$};
\node[above=0 of T2] {$2$};
\node[above=0 of T3] {$3$};
\node[above=0 of T4] {$4$};
\node[above=0 of T5] {$5$};
\node[above=0 of T6] {$6$};
\node[above=0 of T7] {$7$};
\node[above=0 of T8] {$8$};
\node[above=0 of T9] {$9$};
\node[above=0 of T10] {$10$};

\draw (B1) to[out=90,in=-90] (T1);
\draw (B2) to[out=90,in=-90] (T3);
\draw (B3) to[out=90,in=-90] (T9);
\draw (B4) to[out=90,in=-90] (T6);
\draw (B5) to[out=90,in=-90] (T10);
\draw (B6) to[out=90,in=-90] (T2);
\draw (B7) to[out=90,in=-90] (T4);

\draw[black, thick] (-0.3+\s,1-\t) -- (1.3+\s,1-\t);
\draw[black, thick] (-0.3+\s,1-\t) -- (-0.3+\s,0.7-\t);
\draw[black, thick] (1.3+\s,1-\t) -- (1.3+\s,0.7-\t);

\draw[black, thick] (1.7+\s,1-\t) -- (4.3+\s,1-\t);
\draw[black, thick] (1.7+\s,1-\t) -- (1.7+\s,0.7-\t);
\draw[black, thick] (4.3+\s,1-\t) -- (4.3+\s,0.7-\t);

\draw[black, thick] (4.7+\s,1-\t) -- (6.3+\s,1-\t);
\draw[black, thick] (4.7+\s,1-\t) -- (4.7+\s,0.7-\t);
\draw[black, thick] (6.3+\s,1-\t) -- (6.3+\s,0.7-\t);

\draw[black, thick] (6.7+\s,1-\t) -- (7.3+\s,1-\t);
\draw[black, thick] (6.7+\s,1-\t) -- (6.7+\s,0.7-\t);
\draw[black, thick] (7.3+\s,1-\t) -- (7.3+\s,0.7-\t);

\draw[black, thick] (7.7+\s,1-\t) -- (9.5+\s,1-\t);
\draw[black, thick] (7.7+\s,1-\t) -- (7.7+\s,0.7-\t);
\draw[black, thick] (9.5+\s,1-\t) -- (9.5+\s,0.7-\t);

\draw[black, thick] (-0.3+\s,-4-\t) -- (2.3+\s,-4-\t);
\draw[black, thick] (-0.3+\s,-4-\t) -- (-0.3+\s,-3.7-\t);
\draw[black, thick] (2.3+\s,-4-\t) -- (2.3+\s,-3.7-\t);

\draw[black, thick] (2.7+\s,-4-\t) -- (4.3+\s,-4-\t);
\draw[black, thick] (2.7+\s,-4-\t) -- (2.7+\s,-3.7-\t);
\draw[black, thick] (4.3+\s,-4-\t) -- (4.3+\s,-3.7-\t);

\draw[black, thick] (4.7+\s,-4-\t) -- (6.3+\s,-4-\t);
\draw[black, thick] (4.7+\s,-4-\t) -- (4.7+\s,-3.7-\t);
\draw[black, thick] (6.3+\s,-4-\t) -- (6.3+\s,-3.7-\t);

\draw[black, thick] (6.7+\s,-4-\t) -- (9.5+\s,-4-\t);
\draw[black, thick] (6.7+\s,-4-\t) -- (6.7+\s,-3.7-\t);
\draw[black, thick] (9.5+\s,-4-\t) -- (9.5+\s,-3.7-\t);

\def \s {13};
\def \t {7};

\draw[-to] decorate[decoration=zigzag] {(10,-1.5-\t) -- (12,-1.5-\t)};

\coordinate (B1) at (0+\s,-3-\t);
\coordinate (B2) at (1+\s,-3-\t);
\coordinate (B3) at (2+\s,-3-\t);
\coordinate (B4) at (3+\s,-3-\t);
\coordinate (B5) at (4+\s,-3-\t);
\coordinate (B6) at (5+\s,-3-\t);
\coordinate (B7) at (6+\s,-3-\t);
\coordinate (B8) at (7+\s,-3-\t);
\coordinate (B9) at (8+\s,-3-\t);
\coordinate (B10) at (9+\s,-3-\t);

\node[below=0 of B1] {$1$};
\node[below=0 of B2] {$2$};
\node[below=0 of B3] {$3$};
\node[below=0 of B4] {$4$};
\node[below=0 of B5] {$5$};
\node[below=0 of B6] {$6$};
\node[below=0 of B7] {$7$};
\node[below=0 of B8] {$8$};
\node[below=0 of B9] {$9$};
\node[below=0 of B10] {$10$};

\coordinate (T1) at (0+\s,0-\t);
\coordinate (T2) at (1+\s,0-\t);
\coordinate (T3) at (2+\s,0-\t);
\coordinate (T4) at (3+\s,0-\t);
\coordinate (T5) at (4+\s,0-\t);
\coordinate (T6) at (5+\s,0-\t);
\coordinate (T7) at (6+\s,0-\t);
\coordinate (T8) at (7+\s,0-\t);
\coordinate (T9) at (8+\s,0-\t);
\coordinate (T10) at (9+\s,0-\t);

\node[above=0 of T1] {$1$};
\node[above=0 of T2] {$2$};
\node[above=0 of T3] {$3$};
\node[above=0 of T4] {$4$};
\node[above=0 of T5] {$5$};
\node[above=0 of T6] {$6$};
\node[above=0 of T7] {$7$};
\node[above=0 of T8] {$8$};
\node[above=0 of T9] {$9$};
\node[above=0 of T10] {$10$};

\draw (B1) to[out=90,in=-90] (T1);
\draw (B2) to[out=90,in=-90] (T3);
\draw (B3) to[out=90,in=-90] (T9);
\draw (B4) to[out=90,in=-90] (T6);
\draw (B5) to[out=90,in=-90] (T10);
\draw (B6) to[out=90,in=-90] (T2);
\draw (B7) to[out=90,in=-90] (T4);
\draw (B8) to[out=90,in=-90] (T5);
\draw (B9) to[out=90,in=-90] (T7);
\draw (B10) to[out=90,in=-90] (T8);

\draw[black, thick] (-0.3+\s,1-\t) -- (1.3+\s,1-\t);
\draw[black, thick] (-0.3+\s,1-\t) -- (-0.3+\s,0.7-\t);
\draw[black, thick] (1.3+\s,1-\t) -- (1.3+\s,0.7-\t);

\draw[black, thick] (1.7+\s,1-\t) -- (4.3+\s,1-\t);
\draw[black, thick] (1.7+\s,1-\t) -- (1.7+\s,0.7-\t);
\draw[black, thick] (4.3+\s,1-\t) -- (4.3+\s,0.7-\t);

\draw[black, thick] (4.7+\s,1-\t) -- (6.3+\s,1-\t);
\draw[black, thick] (4.7+\s,1-\t) -- (4.7+\s,0.7-\t);
\draw[black, thick] (6.3+\s,1-\t) -- (6.3+\s,0.7-\t);

\draw[black, thick] (6.7+\s,1-\t) -- (7.3+\s,1-\t);
\draw[black, thick] (6.7+\s,1-\t) -- (6.7+\s,0.7-\t);
\draw[black, thick] (7.3+\s,1-\t) -- (7.3+\s,0.7-\t);

\draw[black, thick] (7.7+\s,1-\t) -- (9.5+\s,1-\t);
\draw[black, thick] (7.7+\s,1-\t) -- (7.7+\s,0.7-\t);
\draw[black, thick] (9.5+\s,1-\t) -- (9.5+\s,0.7-\t);

\draw[black, thick] (-0.3+\s,-4-\t) -- (2.3+\s,-4-\t);
\draw[black, thick] (-0.3+\s,-4-\t) -- (-0.3+\s,-3.7-\t);
\draw[black, thick] (2.3+\s,-4-\t) -- (2.3+\s,-3.7-\t);

\draw[black, thick] (2.7+\s,-4-\t) -- (4.3+\s,-4-\t);
\draw[black, thick] (2.7+\s,-4-\t) -- (2.7+\s,-3.7-\t);
\draw[black, thick] (4.3+\s,-4-\t) -- (4.3+\s,-3.7-\t);

\draw[black, thick] (4.7+\s,-4-\t) -- (6.3+\s,-4-\t);
\draw[black, thick] (4.7+\s,-4-\t) -- (4.7+\s,-3.7-\t);
\draw[black, thick] (6.3+\s,-4-\t) -- (6.3+\s,-3.7-\t);

\draw[black, thick] (6.7+\s,-4-\t) -- (9.5+\s,-4-\t);
\draw[black, thick] (6.7+\s,-4-\t) -- (6.7+\s,-3.7-\t);
\draw[black, thick] (9.5+\s,-4-\t) -- (9.5+\s,-3.7-\t);
\end{tikzpicture}
\end{center}
We denote the resulting permutation by $\tilde w_A$, i.e. $\tilde w_A=13961024578$. Our condition to pick always the smallest entry in
$\{C_{j-1}+1,\ldots,C_j \}$ that is not already the endpoint of a strand implies 
$
\tilde w_A^{-1}(C_{j-1}+1)<\ldots <\tilde w_A^{-1}(C_{j})
$, for all $j$. In addition, as the functions $F_{A,i}$ strictly increase, we also have
$
\tilde w_A(R_{i-1}+1)<\ldots <\tilde w_A(R_{i})
$, for all $i$. Thus, $\tilde w_A$ is a shortest $(S_{\textbf c},S_{\textbf r})$-double coset representative. As $\tilde w_A$ satisfies 
$\tilde w_A(R_{i-1}+l)\in \{C_{F_{A,i}(l)-1}+1,\ldots C_{F_{A,i}(l)}\}$ for all $i,l$, we conclude
\[
Z(\tilde w_A)_{R_{i-1}+l,j}= 
\begin{cases}
1 &\textup{if }j=F_{A,i}(l),\\
0 &\textup{if }j\ne F_{A,i}(l).
\end{cases}
\]
Therefore, $Z(\tilde w_A)=A$ which implies that $\tilde w_A$ is indeed the shortest representative of $\bar Z^{-1}(A)$.
\end{example}

We return to the general setup: let $\mathcal D$ be a brane diagram and $A\in\mathrm{bct}(\mathcal D)$. 

As in the previous example, let 
$
F_{A,i}\colon\{1,\ldots, r_i \}\rightarrow \{1,\ldots ,N\}
$
be the function assigning to $l$ the column index of the $l$-th $1$-entry in the $i$-th row of $A$.
Likewise, let $
G_{A,j}\colon\{1,\ldots, c_j \}\rightarrow \{1,\ldots ,M\}
$
be the function assigning to $l$ the row index of the $l$-th $1$-entry in the $j$-th column of $A$.
Furthermore, we set
$
n_{A,i,j}= \sum_{l=1}^{i}A_{l,j}
$. That is, $n_{A,i,j}$ is the number of $1$-entries that are in the $j$-th column of $A$ and strictly above the entry $A_{i+1,j}$.

\begin{definition}\label{definition:DefinitionTildeWA}
We define the permutation $\tilde w_A\in S_n$ as
\[
\tilde w_A(R_{i-1}+l)= C_{F_{A,i}(l)-1}+n_{A,i,F_{A,i(l)}},\quad\textup{for $i=1,\ldots,M$ and $l=1,\ldots, r_i$.}
\]
\end{definition}

To see that $\tilde w_A$ is a permutation, note that if we are given $j\in\{1,\ldots,N\}$ and $l\in \{1,\ldots, c_j\}$ then let $i=G_{A,j}(l)$ and $l'\in\{1,\ldots, r_i\}$ such that
$A_{i,j}$ is the $l'$-th $1$-entry in the $i$-th row of $A$. By construction, we have $\tilde w_A(R_{i-1}+l')=C_{j-1}+l$ which proves the surjectivity of $\tilde w_A$. Hence, $\tilde w_A$ is a permutation.

The next proposition shows that $\tilde w_A$ satisfies the desired properties.
\begin{prop}\label{prop:PropertiesOfTildeWA}
The following holds:
\begin{enumerate}[label=(\roman*)]
\item\label{item:tildeWAFullySep} $Z(\tilde w_A)=A$,
\item\label{item:tildeWAShortest} $\tilde w_A$ is the shortest representative of $\bar Z^{-1}(A)$,
\item\label{item:tildeWALength} we have $\ell(\tilde w_A)=|\mathrm{Inv}(A)|,$ where
\[
\mathrm{Inv}(A) =\{((i_1,j_1),(i_2,j_2))\mid A_{i_1,j_1}=A_{i_2,j_2}=1,i_1<i_2,j_2<j_1\}.
\]
\end{enumerate}
\end{prop}
\begin{proof} By construction, $\tilde w_A(R_{i-1}+l)\in \{C_{F_{A,i}(l)-1}+1,\ldots,C_{F_{A,i}(l)} \}$ which gives
\[
Z(\tilde w_A)_{R_{i-1}+l,j}= 
\begin{cases}
1 &\textup{if }j=F_{A,i}(l),\\
0 &\textup{otherwise.}
\end{cases}
\]
Hence, $Z(\tilde w_A)=A$. Moreover, we conclude
$
\tilde w_A(R_{i-1}+1)<\ldots<\tilde w_A(R_i)
$
for all $i$. The explicit description of $\tilde w_A$ above gives $\tilde w_A^{-1}(C_{j-1}+l)\in \{R_{G_{A,j}(l)-1}+1,\ldots,R_{G_{A,j}(l)} \}$ which implies $
\tilde w_A^{-1}(C_{j-1}+1)<\ldots<\tilde w_A^{-1}(C_j)
$
for all $j$. Thus, $\tilde w_A$ is the shortest representative of $\bar Z^{-1}(A)$.
Finally, note that since $\tilde w_A$ is a shortest left $S_{\textbf r}$-coset representative, the inversions of $\tilde w_A$ are exactly the ordered pairs $(R_{i_1}+l_1,R_{i_2}+l_2)$ with 
\[
1\le i_1<i_2\le M,\quad 1\le l_1 \le r_{i_1},\quad 1\le l_2 \le r_{i_2}, \quad F_{A,i_1}(l_1)>F_{A,i_2}(l_2).
\]
It follows that we have a bijection
$
\mathrm{Inv}(\tilde w_D)\xrightarrow\sim \mathrm{Inv}(D),
$
where an inversion $(R_{i_1}+l_1,R_{i_2}+l_2)$ of $\tilde w_A$ is mapped to $((i_1,F_{D,i_1}(l_1)),(i_2,F_{D,i_2}(l_2))$. Hence, $\ell(\tilde w_D)=|\mathrm{Inv}(D)|$.
\end{proof}

\begin{example}\label{ex:RealityCheckExercise} Let $w$, $w'\in S_5$ be as in Example~\ref{example:StabsOfCotangentPartialFlag} and choose $\textbf r=(2,2,1)$, $\textbf c=(2,1,2)$. Then, we have
\[
Z(w)= \begin{pmatrix}
1&0&1\\
1&0&1\\
0&1&0
\end{pmatrix},\quad Z(w')=
 \begin{pmatrix}
1&0&1\\
1&1&0\\
0&0&1
\end{pmatrix}.
\]
We leave it as an exercise to the reader to that $\tilde w_{Z(w)}=14253$ and $\tilde w_{Z(w')}=14235$.
\end{example}

\subsection{Uniqueness properties}\label{subsection:UniquenessFullySeparatedness}

In this subsection, we discuss strong uniqueness properties of fully separated permutations that distinguish them from general permutations. The central result is the following proposition:

\begin{prop}\label{prop:UniquenessOfPrePostComposition} Assume $w\in S_n$ is fully separated. Let $v$, $v'\in S_{\textbf r}$ and $u$, $u'\in S_{\textbf c}$ such that
$
uwv=u'wv'.
$
Then, $u=u'$ and $v=v'$.
\end{prop}

Before we prove Proposition~\ref{prop:UniquenessOfPrePostComposition}, we illustrate the idea of the proof in the following example:

\begin{example}\label{example:fsepReconstruction} Let $n,\textbf r,\textbf c$ and $A$ be as in Example~\ref{example:ShortestCosetRepresentative}. For a permutation $w\in S_n$, we define the function
\[
F_{w}\colon\{1,\ldots,n\}\xrightarrow{\phantom{xxx}} \{1,\ldots,N\}\quad i\mapsto F_w(i),
\]
where $F_w(i)$ is the unique element in $\{1,\ldots, N\}$ such that
$
C_{F_w(i)-1}+1\le w(i) \le C_{F_w(i)}.
$
In terms of diagrammatic calculus, the function $F_w$ can be characterized as follows: pick a diagram for $w$. Then, on the top, draw $N$ square brackets around the intervals $\{1,\ldots,C_1\},\{C_1+1,\ldots,C_2\},\ldots,\{C_{N-1}+1,\ldots,C_N\}$ and label them with $1,\ldots,N$ from left to right. Then, $F_w(i)$ is the index of the square bracket containing the endpoint of the unique strand starting in $i$.

Let for instance $v=v_1\times v_2\times v_3\times v_4 \in S_{\textbf r}$, where $v_1=312$, $v_2=21$, $v_3=12$, $v_4=231$.
A diagram for $\tilde w_Av$ is given by
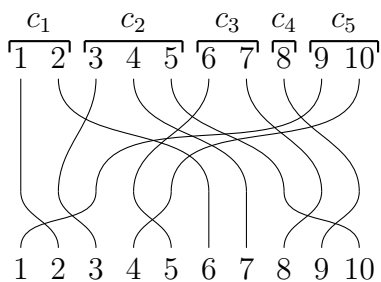
\begin{center}
\begin{tikzpicture}[scale=.5]
\coordinate (B1) at (0,-3);
\coordinate (B2) at (1,-3);
\coordinate (B3) at (2,-3);
\coordinate (B4) at (3,-3);
\coordinate (B5) at (4,-3);
\coordinate (B6) at (5,-3);
\coordinate (B7) at (6,-3);
\coordinate (B8) at (7,-3);
\coordinate (B9) at (8,-3);
\coordinate (B10) at (9,-3);

\coordinate (BB1) at (0,-4.5);
\coordinate (BB2) at (1,-4.5);
\coordinate (BB3) at (2,-4.5);
\coordinate (BB4) at (3,-4.5);
\coordinate (BB5) at (4,-4.5);
\coordinate (BB6) at (5,-4.5);
\coordinate (BB7) at (6,-4.5);
\coordinate (BB8) at (7,-4.5);
\coordinate (BB9) at (8,-4.5);
\coordinate (BB10) at (9,-4.5);

\node[below=0 of BB1] {$1$};
\node[below=0 of BB2] {$2$};
\node[below=0 of BB3] {$3$};
\node[below=0 of BB4] {$4$};
\node[below=0 of BB5] {$5$};
\node[below=0 of BB6] {$6$};
\node[below=0 of BB7] {$7$};
\node[below=0 of BB8] {$8$};
\node[below=0 of BB9] {$9$};
\node[below=0 of BB10] {$10$};

\coordinate (T1) at (0,0);
\coordinate (T2) at (1,0);
\coordinate (T3) at (2,0);
\coordinate (T4) at (3,0);
\coordinate (T5) at (4,0);
\coordinate (T6) at (5,0);
\coordinate (T7) at (6,0);
\coordinate (T8) at (7,0);
\coordinate (T9) at (8,0);
\coordinate (T10) at (9,0);

\node[above=0 of T1] {$1$};
\node[above=0 of T2] {$2$};
\node[above=0 of T3] {$3$};
\node[above=0 of T4] {$4$};
\node[above=0 of T5] {$5$};
\node[above=0 of T6] {$6$};
\node[above=0 of T7] {$7$};
\node[above=0 of T8] {$8$};
\node[above=0 of T9] {$9$};
\node[above=0 of T10] {$10$};

\draw (B1) to[out=90,in=-90] (T1);
\draw (B2) to[out=90,in=-90] (T3);
\draw (B3) to[out=90,in=-90] (T9);
\draw (B4) to[out=90,in=-90] (T6);
\draw (B5) to[out=90,in=-90] (T10);
\draw (B6) to[out=90,in=-90] (T2);
\draw (B7) to[out=90,in=-90] (T4);
\draw (B8) to[out=90,in=-90] (T5);
\draw (B9) to[out=90,in=-90] (T7);
\draw (B10) to[out=90,in=-90] (T8);

\draw (BB1) to[out=90,in=-90] (B3);
\draw (BB2) to[out=90,in=-90] (B1);
\draw (BB3) to[out=90,in=-90] (B2);
\draw (BB4) to[out=90,in=-90] (B5);
\draw (BB5) to[out=90,in=-90] (B4);
\draw (BB6) to[out=90,in=-90] (B6);
\draw (BB7) to[out=90,in=-90] (B7);
\draw (BB8) to[out=90,in=-90] (B9);
\draw (BB9) to[out=90,in=-90] (B10);
\draw (BB10) to[out=90,in=-90] (B8);

\draw[black, thick] (-0.3,1) -- (1.3,1);
\draw[black, thick] (-0.3,1) -- (-0.3,0.7);
\draw[black, thick] (1.3,1) -- (1.3,0.7);

\draw[black, thick] (1.7,1) -- (4.3,1);
\draw[black, thick] (1.7,1) -- (1.7,0.7);
\draw[black, thick] (4.3,1) -- (4.3,0.7);

\draw[black, thick] (4.7,1) -- (6.3,1);
\draw[black, thick] (4.7,1) -- (4.7,0.7);
\draw[black, thick] (6.3,1) -- (6.3,0.7);

\draw[black, thick] (6.7,1) -- (7.3,1);
\draw[black, thick] (6.7,1) -- (6.7,0.7);
\draw[black, thick] (7.3,1) -- (7.3,0.7);

\draw[black, thick] (7.7,1) -- (9.5,1);
\draw[black, thick] (7.7,1) -- (7.7,0.7);
\draw[black, thick] (9.5,1) -- (9.5,0.7);

\node at (0.5,1.5) {$c_1$};
\node at (3,1.5) {$c_2$};
\node at (5.5,1.5) {$c_3$};
\node at (7,1.5) {$c_4$};
\node at (8.6,1.5) {$c_5$};
\end{tikzpicture}
\end{center}
The functions $F_{\tilde w_A}$ and $F_{\tilde w_Av}$ can be easily read of from their diagrams:
\begin{center}
\begin{tabular}{| c | c | c | c | c | c | c | c | c | c | c |}
  \hline			
  $i$ & $1$ & $2$ & $3$ & $4$ & $5$ & $6$ & $7$ & $8$  & $9$  & $10$ \\ \hline\hline
  $F_{\tilde w_A}(i)$  & $1$ & $2$ & $5$ & $3$ & $5$ & $1$ & $2$ & $2$ & $3$ & $4$ \\ \hline
  $F_{\tilde w_Av}(i)$ & $5$ & $1$ & $2$ & $5$ & $3$ & $1$ & $2$ & $3$ & $4$ & $5$\\
  \hline  
\end{tabular}
\end{center}
Next, we show that if we know $F_{\tilde w_Av}$ then we can reconstruct the permutation $v$. We begin with reconstructing the factor $v_1$. 
The first three letters in the row of $F_{\tilde w_Av}$ give the word $512$ then using the identification $1\mapsto 1$, $2\mapsto 2$, $5\mapsto 3$, we see that $512$ corresponds to $312=v_1$. Next, the fourth and the fifth letters in the row of $F_{\tilde w_Av}$ give the word $53$. Using the identification $3\mapsto 1$, $5\mapsto 2$, we get the word $21=v_2$. In the same way one can construct $v_3$ and $v_4$ and thus the permutation $v$.

In our reasoning, the fully separatedness property was essential because this property ensures that the restriction of $F_{\tilde w_Av}$ to $\{1,2,3\},\{4,5\},\{6,7\},\{8,9,10\}$ is injective.
\end{example}

We proceed with the general setup. As in Example~\ref{example:fsepReconstruction} we define for given
$w\in S_n$ the function
\begin{equation}\label{eq:DefinitionFW}
F_w\colon\{1,\ldots,n\}\xrightarrow{\phantom{xxx}} \{1,\ldots, N\},\quad i\mapsto F_w(i),
\end{equation}
where $F_w(i)$ is the unique element on $\{1,\ldots, N\}$ such that
\[
C_{F_w(i)-1}+1\le w(i) \le C_{F_w(i)}.
\]
Likewise, we define
\[
G_w\colon\{1,\ldots,n\}\xrightarrow{\phantom{xxx}} \{1,\ldots, M\},\quad j\mapsto G_w(j),
\]
where $G_w(j)$ is the unique element on $\{1,\ldots, M\}$ such that
\[
R_{G_w(j)-1}+1\le w^{-1}(j) \le R_{G_w(j)}.
\]
Similarly as $F_w$, the function $G_w$ admits the following diagrammatic interpretation: pick a diagram for $w$ and draw $M$ square brackets on the bottom around the discrete intervals $\{1,\ldots,R_1\},\{R_1+1,\ldots,R_2\},\ldots,\{R_{M-1}+1,\ldots,R_M\}$. Label the square brackets with $1,\ldots,M$ from left to right. Then, $G_w(j)$ is the index of the square bracket containing the starting point of the unique strand with endpoint $j$.

If $w$ is fully separated then the restrictions of $F_w$ to the sets
$
\{R_{i-1}+1,\ldots, R_i\}
$
is injective for $i=1,\ldots,M$. Likewise, the restriction of $G_w$ to
$
\{C_{j-1}+1,\ldots, C_j\}
$
is injective for $j=1,\ldots,N$. Moreover, the following properties are satisfied:

\begin{lemma}\label{lemma:MatchingFunctionsProperties} Assume $w\in S_n$ is fully separated.
Given $u\in S_{\textbf c}$ and $v\in S_{\textbf r}$ then we have 
\begin{enumerate}[label=(\roman*)]
\item\label{item:MatchingFProperty1} $F_{uw}=F_w$,
\item\label{item:MatchingGProperty1} $G_{wv}=G_w$,
\item\label{item:MatchingFProperty2} $F_{wv}=F_w$ if and only if $v=\operatorname{id}$,
\item\label{item:MatchingGProperty2} $G_{uw}=G_w$ if and only if $u=\operatorname{id}$.
\end{enumerate}
\end{lemma}
\begin{proof}
Since $u$ leaves the sets $\{C_{j-1}+1,\ldots,C_j\}$ invariant, we get~\ref{item:MatchingFProperty1}. Likewise, $v$ leaves the sets $\{R_{i-1}+1,\ldots,R_i\}$ invariant which gives~\ref{item:MatchingGProperty1}.
For~\ref{item:MatchingFProperty2}, suppose that $v\ne\operatorname{id}$ and $F_{wv}=F_w$. Let $l\in\{1,\ldots,n\}$ such that $v(l)\ne l$. Choose $i\in\{1,\ldots,M\}$ such that $l\in \{R_{i-1}+1,\ldots, R_i\}$. As $v\in S_{\textbf r}$, we have that $v(l)$ is also contained in $\{R_{i-1}+1,\ldots, R_i\}$. In addition, we have by definition $F_{wv}(l)=F_{w}(v(l))$ and hence $F_{w}(v(l))=F_w(l)$ which contradicts the fact that the restriction of $F_w$ to $\{R_{i-1}+1,\ldots, R_i\}$ is injective. The proof of~\ref{item:MatchingGProperty2} is analogous.
\end{proof}

\begin{proof}[Proof of Proposition~\ref{prop:UniquenessOfPrePostComposition}]
Without loss of generality, $u'=\operatorname{id}$, $v'=\operatorname{id}$. By Lemma~\ref{lemma:MatchingFunctionsProperties}.\ref{item:MatchingFProperty1},
$
F_{w}=F_{uwv}=F_{wv}
$.
Thus, Lemma~\ref{lemma:MatchingFunctionsProperties}.\ref{item:MatchingFProperty2} implies $v=\operatorname{id}$. Likewise, Lemma~\ref{lemma:MatchingFunctionsProperties}.\ref{item:MatchingGProperty1} gives $G_w=G_{uwv}=G_{uw}$ which implies $u=\operatorname{id}$ by Lemma~\ref{lemma:MatchingFunctionsProperties}.\ref{item:MatchingGProperty2}.
\end{proof}

Given $A\in \mathrm{bct}(\mathcal D)$ then by construction, we have
\begin{equation}\label{eq:MatchingFunctionsFMatrixVSPermutation}
F_{\tilde w_A}(R_{i-1}+l)=F_{A,i}(l),\quad \textup{for $i=1,\ldots,M,$ $l=R_{i-1}+1,\ldots,R_i$.}
\end{equation}
This observation combined with Lemma~\ref{lemma:MatchingFunctionsProperties} leads to the following characterization of shortest representatives of fully separated left resp. right cosets:

\begin{cor}\label{cor:TildeWDGivesMinimalLeftAndRightCosets} Given $A\in \mathrm{bct}(\mathcal D)$ then the following holds:
\begin{enumerate}[label=(\roman*)]
\item\label{item:TildeWDGivesMinimalLeft} $u\tilde w_A$ is a shortest left $S_{\textbf r}$-coset representative for all $u\in S_{\textbf c}$,
\item\label{item:TildeWDGivesMinimalLeftExistence} if $wS_{\textbf r}$ is a fully separated left coset then there exist $A\in\mathrm{bct}(\mathcal D),u\in S_{\textbf c}$ such that $u\tilde w_A$ is the shortest representative of $wS_{\textbf r}$,
\item\label{item:TildeWDGivesMinimalRight} $\tilde w_Av$ is a shortest right $S_{\textbf c}$-coset representative for all $v\in S_{\textbf r}$,
\item\label{item:TildeWDGivesMinimalRightExistence} if $S_{\textbf c}w$ is a fully separated right coset then there exist $A\in\mathrm{bct}(\mathcal D),v\in S_{\textbf r}$ such that $\tilde w_Av$ is the shortest representative of $S_{\textbf c}w$.
\end{enumerate}
\end{cor}

\begin{proof} By Lemma~\ref{lemma:MatchingFunctionsProperties}.\ref{item:MatchingFProperty1} and \eqref{eq:MatchingFunctionsFMatrixVSPermutation}, we have
\[
F_{u\tilde w_A}(R_{i-1}+1)<F_{u\tilde w_A}(R_{i-1}+2)<\ldots < F_{u\tilde w_A}(R_{i}),\quad \textup{for $i=1,\ldots,M$.}
\]
This implies $u\tilde w_A(R_{i-1}+1)<u\tilde w_A(R_{i-1}+2)<\ldots < u\tilde w_A(R_{i})$ for all $i$. Hence, $u\tilde w_A$ is the shortest representative of $u\tilde w_AS_{\textbf r}$ which gives \ref{item:TildeWDGivesMinimalLeft}. For \ref{item:TildeWDGivesMinimalLeftExistence}, we use that if $wS_{\textbf r}$ is fully separated then $Z(w)\in \mathrm{bct}(\mathcal D)$ and there exist $u\in S_{\textbf c},v\in S_{\textbf r}$ such that $w=u\tilde w_{Z(w)}v$. Thus, \ref{item:TildeWDGivesMinimalLeft} gives that $u\tilde w_{Z(w)}$ is the shortest representative of $wS_{\textbf r}$. The proofs of~\ref{item:TildeWDGivesMinimalRight} and \ref{item:TildeWDGivesMinimalRightExistence} are analogous.
\end{proof}

\section{Equivariant resolution theorem}\label{section:EquivariantResolutionTheorem}

In this section, we discuss a version of the equivariant resolution theorem for bow varieties of Botta and Rim\'anyi~\cite[Theorem~6.13]{botta2023mirror}. This theorem allows to express localization coefficients of stable basis elements of bow varieties in terms of localization coefficients of stable basis elements of cotangent bundles of partial flag varieties.

%In this section, we discuss a version of the equivariant resolution theorem for bow varieties of Botta and Rim\'anyi \cite[Theorem~6.13]{botta2023mirror}, see . This theorem allows to express localization coefficients of stable basis elements of bow varieties in terms of localization coefficients of stable basis elements of cotangent bundles of partial flag varieties. 

We begin with describing the underlying combinatorics.

\subsection{Resolutions of brane and tie diagrams}

Let $\mathcal D$ be an essential separated brane diagram where we denote the numbers $d_X$ as follows:
\[
\begin{tikzpicture}
[scale=.475]
\draw[thick] (0,0)--(2,0);
\draw[thick] (3,0)--(5,0);
\draw[thick] (6,0)--(8,0);
\draw[dotted] (9,0)--(11,0);
\draw[thick] (12,0)--(14,0);
\draw[thick] (15,0)--(17,0);
\draw[thick] (18,0)--(20,0);
\draw[dotted] (21,0)--(23,0);
\draw[thick] (24,0)--(26,0);
\draw[thick] (27,0)--(29,0);
\draw[thick] (30,0)--(32,0);

\draw [thick,red] (2,-1) --(3,1); 
\draw [thick,red] (5,-1) --(6,1); 
\draw [thick,red] (8,-1) --(9,1);
\draw [thick,red] (11,-1) --(12,1);
\draw [thick,red] (14,-1) --(15,1);

\draw [thick,blue] (18,-1) --(17,1); 
\draw [thick,blue] (21,-1) --(20,1); 
\draw [thick,blue] (24,-1) --(23,1);
\draw [thick,blue] (27,-1) --(26,1); 
\draw [thick,blue] (30,-1) --(29,1); 

\node at (1,0.6){$0$};
\node at (4,0.6){$\bar R_{M-1}$};
\node at (7,0.6){$\bar R_{M-2}$};
\node at (13,0.6){$\bar R_1$};
\node at (16,0.6){$n$};
\node at (19,0.6){$\bar C_1$};
\node at (25,0.6){$\bar C_{N-2}$};
\node at (28,0.6){$\bar C_{N-1}$};
\node at (31,0.6){$0$};
\end{tikzpicture}
\]
We have
\[
\bar R_i=\sum_{l=M-i}^{M} r_{l},\quad i=1,\ldots,M-1,\quad \bar C_j=\sum_{l=j+1}^{N} c_{l},\quad
j=1,\ldots,N-1.
\]
Note that the assumption that $\mathcal D$ is admissible and essential gives
that all $r_i$ and $c_j$ are nonzero.

\begin{definition}
The \textit{resolution $\mathrm{Res}(\mathcal D)$ of $\mathcal D$} is the brane diagram defined as
\[
\begin{tikzpicture}
[scale=.475]
\draw[thick] (0,0)--(2,0);
\draw[thick] (3,0)--(5,0);
\draw[thick] (6,0)--(8,0);
\draw[dotted] (9,0)--(11,0);
\draw[thick] (12,0)--(14,0);
\draw[thick] (15,0)--(17,0);
\draw[thick] (18,0)--(20,0);
\draw[dotted] (21,0)--(23,0);
\draw[thick] (24,0)--(26,0);
\draw[thick] (27,0)--(29,0);
\draw[thick] (30,0)--(32,0);

\draw [thick,red] (2,-1) --(3,1); 
\draw [thick,red] (5,-1) --(6,1); 
\draw [thick,red] (8,-1) --(9,1);
\draw [thick,red] (11,-1) --(12,1);
\draw [thick,red] (14,-1) --(15,1);

\draw [thick,blue] (18,-1) --(17,1); 
\draw [thick,blue] (21,-1) --(20,1); 
\draw [thick,blue] (24,-1) --(23,1);
\draw [thick,blue] (27,-1) --(26,1); 
\draw [thick,blue] (30,-1) --(29,1); 

\node at (1,0.6){$0$};
\node at (4,0.6){$\bar R_{M-1}$};
\node at (7,0.6){$\bar R_{M-2}$};
\node at (13,0.6){$\bar R_1$};
\node at (16,0.6){$n$};
\node at (19,0.6){$n-1$};
\node at (25,0.6){$2$};
\node at (28,0.6){$1$};
\node at (31,0.6){$0$};
\end{tikzpicture}
\]
\end{definition}

That is, the resolved brane diagram $\mathrm{Res}(\mathcal D)$ is obtained from $\mathcal D$ by replacing the part
$n\textcolor{blue}{\backslash}\bar C_1\textcolor{blue}{\backslash}\ldots\textcolor{blue}{\backslash} \bar C_{N-1}\textcolor{blue}{\backslash}0$ with
$n\textcolor{blue}{\backslash}n-1\textcolor{blue}{\backslash}\ldots\textcolor{blue}{\backslash}2\textcolor{blue}{\backslash}1\textcolor{blue}{\backslash}0$. 
Thus, with the notation from Subsection~\ref{subsec:TFlagAsBow}, $\mathrm{Res}(\mathcal D)$ is equal to the brane diagram $\mathcal D(R_1,\ldots,R_{M-1};n)$.

Given $u=u_1\times \cdots\times u_N\in S_{\textbf c}$, we obtain an inclusion
\begin{equation}\label{eq:ResolutionForTieDiagrams}
\begin{tikzcd}
\mathrm{Res}_u\colon\mathrm{Tie}(\mathcal D) \arrow[hookrightarrow]{r} &\mathrm{Tie}(\mathrm{Res}(\mathcal D)),
\end{tikzcd}
\end{equation}
where for a tie diagram $D\in \mathrm{Tie}(\mathcal D)$ the resolved tie diagram $\mathrm{Res}_u(D)$ is obtained via performing at each blue line $U_i$ with ties the local move: 
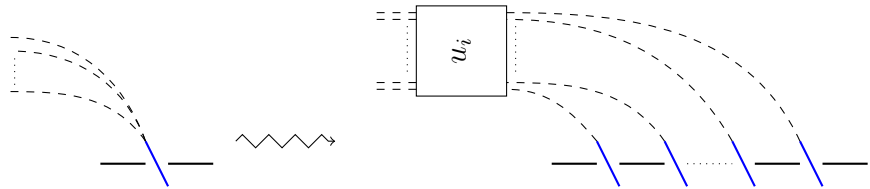
\begin{center}
\begin{tikzpicture}[scale=.3]
\def\s{-20}
\draw[thick] (0+\s,0)--(2+\s,0);
\draw [thick,blue] (3+\s,-1) --(2+\s,1); 
\draw[thick] (3+\s,0)--(5+\s,0);

\draw[dashed] (2+\s,1) to[out=110, in=0] (-4+\s,5);
\draw[dashed] (2+\s,1) to[out=110, in=0] (-4+\s,5.6);
\draw[dashed] (2+\s,1) to[out=125, in=0] (-4+\s,3.2);
\draw[dotted] (-3.8+\s,3.5) to (-3.8+\s,4.7);

\draw[-to] decorate [decoration=zigzag]{(6+\s,1) -- (-9.6,1)};

\draw[thick] (0,0)--(2,0);
\draw[thick] (3,0)--(5,0);
\draw[dotted] (6,0)--(8,0);
\draw[thick] (9,0)--(11,0);
\draw[thick] (12,0)--(14,0);

\draw [thick,blue] (3,-1) --(2,1);
\draw [thick,blue] (6,-1) --(5,1);
\draw [thick,blue] (9,-1) --(8,1);
\draw [thick,blue] (12,-1) --(11,1);

\node[rotate=90] (w) at (-4,5) {$u_i$};
\draw[draw=black] (-6,3) rectangle ++(4,4);

\draw[dashed] (2,1) to[out=130, in=0] (-2,3.3);
\draw[dashed] (5,1) to[out=125, in=0] (-2,3.6);
\draw[dashed] (8,1) to[out=120, in=0] (-2,6.4);
\draw[dashed] (11,1) to[out=115, in=0] (-2,6.7);

\draw[dotted] (-1.6,6.1) to (-1.6,3.9);
\draw[dotted] (-6.4,6.1) to (-6.4,3.9);

\draw[dashed] (-6,3.3) to (-8,3.3);
\draw[dashed] (-6,3.6) to (-8,3.6);
\draw[dashed] (-6,6.4) to (-8,6.4);
\draw[dashed] (-6,6.7) to (-8,6.7);
\end{tikzpicture}
\end{center}
Here, the box around $u_i$ represents an arbitrary diagram for $u_i$.

\begin{example}
Let $D$ be the tie diagram:
\begin{center}
\begin{tikzpicture}
[scale=.3]
\draw[thick] (0,0)--(2,0);
\draw[thick] (3,0)--(5,0);
\draw[thick] (6,0)--(8,0);
\draw[thick] (9,0)--(11,0);
\draw[thick] (12,0)--(14,0);
\draw[thick] (15,0)--(17,0);
\draw[thick] (18,0)--(20,0);

\draw [thick,red] (2,-1) -- (3,1);
\draw [thick,red] (5,-1) -- (6,1);
\draw [thick,red] (8,-1) -- (9,1);
\draw [thick,red] (11,-1) -- (12,1);
\draw [thick,blue] (15,-1) -- (14,1);
\draw [thick,blue] (18,-1) -- (17,1);

\node at (1,0.7){$0$};
\node at (4,0.7){$1$};
\node at (7,0.7){$2$};
\node at (10,0.7){$3$};
\node at (13,0.7){$5$};
\node at (16,0.7){$3$};
\node at (19,0.7){$0$};

\draw[dashed] (3,1) to[out=45,in=135] (17,1);
\draw[dashed] (6,1) to[out=45,in=135] (14,1);
\draw[dashed] (9,1) to[out=45,in=135] (17,1);
\draw[dashed] (12,1) to[out=45,in=135] (17,1);
\draw[dashed] (12,1) to[out=45,in=135] (14,1);
\node at (28,0.5) {$\displaystyle {\begin{pmatrix}
 1 & 1 \\
 0 & 1 \\
 1 & 0 \\
 0 & 1
\end{pmatrix}}$};
\end{tikzpicture}
\end{center}
We choose $u_1=21$ and $u_2=231$. Then, $\mathrm{Res}_u(D)$ is given by

\begin{center}
\begin{tikzpicture}
[scale=.3]
\draw[thick] (0,0)--(2,0);
\draw[thick] (3,0)--(5,0);
\draw[thick] (6,0)--(8,0);
\draw[thick] (9,0)--(11,0);
\draw[thick] (12,0)--(16,0);
\draw[thick] (17,0)--(19,0);
\draw[thick] (20,0)--(22,0);
\draw[thick] (23,0)--(25,0);
\draw[thick] (26,0)--(28,0);
\draw[thick] (29,0)--(31,0);

\draw [thick,red] (2,-1) -- (3,1);
\draw [thick,red] (5,-1) -- (6,1);
\draw [thick,red] (8,-1) -- (9,1);
\draw [thick,red] (11,-1) -- (12,1);
\draw [thick,blue] (17,-1) -- (16,1);
\draw [thick,blue] (20,-1) -- (19,1);
\draw [thick,blue] (23,-1) -- (22,1);
\draw [thick,blue] (26,-1) -- (25,1);
\draw [thick,blue] (29,-1) -- (28,1);

\node at (1,0.7){$0$};
\node at (4,0.7){$1$};
\node at (7,0.7){$2$};
\node at (10,0.7){$3$};
\node at (14,0.7){$5$};

\node at (18,0.7){$4$};
\node at (21,0.7){$3$};
\node at (24,0.7){$2$};
\node at (27,0.7){$1$};
\node at (30,0.7){$0$};

\draw[dashed] (6,1) to[out=45,in=180] (13,2.5) to[out=0,in=180] (15,1.5) to[out=0,in=135] (16,1);
\draw[dashed] (12,1) to[out=45,in=180] (13,1.5) to[out=0,in=180] (15,2.5) to[out=0,in=135] (19,1);

\draw[dashed] (12,1) to[out=45,in=180] (15,3.5) to[out=0,in=180] (18,5.5) to[out=0,in=135] (28,1);
\draw[dashed] (9,1) to[out=45,in=180] (15,4.5) to[out=0,in=180] (18,3.5) to[out=0,in=135] (22,1);
\draw[dashed] (3,1) to[out=45,in=180] (15,5.5) to[out=0,in=180] (18,4.5) to[out=0,in=135] (25,1);

\node at (40,1.5) {$\displaystyle{\begin{pmatrix}
0 & 1 & 0 & 0 & 1\\
0 & 0 & 1 & 0 & 0\\
1 & 0 & 0 & 0 & 0\\
0 & 0 & 0 & 1 & 0
\end{pmatrix}}$};

\end{tikzpicture}
\end{center}
Here, we see the rotated diagrams for $u_1$ and $u_2$. The diagram for $u_1$  involves the ties of the first and the second blue line and the diagram for $u_2$ involves the ties of the third, fourth and fifth blue line.
\end{example}

For $D\in\mathrm{Tie}(\mathcal D)$, we set $\tilde w_D\coloneqq \tilde w_{M(D)}$, where $\tilde w_{M(D)}$ is the permutation defined in Definition~\ref{definition:DefinitionTildeWA}. 

In terms of permutations, we can characterize $\mathrm{Res}_u(D)$ as follows:

\begin{prop}\label{prop:ResolutionsAndPermutations}
We have
\[
\mathrm{Res}_u(D) = D_{u\tilde w_D S_{\textbf r}},
\]
where $D_{u\tilde w_D S_{\textbf r}}$ is defined as in Subsection~\ref{subsection:DepBraneDiagPartialFlag}.
\end{prop}
\begin{proof}
Suppose the blue line $U_j$ in $D$ is connected to the red lines $V_{i_1},\ldots V_{i_{c_j}}$, where $i_1<\ldots<i_{c_j}$. 
Then, in $\mathrm{Res}_u(D)$, the blue line $U_{C_{j-1}+l}$ is connected to $V_{i_{u(l)}}$ for $l=1,\ldots,c_j$. On the other hand, by construction of $\tilde w_D$, we have
$
\tilde w_{D}^{-1}(C_{j-1}+l)\in \{R_{i_l-1}+1,\ldots,R_{i_1}\}
$
for all $l$. Thus, $\tilde w_{D}^{-1}u^{-1}(C_{j-1}+l)\in \{R_{i_{u_j(l)}-1}+1,\ldots,R_{i_{u_j(l)}}\}$ which proves the claim.
\end{proof}

\subsection{Equivariant resolution theorem}\label{subsection:EquivariantResolutionTheorem}
The equivariant resolution theorem connects the localization coefficients of stable basis elements of $\mathcal D$ and $\mathrm{Res}(\mathcal D)$ as follows:

\begin{theorem}[Equivariant resolution theorem]\label{thm:EquivariantResolution}
Let $D$ and $D'$ be tie diagrams of $\mathcal D$. Then,
\[
\Big(\prod_{i=1}^N\prod_{j=1}^{c_i-1} (jh)^{c_i-j}\Big) \iota_{D}^\ast (\widetilde{\mathrm{Stab}}_{\mathfrak C_-}(D'))
= \Psi_{\mathcal D}( \iota_{\mathrm{Res}_{u_0}(D)}^\ast (\mathrm{Stab}_{\mathfrak C_-}(\mathrm{Res}_{u}(D')))),
\]
where $\widetilde{\mathrm{Stab}}_{\mathfrak C_-}$ is the normalized version of $\mathrm{Stab}_{\mathfrak C_-}$ from \eqref{equation:DefinitionStabTilde}, $u\in S_{\textbf c}$, $u_0=w_{0,c_1}\times \ldots \times w_{0,c_N}$ and $w_{0,l}$ denotes the longest element in $S_l$ for all $l$.
The $\mathbb Q[h]$-algebra homomorphism \[ 
\Psi_{\mathcal D}\colon\mathbb Q[t_1,\ldots,t_n,h]\xrightarrow{\phantom{xxx}} \mathbb Q[t_1,\ldots,t_N,h] \]
is given by
\[
\Psi_{\mathcal D}(t_{C_{i-1}+k})=t_i-(k-1)h,\quad \textit{for $i=1,\ldots,N$, $k=1,\ldots,c_i$}.
\]
\end{theorem}

Applying Proposition~\ref{prop:StableEnvelopesPartialFlagVarietyInvariantUnderHW} and Proposition~\ref{prop:ResolutionsAndPermutations} then directly implies the following:

\begin{cor}\label{cor:ResolvedStableEnvelopesFormula}
With the notation of Theorem~\ref{thm:EquivariantResolution}, we have
\[
\Big(\prod_{i=1}^N\prod_{j=1}^{c_i-1} (jh)^{c_i-j}\Big) \iota_{D}^\ast (\widetilde{\mathrm{Stab}}_{\mathfrak C_-}(D'))
= \Psi_{\mathcal D}( \iota_{w_DS_{\textbf r}}^\ast (\mathrm{Stab}_{\mathfrak C_-}(w'S_{\textbf r}))),
\]
where the stable basis element on the right is on $T^\ast F(R_1,\ldots,R_{M-1};n)$, $w'\in S_{\textbf c} \tilde w_{D'} S_{\textbf r}$ and 
$
w_D=(w_{0,c_1}\times \ldots \times w_{0,c_N}) \tilde w_D
$.
\end{cor}
\begin{proof} By Proposition~\ref{prop:ResolutionsAndPermutations}, $\mathrm{Res}_{u_0}(D)=D_{w_DS_{\textbf r}}$ and $\mathrm{Res}_{u}(D')=D_{w'S_{\textbf r}}$ for some $u'\in S_{\textbf c}$.
Thus, Proposition~\ref{prop:StableEnvelopesPartialFlagVarietyInvariantUnderHW} yields
\[
\iota_{w_DS_{\textbf r}}^\ast (\mathrm{Stab}_{\mathfrak C_-}(w'S_{\textbf r}))
=
\iota_{\mathrm{Res}_{u_0}(D)}^\ast (\mathrm{Stab}_{\mathfrak C_-}(\mathrm{Res}_{u}(D'))).
\]
Hence, the claim follows from the equivariant resolution theorem.
\end{proof}

\begin{example} Given the tie diagrams
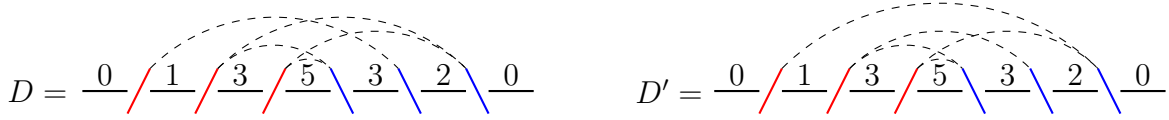
\begin{center}
\begin{tikzpicture}
[scale=.3]
\draw[thick] (0,0)--(2,0);
\draw[thick] (3,0)--(5,0);
\draw[thick] (6,0)--(8,0);
\draw[thick] (9,0)--(11,0);
\draw[thick] (12,0)--(14,0);
\draw[thick] (15,0)--(17,0);
\draw[thick] (18,0)--(20,0);

\draw [thick,red] (2,-1) -- (3,1);
\draw [thick,red] (5,-1) -- (6,1);
\draw [thick,red] (8,-1) -- (9,1);

\draw [thick,blue] (12,-1) -- (11,1);
\draw [thick,blue] (15,-1) -- (14,1);
\draw [thick,blue] (18,-1) -- (17,1);

\draw[dashed] (3,1) to[out=45, in=135] (14,1);
\draw[dashed] (6,1) to[out=45, in=135] (11,1);
\draw[dashed] (6,1) to[out=45, in=135] (17,1);
\draw[dashed] (9,1) to[out=45, in=135] (11,1);
\draw[dashed] (9,1) to[out=45, in=135] (17,1);

\node at (1,0.7) {$0$};
\node at (4,0.7) {$1$};
\node at (7,0.7) {$3$};
\node at (10,0.7) {$5$};
\node at (13,0.7) {$3$};
\node at (16,0.7) {$2$};
\node at (19,0.7) {$0$};

\node at (-2,0) {$D=$};

\def\s{28};

\node at (-2+\s,0) {$D'=$};
\node at (1+\s,0.7) {$0$};
\node at (4+\s,0.7) {$1$};
\node at (7+\s,0.7) {$3$};
\node at (10+\s,0.7) {$5$};
\node at (13+\s,0.7) {$3$};
\node at (16+\s,0.7) {$2$};
\node at (19+\s,0.7) {$0$};

\draw[thick] (0+\s,0)--(2+\s,0);
\draw[thick] (3+\s,0)--(5+\s,0);
\draw[thick] (6+\s,0)--(8+\s,0);
\draw[thick] (9+\s,0)--(11+\s,0);
\draw[thick] (12+\s,0)--(14+\s,0);
\draw[thick] (15+\s,0)--(17+\s,0);
\draw[thick] (18+\s,0)--(20+\s,0);

\draw [thick,red] (2+\s,-1) -- (3+\s,1);
\draw [thick,red] (5+\s,-1) -- (6+\s,1);
\draw [thick,red] (8+\s,-1) -- (9+\s,1);

\draw [thick,blue] (12+\s,-1) -- (11+\s,1);
\draw [thick,blue] (15+\s,-1) -- (14+\s,1);
\draw [thick,blue] (18+\s,-1) -- (17+\s,1);

\draw[dashed] (3+\s,1) to[out=45, in=135] (17+\s,1);
\draw[dashed] (6+\s,1) to[out=45, in=135] (11+\s,1);
\draw[dashed] (6+\s,1) to[out=45, in=135] (14+\s,1);
\draw[dashed] (9+\s,1) to[out=45, in=135] (11+\s,1);
\draw[dashed] (9+\s,1) to[out=45, in=135] (17+\s,1);
\end{tikzpicture}
\end{center}
In the following, we compute the localization coefficient $\iota_{D}^\ast (\widetilde{\mathrm{Stab}}_{\mathfrak C_-}(D'))$. Note that $n=5$, $\textbf r=(2,2,1)$ and $\textbf c=(2,1,2)$.
Let $w$, $w'\in S_5$ be as in Example~\ref{example:StabsOfCotangentPartialFlag}. Then, we have $Z(w)=M(D)$ and $Z(w')=M(D')$. Hence, we know from Example~\ref{example:StabsOfCotangentPartialFlag} that
$\tilde w_D=14253$ and $\tilde w_{D'}=14235$. Thus, $w=(s\times \operatorname{id}\times s)\tilde w_D$, where $s=12\in S_2$ which gives $w=w_D$. Therefore, by Corollary~\ref{cor:ResolvedStableEnvelopesFormula}, we have 
\begin{equation}\label{eq:ExampleEquivariantMultiplicityViaResolution}
h^2\iota_{D}^\ast (\widetilde{\mathrm{Stab}}_{\mathfrak C_-}(D'))
=
\Psi_{\mathcal D}( \iota_{wS_{\textbf r}}^\ast ({\mathrm{Stab}}_{\mathfrak C_-}(w'S_{\textbf r}))),
\end{equation}
where \[
\Psi_{\mathcal D}\colon\mathbb Q[t_1,t_2,t_3,t_4,t_5,h]\xrightarrow{\phantom{xxx}}\mathbb Q[t_1,t_2,t_3,h]
\]
is the $\mathbb Q[h]$-algebra homomorphism given by
$t_1\mapsto t_1$, $t_2\mapsto t_1-h$, $t_3\mapsto t_2$, $t_4\mapsto t_3$, $t_5\mapsto t_3-h$. 
From~\eqref{eq:ExampleStableEnvelopesPartialFlag}, we know
\begin{equation*}
 \iota_{wS_{\textbf r}}^\ast ({\mathrm{Stab}}_{\mathfrak C_-}(w'S_{\textbf r}))
 = h
 (t_1-t_3+h)(t_2-t_3+h)(t_2-t_4+h) (t_4-t_5)(t_1-t_2)(t_3-t_5)(t_1-t_5),
\end{equation*}
which yields
\begin{equation}
\label{eq:ExampleEquivariantMultiplicityViaResolution2}
\Psi_{\mathcal D}( \iota_{wS_{\textbf r}}^\ast ({\mathrm{Stab}}_{\mathfrak C_-}(w'S_{\textbf r})))
=
h^3 (t_1-t_2+h)(t_1-t_2)(t_1-t_2+h)(t_2-t_3)(t_1-t_3).
\end{equation}
Finally, inserting~\eqref{eq:ExampleEquivariantMultiplicityViaResolution2} in \eqref{eq:ExampleEquivariantMultiplicityViaResolution} gives
\[
\iota_{D}^\ast (\widetilde{\mathrm{Stab}}_{\mathfrak C_-}(D'))
=
h (t_1-t_2+h)(t_1-t_2)(t_1-t_2+h)(t_2-t_3)(t_1-t_3).
\]
\end{example}

\begin{remark}
The equivariant resolution theorem as stated in \cite[Theorem~6.13]{botta2023mirror} connects the stable basis elements of different bow varieties which differ by resolving just one blue line. As explained in e.g.~\cite[Section~9.6]{wehrhanphd}, Theorem~\ref{thm:EquivariantResolution} can be deduced from \cite[Theorem~6.13]{botta2023mirror} by an inductive procedure.
\end{remark}

\subsection{Approximations of localization coefficients}\label{subsection:ApproximationReducedStableEnvelopes}
Next, we combine the diagrammatic localization formula and Corollary~\ref{cor:ResolvedStableEnvelopesFormula} to approximate localization coefficients of stable basis elements modulo powers of $h$.

For this, we like to choose the reduced diagrams for permutations of a particular form: let $w=u_0\circ\tilde w_D\circ (v_1\times\cdots\times v_M)$, where, as in Theorem~\ref{thm:EquivariantResolution}, $u_0=w_{0,c_1}\times \ldots\times w_{0,c_N}$ and $v_j\in S_{r_j}$ is an arbitrary element for $j=1,\ldots,M$. By Corollary~\ref{cor:TildeWDGivesMinimalLeftAndRightCosets}, we can choose a reduced diagram for $w$ of the form:
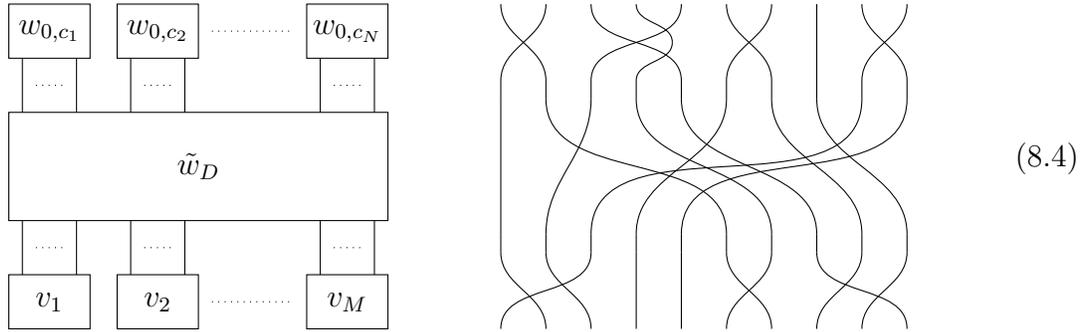
\begin{equation}\label{eq:ShapeOfDiagrams}
\begin{tikzpicture}[baseline=(current  bounding  box.center), scale=.6]
%Boxes
\node(ru1) at (0,0) {$w_{0,c_1}$};
\draw[draw=black] (-0.9,-0.6) rectangle ++(1.8,1.2);
\node (ru2) at (2.4,0) {$w_{0,c_2}$};
\draw[draw=black] (1.5,-0.6) rectangle ++(1.8,1.2);
\node (ruN) at (6.6,0) {$w_{0,c_N}$};
\draw[draw=black] (5.7,-0.6) rectangle ++(1.8,1.2);
\node (rw) at (3.3,-3) {$\tilde w_{D}$};
\draw[draw=black] (-0.9,-4.2) rectangle ++(8.4,2.4);
\node (rv1) at (0,-6) {$v_1$};
\draw[draw=black] (-0.9,-6.6) rectangle ++(1.8,1.2);
\node (rv2) at (2.4,-6) {$v_2$};
\draw[draw=black] (1.5,-6.6) rectangle ++(1.8,1.2);
\node (rvM) at (6.6,-6) {$v_M$};
\draw[draw=black] (5.7,-6.6) rectangle ++(1.8,1.2);

%Strands
\draw (-0.6,-0.6) -- (-0.6,-1.8);
\draw (0.6,-0.6) -- (0.6,-1.8);

\draw (1.8,-0.6) -- (1.8,-1.8);
\draw (3.0,-0.6) -- (3.0,-1.8);

\draw (6,-0.6) -- (6,-1.8);
\draw (7.2,-0.6) -- (7.2,-1.8);

\draw (-0.6,-4.2) -- (-0.6,-5.4);
\draw (0.6,-4.2) -- (0.6,-5.4);

\draw (1.8,-4.2) -- (1.8,-5.4);
\draw (3.0,-4.2) -- (3.0,-5.4);

\draw (6,-4.2) -- (6,-5.4);
\draw (7.2,-4.2) -- (7.2,-5.4);

%Dots
\draw[dotted] (-0.3,-1.2)--(0.3,-1.2);
\draw[dotted] (2.1,-1.2)--(2.7,-1.2);
\draw[dotted] (6.3,-1.2)--(6.9,-1.2);
\draw[dotted] (3.6,0)--(5.4,0);
\draw[dotted] (-0.3,-4.8)--(0.3,-4.8);
\draw[dotted] (2.1,-4.8)--(2.7,-4.8);
\draw[dotted] (6.3,-4.8)--(6.9,-4.8);
\draw[dotted] (3.6,-6)--(5.4,-6);

\def \s {10};
\def \t {1.5};

\coordinate (B1) at (0+\s,-3-\t);
\coordinate (B2) at (1+\s,-3-\t);
\coordinate (B3) at (2+\s,-3-\t);
\coordinate (B4) at (3+\s,-3-\t);
\coordinate (B5) at (4+\s,-3-\t);
\coordinate (B6) at (5+\s,-3-\t);
\coordinate (B7) at (6+\s,-3-\t);
\coordinate (B8) at (7+\s,-3-\t);
\coordinate (B9) at (8+\s,-3-\t);
\coordinate (B10) at (9+\s,-3-\t);

\coordinate (BP1) at (0+\s,-3.4-\t);
\coordinate (BP2) at (1+\s,-3.4-\t);
\coordinate (BP3) at (2+\s,-3.4-\t);
\coordinate (BP4) at (3+\s,-3.4-\t);
\coordinate (BP5) at (4+\s,-3.4-\t);
\coordinate (BP6) at (5+\s,-3.4-\t);
\coordinate (BP7) at (6+\s,-3.4-\t);
\coordinate (BP8) at (7+\s,-3.4-\t);
\coordinate (BP9) at (8+\s,-3.4-\t);
\coordinate (BP10) at (9+\s,-3.4-\t);

\coordinate (BB1) at (0+\s,-5.1-\t);
\coordinate (BB2) at (1+\s,-5.1-\t);
\coordinate (BB3) at (2+\s,-5.1-\t);
\coordinate (BB4) at (3+\s,-5.1-\t);
\coordinate (BB5) at (4+\s,-5.1-\t);
\coordinate (BB6) at (5+\s,-5.1-\t);
\coordinate (BB7) at (6+\s,-5.1-\t);
\coordinate (BB8) at (7+\s,-5.1-\t);
\coordinate (BB9) at (8+\s,-5.1-\t);
\coordinate (BB10) at (9+\s,-5.1-\t);

\foreach \x in {1,...,10}
	\draw[black] (B\x) -- (BP\x);
	
\draw[black] (BB1) to[out=90, in=-90] (BP3);
\draw[black] (BB2) to[out=90, in=-90] (BP1);
\draw[black] (BB3) to[out=90, in=-90] (BP2);

\draw[black] (BB4) to[out=90, in=-90] (BP4);
\draw[black] (BB5) to[out=90, in=-90] (BP5);

\draw[black] (BB6) to[out=90, in=-90] (BP7);
\draw[black] (BB7) to[out=90, in=-90] (BP6);

\draw[black] (BB8) to[out=90, in=-90] (BP9);
\draw[black] (BB9) to[out=90, in=-90] (BP10);
\draw[black] (BB10) to[out=90, in=-90] (BP8);

\coordinate (T1) at (0+\s,0-\t);
\coordinate (T2) at (1+\s,0-\t);
\coordinate (T3) at (2+\s,0-\t);
\coordinate (T4) at (3+\s,0-\t);
\coordinate (T5) at (4+\s,0-\t);
\coordinate (T6) at (5+\s,0-\t);
\coordinate (T7) at (6+\s,0-\t);
\coordinate (T8) at (7+\s,0-\t);
\coordinate (T9) at (8+\s,0-\t);
\coordinate (T10) at (9+\s,0-\t);

\coordinate (TP1) at (0+\s,0.4-\t);
\coordinate (TP2) at (1+\s,0.4-\t);
\coordinate (TP3) at (2+\s,0.4-\t);
\coordinate (TP4) at (3+\s,0.4-\t);
\coordinate (TP5) at (4+\s,0.4-\t);
\coordinate (TP6) at (5+\s,0.4-\t);
\coordinate (TP7) at (6+\s,0.4-\t);
\coordinate (TP8) at (7+\s,0.4-\t);
\coordinate (TP9) at (8+\s,0.4-\t);
\coordinate (TP10) at (9+\s,0.4-\t);

\foreach \x in {1,...,10}
	\draw[black] (T\x) -- (TP\x);

\coordinate (TT1) at (0+\s,2.1-\t);
\coordinate (TT2) at (1+\s,2.1-\t);
\coordinate (TT3) at (2+\s,2.1-\t);
\coordinate (TT4) at (3+\s,2.1-\t);
\coordinate (TT5) at (4+\s,2.1-\t);
\coordinate (TT6) at (5+\s,2.1-\t);
\coordinate (TT7) at (6+\s,2.1-\t);
\coordinate (TT8) at (7+\s,2.1-\t);
\coordinate (TT9) at (8+\s,2.1-\t);
\coordinate (TT10) at (9+\s,2.1-\t);

\coordinate (TTAux) at (3.8+\s,1.25-\t);

\draw[black] (TP1) to[out=90, in=-90]  (TT2);
\draw[black] (TP2) to[out=90, in=-90]  (TT1);

\draw[black] (TP3) to[out=90, in=-90]  (TT5);
\draw[black] (TP5) to[out=90, in=-90]  (TT3);
\draw[black] (TP4) to[out=90, in=-90]  (TTAux);
\draw[black] (TTAux) to[out=90, in=-90]  (TT4);

\draw[black] (TP6) to[out=90, in=-90]  (TT7);
\draw[black] (TP7) to[out=90, in=-90]  (TT6);

\draw[black] (TP8) to[out=90, in=-90]  (TT8);
\draw[black] (TP9) to[out=90, in=-90]  (TT10);
\draw[black] (TP10) to[out=90, in=-90]  (TT9);

\draw (B1) to[out=90,in=-90] (T1);
\draw (B2) to[out=90,in=-90] (T3);
\draw (B3) to[out=90,in=-90] (T9);
\draw (B4) to[out=90,in=-90] (T6);
\draw (B5) to[out=90,in=-90] (T10);
\draw (B6) to[out=90,in=-90] (T2);
\draw (B7) to[out=90,in=-90] (T4);
\draw (B8) to[out=90,in=-90] (T5);
\draw (B9) to[out=90,in=-90] (T7);
\draw (B10) to[out=90,in=-90] (T8);
\end{tikzpicture}
\end{equation}
Here, the boxes represent reduced diagrams of the respective permutations. 
The example on the right shows the permutation $u_0\circ \tilde w_D \circ v$, where $\tilde w_D$ is the shortest $(S_{\textbf c},S_{\textbf r})$-double coset representative from Example~\ref{example:ShortestCosetRepresentative} for $\textbf r=(3,2,2,3)$, $\textbf c=(2,3,2,1,2)$ and
$v=v_1\times v_2\times v_3\times v_4$, where
$v_1=312$, $v_2=12$, $v_3=21$, $v_4=231$.

If $d_w$ is a diagram of shape~\eqref{eq:ShapeOfDiagrams} then, according to their position in the diagram, we define the following subsets of crossings in $d_w$:
\begin{align*}
K_U(d_w) &= \{\kappa\in K(d_w)\mid \textup{$\kappa$ belongs to some $w_{0,c_i}$ for $i=1,\ldots,N$}\},\\
K_W(d_w) &= \{\kappa\in K(d_w)\mid \textup{$\kappa$ belongs to $\tilde w_D$}\},\\
K_V(d_w) &= \{\kappa\in K(d_w)\mid \textup{$\kappa$ belongs to some $v_i$ for $i=1,\ldots,M$}\}.
\end{align*}
The next proposition shows that the weights of crossings in $K_U(d_w)$ precisely contribute the normalization factor which appears in Corollary~\ref{cor:ResolvedStableEnvelopesFormula}.

\begin{prop}\label{prop:UPartVSNormalizationFactor}
We have
\[
\Psi_{\mathcal D}\Big( \prod_{\kappa\in K_U(d_w)} \mathrm{wt}(\kappa) \Big) = \prod_{i=1}^N \prod_{j=1}^{c_i-1}(jh)^{c_i-j}.
\]
\end{prop}
The proof is immediate from the following lemma:
\begin{lemma} Let $w_{0,n}\in S_n$ be the longest element and $d_{w_{0,n}}$ a reduced diagram of $w_{0,n}$. Then,
\[
\Psi\Big(\prod_{\kappa\in K(d_{w_{0,n}})} \mathrm{wt}(\kappa) \Big)=\prod_{j=1}^{n-1}(jh)^{n-j},
\]
where $\Psi\colon\mathbb Q[t_1,\ldots,t_n,h]\rightarrow \mathbb Q[t,h]$ is the $\mathbb Q[h]$-algebra homomorphism given by
\[
t_i\mapsto t-(i-1)h,\quad i=1,\ldots, n.
 \]
\end{lemma}
\begin{proof} By~\eqref{eq:SetOfBetasIndependentFromReducedExpression}, the product $\prod_{\kappa\in K(d_{w_{0,n}})} \mathrm{wt}(\kappa)$ does not depend on the choice of reduced diagram. We prove the statement by induction on $n$ where the case $n=1$ is clear. For $n>1$, we choose $d_{w_{0,n}}$ to be of the following shape:
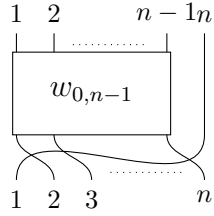
\[
\begin{tikzpicture}[scale=.5]
\coordinate (B1) at (0,-3.9);
\coordinate (B2) at (1,-3.9);
\coordinate (B3) at (2,-3.9);
\coordinate (B4) at (3,-3.9);
\coordinate (B5) at (4,-3.9);
\coordinate (B6) at (5,-3.9);

\draw[dotted] (2.5,-3.7) -- (4.5,-3.7);

\coordinate (M1) at (0,-2.7);
\coordinate (M2) at (1,-2.7);
\coordinate (M3) at (2,-2.7);
\coordinate (M4) at (3,-2.7);
\coordinate (M5) at (4,-2.7);
\coordinate (M6) at (5,-2.7);

\coordinate (MM1) at (0,-0.5);
\coordinate (MM2) at (1,-0.5);
\coordinate (MM3) at (2,-0.5);
\coordinate (MM4) at (3,-0.5);
\coordinate (MM5) at (4,-0.5);
\coordinate (MM6) at (5,-0.5);

\coordinate (T1) at (0,0);
\coordinate (T2) at (1,0);
\coordinate (T3) at (2,0);
\coordinate (T4) at (3,0);
\coordinate (T5) at (4,0);
\coordinate (T6) at (5,0);

\draw (B1) to[out=90, in =-90] (M6) to[out=90, in=-90] (T6);
\draw (B2) to[out=90, in =-90] (M1);
\draw (B3) to[out=90, in =-90] (M2);
\draw (B6) to[out=90, in =-90] (M5);

\node[below=0 of B1] {\text{\footnotesize $1$}};
\node[below=0 of B2] {\text{\footnotesize $2$}};
\node[below=0 of B3] {\text{\footnotesize $3$}};
\node[below=0 of B6] {\text{\footnotesize $n$}};

\draw[dotted] (1.5,-0.3) -- (3.5,-0.3);

\draw (MM1) -- (T1);
\draw (MM2) -- (T2);
\draw (MM5) -- (T5);

\node[above=0 of T1] {\text{\footnotesize $1$}};
\node[above=0 of T2] {\text{\footnotesize $2$}};
\node[above=0 of T5] {\text{\footnotesize $n-1$}};
\node[above=0 of T6] {\text{\footnotesize $n$}};

\node (w) at (2,-1.6) {$w_{0,n-1}$};
\draw[draw=black] (-0.1,-2.7) rectangle ++(4.2,2.2);
\end{tikzpicture}
\]
Here, the box represents a reduced diagram for $w_{0,n-1}$. Let $K'$ be the set of crossings contained in the box of $w_{0,n-1}$ and $K''$ be the set of crossings outside of the box of $w_{0,n-1}$.
From the diagram $d_{w_{0,n}}$, we can read off that the crossings in $K''$ have weights $t_1-t_n,\ldots,t_{n-1}-t_n$. Thus, we have
$\Psi(\prod_{\kappa\in K''}\mathrm{wt}(\kappa))=\prod_{i=1}^n(ih)$.
Applying the induction hypothesis to $K'$ yields
\begin{align*}
\Psi\Big(\prod_{\kappa\in K(d_{w_{0,n}})}\mathrm{wt}(\kappa) \Big)&= 
\Psi\Big(\prod_{\kappa\in K'} \mathrm{wt}(\kappa)\Big)\cdot
\Psi\Big(\prod_{\kappa\in K''} \mathrm{wt}(\kappa) \Big)\\
&=
\Big(\prod_{j=1}^{n-2}(jh)^{n-1-j}\Big)\cdot \Psi\Big(\prod_{i=1}^{n-1}(t_i-t_n)\Big)\\
&= \prod_{j=1}^{n-1}(jh)^{n-j}
\end{align*}
which finishes the proof.
\end{proof}

\begin{prop}[Approximation]\label{prop:ApproximationStableEnvelopeModuloPowersOfH}
Under the assumptions as in Corollary~\ref{cor:ResolvedStableEnvelopesFormula}, we have
\begin{equation}\label{eq:ApproximationStableEnvelopeModuloPowersOfH}
\iota_{D}^\ast (\widetilde{\mathrm{Stab}}_{\mathfrak C_-}(D'))
\equiv
\sum_{z\in w_D S_{\textbf r}} \frac{(-1)^{\ell(w')+\ell(w'S_{\textbf r})}
\Big(\prod_{\alpha\in L'_z}\Psi_{\mathcal D}(\alpha+h) \Big)
\cdot P_{d_z,w',m} }
{\prod_{\beta\in R_{\textbf r}}\Psi_{\mathcal D}(z.\beta)}\quad\mathrm{mod}\;h^m,
\end{equation}
where $L'_z$ is defined as in Proposition~\ref{prop:SuLocalizationFormulaDiagrammatic},
\[
P_{d_z,w',m} = \sum_{K'\in K(d_z,w',m-1)}\biggl( h^{|K'\setminus K_U(d_z)|} f_{K'}\cdot\biggl( \prod_{\substack{\kappa \in K(d_z)\\ \kappa\notin K'\cup K_U(d_z)}}\Psi_{\mathcal D}(\mathrm{wt}(\kappa))\biggr)\biggr)
\]
and
\begin{align*}
 K(d_z,w',m-1) &= \{K'\in K_{d_z,w'}\mid  |K'\setminus K_U(d_z)| \le m-1 \},\\
 f_{K'} &= \frac{h^{|K'\cap K_U(d_z)|}\prod_{\kappa\in K_U(d_z)\setminus K'}\Psi_{\mathcal D}(\mathrm{wt}(\kappa)) }{\prod_{i=1}^N\prod_{j=1}^{c_i-1} (jh)^{c_i-j}}.
\end{align*}
\end{prop}

\begin{remark} By Proposition~\ref{prop:UPartVSNormalizationFactor},  the factor $f_{K'}$ is always contained in $\mathbb Q$. Moreover, as $F_z(i)\ne F_z(j)$  for all $t_i-t_j\in R_{\textbf r}$ and $z\in S_{\textbf c}\tilde w_D S_{\textbf r}$, we deduce that for all $\alpha\in R_{\textbf r}$, the polynomial $\Psi_{\mathcal D}(z.\alpha)$ is contained in $S$, as defined in~\eqref{eq:DefinitionS}.
\end{remark}
%Since $h$ divides $\Psi_{\mathcal D}(t_i-t_j)$ if and only if $C_{l-1}+1\le i<j\le C_l$ for some $l=1,\ldots,N$, the factor $f_{K'}$ is always contained in $\mathbb Q$.

\begin{proof}[Proof of Proposition~\ref{prop:ApproximationStableEnvelopeModuloPowersOfH}]
For $z\in w_DS_{\textbf r}$ with reduced diagram $d_z$ of shape \eqref{eq:ShapeOfDiagrams} let $m_0(z)=|K(d_z)\setminus K_U(d_z)|$. By Corollary~\ref{cor:ResolvedStableEnvelopesFormula}, we have
\[
\iota_{D}^\ast (\widetilde{\mathrm{Stab}}_{\mathfrak C_-}(D'))
=
\sum_{z\in w_D S_{\textbf r}} \frac{(-1)^{\ell(w')+\ell(w'S_{\textbf r})}
\Big(\prod_{\alpha\in L'_z}\Psi_{\mathcal D}(\alpha+h) \Big)
\cdot P_{d_z,w',m_0(z)} }
{\prod_{\beta\in R_{\textbf r}}\Psi_{\mathcal D}(z.\beta)}.
\]
If $K'\in K_{d_z,w'} \setminus K(d_z,w',m-1)$ then by Proposition~\ref{prop:UPartVSNormalizationFactor},
$
h^{|K'|}\prod_{\kappa\in K_U(d_z)\setminus K'}\Psi_{\mathcal D}(\mathrm{wt}(\kappa))
$
is divisible by $h^{\frac12(c_1(c_1-1)+\cdots +c_N(c_N-1))+m}$. Hence, the contribution of those $K'$ in $P_{d_z,w',m_0(z)}$ vanishes modulo $h^m$. Thus, we have
\[
P_{d_z,w',m_0(z)}\equiv P_{d_z,w',m}\quad\mathrm{mod}\;h^m,\quad\textup{for all $z\in w_DS_{\textbf r}$}
\]
which proves the proposition.
\end{proof}

\section{Chevalley--Monk formulas in the separated case}\label{section:CMSeparatedCase}

In this section, we state and prove Chevalley--Monk formulas for the tautological bundles of bow varieties corresponding to separated brane diagrams. In the upcoming section, we then derive Chevalley--Monk formulas for tautological bundles corresponding to general bow varieties using Hanany--Witten transition.

\begin{assumption}
For this section, we assume that $\mathcal D$ is a separated brane diagram.
\end{assumption}

Recall from Subsection~\ref{subsection:TautologicalBundles} that the tautological bundles $\xi_{M+1},\ldots,\xi_{M+N+1}$ are trivial. Hence, we focus on characterizing the multiplication of $c_1(\xi_1),\ldots,c_1(\xi_M)$.

\subsection{Chevalley--Monk formula for the antidominant chamber}
We first restrict our attention to the antidominant chamber $\mathfrak C_-$. In this case, the Chevalley--Monk formula is given as follows:

\begin{theorem}[Chevalley--Monk for antidominant chamber]\label{thm:CMAntidominantChamber}
Let $D\in\mathrm{Tie}(\mathcal D)$. Then, we have the following identity in $H_{\mathbb T}^\ast(\mathcal C(\mathcal D))_{\mathrm{loc}}$:
\[
c_1(\xi_i)\cup \mathrm{Stab}_{\mathfrak C_-}(D)=\iota_{D}^\ast(c_1(\xi_i))\cdot \mathrm{Stab}_{\mathfrak C_-}(D) + \sum_{D'\in \mathrm{SM}_{D,i}} \operatorname{sgn}(D,D')h\cdot\mathrm{Stab}_{\mathfrak C_-}(D'),
\]
for $i=1,\ldots,M$. Here, the set of simple moves $\mathrm{SM}_{D,i}$ is defined in~\eqref{eq:SMiDefinition} and the signs of simple moves $\operatorname{sgn}(D,D')\in\{\pm1\}$ are defined in Definition~\ref{definition:SignsOfSimpleMoves}.
\end{theorem}

The proof of Theorem~\ref{thm:CMAntidominantChamber} is given in Subsection~\ref{subsection:ProofOfTHeoremAntidominant}. We first give the definitions relevant for the theorem. We begin with introducing notion of moving ties:

\begin{definition}\label{definition:SimpleMoves}
Let $D$, $D'\in\mathrm{Tie}(\mathcal D)$. We say that $D'$ is obtained from $D$ via a \textit{simple move} if there exist $1\le i_1<i_2\le M$ and $1\le j_1 < j_2 \le N$ such that
 $(V_{i_1},U_{j_1})$, $(V_{i_2},U_{j_2})\in D$, $(V_{i_1},U_{j_2})$, $(V_{i_2},U_{j_1})\in D'$ and
\[
D\setminus\{(V_{i_1},U_{j_2}),(V_{i_2},U_{j_1})\} = D'\setminus\{(V_{i_1},U_{j_1}),(V_{i_2},U_{j_2})\}.
\]
We call $(V_{i_1},U_{j_1})$ the \textit{right moving tie} and $(V_{i_2},U_{j_2})$ the \textit{left moving tie of} $D$. Let $\mathrm{SM}_{D}$ be the set of all tie diagrams that are obtained from $D$ via a simple move.
\end{definition}

Pictorially, simple moves can be described as switching two ties as illustrated:
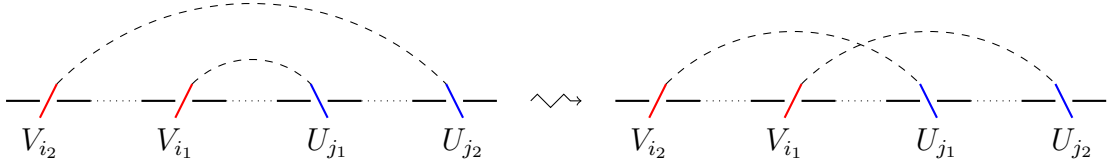
\begin{center}
\begin{tikzpicture}
[scale=.225]
\draw[thick] (0,0)--(2,0);
\draw[thick] (3,0)--(5,0);
\draw[dotted] (5,0) -- (8,0);
\draw[thick] (8,0)--(10,0);
\draw[thick] (11,0)--(13,0);
\draw[dotted] (13,0) -- (16,0);
\draw[thick] (16,0)--(18,0);
\draw[thick] (19,0)--(21,0);
\draw[dotted] (21,0) -- (24,0);
\draw[thick] (24,0)--(26,0);
\draw[thick] (27,0)--(29,0);

\draw [thick,red] (2,-1) --(3,1); 
\draw [thick,red] (10,-1) --(11,1); 
\node[black] at (2,-2.5) {$V_{i_2}$};
\node[black] at (10,-2.5) {$V_{i_1}$};

\draw [thick,blue] (27,-1) --(26,1); 
\draw [thick,blue] (19,-1) --(18,1);
\node[black] at (19,-2.5) {$U_{j_1}$};
\node[black] at (27,-2.5) {$U_{j_2}$};

\draw[dashed] (3,1) to[out=45,in=135] (26,1);
\draw[dashed] (11,1) to[out=45,in=135] (18,1);

\draw[-to] decorate[decoration=zigzag] {(31,0)--(34,0)};
\def\s {36};
\def\t {9};
%\node at (-2,-9) {$D'$};
\draw[thick] (0+\s,-9+\t)--(2+\s,-9+\t);
\draw[thick] (3+\s,-9+\t)--(5+\s,-9+\t);
\draw[dotted] (5+\s,-9+\t) -- (8+\s,-9+\t);
\draw[thick] (8+\s,-9+\t)--(10+\s,-9+\t);
\draw[thick] (11+\s,-9+\t)--(13+\s,-9+\t);
\draw[dotted] (13+\s,-9+\t) -- (16+\s,-9+\t);
\draw[thick] (16+\s,-9+\t)--(18+\s,-9+\t);
\draw[thick] (19+\s,-9+\t)--(21+\s,-9+\t);
\draw[dotted] (21+\s,-9+\t) -- (24+\s,-9+\t);
\draw[thick] (24+\s,-9+\t)--(26+\s,-9+\t);
\draw[thick] (27+\s,-9+\t)--(29+\s,-9+\t);

\draw [thick,red] (2+\s,-10+\t) --(3+\s,-8+\t); 
\draw [thick,red] (10+\s,-10+\t) --(11+\s,-8+\t); 
\node[black] at (2+\s,-11.5+\t) {$V_{i_2}$};
\node[black] at (10+\s,-11.5+\t) {$V_{i_1}$};

\draw [thick,blue] (27+\s,-10+\t) --(26+\s,-8+\t); 
\draw [thick,blue] (19+\s,-10+\t) --(18+\s,-8+\t);
\node[black] at (19+\s,-11.5+\t) {$U_{j_1}$};
\node[black] at (27+\s,-11.5+\t) {$U_{j_2}$};

\draw[dashed] (3+\s,-8+\t) to[out=45,in=135] (18+\s,-8+\t);
\draw[dashed] (11+\s,-8+\t) to[out=45,in=135](26+\s,-8+\t) ;

\end{tikzpicture}
\end{center}

Translating between tie diagrams and binary contingency tables immediately gives the following equivalent characterization of simple moves:
\begin{lemma}\label{lemma:SImpleMovesBCT}
Let $D$, $D'\in \mathrm{Tie}(\mathcal D)$. Then, $D'$ is obtained from $D$ via a simple move if and only if there exist $1\le i_1<i_2\le M$ and $1\le j_1 < j_2 \le N$ such that
\begin{enumerate}[label=(\roman*)]
\item $M(D)_{i_1,j_1}=M(D)_{i_2,j_2}=1$ and $M(D)_{i_1,j_2}=M(D)_{i_1,j_2}=0$,
\item $M(D')_{i_1,j_1}=M(D')_{i_2,j_2}=0$ and $M(D')_{i_1,j_2}=M(D')_{i_1,j_2}=1$,
\item $M(D)_{l,k}=M(D')_{l,k}$, for $(l,k)\notin \{(i_1,j_1),(i_2,j_1),(i_1,j_2),(i_2,j_2)\}$.
\end{enumerate}
\end{lemma}

For $X_i\in\mathrm{h}(\mathcal D)$ with $i\in \{1,\ldots,M\}$, we define the \textit{set of simple moves relative to $X_i$} as
\begin{equation}\label{eq:SMiDefinition}
\mathrm{SM}_{D,i}=\{D'\in \mathrm{Tie}(\mathcal D)\mid \textup{$D'$ satisfies \ref{item:SMi}, \ref{item:SMi1} and \ref{item:SMi2}}\},
\end{equation}
where
\begin{enumerate}[label=(\alph*)]
\item\label{item:SMi} $D'$ is obtained from $D$ via a simple move,
\item\label{item:SMi1} if $(V_{i_1},U_{j_1})$ is the right moving tie of $D$ then $X_i \triangleleft V_{i_1} $,
\item\label{item:SMi2} if $(V_{i_2},U_{j_2})$ is the left moving tie of $D$ then $V_{i_2}\triangleleft X_i$.
\end{enumerate}

Next, we define the sign of a simple move:
\begin{definition}\label{definition:SignsOfSimpleMoves} Given $D'\in \mathrm{SM}_D$ with right moving tie $(V_{i_1},U_{j_1})$ and left moving tie $(V_{i_2},U_{j_2})$. Then, we define
\[
\operatorname{sgn}(D,D')\coloneqq \begin{cases}
1 &\textup{if $n_1+n_2$ is even,}\\
-1 &\textup{if $n_1+n_2$ is odd,}
\end{cases}
\]
where
\[
n_1\coloneqq |\{(V_{i_1},U_j)\in D\mid j_1<j<j_2\}|,\quad n_2\coloneqq |\{(V_{i_2},U_j)\in D\mid j_1<j<j_2\}|.
\]
We call $\operatorname{sgn}(D,D')$ the \textit{sign of the simple move between $D$ and $D'$}.
\end{definition}

Thus, all notions appearing in Theorem~\ref{thm:CMAntidominantChamber} are introduced.

\begin{example}\label{example:SimpleMovesTieDiagrams} Let $\mathcal D=0
\textcolor{red}{\slash} 1
\textcolor{red}{\slash} 3
\textcolor{red}{\slash} 4
\textcolor{red}{\slash} 5
\textcolor{blue}{\backslash} 4
\textcolor{blue}{\backslash} 3
\textcolor{blue}{\backslash} 1
\textcolor{blue}{\backslash} 0$ 
and 
\begin{center}
\begin{tikzpicture}
[scale=.45]
\node at (-2,0){$D=$};
\draw[thick] (0,0)--(2,0);
\draw[thick] (3,0)--(5,0);
\draw[thick] (6,0)--(8,0);
\draw[thick] (9,0)--(11,0);
\draw[thick] (12,0)--(14,0);
\draw[thick] (15,0)--(17,0);
\draw[thick] (18,0)--(20,0);
\draw[thick] (21,0)--(23,0);
\draw[thick] (24,0)--(26,0);

\draw [thick,red] (2,-1) --(3,1); 
\draw [thick,red] (5,-1) --(6,1); 
\draw [thick,red] (8,-1) --(9,1);
\draw [thick,red] (11,-1) --(12,1);

\draw [thick,blue] (15,-1) --(14,1); 
\draw [thick,blue] (18,-1) --(17,1); 
\draw [thick,blue] (21,-1) --(20,1); 
\draw [thick,blue] (24,-1) --(23,1); 

\node at (1,0.5){$0$};
\node at (4,0.5){$1$};
\node at (7,0.5){$3$};
\node at (10,0.5){$4$};
\node at (13,0.5){$5$};
\node at (16,0.5){$4$};
\node at (19,0.5){$3$};
\node at (22,0.5){$1$};
\node at (25,0.5){$0$};

\draw[dashed] (3,1) to[out=45,in=135] (20,1);
\draw[dashed] (6,1) to[out=45,in=135] (17,1);
\draw[dashed] (6,1) to[out=45,in=135] (23,1);
\draw[dashed] (9,1) to[out=45,in=135] (14,1);
\draw[dashed] (12,1) to[out=45,in=135] (20,1);

\end{tikzpicture}
\end{center}
In the following, we determine $c_1(\xi_i)\cup \mathrm{Stab}_{\mathfrak C_-}(D)$ for $i=3$. The simple moves that are contained in $\mathrm{SM}_{D,i}$ are illustrated as follows:
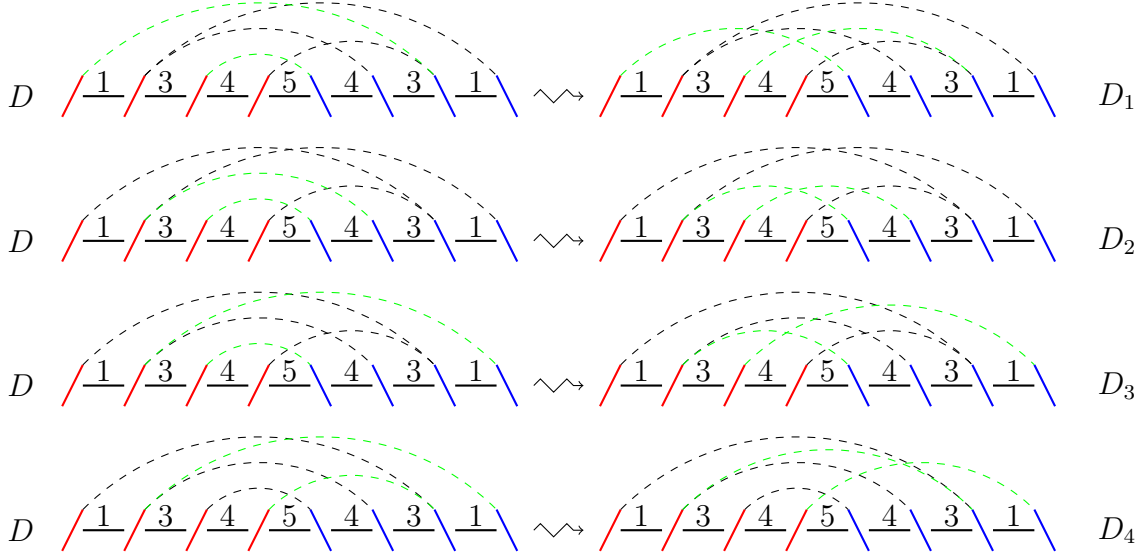
\begin{center}
\begin{tikzpicture}
[scale=.275]

\def\z{0.2};
\node at (0,0) {$D$};
\draw[thick] (3,0)--(5,0);
\draw[thick] (6,0)--(8,0);
\draw[thick] (9,0)--(11,0);
\draw[thick] (12,0)--(14,0);
\draw[thick] (15,0)--(17,0);
\draw[thick] (18,0)--(20,0);
\draw[thick] (21,0)--(23,0);

\draw [thick,red] (2,-1) --(3,1); 
\draw [thick,red] (5,-1) --(6,1); 
\draw [thick,red] (8,-1) --(9,1);
\draw [thick,red] (11,-1) --(12,1);

\draw [thick,blue] (15,-1) --(14,1); 
\draw [thick,blue] (18,-1) --(17,1); 
\draw [thick,blue] (21,-1) --(20,1); 
\draw [thick,blue] (24,-1) --(23,1); 

\node at (4,0.7){$1$};
\node at (7,0.7){$3$};
\node at (10,0.7){$4$};
\node at (13,0.7){$5$};
\node at (16,0.7){$4$};
\node at (19,0.7){$3$};
\node at (22,0.7){$1$};
%\node at (25,0.7){$0$};

\draw[dashed] (6,1) to[out=45,in=135] (17,1);
\draw[dashed] (6,1) to[out=45,in=135] (23,1);

\draw[dashed] (12,1) to[out=45,in=135] (20,1);

\draw[dashed, very thick, ForestGreen] (3,1) to[out=45,in=135] (20,1);
\draw[dashed, very thick, ForestGreen] (9,1) to[out=45,in=135] (14,1);

\draw[-to] decorate[decoration=zigzag] {(25-\z,0)--(27+\z,0)};
\def\s {26};

\draw[thick] (3+\s,0)--(5+\s,0);
\draw[thick] (6+\s,0)--(8+\s,0);
\draw[thick] (9+\s,0)--(11+\s,0);
\draw[thick] (12+\s,0)--(14+\s,0);
\draw[thick] (15+\s,0)--(17+\s,0);
\draw[thick] (18+\s,0)--(20+\s,0);
\draw[thick] (21+\s,0)--(23+\s,0);
\node at (27+\s,0) {$D_1$};

\draw [thick,red] (2+\s,-1) --(3+\s,1); 
\draw [thick,red] (5+\s,-1) --(6+\s,1); 
\draw [thick,red] (8+\s,-1) --(9+\s,1);
\draw [thick,red] (11+\s,-1) --(12+\s,1);

\draw [thick,blue] (15+\s,-1) --(14+\s,1); 
\draw [thick,blue] (18+\s,-1) --(17+\s,1); 
\draw [thick,blue] (21+\s,-1) --(20+\s,1); 
\draw [thick,blue] (24+\s,-1) --(23+\s,1); 

\node at (4+\s,0.7){$1$};
\node at (7+\s,0.7){$3$};
\node at (10+\s,0.7){$4$};
\node at (13+\s,0.7){$5$};
\node at (16+\s,0.7){$4$};
\node at (19+\s,0.7){$3$};
\node at (22+\s,0.7){$1$};

\draw[dashed] (6+\s,1) to[out=45,in=135] (17+\s,1);
\draw[dashed] (6+\s,1) to[out=45,in=135] (23+\s,1);
\draw[dashed] (12+\s,1) to[out=45,in=135] (20+\s,1);

\draw[dashed, very thick, ForestGreen] (3+\s,1) to[out=45,in=135] (14+\s,1);
\draw[dashed, very thick, ForestGreen] (9+\s,1) to[out=45,in=135] (20+\s,1);

\def\t{-7};
\node at (0,0+\t) {$D$};
\draw[thick] (3,0+\t)--(5,0+\t);
\draw[thick] (6,0+\t)--(8,0+\t);
\draw[thick] (9,0+\t)--(11,0+\t);
\draw[thick] (12,0+\t)--(14,0+\t);
\draw[thick] (15,0+\t)--(17,0+\t);
\draw[thick] (18,0+\t)--(20,0+\t);
\draw[thick] (21,0+\t)--(23,0+\t);

\draw [thick,red] (2,-1+\t) --(3,1+\t); 
\draw [thick,red] (5,-1+\t) --(6,1+\t); 
\draw [thick,red] (8,-1+\t) --(9,1+\t);
\draw [thick,red] (11,-1+\t) --(12,1+\t);

\draw [thick,blue] (15,-1+\t) --(14,1+\t); 
\draw [thick,blue] (18,-1+\t) --(17,1+\t); 
\draw [thick,blue] (21,-1+\t) --(20,1+\t); 
\draw [thick,blue] (24,-1+\t) --(23,1+\t); 

\node at (4,0.7+\t){$1$};
\node at (7,0.7+\t){$3$};
\node at (10,0.7+\t){$4$};
\node at (13,0.7+\t){$5$};
\node at (16,0.7+\t){$4$};
\node at (19,0.7+\t){$3$};
\node at (22,0.7+\t){$1$};

\draw[dashed] (3,1+\t) to[out=45,in=135] (20,1+\t);
\draw[dashed] (6,1+\t) to[out=45,in=135] (23,1+\t);
\draw[dashed] (12,1+\t) to[out=45,in=135] (20,1+\t);

\draw[dashed, very thick, ForestGreen] (6,1+\t) to[out=45,in=135] (17,1+\t);
\draw[dashed, very thick, ForestGreen] (9,1+\t) to[out=45,in=135] (14,1+\t);

\draw[-to] decorate[decoration=zigzag] {(25-\z,0+\t)--(27+\z,0+\t)};

\draw[thick] (3+\s,-7)--(5+\s,-7);
\draw[thick] (6+\s,-7)--(8+\s,-7);
\draw[thick] (9+\s,-7)--(11+\s,-7);
\draw[thick] (12+\s,-7)--(14+\s,-7);
\draw[thick] (15+\s,-7)--(17+\s,-7);
\draw[thick] (18+\s,-7)--(20+\s,-7);
\draw[thick] (21+\s,-7)--(23+\s,-7);

\node at (27+\s,-7){$D_2$};

\draw [thick,red] (2+\s,-8) --(3+\s,-6); 
\draw [thick,red] (5+\s,-8) --(6+\s,-6); 
\draw [thick,red] (8+\s,-8) --(9+\s,-6);
\draw [thick,red] (11+\s,-8) --(12+\s,-6);

\draw [thick,blue] (15+\s,-8) --(14+\s,-6); 
\draw [thick,blue] (18+\s,-8) --(17+\s,-6); 
\draw [thick,blue] (21+\s,-8) --(20+\s,-6); 
\draw [thick,blue] (24+\s,-8) --(23+\s,-6); 

\node at (4+\s,-6.3){$1$};
\node at (7+\s,-6.3){$3$};
\node at (10+\s,-6.3){$4$};
\node at (13+\s,-6.3){$5$};
\node at (16+\s,-6.3){$4$};
\node at (19+\s,-6.3){$3$};
\node at (22+\s,-6.3){$1$};

\draw[dashed] (3+\s,-6) to[out=45,in=135] (20+\s,-6);
\draw[dashed] (6+\s,-6) to[out=45,in=135] (23+\s,-6);
\draw[dashed] (12+\s,-6) to[out=45,in=135] (20+\s,-6);

\draw[dashed, very thick, ForestGreen] (6+\s,-6) to[out=45,in=135] (14+\s,-6);
\draw[dashed, very thick, ForestGreen] (9+\s,-6) to[out=45,in=135] (17+\s,-6);
\def\t{-14};
\node at (0,0+\t) {$D$};
\draw[thick] (3,0+\t)--(5,0+\t);
\draw[thick] (6,0+\t)--(8,0+\t);
\draw[thick] (9,0+\t)--(11,0+\t);
\draw[thick] (12,0+\t)--(14,0+\t);
\draw[thick] (15,0+\t)--(17,0+\t);
\draw[thick] (18,0+\t)--(20,0+\t);
\draw[thick] (21,0+\t)--(23,0+\t);

\draw [thick,red] (2,-1+\t) --(3,1+\t); 
\draw [thick,red] (5,-1+\t) --(6,1+\t); 
\draw [thick,red] (8,-1+\t) --(9,1+\t);
\draw [thick,red] (11,-1+\t) --(12,1+\t);

\draw [thick,blue] (15,-1+\t) --(14,1+\t); 
\draw [thick,blue] (18,-1+\t) --(17,1+\t); 
\draw [thick,blue] (21,-1+\t) --(20,1+\t); 
\draw [thick,blue] (24,-1+\t) --(23,1+\t); 

\node at (4,0.7+\t){$1$};
\node at (7,0.7+\t){$3$};
\node at (10,0.7+\t){$4$};
\node at (13,0.7+\t){$5$};
\node at (16,0.7+\t){$4$};
\node at (19,0.7+\t){$3$};
\node at (22,0.7+\t){$1$};

\draw[dashed] (3,1+\t) to[out=45,in=135] (20,1+\t);
\draw[dashed] (6,1+\t) to[out=45,in=135] (17,1+\t);
\draw[dashed] (12,1+\t) to[out=45,in=135] (20,1+\t);

\draw[dashed, very thick, ForestGreen] (6,1+\t) to[out=45,in=135] (23,1+\t);
\draw[dashed, very thick, ForestGreen] (9,1+\t) to[out=45,in=135] (14,1+\t);

\draw[-to] decorate[decoration=zigzag] {(25-\z,0+\t)--(27+\z,0+\t)};
\draw[thick] (3+\s,-14)--(5+\s,-14);
\draw[thick] (6+\s,-14)--(8+\s,-14);
\draw[thick] (9+\s,-14)--(11+\s,-14);
\draw[thick] (12+\s,-14)--(14+\s,-14);
\draw[thick] (15+\s,-14)--(17+\s,-14);
\draw[thick] (18+\s,-14)--(20+\s,-14);
\draw[thick] (21+\s,-14)--(23+\s,-14);
\node at (27+\s,-14){$D_3$};

\draw [thick,red] (2+\s,-15) --(3+\s,-13); 
\draw [thick,red] (5+\s,-15) --(6+\s,-13); 
\draw [thick,red] (8+\s,-15) --(9+\s,-13);
\draw [thick,red] (11+\s,-15) --(12+\s,-13);

\draw [thick,blue] (15+\s,-15) --(14+\s,-13); 
\draw [thick,blue] (18+\s,-15) --(17+\s,-13); 
\draw [thick,blue] (21+\s,-15) --(20+\s,-13); 
\draw [thick,blue] (24+\s,-15) --(23+\s,-13); 

\node at (4+\s,-13.3){$1$};
\node at (7+\s,-13.3){$3$};
\node at (10+\s,-13.3){$4$};
\node at (13+\s,-13.3){$5$};
\node at (16+\s,-13.3){$4$};
\node at (19+\s,-13.3){$3$};
\node at (22+\s,-13.3){$1$};

\draw[dashed] (3+\s,-13) to[out=45,in=135] (20+\s,-13);
\draw[dashed] (6+\s,-13) to[out=45,in=135] (17+\s,-13);
\draw[dashed] (12+\s,-13) to[out=45,in=135] (20+\s,-13);

\draw[dashed, very thick, ForestGreen] (6+\s,-13) to[out=45,in=135] (14+\s,-13);
\draw[dashed, very thick, ForestGreen] (9+\s,-13) to[out=45,in=135] (23+\s,-13);

\def\t{-21};
\node at (0,0+\t) {$D$};
\draw[thick] (3,0+\t)--(5,0+\t);
\draw[thick] (6,0+\t)--(8,0+\t);
\draw[thick] (9,0+\t)--(11,0+\t);
\draw[thick] (12,0+\t)--(14,0+\t);
\draw[thick] (15,0+\t)--(17,0+\t);
\draw[thick] (18,0+\t)--(20,0+\t);
\draw[thick] (21,0+\t)--(23,0+\t);

\draw [thick,red] (2,-1+\t) --(3,1+\t); 
\draw [thick,red] (5,-1+\t) --(6,1+\t); 
\draw [thick,red] (8,-1+\t) --(9,1+\t);
\draw [thick,red] (11,-1+\t) --(12,1+\t);

\draw [thick,blue] (15,-1+\t) --(14,1+\t); 
\draw [thick,blue] (18,-1+\t) --(17,1+\t); 
\draw [thick,blue] (21,-1+\t) --(20,1+\t); 
\draw [thick,blue] (24,-1+\t) --(23,1+\t);

\node at (4,0.7+\t){$1$};
\node at (7,0.7+\t){$3$};
\node at (10,0.7+\t){$4$};
\node at (13,0.7+\t){$5$};
\node at (16,0.7+\t){$4$};
\node at (19,0.7+\t){$3$};
\node at (22,0.7+\t){$1$};

\draw[dashed] (3,1+\t) to[out=45,in=135] (20,1+\t);
\draw[dashed] (6,1+\t) to[out=45,in=135] (17,1+\t);
\draw[dashed] (9,1+\t) to[out=45,in=135] (14,1+\t);
\draw[dashed, very thick, ForestGreen] (12,1+\t) to[out=45,in=135] (20,1+\t);

\draw[dashed, very thick, ForestGreen] (6,1+\t) to[out=45,in=135] (23,1+\t);

\draw[-to] decorate[decoration=zigzag] {(25-\z,0+\t)--(27+\z,0+\t)};

\draw[thick] (3+\s,-21)--(5+\s,-21);
\draw[thick] (6+\s,-21)--(8+\s,-21);
\draw[thick] (9+\s,-21)--(11+\s,-21);
\draw[thick] (12+\s,-21)--(14+\s,-21);
\draw[thick] (15+\s,-21)--(17+\s,-21);
\draw[thick] (18+\s,-21)--(20+\s,-21);
\draw[thick] (21+\s,-21)--(23+\s,-21);

\node at (27+\s,-21){$D_4$};

\draw [thick,red] (2+\s,-22) --(3+\s,-20); 
\draw [thick,red] (5+\s,-22) --(6+\s,-20); 
\draw [thick,red] (8+\s,-22) --(9+\s,-20);
\draw [thick,red] (11+\s,-22) --(12+\s,-20);

\draw [thick,blue] (15+\s,-22) --(14+\s,-20); 
\draw [thick,blue] (18+\s,-22) --(17+\s,-20); 
\draw [thick,blue] (21+\s,-22) --(20+\s,-20); 
\draw [thick,blue] (24+\s,-22) --(23+\s,-20); 

\node at (4+\s,-20.3){$1$};
\node at (7+\s,-20.3){$3$};
\node at (10+\s,-20.3){$4$};
\node at (13+\s,-20.3){$5$};
\node at (16+\s,-20.3){$4$};
\node at (19+\s,-20.3){$3$};
\node at (22+\s,-20.3){$1$};

\draw[dashed] (3+\s,-20) to[out=45,in=135] (20+\s,-20);
\draw[dashed] (6+\s,-20) to[out=45,in=135] (17+\s,-20);

\draw[dashed] (9+\s,-20) to[out=45,in=135] (14+\s,-20);
\draw[dashed, very thick, ForestGreen] (12+\s,-20) to[out=45,in=135] (23+\s,-20);

\draw[dashed, very thick, ForestGreen] (6+\s,-20) to[out=45,in=135] (20+\s,-20);
\end{tikzpicture}
\end{center}
%\caption{Illustration of simple moves from Example~\ref{example:SimpleMovesTieDiagrams}. The figure shows the set $\mathrm{SM}_{D,3}$ where we highlighted the moving ties of each respective simple move. The tie diagrams $D_1,D_2,D_3,D_4$ coincide with those of \eqref{eq:ExamplesOfSimpleMoves}. We omitted the horizontal black lines on the boundary of the tie diagrams.}
%\label{figure:ExampleSimpleMoves}
%\end{figure}
%Figure~\ref{figure:ExampleSimpleMoves} shows all the simple moves that are contained in $\mathrm{SM}_{D,i}$.
That is $\mathrm{SM}_{D,i}=\{D_1,D_2,D_3,D_4\}$, where
\begin{equation}\label{eq:ExamplesOfSimpleMoves}
\begin{split}
D_1&= (D\cup \{(V_2,U_3),(V_4,U_1)\})\setminus \{(V_2,U_1),(V_4,U_3)\}, \\
D_2&= (D\cup \{(V_2,U_2),(V_3,U_1)\})\setminus \{(V_2,U_1),(V_3,U_2)\}, \\
D_3&= (D\cup \{(V_2,U_4),(V_3,U_1)\})\setminus \{(V_2,U_1),(V_3,U_4)\}, \\
D_4&= (D\cup \{(V_1,U_4),(V_3,U_3)\})\setminus \{(V_1,U_3),(V_3,U_4)\}.
\end{split}
\end{equation}
Note that in the illustration of the simple moves, we omitted the horizontal black lines on the boundary of the tie diagrams. From the diagrams one can easily read off the respective signs:
\[
\operatorname{sgn}(D,D_1)=\operatorname{sgn}(D,D_2)=\operatorname{sgn}(D,D_4)=1,\quad \operatorname{sgn}(D,D_3)=-1.
\]
By~\eqref{eq:RestrictionTautologicalBundles}, we have an isomorphism of $\mathbb T$-representations
$
\iota_{D}^\ast(\xi_i)\cong \mathbb C_{U_2}\oplus \mathbb C_{U_3}\oplus \mathbb C_{U_4}.
$
Thus, $\iota_{D}^\ast(c_1(\xi_i))=t_2+t_3+t_4$. Hence, Theorem~\ref{thm:CMAntidominantChamber} gives 
\begin{align*}
c_1(\xi_i)\cup \mathrm{Stab}_{\mathfrak C_-}(D)&=
(t_2+t_3+t_4)\mathrm{Stab}_{\mathfrak C_-}(D)
+h\mathrm{Stab}_{\mathfrak C_-}({D_1})+h\mathrm{Stab}_{\mathfrak C_-}({D_2})\\ &\phantom{=}-h\mathrm{Stab}_{\mathfrak C_-}({D_3})+h\mathrm{Stab}_{\mathfrak C_-}({D_4})
.
\end{align*}
\end{example}

\begin{remark} The notion of simple moves on brane diagrams also appeared in \cite{foster2023tangent}, where they are called \textit{swap moves}. There these moves were used to classify torus invariant curves on bow varieties.
\end{remark}

\subsection{The sign}\label{subsection:TheSign} In this subsection, we give an interpretation of $\operatorname{sgn}(D,D')$ in terms of permutations assigned to the double cosets of $D$ and $D'$. From this, we deduce that after appropriate normalization of the stable basis all off-diagonal entries in the Chevalley--Monk formula become equal to $-h$.

Given a tie diagram $D$ and $D'\in \mathrm{SM}_D$ with left moving tie $(V_{i_1},U_{j_1})$ and right moving tie $(V_{i_2}, U_{j_2})$. Let  
$
\tilde w_{D'}=\tilde w_{M(D')}\in S_n
$
be the shortest $(S_{\textbf c},S_{\textbf c})$-double coset representative from  Definition~\ref{definition:DefinitionTildeWA}. By definition, there exist a unique
$f_1\in \{ R_{i_1-1}+1,\ldots, R_{i_1}\}$ such that $\tilde w_{D'}(f_1)\in \{C_{j_2-1}+1,\ldots,C_{j_2}\}$
and a unique $f_2\in \{ R_{i_2-1}+1,\ldots, R_{i_2}\}$ with $\tilde w_{D'}(f_2)\in \{C_{j_1-1}+1,\ldots,C_{j_1}\}$. We set
\begin{equation}\label{eq:DefinitionTildeYD}
\tilde y_D\coloneqq \tilde w_{D'} \circ (f_1,f_2).
\end{equation}
Then,  by construction, $\tilde y_D\in \tilde w_D S_{\textbf r}$.

The permutation $\tilde y_D$ has the following diagrammatic interpretation: let
$d_{\tilde w_{D'}}$ be  a reduced diagram of $\tilde w_{D'}$.
Since $(V_{i_1},U_{j_2})\in D'$ there exists a unique strand $\lambda_1$ in $d_{\tilde w_{D'}}$ starting in $\{R_{i_1-1}+1,\ldots,R_{i_1} \}$ and ending in $\{C_{j_2-1}+1,\ldots, C_{j_2}\}$. Likewise, as $(V_{i_2},U_{j_1})\in D'$ there is also a unique strand $\lambda_2$ in $d_{w_{D'}}$ which starts in $\{R_{i_2-1}+1,\ldots,R_{i_2} \}$ and ends in $\{C_{j_1-1}+1,\ldots, C_{j_1}\}$. As $i_1<i_2$ and $j_1<j_2$, the strands $\lambda_1$ and $\lambda_2$ intersect exactly once. Resolving the crossing of $\lambda_1$ and $\lambda_2$ then gives a diagram for the permutation $\tilde y_D$.

\begin{example}\label{example:ConstructionYDTilde} Consider the simple move:
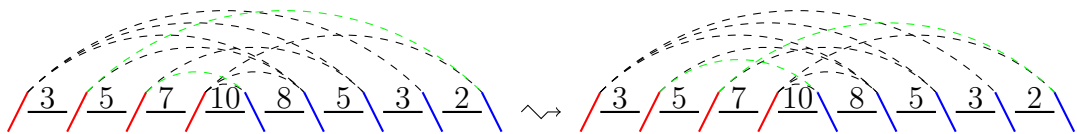
\begin{center}
\begin{tikzpicture}
[scale=.2625]
%\node at (0,0) {$D$};
\draw[thick] (0,0)--(2,0);
\draw[thick] (3,0)--(5,0);
\draw[thick] (6,0)--(8,0);
\draw[thick] (9,0)--(11,0);
\draw[thick] (12,0)--(14,0);
\draw[thick] (15,0)--(17,0);
\draw[thick] (18,0)--(20,0);
\draw[thick] (21,0)--(23,0);
%\draw[thick] (24,0)--(26,0);

\draw [thick,red] (-1,-1) --(0,1); 
\draw [thick,red] (2,-1) --(3,1); 
\draw [thick,red] (5,-1) --(6,1); 
\draw [thick,red] (8,-1) --(9,1);
%\draw [thick,red] (11,-1) --(12,1);

\draw [thick,blue] (12,-1) --(11,1); 
\draw [thick,blue] (15,-1) --(14,1); 
\draw [thick,blue] (18,-1) --(17,1); 
\draw [thick,blue] (21,-1) --(20,1); 
\draw [thick,blue] (24,-1) --(23,1); 

\node at (1,0.7){$3$};
\node at (4,0.7){$5$};
\node at (7,0.7){$7$};
\node at (10,0.7){$10$};
\node at (13,0.7){$8$};
\node at (16,0.7){$5$};
\node at (19,0.7){$3$};
\node at (22,0.7){$2$};

\draw[dashed] (9,1) to[out=45, in=135] (11,1);
\draw[dashed] (9,1) to[out=45, in=135] (14,1);
\draw[dashed] (9,1) to[out=45, in=135] (23,1);
\draw[dashed] (6,1) to[out=45, in=135] (17,1);
\draw[dashed] (3,1) to[out=45, in=135] (14,1);
\draw[dashed] (0,1) to[out=45, in=135] (14,1);
\draw[dashed] (0,1) to[out=45, in=135] (17,1);
\draw[dashed] (0,1) to[out=45, in=135] (20,1);

\draw[dashed, very thick, ForestGreen] (3,1) to[out=45, in=135] (23,1);
\draw[dashed, very thick, ForestGreen] (6,1) to[out=45, in=135] (11,1);

\draw[-to] decorate[decoration=zigzag] {(25,0)--(27,0)};
\def\s{29};

\draw[thick] (0+\s,0)--(2+\s,0);
\draw[thick] (3+\s,0)--(5+\s,0);
\draw[thick] (6+\s,0)--(8+\s,0);
\draw[thick] (9+\s,0)--(11+\s,0);
\draw[thick] (12+\s,0)--(14+\s,0);
\draw[thick] (15+\s,0)--(17+\s,0);
\draw[thick] (18+\s,0)--(20+\s,0);
\draw[thick] (21+\s,0)--(23+\s,0);
%\draw[thick] (24,0)--(26,0);

\draw [thick,red] (-1+\s,-1) --(0+\s,1); 
\draw [thick,red] (2+\s,-1) --(3+\s,1); 
\draw [thick,red] (5+\s,-1) --(6+\s,1); 
\draw [thick,red] (8+\s,-1) --(9+\s,1);
%\draw [thick,red] (11,-1) --(12,1);

\draw [thick,blue] (12+\s,-1) --(11+\s,1); 
\draw [thick,blue] (15+\s,-1) --(14+\s,1); 
\draw [thick,blue] (18+\s,-1) --(17+\s,1); 
\draw [thick,blue] (21+\s,-1) --(20+\s,1); 
\draw [thick,blue] (24+\s,-1) --(23+\s,1); 

\node at (1+\s,0.7){$3$};
\node at (4+\s,0.7){$5$};
\node at (7+\s,0.7){$7$};
\node at (10+\s,0.7){$10$};
\node at (13+\s,0.7){$8$};
\node at (16+\s,0.7){$5$};
\node at (19+\s,0.7){$3$};
\node at (22+\s,0.7){$2$};

\draw[dashed] (9+\s,1) to[out=45, in=135] (11+\s,1);
\draw[dashed] (9+\s,1) to[out=45, in=135] (14+\s,1);
\draw[dashed] (9+\s,1) to[out=45, in=135] (23+\s,1);
\draw[dashed] (6+\s,1) to[out=45, in=135] (17+\s,1);
\draw[dashed] (3+\s,1) to[out=45, in=135] (14+\s,1);
\draw[dashed] (0+\s,1) to[out=45, in=135] (14+\s,1);
\draw[dashed] (0+\s,1) to[out=45, in=135] (17+\s,1);
\draw[dashed] (0+\s,1) to[out=45, in=135] (20+\s,1);

\draw[dashed, very thick, ForestGreen] (3+\s,1) to[out=45, in=135] (11+\s,1);
\draw[dashed, very thick, ForestGreen] (6+\s,1) to[out=45, in=135] (23+\s,1);
\end{tikzpicture}
\end{center}
We denote the tie diagram on the left by $D$ and the one on the right by $D'$. The moving ties are $(V_2,U_1)$ and $(V_3,U_5)$. Note that $n=10$, $\textbf r=(3,2,2,3)$ and $\textbf c=(2,3,2,1,2)$.
By construction, the binary contingency table $M(D')$ equals the matrix $A$ from Example~\ref{example:ShortestCosetRepresentative} where we also constructed the corresponding shortest $(S_{\textbf c},S_{\textbf r})$-double coset representative $\tilde w_{D'}=13961024578$. We have
$\tilde w_{D'}(5)\in\{9,10\}=\{C_3+1,\ldots,C_4\}$ and $\tilde w_{D'}(6)\in\{1,2\}=\{C_0+1,\ldots,C_1\}$
, so $f_1=6$ and $f_2=7$. This gives $\tilde y_D=13962104578$. The diagrammatic construction of $\tilde y_D$ is illustrated as follows:
%in Figure~\ref{figure:ExampleConstructionTildeYD}.
%\begin{figure}[h]
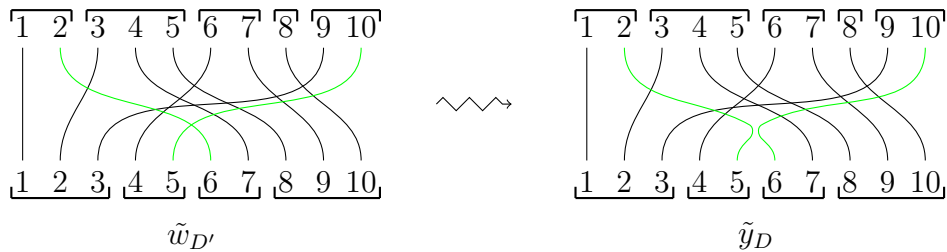
\begin{center}
\begin{tikzpicture}[scale=.5]

\def \s {0};
\def \t {0};

\coordinate (B1) at (0+\s,-3-\t);
\coordinate (B2) at (1+\s,-3-\t);
\coordinate (B3) at (2+\s,-3-\t);
\coordinate (B4) at (3+\s,-3-\t);
\coordinate (B5) at (4+\s,-3-\t);
\coordinate (B6) at (5+\s,-3-\t);
\coordinate (B7) at (6+\s,-3-\t);
\coordinate (B8) at (7+\s,-3-\t);
\coordinate (B9) at (8+\s,-3-\t);
\coordinate (B10) at (9+\s,-3-\t);

\node[below=0 of B1] {$1$};
\node[below=0 of B2] {$2$};
\node[below=0 of B3] {$3$};
\node[below=0 of B4] {$4$};
\node[below=0 of B5] {$5$};
\node[below=0 of B6] {$6$};
\node[below=0 of B7] {$7$};
\node[below=0 of B8] {$8$};
\node[below=0 of B9] {$9$};
\node[below=0 of B10] {$10$};

\coordinate (T1) at (0+\s,0-\t);
\coordinate (T2) at (1+\s,0-\t);
\coordinate (T3) at (2+\s,0-\t);
\coordinate (T4) at (3+\s,0-\t);
\coordinate (T5) at (4+\s,0-\t);
\coordinate (T6) at (5+\s,0-\t);
\coordinate (T7) at (6+\s,0-\t);
\coordinate (T8) at (7+\s,0-\t);
\coordinate (T9) at (8+\s,0-\t);
\coordinate (T10) at (9+\s,0-\t);

\node[above=0 of T1] {$1$};
\node[above=0 of T2] {$2$};
\node[above=0 of T3] {$3$};
\node[above=0 of T4] {$4$};
\node[above=0 of T5] {$5$};
\node[above=0 of T6] {$6$};
\node[above=0 of T7] {$7$};
\node[above=0 of T8] {$8$};
\node[above=0 of T9] {$9$};
\node[above=0 of T10] {$10$};

\draw (B1) to[out=90,in=-90] (T1);
\draw (B2) to[out=90,in=-90] (T3);
\draw (B3) to[out=90,in=-90] (T9);
\draw (B4) to[out=90,in=-90] (T6);
\draw (B7) to[out=90,in=-90] (T4);
\draw (B8) to[out=90,in=-90] (T5);
\draw (B9) to[out=90,in=-90] (T7);
\draw (B10) to[out=90,in=-90] (T8);

\draw[thick, ForestGreen] (B5) to[out=90,in=-90] (T10);
\draw[thick, ForestGreen] (B6) to[out=90,in=-90] (T2);

\draw[black, thick] (-0.3+\s,1-\t) -- (1.3+\s,1-\t);
\draw[black, thick] (-0.3+\s,1-\t) -- (-0.3+\s,0.7-\t);
\draw[black, thick] (1.3+\s,1-\t) -- (1.3+\s,0.7-\t);

\draw[black, thick] (1.7+\s,1-\t) -- (4.3+\s,1-\t);
\draw[black, thick] (1.7+\s,1-\t) -- (1.7+\s,0.7-\t);
\draw[black, thick] (4.3+\s,1-\t) -- (4.3+\s,0.7-\t);

\draw[black, thick] (4.7+\s,1-\t) -- (6.3+\s,1-\t);
\draw[black, thick] (4.7+\s,1-\t) -- (4.7+\s,0.7-\t);
\draw[black, thick] (6.3+\s,1-\t) -- (6.3+\s,0.7-\t);

\draw[black, thick] (6.7+\s,1-\t) -- (7.3+\s,1-\t);
\draw[black, thick] (6.7+\s,1-\t) -- (6.7+\s,0.7-\t);
\draw[black, thick] (7.3+\s,1-\t) -- (7.3+\s,0.7-\t);

\draw[black, thick] (7.7+\s,1-\t) -- (9.5+\s,1-\t);
\draw[black, thick] (7.7+\s,1-\t) -- (7.7+\s,0.7-\t);
\draw[black, thick] (9.5+\s,1-\t) -- (9.5+\s,0.7-\t);

\draw[black, thick] (-0.3+\s,-4-\t) -- (2.3+\s,-4-\t);
\draw[black, thick] (-0.3+\s,-4-\t) -- (-0.3+\s,-3.7-\t);
\draw[black, thick] (2.3+\s,-4-\t) -- (2.3+\s,-3.7-\t);

\draw[black, thick] (2.7+\s,-4-\t) -- (4.3+\s,-4-\t);
\draw[black, thick] (2.7+\s,-4-\t) -- (2.7+\s,-3.7-\t);
\draw[black, thick] (4.3+\s,-4-\t) -- (4.3+\s,-3.7-\t);

\draw[black, thick] (4.7+\s,-4-\t) -- (6.3+\s,-4-\t);
\draw[black, thick] (4.7+\s,-4-\t) -- (4.7+\s,-3.7-\t);
\draw[black, thick] (6.3+\s,-4-\t) -- (6.3+\s,-3.7-\t);

\draw[black, thick] (6.7+\s,-4-\t) -- (9.5+\s,-4-\t);
\draw[black, thick] (6.7+\s,-4-\t) -- (6.7+\s,-3.7-\t);
\draw[black, thick] (9.5+\s,-4-\t) -- (9.5+\s,-3.7-\t);

\node at (4.5,-5) {$\tilde w_{D'}$};

\draw[-to] decorate[decoration=zigzag] {(11,-1.5) -- (13,-1.5)};

\def \s {15};
\def \t {0};

\node at (4.5+\s,-5) {$\tilde y_{D}$};

\coordinate (B1) at (0+\s,-3-\t);
\coordinate (B2) at (1+\s,-3-\t);
\coordinate (B3) at (2+\s,-3-\t);
\coordinate (B4) at (3+\s,-3-\t);
\coordinate (B5) at (4+\s,-3-\t);
\coordinate (B6) at (5+\s,-3-\t);
\coordinate (B7) at (6+\s,-3-\t);
\coordinate (B8) at (7+\s,-3-\t);
\coordinate (B9) at (8+\s,-3-\t);
\coordinate (B10) at (9+\s,-3-\t);

\node[below=0 of B1] {$1$};
\node[below=0 of B2] {$2$};
\node[below=0 of B3] {$3$};
\node[below=0 of B4] {$4$};
\node[below=0 of B5] {$5$};
\node[below=0 of B6] {$6$};
\node[below=0 of B7] {$7$};
\node[below=0 of B8] {$8$};
\node[below=0 of B9] {$9$};
\node[below=0 of B10] {$10$};

\coordinate (T1) at (0+\s,0-\t);
\coordinate (T2) at (1+\s,0-\t);
\coordinate (T3) at (2+\s,0-\t);
\coordinate (T4) at (3+\s,0-\t);
\coordinate (T5) at (4+\s,0-\t);
\coordinate (T6) at (5+\s,0-\t);
\coordinate (T7) at (6+\s,0-\t);
\coordinate (T8) at (7+\s,0-\t);
\coordinate (T9) at (8+\s,0-\t);
\coordinate (T10) at (9+\s,0-\t);

\node[above=0 of T1] {$1$};
\node[above=0 of T2] {$2$};
\node[above=0 of T3] {$3$};
\node[above=0 of T4] {$4$};
\node[above=0 of T5] {$5$};
\node[above=0 of T6] {$6$};
\node[above=0 of T7] {$7$};
\node[above=0 of T8] {$8$};
\node[above=0 of T9] {$9$};
\node[above=0 of T10] {$10$};

\draw (B1) to[out=90,in=-90] (T1);
\draw (B2) to[out=90,in=-90] (T3);
\draw (B3) to[out=90,in=-90] (T9);
\draw (B4) to[out=90,in=-90] (T6);
\draw (B7) to[out=90,in=-90] (T4);
\draw (B8) to[out=90,in=-90] (T5);
\draw (B9) to[out=90,in=-90] (T7);
\draw (B10) to[out=90,in=-90] (T8);

\draw[draw=none,name path=line5] (B5) to[out=90,in=-90] (T10);
\draw[draw=none,name path=line6] (B6) to[out=90,in=-90] (T2);

\path [name intersections={of=line5 and line6,by=Y}];
\coordinate[below left=0.18 and 0.08 of Y] (Y1);
\coordinate[above left =0 and 0.08 of Y] (Y2);

\coordinate[below right=0.18 and 0.08 of Y] (Y3);
\coordinate[above right =0 and 0.08 of Y] (Y4);

\draw[thick, ForestGreen] (B5) to[out=90,in=-135] (Y1) to[out=45,in=-30] (Y2) to[out=150,in=-90] (T2);
\draw[thick, ForestGreen] (B6) to[out=90,in=-45] (Y3) to[out=135,in=-150] (Y4) to[out=30,in=-90] (T10);

\draw[black, thick] (-0.3+\s,1-\t) -- (1.3+\s,1-\t);
\draw[black, thick] (-0.3+\s,1-\t) -- (-0.3+\s,0.7-\t);
\draw[black, thick] (1.3+\s,1-\t) -- (1.3+\s,0.7-\t);

\draw[black, thick] (1.7+\s,1-\t) -- (4.3+\s,1-\t);
\draw[black, thick] (1.7+\s,1-\t) -- (1.7+\s,0.7-\t);
\draw[black, thick] (4.3+\s,1-\t) -- (4.3+\s,0.7-\t);

\draw[black, thick] (4.7+\s,1-\t) -- (6.3+\s,1-\t);
\draw[black, thick] (4.7+\s,1-\t) -- (4.7+\s,0.7-\t);
\draw[black, thick] (6.3+\s,1-\t) -- (6.3+\s,0.7-\t);

\draw[black, thick] (6.7+\s,1-\t) -- (7.3+\s,1-\t);
\draw[black, thick] (6.7+\s,1-\t) -- (6.7+\s,0.7-\t);
\draw[black, thick] (7.3+\s,1-\t) -- (7.3+\s,0.7-\t);

\draw[black, thick] (7.7+\s,1-\t) -- (9.5+\s,1-\t);
\draw[black, thick] (7.7+\s,1-\t) -- (7.7+\s,0.7-\t);
\draw[black, thick] (9.5+\s,1-\t) -- (9.5+\s,0.7-\t);

\draw[black, thick] (-0.3+\s,-4-\t) -- (2.3+\s,-4-\t);
\draw[black, thick] (-0.3+\s,-4-\t) -- (-0.3+\s,-3.7-\t);
\draw[black, thick] (2.3+\s,-4-\t) -- (2.3+\s,-3.7-\t);

\draw[black, thick] (2.7+\s,-4-\t) -- (4.3+\s,-4-\t);
\draw[black, thick] (2.7+\s,-4-\t) -- (2.7+\s,-3.7-\t);
\draw[black, thick] (4.3+\s,-4-\t) -- (4.3+\s,-3.7-\t);

\draw[black, thick] (4.7+\s,-4-\t) -- (6.3+\s,-4-\t);
\draw[black, thick] (4.7+\s,-4-\t) -- (4.7+\s,-3.7-\t);
\draw[black, thick] (6.3+\s,-4-\t) -- (6.3+\s,-3.7-\t);

\draw[black, thick] (6.7+\s,-4-\t) -- (9.5+\s,-4-\t);
\draw[black, thick] (6.7+\s,-4-\t) -- (6.7+\s,-3.7-\t);
\draw[black, thick] (9.5+\s,-4-\t) -- (9.5+\s,-3.7-\t);
\end{tikzpicture}

%\caption{Construction of $\tilde y_D$ from Example~\ref{example:ConstructionYDTilde}. We highlighted the strands and the crossing whose resolution gives $\tilde y_D$.}
%\label{figure:ExampleConstructionTildeYD}
%%\end{figure}
\end{center}
Here, we highlighted the strands whose resolution gives $\tilde y_D$.
\end{example}

Comparing the length of $\tilde w_D$ and $\tilde y_D$ gives the sign we attached to $D$ and $D'$:

\begin{prop}\label{prop:SignSMandPermutations} We have
$
(-1)^{\ell(\tilde w_D)+\ell(\tilde y_D)}=\operatorname{sgn}(D,D').
$
\end{prop}
\begin{proof} By construction, $\tilde w_{D'}(R_{i-1}+l)\in \{ C_{F_{M(D'),i}(l)-1}+1,\ldots C_{F_{M(D'),i}(l)} \}$ for all $i,l$, where we used the notation from Subsection~\ref{subsection:ConstructionShortestRepresentative}. This directly implies that the set
\[
\{(i,j)\mid 
\textup{there exists $l$ with $R_{l-1}+1\le i<j\le R_l$ and $\tilde y_D(i)>\tilde y_D(j)$}\}
\]
equals
\[
\{(f_1,f_1+1),\ldots ,(f_1,f_1+n_1)\} \cup \{(f_2-1,f_2),\ldots ,(f_2-n_2,f_2)\}.
\]
Here, $n_1$ and $n_2$ are defined as is Definition~\ref{definition:SignsOfSimpleMoves}.
Since $\tilde w_D$ is the shortest representative of $\tilde y_DS_{\textbf r}$, we conclude
$\ell(\tilde y_D)=\ell(\tilde w_D)+n_1+n_2$ which proves the proposition.
\end{proof}

Theorem~\ref{thm:CMAntidominantChamber} implies now a Chevalley--Monk formula with simplified signs:

\begin{cor}\label{cor:CMAntidominantWithoutSigns}
Let $D\in\mathrm{Tie}(\mathcal D)$. Then, the following identity holds in $H_{\mathbb T}^\ast(\mathcal C(\mathcal D))_{\mathrm{loc}}:$
\[
c_1(\xi_i)\cup \mathrm{Stab}'_{\mathfrak C_-}(D)=\iota_{D}^\ast(c_1(\xi_i))\cdot \mathrm{Stab}'_{\mathfrak C_-}(D) -h\cdot \Big( \sum_{D'\in \mathrm{SM}_{D,i}} \mathrm{Stab}'_{\mathfrak C_-}({D'})\Big),
\]
for $i=1,\ldots,M$, where
$
\mathrm{Stab'}_{\mathfrak C_-}(T)=(-1)^{\ell(\tilde w_T)}\mathrm{Stab}_{\mathfrak C_-}(T)
$
for $T\in\mathrm{Tie}(\mathcal D)$.
\end{cor}
\begin{proof} By construction, $\tilde y_D$ is obtained from $\tilde w_{D'}$ by pre\-com\-po\-si\-tion with a transposition. Hence, we have  $(-1)^{\ell(\tilde y_D)}=(-1)^{\ell(\tilde w_{D'})+1}$. Thus, Pro\-po\-si\-tion~\ref{prop:SignSMandPermutations} yields
\[
\operatorname{sgn}(D,D')
=
(-1)^{\ell(\tilde w_D)+\ell(\tilde y_D)}
= 
(-1)^{\ell(\tilde w_D)+\ell(\tilde w_{D'})+1}
\]
which proves the corollary.
\end{proof}

\subsection{Divisibility and Approximation} \label{subsection:DivAndApprox}
We now consider divisibility and approximation results for localization coefficients of stable basis elements. These results are essential ingredients of the proof of Theorem~\ref{thm:CMAntidominantChamber}. We first formulate the results and deduce some consequences. The proofs are then given in Subsections~\ref{subsection:ProofPropDivisibility} and~\ref{subsection:ProofPropClosetsNeighbors}.

\begin{prop}[$h^2$-Divisibility]\label{prop:DivisibilityResult}
The localization coefficient $\iota_{D'}^\ast (\mathrm{Stab}_{\mathfrak C_-}(D))$ is divisible by $h^2$, for $D\in\mathrm{Tie}(\mathcal D)$ and $D'\notin \mathrm{SM}_{D}\cup\{D\}$.
\end{prop}

We prove Proposition~\ref{prop:DivisibilityResult} in Subsection~\ref{subsection:ProofPropDivisibility}. By applying Theorem~\ref{thm:ComparisonOfChambers}, we deduce the analogous result for the chamber $\mathfrak C_+$:

\begin{cor}\label{cor:DivisibilityResultOppositeChamber}
We have that $\iota_{D}^\ast (\mathrm{Stab}_{\mathfrak C_+}(D'))$ is divisible by $h^2$ for $D\in\mathrm{Tie}(\mathcal D)$ and $D'\notin \mathrm{SM}_{D}\cup\{D\}$.
\end{cor}
\begin{proof} As before, let $w_{0,N}\in S_N$ be the longest element.
By Theorem~\ref{thm:ComparisonOfChambers}, we have
\begin{equation}\label{eq:DivisibilityOppositeChamber}
\iota_{D}^\ast (\widetilde{\mathrm{Stab}}_{\mathfrak C_+}({D'}))
=
w_{0,N}.(\iota_{w_{0,N}.D}^\ast (\widetilde{\mathrm{Stab}}_{\mathfrak C_-}(w_{0,N}.D'))).
\end{equation}
Since $D'\notin \mathrm{SM}_{D}$ if and only if $w_{0,N}.D\notin \mathrm{SM}_{w_{0,N}.D'}$, Proposition~\ref{prop:DivisibilityResult} implies that the right hand side of~\eqref{eq:DivisibilityOppositeChamber} is divisible by $h^2$. Thus, $\iota_{D}^\ast (\widetilde{\mathrm{Stab}}_{\mathfrak C_+}(D'))$ is divisible by $h^2$ and hence also $\iota_{D}^\ast (\mathrm{Stab}_{\mathfrak C_+}({D'}))$.
\end{proof}

Combining Proposition~\ref{prop:DivisibilityResult} and Corollary~\ref{cor:DivisibilityResultOppositeChamber} gives the following divisibility :
\begin{cor}[$h^2$-Divisibility of products]\label{cor:RestrictionCongruence} Let $D$, $D'$, $T\in\mathrm{Tie}(\mathcal D)$ such that $T\notin \{D,{D'}\}$ or $D'\notin \mathrm{SM}_D\cup\{D\}$. Then, we have
\begin{equation}\label{eq:CorollaryRestrictionCongruence}
\iota_T^\ast(\mathrm{Stab}_{\mathfrak C_-}(D)\cup \mathrm{Stab}_{\mathfrak C_+}(D')) \equiv 0\quad\mathrm{mod}\;h^2.
\end{equation}
\end{cor}
\begin{proof} If $T\notin \{D,D'\}$ then the smallness condition implies that both $\iota_T^\ast(\mathrm{Stab}_{\mathfrak C_-}(D))$ and $\iota_T^\ast(\mathrm{Stab}_{\mathfrak C_+}(D'))$ are divisible by $h$ which gives that \eqref{eq:CorollaryRestrictionCongruence} is divisible by $h^2$. If $T=D$ and $D'\notin \mathrm{SM}_D\cup \{D\}$ then, by Corollary~\ref{cor:DivisibilityResultOppositeChamber}, $\iota_T^\ast(\mathrm{Stab}_{\mathfrak C_+}(D'))$ is divisible by $h^2$ and so is \eqref{eq:CorollaryRestrictionCongruence}. Likewise, if $T=D'$ and $D'\notin \mathrm{SM}_D\cup \{D\}$ then Proposition~\ref{prop:DivisibilityResult} implies that $\iota_T^\ast(\mathrm{Stab}_{\mathfrak C_-}(D))$ is divisible by $h^2$ and hence also  \eqref{eq:CorollaryRestrictionCongruence}.
\end{proof}

We proceed with the following statement about $h^2$-approximations of localization coefficients:

\begin{prop}[$h^2$-Approximation]\label{prop:ClosestNeighbors} Let $D\in\mathrm{Tie}(\mathcal D)$ and $D'\in \mathrm{SM}_D$ with $(V_{i_1},U_{j_1})$ be the right moving tie and $(V_{i_2},U_{j_2})$ be the left moving tie of $D$. Then, we have
\[
\frac{\iota_{D'}^\ast (\mathrm{Stab}_{\mathfrak C_-}(D))}{e_{\mathbb T}(T_{D'}\mathcal C(\mathcal D)_{\mathfrak C_-}^- )} \equiv  \operatorname{sgn}(D,D')\frac{h}{t_{j_1}-t_{j_2}}\quad\mathrm{mod}\;h^2
\]
in $S^{-1}H^\ast_{\mathbb T}(\mathcal C(\mathcal D))$. Here, $S$ is defined as in~\eqref{eq:DefinitionS}.
\end{prop}

As before, Theorem~\ref{thm:ComparisonOfChambers} gives the analogous statement for $\mathfrak C_+$:

\begin{cor}\label{cor:ClosestNeighborsOppositeChamber}
Let $D\in\mathrm{Tie}(\mathcal D)$ and $D'\in \mathrm{SM}_D$ with $(V_{i_1},U_{j_1})$ be the right moving tie and $(V_{i_2},U_{j_2})$ be the left moving tie of $D$. Then, we have
\[
\frac{\iota_{D}^\ast (\mathrm{Stab}_{\mathfrak C_+}(D'))}{e_{\mathbb T}(T_{D}\mathcal C(\mathcal D)_{\mathfrak C_+}^- )} \equiv  \operatorname{sgn}(D,D')\frac{h}{t_{j_2}-t_{j_1}}\quad\mathrm{mod}\;h^2
\]
in $S^{-1}H^\ast_{\mathbb T}(\mathcal C(\mathcal D))$.
\end{cor}
\begin{proof} 
By the definition of $\widetilde{\mathrm{Stab}}$, see \eqref{equation:DefinitionStabTilde}, we have
\begin{equation}\label{eq:ProofCorClosestNeighborsOpposite1}
\frac{\iota_{D}^\ast(\mathrm{Stab}_{\mathfrak C_+}(D'))}{e_{\mathbb T}(T_D\mathcal C(\mathcal D)_{\mathfrak C_+}^-)}
=\frac{\iota_{D}^\ast(\widetilde{\mathrm{Stab}}_{\mathfrak C_+}(D'))}
{\iota_{D}^\ast(\widetilde{\mathrm{Stab}}_{\mathfrak C_+}(D))},
\quad
\frac{\iota_{D}^\ast(\mathrm{Stab}_{\mathfrak C_-}(D'))}{e_{\mathbb T}(T_D\mathcal C(\mathcal D)_{\mathfrak C_-}^-)}
=\frac{\iota_{D}^\ast(\widetilde{\mathrm{Stab}}_{\mathfrak C_-}(D'))}
{\iota_{D}^\ast(\widetilde{\mathrm{Stab}}_{\mathfrak C_-}(D))}.
\end{equation}
In addition, Theorem~\ref{thm:ComparisonOfChambers} yields
\begin{equation}\label{eq:ProofCorClosestNeighborsOpposite2}
\frac{\iota_{D}^\ast(\widetilde{\mathrm{Stab}}_{\mathfrak C_+}(D'))}
{\iota_{D}^\ast(\widetilde{\mathrm{Stab}}_{\mathfrak C_+}(D))}
=
w_{0,N}.\bigg( \frac{\iota_{w_{0,N}.D }^\ast(\widetilde{\mathrm{Stab}}_{\mathfrak C_-}(w_{0,N}.D'))}
{\iota_{w_{0,N}.D}^\ast(\widetilde{\mathrm{Stab}}_{\mathfrak C_-}(w_{0,N}.D))} \bigg).
\end{equation}
The tie diagram $w_{0,N}.D$ is obtained from $w_{0,N}.D'$ via a simple move where $(V_{i_1},U_{N-j_2+1})$ is the right moving tie and $(V_{i_2},U_{N-j_1+1})$ is the left moving tie of $w_{0,N}.D'$.
Thus, we have
\begin{align*}
\frac{\iota_{D}^\ast(\mathrm{Stab}_{\mathfrak C_+}(D'))}{e_{\mathbb T}(T_D\mathcal C(\mathcal D)_{\mathfrak C_+}^-)} &=
w_{0,N}.\bigg( \frac{\iota_{w_{0,N}.D}^\ast(\mathrm{Stab}_{\mathfrak C_-}(w_{0,N}.D'))}{e_{\mathbb T}(T_{w_{0,N}.D}\mathcal C(\mathcal D)_{\mathfrak C_-}^- )}\bigg) \\
&\equiv w_{0,N}.\bigg(\operatorname{sgn}(w_{0,N}.D',w_{0,N}.D)\frac{h}{t_{N-j_2+1}-t_{N-j_1+1}} \bigg) \quad \mathrm{mod}\;h^2\\
&\equiv  \operatorname{sgn}(D,D')\frac{h}{t_{j_2}-t_{j_1}} \quad\mathrm{mod}\;h^2,
\end{align*}
where the first equality follows from \eqref{eq:ProofCorClosestNeighborsOpposite1} and \eqref{eq:ProofCorClosestNeighborsOpposite2}, the subsequent congruence from Proposition~\ref{prop:ClosestNeighbors} and the final congruence from  $
\operatorname{sgn}(D,D')=\operatorname{sgn}(w_{0,N}.D',w_{0,N}.D).
$
\end{proof}

\begin{remark}
In the framework of partial flag varieties,  the results of this subsection are contained in~\cite[Corollary~3.11]{su2017restriction}.
\end{remark}
\subsection{Proof of Theorem~\ref{thm:CMAntidominantChamber}} \label{subsection:ProofOfTHeoremAntidominant}
We begin with the following auxiliary statement:

\begin{lemma}\label{lemma:ProofAntiDomCMSideComputation}
Let $D\in \mathrm{Tie}(\mathcal D)$, $D'\in \mathrm{SM}_D$ with right moving tie $(V_{i_1},U_{j_1})$ and left moving tie $(V_{i_2},U_{j_2})$. Then, we have
\[
\iota_{D'}^\ast(c_1(\xi_i))-\iota_{D}^\ast(c_1(\xi_i))
\equiv
\begin{cases}
t_{j_1}-t_{j_2}\;\mathrm{mod}\;h &\textit{if $D\in \mathrm{SM}_{D,i}$,}\\
0\;\mathrm{mod}\;h &\textit{if $D\notin \mathrm{SM}_{D,i}$.}
\end{cases}
\]
\end{lemma}
\begin{proof} From \eqref{eq:RestrictionTautologicalBundles}, we immediately obtain
\begin{equation}\label{eq:ProofAntiDomCMSideComputation1}
\iota_T^\ast (c_1(\xi_i))=\sum_{U\in\mathrm{b}(\mathcal D)} d_{T,U,X_i}t_i \quad\mathrm{mod}\;h,
\quad\textup{for all $T\in\mathrm{Tie}(\mathcal D)$,}
\end{equation}
where $d_{T,U,X_i}$ is defined as in Subsection~\ref{subsection:TautologicalBundles}. According to the relative position of $X_i$ with respect to $V_{i_1}$ and $V_{i_2}$ we have the three cases illustrated in Figure~\ref{figure:RelativePositionSM}.
\begin{figure}
\begin{tikzpicture}[scale=.275]
	\node at (-2,0) {(1)};
	
	\draw[dotted] (0,0) -- (2,0);
	\draw[black, thick] (2,0) -- (4,0);
	\draw[dotted] (4,0) -- (6,0);
	\draw[dotted] (7,0) -- (9,0);
	\draw[dotted] (10,0) -- (13,0);
	\draw[dotted] (14,0) -- (16,0);
	\draw[dotted] (17,0) -- (19,0);
	
	\node at (3,-1) {$X_i$};

	\draw[red, thick] (6,-1) -- (7,1);
	\node at (6,-2) {$V_{i_2}$};
	\draw[red, thick] (9,-1) -- (10,1);
	\node at (9,-2) {$V_{i_1}$};
	\draw[blue, thick] (14,-1) -- (13,1);
	\node at (14,-2) {$U_{j_1}$};
	\draw[blue, thick] (17,-1) -- (16,1);
	\node at (17,-2) {$U_{j_2}$};

	\draw[dashed] (7,1) to[out=45, in=135] (16,1);
	\draw[dashed] (10,1) to[out=45, in=135] (13,1);
	
	\draw[-to] decorate[decoration=zigzag] {(20,0) -- (23,0)};
	
	\def\x{24};

	\draw[dotted] (0+\x,0) -- (2+\x,0);
	\draw[black, thick] (2+\x,0) -- (4+\x,0);
	\draw[dotted] (4+\x,0) -- (6+\x,0);
	\draw[dotted] (7+\x,0) -- (9+\x,0);
	\draw[dotted] (10+\x,0) -- (13+\x,0);
	\draw[dotted] (14+\x,0) -- (16+\x,0);
	\draw[dotted] (17+\x,0) -- (19+\x,0);

	\draw[red, thick] (6+\x,-1) -- (7+\x,1);
	\node at (6+\x,-2) {$V_{i_2}$};
	\draw[red, thick] (9+\x,-1) -- (10+\x,1);
	\node at (9+\x,-2) {$V_{i_1}$};
	\draw[blue, thick] (14+\x,-1) -- (13+\x,1);
	\node at (14+\x,-2) {$U_{j_1}$};
	\draw[blue, thick] (17+\x,-1) -- (16+\x,1);
	\node at (17+\x,-2) {$U_{j_2}$};
	
	\draw[dashed] (7+\x,1) to[out=45, in=135] (13+\x,1);
	\draw[dashed] (10+\x,1) to[out=45, in=135] (16+\x,1);
	
	\node at (3+\x,-1) {$X_i$};
	
	\def\t {-9};
	\node at (-2,0+\t) {(2)};
	
	\draw[dotted] (0,0+\t) -- (2,0+\t);
	\draw[black, thick] (5,0+\t) -- (7,0+\t);
	\draw[dotted] (3,0+\t) -- (5,0+\t);
	\draw[dotted] (7,0+\t) -- (9,0+\t);
	\draw[dotted] (10,0+\t) -- (13,0+\t);
	\draw[dotted] (14,0+\t) -- (16,0+\t);
	\draw[dotted] (17,0+\t) -- (19,0+\t);

	\draw[red, thick] (2,-1+\t) -- (3,1+\t);
	\node at (2,-2+\t) {$V_{i_2}$};
	\draw[red, thick] (9,-1+\t) -- (10,1+\t);
	\node at (9,-2+\t) {$V_{i_1}$};
	\draw[blue, thick] (14,-1+\t) -- (13,1+\t);
	\node at (14,-2+\t) {$U_{j_1}$};
	\draw[blue, thick] (17,-1+\t) -- (16,1+\t);
	\node at (17,-2+\t) {$U_{j_2}$};

	\draw[dashed] (3,1+\t) to[out=45, in=135] (16,1+\t);
	\draw[dashed] (10,1+\t) to[out=45, in=135] (13,1+\t);
	
	\node at (6,-1+\t) {$X_i$};
	
	\draw[-to] decorate[decoration=zigzag] {(20,0+\t) -- (23,0+\t)};

	\draw[dotted] (0+\x,0+\t) -- (2+\x,0+\t);
	\draw[black, thick] (5+\x,0+\t) -- (7+\x,0+\t);
	\draw[dotted] (3+\x,0+\t) -- (5+\x,0+\t);
	\draw[dotted] (7+\x,0+\t) -- (9+\x,0+\t);
	\draw[dotted] (10+\x,0+\t) -- (13+\x,0+\t);
	\draw[dotted] (14+\x,0+\t) -- (16+\x,0+\t);
	\draw[dotted] (17+\x,0+\t) -- (19+\x,0+\t);

	\draw[red, thick] (2+\x,-1+\t) -- (3+\x,1+\t);
	\node at (2+\x,-2+\t) {$V_{i_2}$};
	\draw[red, thick] (9+\x,-1+\t) -- (10+\x,1+\t);
	\node at (9+\x,-2+\t) {$V_{i_1}$};
	\draw[blue, thick] (14+\x,-1+\t) -- (13+\x,1+\t);
	\node at (14+\x,-2+\t) {$U_{j_1}$};
	\draw[blue, thick] (17+\x,-1+\t) -- (16+\x,1+\t);
	\node at (17+\x,-2+\t) {$U_{j_2}$};

	\draw[dashed] (3+\x,1+\t) to[out=45, in=135] (13+\x,1+\t);
	\draw[dashed] (10+\x,1+\t) to[out=45, in=135] (16+\x,1+\t);
	
	\node at (6+\x,-1+\t) {$X_i$};
	
	\def\t {-18};
	\node at (-2,0+\t) {(3)};

	\draw[dotted] (0,0+\t) -- (2,0+\t);
	\draw[black, thick] (9,0+\t) -- (11,0+\t);
	\draw[dotted] (3,0+\t) -- (6,0+\t);
	\draw[dotted] (7,0+\t) -- (9,0+\t);
	\draw[dotted] (11,0+\t) -- (16,0+\t);
	\draw[dotted] (17,0+\t) -- (19,0+\t);

	\draw[red, thick] (2,-1+\t) -- (3,1+\t);
	\node at (2,-2+\t) {$V_{i_2}$};
	\draw[red, thick] (6,-1+\t) -- (7,1+\t);
	\node at (6,-2+\t) {$V_{i_1}$};
	\draw[blue, thick] (14,-1+\t) -- (13,1+\t);
	\node at (14,-2+\t) {$U_{j_1}$};
	\draw[blue, thick] (17,-1+\t) -- (16,1+\t);
	\node at (17,-2+\t) {$U_{j_2}$};

	\draw[dashed] (3,1+\t) to[out=45, in=135] (16,1+\t);
	\draw[dashed] (7,1+\t) to[out=45, in=135] (13,1+\t);
	
	\node at (10,-1+\t) {$X_i$};
	
	\draw[-to] decorate[decoration=zigzag] {(20,0+\t) -- (23,0+\t)};
	
	\draw[dotted] (0+\x,0+\t) -- (2+\x,0+\t);
	\draw[black, thick] (9+\x,0+\t) -- (11+\x,0+\t);
	\draw[dotted] (3+\x,0+\t) -- (6+\x,0+\t);
	\draw[dotted] (7+\x,0+\t) -- (9+\x,0+\t);
	\draw[dotted] (11+\x,0+\t) -- (16+\x,0+\t);
	\draw[dotted] (17+\x,0+\t) -- (19+\x,0+\t);

	\draw[red, thick] (2+\x,-1+\t) -- (3+\x,1+\t);
	\node at (2+\x,-2+\t) {$V_{i_2}$};
	\draw[red, thick] (6+\x,-1+\t) -- (7+\x,1+\t);
	\node at (6+\x,-2+\t) {$V_{i_1}$};
	\draw[blue, thick] (14+\x,-1+\t) -- (13+\x,1+\t);
	\node at (14+\x,-2+\t) {$U_{j_1}$};
	\draw[blue, thick] (17+\x,-1+\t) -- (16+\x,1+\t);
	\node at (17+\x,-2+\t) {$U_{j_2}$};

	\draw[dashed] (3+\x,1+\t) to[out=45, in=135] (13+\x,1+\t);
	\draw[dashed] (7+\x,1+\t) to[out=45, in=135] (16+\x,1+\t);
	
	\node at (10+\x,-1+\t) {$X_i$};
\end{tikzpicture}
\caption{Relative positions of simple moves.}\label{figure:RelativePositionSM}
\end{figure}
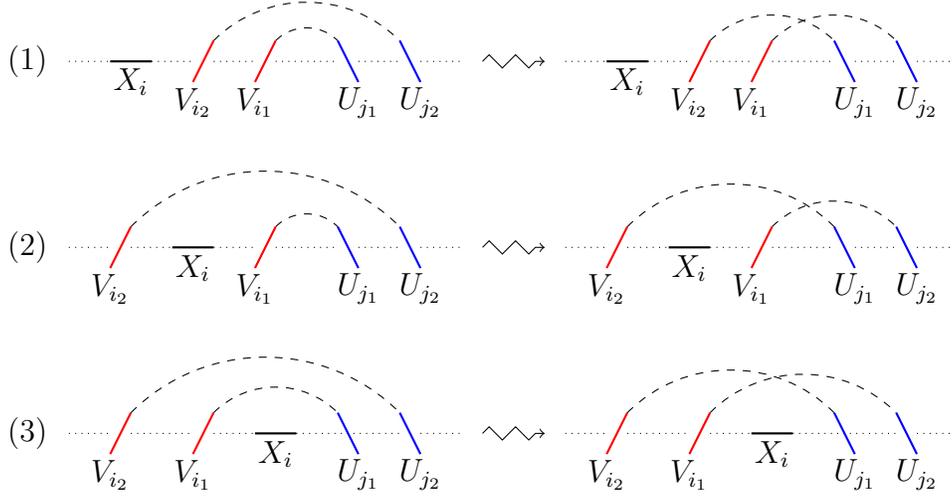
By construction, in the first and third case holds $d_{D,U,X_i}=d_{D',U,X_i}$ for all $U\in \mathrm{b}(\mathcal D)$. Hence, by \eqref{eq:ProofAntiDomCMSideComputation1}, we have $
\iota_{D'}^\ast(c_1(\xi_i))-\iota_{D}^\ast(c_1(\xi_i)) \equiv0\;\mathrm{mod}\;h
$
for $D'\notin \mathrm{SM}_{D,i}$. The second case is equivalent to $D'\in\mathrm{SM}_{D,i}$. In this case, we have 
\[
d_{D,U,X_i}=\begin{cases} d_{D',U,X_i} &\textup{if $U\in \mathrm{b}(\mathcal D)\setminus\{U_{j_1},U_{j_2}\}$,} \\
d_{D',U,X_i}-1  &\textup{if $U=U_{j_1}$,} \\
d_{D',U,X_i}+1  &\textup{if $U=U_{j_2}$.}
\end{cases}
\]
Thus, \eqref{eq:ProofAntiDomCMSideComputation1} gives
$
\iota_{D'}^\ast(c_1(\xi_i))-\iota_{D}^\ast(c_1(\xi_i)) \equiv t_{j_1}-t_{j_2}\;\mathrm{mod}\;h
$
for $D'\in \mathrm{SM}_{D,i}$.
\end{proof}

\begin{proof}[Proof of Theorem~\ref{thm:CMAntidominantChamber}]
By Theorem~\ref{thm:OrthogonalityOfStableEnvelopes}, we have to show 
\begin{equation}\label{eq:ProofOfCMAntidominantAim}
(c_{1}(\xi_i)\cup \mathrm{Stab}_{\mathfrak C_-}(D), \mathrm{Stab}_{\mathfrak C_+}(D'))_{\mathrm{virt}} =\begin{cases}
\iota_{D}^\ast(c_1(\xi_i)) &\textup{if $D'=D$,}\\
\operatorname{sgn}(D,D') h &\textup{if $D'\in \mathrm{SM}_{D,i}$,}\\
0 &\textup{otherwise.}
\end{cases}
\end{equation}
By Proposition~\ref{prop:MatrixCoefficients}, the virtual scalar product in~\eqref{eq:ProofOfCMAntidominantAim} is a linear polynomial in the equivariant parameters. Hence, it suffices to prove the above equality in $H_{\mathbb T}^\ast (\operatorname{pt})/(h^2)$. 
Suppose first that $D=D'$. As the partial orders $\preceq_+=\preceq_{\mathfrak C_+}$ and $\preceq_-=\preceq_{\mathfrak C_-}$ defined in Subsection~\ref{subsection:AttractionCells} are opposite, the support condition for stable basis elements implies that
$
\iota_T^\ast (\mathrm{Stab}_{\mathfrak C_-}(D))=0 
$
for $T\not\preceq_- D$ and $
\iota_T^\ast (\mathrm{Stab}_{\mathfrak C_+}(D))=0 
$
for $T\prec_- D$. Thus, we have 
\begin{equation}\label{eq:ProofOfAntiDOmCMDiagonalCase1}
\begin{split}
(c_{1}(\xi_i)\cup \mathrm{Stab}_{\mathfrak C_-}(D), \mathrm{Stab}_{\mathfrak C_+}(D))_{\mathrm{virt}} &= \sum_{D''\in \mathrm C(\mathcal D)^{\mathbb T}} \frac{\iota_{D''}^\ast(c_{1}(\xi_i)\cup \mathrm{Stab}_{\mathfrak C_-}(D)\cup \mathrm{Stab}_{\mathfrak C_+}(D))}{e_{\mathbb T}(T_{D''}\mathcal C(\mathcal D))}\\
&= \frac{\iota_{D}^\ast(c_{1}(\xi_i)\cup \mathrm{Stab}_{\mathfrak C_-}(D)\cup \mathrm{Stab}_{\mathfrak C_+}(D))}{e_{\mathbb T}(T_{D}\mathcal C(\mathcal D))}.
\end{split}
\end{equation}
Then, the normalization condition yields
\[
\eqref{eq:ProofOfAntiDOmCMDiagonalCase1}=\frac{\iota_{D}^\ast(c_{1}(\xi_i))\cup e_{\mathbb T}(T_{D}\mathcal C(\mathcal D)_{\mathfrak C_-}^-) \cup e_{\mathbb T}(T_{D}\mathcal C(\mathcal D)_{\mathfrak C_+}^-)}{e_{\mathbb T}(T_{D}\mathcal C(\mathcal D))}
= \iota_{D}^\ast(c_{1}(\xi_i)).
\]
This proves~\eqref{eq:ProofOfCMAntidominantAim} for $D=D'$. Next, we assume $D'\in\mathrm{SM}_D$ and $(V_{i_1},U_{j_1})$ be the right moving tie and $(V_{i_2},U_{j_2})$ be the left moving tie of $D$. By
Corollary~\ref{cor:RestrictionCongruence}, we have modulo $h^2$:
\begin{equation}\label{eq:SimplificationVirtualPairing}
\begin{split}
(c_{1}(\xi_i)\cup \mathrm{Stab}_{\mathfrak C_-}(D), \mathrm{Stab}_{\mathfrak C_+}(D'))_{\mathrm{virt}} &\equiv 
\frac{\iota_{D}^\ast(c_{1}(\xi_i)\cup \mathrm{Stab}_{\mathfrak C_-}(D)\cup \mathrm{Stab}_{\mathfrak C_+}(D'))}{e_{\mathbb T}(T_{D}\mathcal C(\mathcal D))} \\
&{\phantom{\equiv}}+\frac{\iota_{D'}^\ast(c_{1}(\xi_i)\cup \mathrm{Stab}_{\mathfrak C_-}(D)\cup \mathrm{Stab}_{\mathfrak C_+}(D'))}{e_{\mathbb T}(T_{D'}\mathcal C(\mathcal D))}.
\end{split}
\end{equation} 
Then, Corollary~\ref{cor:DivisibilityResultOppositeChamber} and the normalization condition imply 
\begin{equation}\label{eq:SimplificationVirtualPairingI}
\frac{\iota_{D}^\ast(c_{1}(\xi_i)\cup \mathrm{Stab}_{\mathfrak C_-}(D)\cup \mathrm{Stab}_{\mathfrak C_+}({D'}))}{e_{\mathbb T}(T_{D}\mathcal C(\mathcal D))}
\equiv \operatorname{sgn}(D,D')h\frac{\iota_{D}^\ast(c_{1}(\xi_i))}{t_{j_2}-t_{j_1}}\;\mathrm{mod}\;h^2,
\end{equation}
whereas Proposition~\ref{prop:ClosestNeighbors} combined with the normalization condition gives
\begin{equation}\label{eq:SimplificationVirtualPairingII}
\frac{\iota_{{D'}}^\ast(c_{1}(\xi_i)\cup \mathrm{Stab}_{\mathfrak C_-}(D)\cup \mathrm{Stab}_{\mathfrak C_+}({D'}))}{e_{\mathbb T}(T_{{D'}}\mathcal C(\mathcal D))}
\equiv
\operatorname{sgn}(D,D')h\frac{\iota_{{D'}}^\ast(c_{1}(\xi_i))}{t_{j_1}-t_{j_2}}
\;\mathrm{mod}\;h^2.
\end{equation}
Inserting \eqref{eq:SimplificationVirtualPairingI} and \eqref{eq:SimplificationVirtualPairingII} in \eqref{eq:SimplificationVirtualPairing} yields
\begin{equation}\label{eq:SimplificationVirtualPairingIII}
\eqref{eq:SimplificationVirtualPairing} \equiv 
\operatorname{sgn}(D,D')h\frac{\iota_{D}^\ast(c_{1}(\xi_i))}{t_{j_2}-t_{j_1}}\;\mathrm{mod}\;h^2.
\end{equation}
Now, Lemma~\ref{lemma:ProofAntiDomCMSideComputation} implies
\[
\eqref{eq:SimplificationVirtualPairingIII}\equiv \begin{cases}
\operatorname{sgn}(D,D')h\;\mathrm{mod}\;h^2 &\textup{if $D\in\mathrm{SM}_{D,i}$,} \\
0 \;\mathrm{mod}\;h^2 & \textup{if $D\notin\mathrm{SM}_{D,i}$.}
\end{cases}
\]
Thus, we proved \eqref{eq:ProofOfCMAntidominantAim} for $D'\in \mathrm{SM}_D$. Finally, it remains to prove \eqref{eq:ProofOfCMAntidominantAim} for $D'\notin \mathrm{SM}_D\cup\{D\}$. By Corollary~\ref{cor:RestrictionCongruence}, this assumption implies that $h^2$ divides all localization coefficients
$\iota_T^\ast(\mathrm{Stab}_{\mathfrak C_-}(D)\cup\mathrm{Stab}_{\mathfrak C_+}(D'))$. Thus, we conclude
\[
(c_{1}(\xi_i)\cup \mathrm{Stab}_{\mathfrak C_-}(D), \mathrm{Stab}_{\mathfrak C_+}(D'))_{\mathrm{virt}} \equiv 0\quad\mathrm{mod}\;h^2
\]
which proves \eqref{eq:ProofOfCMAntidominantAim} for $D'\notin \mathrm{SM}_D\cup\{D\}$ and hence completes the proof of Theorem~\ref{thm:CMAntidominantChamber}.
\end{proof}

\subsection{Proof of Proposition~\ref{prop:DivisibilityResult}}\label{subsection:ProofPropDivisibility}
Note that by Proposition~\ref{prop:EqMultiplicitiesEssentialBraneDiagram}, it suffices to show Proposition~\ref{prop:DivisibilityResult} for essential brane diagrams as defined in Subsection~\ref{subsection:ForgettingChargelessLines}. Hence, we assume throughout this subsection that $\mathcal D$ is essential.

Our crucial tool is the following:
\begin{lemma}\label{lemma:CloseNeighborDiffersBySimpleMove}
Let $D$, $D'\in\mathrm{Tie}(\mathcal D)$, $z\in w_{D'}S_{\textbf r}$ and $d_z$ a reduced diagram of shape~\eqref{eq:ShapeOfDiagrams}. Suppose $K(d_z,w_D,1)\ne \emptyset$, where $K(d_z,w_D,1)$ is defined as in Proposition~\ref{prop:ApproximationStableEnvelopeModuloPowersOfH}. Then, we have $D'\in\mathrm{SM}_D\cup\{ D\}$.
\end{lemma}

\begin{proof}
Given $K\in K(d_z,w_D,1)$ then, as $d_z$ is of shape \eqref{eq:ShapeOfDiagrams}, we distinguish between the following two cases:
\begin{enumerate}[label=(\arabic*)]
\item\label{item:CloseNeighborDiffersBySimpleMoveCase1} all crossings in $K$ are contained in the boxes corresponding to $w_{0,c_1},\ldots,w_{0,c_N}$ and $v_1,\ldots,v_M$,
\item\label{item:CloseNeighborDiffersBySimpleMoveCase2} exactly one crossing $\kappa_0\in K$ is contained in the box corresponding to $\tilde w_{D'}$ and the remaining crossings of $K$ are contained in the boxes corresponding to $w_{0,c_1},\ldots,w_{0,c_N}$.
\end{enumerate}
If \ref{item:CloseNeighborDiffersBySimpleMoveCase1} is satisfied then resolving all crossings contained in $K$ from $d_z$ still gives a permutation in $S_{\textbf c}w_{D'}S_{\textbf r}$. Thus, we have $S_{\textbf c}w_{D}S_{\textbf r}$=$S_{\textbf c}w_{D'}S_{\textbf r}$ which implies $D=D'$ by Corrolary~\ref{cor:DoubleCosetsFSep}.
Assume now that \ref{item:CloseNeighborDiffersBySimpleMoveCase2} is satisfied. 
Let $d_{\tilde w_{D'}}$ be the reduced diagram for $\tilde w_{D'}$ contained in $d_z$. 
We denote by $\lambda_1$, $\lambda_2$ the strands in $d_{\tilde w_{D'}}$ that intersect in $\kappa_0$ and let $y\in S_n$ be the permutation that is obtained from $d_{\tilde w_{D'}}$ by resolving the crossing $\kappa_0$. In pictures, $y$ is obtained as follows:
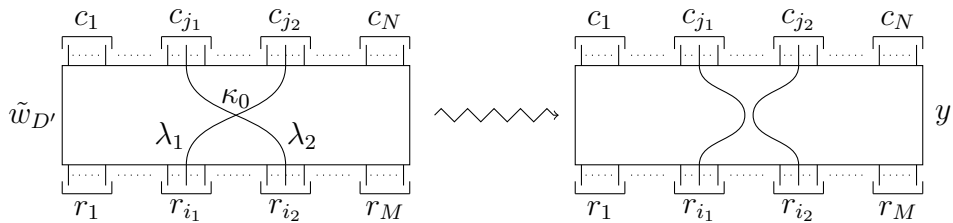
\begin{center}
\begin{tikzpicture}[scale=.55]

%Strands

\draw (0.15,0) -- (0.15,0.5);
\draw (1.05,0) -- (1.05,0.5);

\draw (1.95+0.6,0) -- (1.95+0.6,0.5);
\draw (2.85+0.6,0) -- (2.85+0.6,0.5);

\draw (2.4+0.6,0) -- (2.4+0.6,0.5);
\draw (4.2+1.2,0) -- (4.2+1.2,0.5);

\draw (3.75+1.2,0) -- (3.75+1.2,0.5);
\draw (4.65+1.2,0) -- (4.65+1.2,0.5);

\draw (5.55+1.8,0) -- (5.55+1.8,0.5);
\draw (6.45+1.8,0) -- (6.45+1.8,0.5);

\draw (0,0.7) -- (1.2,0.7);
\draw (0,0.7) -- (0,0.4);
\draw (1.2,0.7) -- (1.2,0.4);

\draw (1.8+0.6,0.7) -- (3+0.6,0.7);
\draw (1.8+0.6,0.7) -- (1.8+0.6,0.4);
\draw (3+0.6,0.7) -- (3+0.6,0.4);

\draw (3.6+1.2,0.7) -- (4.8+1.2,0.7);
\draw (3.6+1.2,0.7) -- (3.6+1.2,0.4);
\draw (4.8+1.2,0.7) -- (4.8+1.2,0.4);

\draw (5.4+1.8,0.7) -- (6.6+1.8,0.7);
\draw (5.4+1.8,0.7) -- (5.4+1.8,0.4);
\draw (6.6+1.8,0.7) -- (6.6+1.8,0.4);

\node at (0.6,1.1) {$c_1$};
\node at (2.4+0.6,1.1) {$c_{j_1}$};
\node at (4.2+1.2,1.1) {$c_{j_2}$};
\node at (6+1.8,1.1) {$c_{N}$};

\draw (2.4+0.6,-2.4) to[out=90, in=-90] (4.2+1.2,0);
\draw (4.2+1.2,-2.4) to[out=90, in=-90] (2.4+0.6,0);

\fill (4.2,-1.2) circle[radius=0.25pt];
\node at (4.2,-0.8) {$\kappa_0$};
\node at (2.6,-1.7) {$\lambda_1$};
\node at (5.8,-1.7) {$\lambda_2$};

\draw[dotted] (0.25,0.25) -- (0.95,0.25);
\draw[dotted] (2.05+0.6,0.25) -- (2.35+0.6,0.25);
\draw[dotted] (2.5+0.6,0.25) -- (2.75+0.6,0.25);
\draw[dotted] (3.85+1.2,0.25) -- (4.05+1.2,0.25);
\draw[dotted] (4.2+1.2,0.25) -- (4.55+1.2,0.25);
\draw[dotted] (5.65+1.8,0.25) -- (6.35+1.8,0.25);

\draw[dotted] (1.35,0.25) -- (2.25,0.25);
\draw[dotted] (3.75,0.25) -- (4.65,0.25);
\draw[dotted] (6.15,0.25) -- (7.05,0.25);

%Bottom part

\draw (0.15,0-2.9) -- (0.15,0.5-2.9);
\draw (1.05,0-2.9) -- (1.05,0.5-2.9);

\draw (1.95+0.6,0-2.9) -- (1.95+0.6,0.5-2.9);
\draw (2.85+0.6,0-2.9) -- (2.85+0.6,0.5-2.9);

\draw (2.4+0.6,0-2.9) -- (2.4+0.6,0.5-2.9);
\draw (4.2+1.2,0-2.9) -- (4.2+1.2,0.5-2.9);

\draw (3.75+1.2,0-2.9) -- (3.75+1.2,0.5-2.9);
\draw (4.65+1.2,0-2.9) -- (4.65+1.2,0.5-2.9);

\draw (5.55+1.8,0-2.9) -- (5.55+1.8,0.5-2.9);
\draw (6.45+1.8,0-2.9) -- (6.45+1.8,0.5-2.9);

\draw (0,0.7-3.8) -- (1.2,0.7-3.8);
\draw (0,0.7-3.8) -- (0,0.4-3.2);
\draw (1.2,0.7-3.8) -- (1.2,0.4-3.2);

\draw (1.8+0.6,0.7-3.8) -- (3+0.6,0.7-3.8);
\draw (1.8+0.6,0.7-3.8) -- (1.8+0.6,0.4-3.2);
\draw (3+0.6,0.7-3.8) -- (3+0.6,0.4-3.2);

\draw (3.6+1.2,0.7-3.8) -- (4.8+1.2,0.7-3.8);
\draw (3.6+1.2,0.7-3.8) -- (3.6+1.2,0.4-3.2);
\draw (4.8+1.2,0.7-3.8) -- (4.8+1.2,0.4-3.2);

\draw (5.4+1.8,0.7-3.8) -- (6.6+1.8,0.7-3.8);
\draw (5.4+1.8,0.7-3.8) -- (5.4+1.8,0.4-3.2);
\draw (6.6+1.8,0.7-3.8) -- (6.6+1.8,0.4-3.2);

\node at (0.6,1.1-4.6) {$r_1$};
\node at (2.4+0.6,1.1-4.6) {$r_{i_1}$};
\node at (4.2+1.2,1.1-4.6) {$r_{i_2}$};
\node at (6+1.8,1.1-4.6) {$r_{M}$};

\draw[dotted] (0.25,0.25-2.9) -- (0.95,0.25-2.9);
\draw[dotted] (2.05+0.6,0.25-2.9) -- (2.35+0.6,0.25-2.9);
\draw[dotted] (2.5+0.6,0.25-2.9) -- (2.75+0.6,0.25-2.9);
\draw[dotted] (3.85+1.2,0.25-2.9) -- (4.05+1.2,0.25-2.9);
\draw[dotted] (4.2+1.2,0.25-2.9) -- (4.55+1.2,0.25-2.9);
\draw[dotted] (5.65+1.8,0.25-2.9) -- (6.35+1.8,0.25-2.9);

\draw[dotted] (1.35,0.25-2.9) -- (2.25,0.25-2.9);
\draw[dotted] (3.75,0.25-2.9) -- (4.65,0.25-2.9);
\draw[dotted] (6.15,0.25-2.9) -- (7.05,0.25-2.9);

\draw[draw=black] (0,-2.4) rectangle ++(8.4,2.4);

\draw[-to] decorate[decoration=zigzag] {(9,-1.2) -- (12,-1.2)};

\def\s{12.4};

\draw (0.15+\s,0) -- (0.15+\s,0.5);
\draw (1.05+\s,0) -- (1.05+\s,0.5);

\draw (1.95+0.6+\s,0) -- (1.95+0.6+\s,0.5);
\draw (2.85+0.6+\s,0) -- (2.85+0.6+\s,0.5);

\draw (2.4+0.6+\s,0) -- (2.4+0.6+\s,0.5);
\draw (4.2+1.2+\s,0) -- (4.2+1.2+\s,0.5);

\draw (3.75+1.2+\s,0) -- (3.75+1.2+\s,0.5);
\draw (4.65+1.2+\s,0) -- (4.65+1.2+\s,0.5);

\draw (5.55+1.8+\s,0) -- (5.55+1.8+\s,0.5);
\draw (6.45+1.8+\s,0) -- (6.45+1.8+\s,0.5);

\draw (0+\s,0.7) -- (1.2+\s,0.7);
\draw (0+\s,0.7) -- (0+\s,0.4);
\draw (1.2+\s,0.7) -- (1.2+\s,0.4);

\draw (1.8+0.6+\s,0.7) -- (3+0.6+\s,0.7);
\draw (1.8+0.6+\s,0.7) -- (1.8+0.6+\s,0.4);
\draw (3+0.6+\s,0.7) -- (3+0.6+\s,0.4);

\draw (3.6+1.2+\s,0.7) -- (4.8+1.2+\s,0.7);
\draw (3.6+1.2+\s,0.7) -- (3.6+1.2+\s,0.4);
\draw (4.8+1.2+\s,0.7) -- (4.8+1.2+\s,0.4);

\draw (5.4+1.8+\s,0.7) -- (6.6+1.8+\s,0.7);
\draw (5.4+1.8+\s,0.7) -- (5.4+1.8+\s,0.4);
\draw (6.6+1.8+\s,0.7) -- (6.6+1.8+\s,0.4);

\node at (0.6+\s,1.1) {$c_1$};
\node at (2.4+0.6+\s,1.1) {$c_{j_1}$};
\node at (4.2+1.2+\s,1.1) {$c_{j_2}$};
\node at (6+1.8+\s,1.1) {$c_{N}$};

\draw (2.4+0.6+\s,-2.4) to[out=90, in=-90] (4.1+\s,-1.2) to[out=90, in=-90] (2.4+0.6+\s,0);
\draw (4.2+1.2+\s,-2.4) to[out=90, in=-90] (4.3+\s,-1.2) to[out=90, in=-90] (4.2+1.2+\s,0);

\draw[dotted] (0.25+\s,0.25) -- (0.95+\s,0.25);
\draw[dotted] (2.05+0.6+\s,0.25) -- (2.35+0.6+\s,0.25);
\draw[dotted] (2.5+0.6+\s,0.25) -- (2.75+0.6+\s,0.25);
\draw[dotted] (3.85+1.2+\s,0.25) -- (4.05+1.2+\s,0.25);
\draw[dotted] (4.3+1.2+\s,0.25) -- (4.55+1.2+\s,0.25);
\draw[dotted] (5.65+1.8+\s,0.25) -- (6.35+1.8+\s,0.25);

\draw[dotted] (1.35+\s,0.25) -- (2.25+\s,0.25);
\draw[dotted] (3.75+\s,0.25) -- (4.65+\s,0.25);
\draw[dotted] (6.15+\s,0.25) -- (7.05+\s,0.25);

%Bottom part

\draw (0.15+\s,0-2.9) -- (0.15+\s,0.5-2.9);
\draw (1.05+\s,0-2.9) -- (1.05+\s,0.5-2.9);

\draw (1.95+0.6+\s,0-2.9) -- (1.95+0.6+\s,0.5-2.9);
\draw (2.85+0.6+\s,0-2.9) -- (2.85+0.6+\s,0.5-2.9);

\draw (2.4+0.6+\s,0-2.9) -- (2.4+0.6+\s,0.5-2.9);
\draw (4.2+1.2+\s,0-2.9) -- (4.2+1.2+\s,0.5-2.9);

\draw (3.75+1.2+\s,0-2.9) -- (3.75+1.2+\s,0.5-2.9);
\draw (4.65+1.2+\s,0-2.9) -- (4.65+1.2+\s,0.5-2.9);

\draw (5.55+1.8+\s,0-2.9) -- (5.55+1.8+\s,0.5-2.9);
\draw (6.45+1.8+\s,0-2.9) -- (6.45+1.8+\s,0.5-2.9);

\draw (0+\s,0.7-3.8) -- (1.2+\s,0.7-3.8);
\draw (0+\s,0.7-3.8) -- (0+\s,0.4-3.2);
\draw (1.2+\s,0.7-3.8) -- (1.2+\s,0.4-3.2);

\draw (1.8+0.6+\s,0.7-3.8) -- (3+0.6+\s,0.7-3.8);
\draw (1.8+0.6+\s,0.7-3.8) -- (1.8+0.6+\s,0.4-3.2);
\draw (3+0.6+\s,0.7-3.8) -- (3+0.6+\s,0.4-3.2);

\draw (3.6+1.2+\s,0.7-3.8) -- (4.8+1.2+\s,0.7-3.8);
\draw (3.6+1.2+\s,0.7-3.8) -- (3.6+1.2+\s,0.4-3.2);
\draw (4.8+1.2+\s,0.7-3.8) -- (4.8+1.2+\s,0.4-3.2);

\draw (5.4+1.8+\s,0.7-3.8) -- (6.6+1.8+\s,0.7-3.8);
\draw (5.4+1.8+\s,0.7-3.8) -- (5.4+1.8+\s,0.4-3.2);
\draw (6.6+1.8+\s,0.7-3.8) -- (6.6+1.8+\s,0.4-3.2);

\node at (0.6+\s,1.1-4.6) {$r_1$};
\node at (2.4+0.6+\s,1.1-4.6) {$r_{i_1}$};
\node at (4.2+1.2+\s,1.1-4.6) {$r_{i_2}$};
\node at (6+1.8+\s,1.1-4.6) {$r_{M}$};

\draw[dotted] (0.25+\s,0.25-2.9) -- (0.95+\s,0.25-2.9);
\draw[dotted] (2.05+0.6+\s,0.25-2.9) -- (2.35+0.6+\s,0.25-2.9);
\draw[dotted] (2.5+0.6+\s,0.25-2.9) -- (2.75+0.6+\s,0.25-2.9);
\draw[dotted] (3.85+1.2+\s,0.25-2.9) -- (4.05+1.2+\s,0.25-2.9);
\draw[dotted] (4.3+1.2+\s,0.25-2.9) -- (4.55+1.2+\s,0.25-2.9);
\draw[dotted] (5.65+1.8+\s,0.25-2.9) -- (6.35+1.8+\s,0.25-2.9);

\draw[dotted] (1.35+\s,0.25-2.9) -- (2.25+\s,0.25-2.9);
\draw[dotted] (3.75+\s,0.25-2.9) -- (4.65+\s,0.25-2.9);
\draw[dotted] (6.15+\s,0.25-2.9) -- (7.05+\s,0.25-2.9);

\draw[draw=black] (0+\s,-2.4) rectangle ++(8.4,2.4);

\node at (8.9+\s,-1.2) {$y$};
\node at (-0.7,-1.2) {$\tilde w_{D'}$};
\end{tikzpicture}
\end{center}
Let $f_1$, $f_2$ resp.~$g_1$, $g_2$ the starting resp. endpoints of $\lambda_1$, $\lambda_2$ in $d_{\tilde w_{D'}}$. As in the above picture, we assume $f_1<f_2$. As $\tilde w_{D'}$ is the shortest element in $S_{\textbf c}\tilde w_{D'} S_{\textbf r}$ there exist $i_1<i_2$ and $j_1<j_2$ such that
\[
R_{i_1-1}<f_1\le R_{i_1},\quad
R_{i_2-1}<f_2\le R_{i_2},\quad
C_{j_1-1}<g_2\le C_{j_1},\quad
C_{j_2-1}<g_1\le C_{j_2}.
\]
Thus, we conclude
\begin{equation}\label{eq:CloseNeighborDiffersBySimpleMoveProof}
F_y(f_1)=F_{\tilde w_{D'}}(f_2)=j_1,\quad
F_y(f_2)=F_{\tilde w_{D'}}(f_1)=j_2,\quad
F_y(i)=F_{\tilde w_{D'}}(i),\quad \textup{for $i\ne f_1$, $f_2$.}
\end{equation}
Here, $F_y$, $F_{\tilde w_{D'}}$ are defined as in \eqref{eq:DefinitionFW}. By assumption, $y\in S_{\textbf c}\tilde w_D$. Thus, we have $F_y=F_{\tilde w_D}$. Hence, by passing to the associated matrices of these double cosets, we deduce that \eqref{eq:CloseNeighborDiffersBySimpleMoveProof} is equivalent to
\[
\begin{split}
    M(D)_{i_1,j_1}&=M(D)_{i_2,j_2}=1, \\
    M(D)_{i_1,j_2}&=M(D)_{i_1,j_2}=0,
\end{split}
\quad\quad
\begin{split}
    M(D')_{i_1,j_1}&=M(D')_{i_2,j_2}=0, \\
    M(D')_{i_1,j_2}&=M(D')_{i_1,j_2}=1,
\end{split}
\]
as well as $M(D)_{l,k}=M(D')_{l,k}$ for $(l,k)\notin \{(i_1,j_1),(i_1,j_2),(i_2,j_1),(i_2,j_2)\}$.
Therefore, by Lemma~\ref{lemma:SImpleMovesBCT}, $D'$ is obtained from $D$ via a simple move.
\end{proof}

The proof of Proposition~\ref{prop:DivisibilityResult} follows now from Proposition~\ref{prop:ApproximationStableEnvelopeModuloPowersOfH} and Lemma~\ref{lemma:CloseNeighborDiffersBySimpleMove}:

\begin{proof}[Proof of Proposition~\ref{prop:DivisibilityResult}]
We have to show that the localization coefficient $\iota_{D'}^\ast (\mathrm{Stab}_{\mathfrak C_-}(D))$ is divisible by $h^2$ for all $D'\notin \mathrm{SM}_D\cup\{D\}$. 
Assume $\iota_{D'}^\ast (\mathrm{Stab}_{\mathfrak C_-}(D))$ is not divisible by $h^2$ for some $D'\notin \mathrm{SM}_D\cup\{D\}$. By definition, this implies that $\iota_{D'}^\ast (\widetilde{\mathrm{Stab}}_{\mathfrak C_-}(D))$ is also not divisible by $h^2$. 
Thus, by Proposition~\ref{prop:ApproximationStableEnvelopeModuloPowersOfH}, there exists $z\in w_{D'}S_{\textbf r}$ with reduced diagram $d_z$ of shape~\eqref{eq:ShapeOfDiagrams} such that $K(d_z,w_D,1)\ne\emptyset$.
Hence, Lemma~\ref{lemma:CloseNeighborDiffersBySimpleMove} yields $D'\in \mathrm{SM}_D\cup\{D\}$ which contradicts our assumption $D'\notin \mathrm{SM}_D\cup\{D\}$.
\end{proof}

\subsection{Proof of Proposition~\ref{prop:ClosestNeighbors}}\label{subsection:ProofPropClosetsNeighbors}
Again, we can assume that $\mathcal D$ is essential.

We need some further notation: let $D\in\mathrm{Tie}(\mathcal D)$ and $D'\in\mathrm{SM}_D$ with right moving tie $(V_{i_1},U_{j_1})$ and left moving tie $(V_{i_2},U_{j_2})$. Given $z\in w_{D'}S_{\textbf r}$ with a reduced diagram $d_z$ of shape~\eqref{eq:ShapeOfDiagrams} then there exist unique strands $\lambda_1$, $\lambda_2$ in $d_z$ with starting points $f_1$, $f_2$ and endpoints $g_1$, $g_2$ such that
\[
R_{i_1-1}<f_1\le R_{i_1},\quad
R_{i_2-1}<f_2\le R_{i_2},\quad
C_{j_1-1}<g_2\le C_{j_1},\quad
C_{j_2-1}<g_1\le C_{j_2}.
\]
Let $\kappa_0$ denote the crossing of $\lambda_1$ and $\lambda_2$.

To prove Proposition~\ref{prop:ClosestNeighbors} we utilize the approximation formula of Proposition~\ref{prop:ApproximationStableEnvelopeModuloPowersOfH}. To apply this formula appropriately, we use the following lemma:

\begin{lemma}\label{lemma:CrossingsAndSimpleMoves} Let $\tilde y_D\in S_n$ be as in \eqref{eq:DefinitionTildeYD} and $y_D=(w_{0,c_1}\times\ldots\times w_{0,c_N})\tilde y_D$. Then, we have
\[
K(d_z,y_D,1)= \begin{cases}
\{\kappa_0\} &\textit{if $z=w_{D'}$,}\\
\emptyset &\textit{if $z\ne w_{D'}$,}
\end{cases}
\]
where $K(d_z,y_D,1)$ is defined as in Proposition~\ref{prop:ApproximationStableEnvelopeModuloPowersOfH}.
\end{lemma}
\begin{proof}
Let $z=w_{D'}v$ where $v\in S_{\textbf r}$ and suppose $K\in K(d_z,y_D,1)$. As in Subsection~\ref{subsection:ApproximationReducedStableEnvelopes}, we denote by $K_U(d_z)$ the crossings in $d_z$ corresponding to the boxes of $w_{0,c_1},\ldots,w_{0,c_N}$. By assumption $|K\setminus K_U(d_z)|\le 1$. Thus, as $z$ is fully separated, we have $K\setminus K_U(d_z)=\{\kappa_0\}$. By construction, resolving the crossing $\kappa_0$ from $d_z$ gives a diagram for $y_Dv$. Hence, Proposition~\ref{prop:UniquenessOfPrePostComposition} implies that $v=\operatorname{id}$ and $K\cap K_U(d_z)=\emptyset$ which proves the lemma.
\end{proof}

By combining Proposition~\ref{prop:ApproximationStableEnvelopeModuloPowersOfH} and Lemma~\ref{lemma:CrossingsAndSimpleMoves}, we obtain the following consequence:

\begin{cor}\label{cor:ApproximationEMOfSimpleMove} We have that $\iota_{D'}^\ast(\widetilde{\mathrm{Stab}}_{\mathfrak C_-}(D))$ is congruent modulo $h^2$ to
\[
\frac{ \operatorname{sgn}(D,D')\cdot h \cdot
\Big(\prod_{\alpha \in L'_{w_{D'}}}\Psi_{\mathcal D}(\alpha+h) \Big)
\cdot \Big( \prod_{\kappa\in K_W(d_{w_{D'}})\setminus\{\kappa_0\}} \Psi_{\mathcal D}(\mathrm{wt}(\kappa))  \Big) }
{ \prod_{\beta \in R_{\textbf r}}\Psi_{\mathcal D}(w_{D'}.\beta)}.
\]
Here, we used the notation from Proposition~\ref{prop:ApproximationStableEnvelopeModuloPowersOfH} and $ K_W(d_{w_{D'}})$ is defined as in Subsection~\ref{subsection:ApproximationReducedStableEnvelopes}.
\end{cor}

\begin{proof} 
If we choose $w'=y_D$ in \eqref{eq:ApproximationStableEnvelopeModuloPowersOfH} then, by Lemma~\ref{lemma:CrossingsAndSimpleMoves}, the only set of crossings that contributes to \eqref{eq:ApproximationStableEnvelopeModuloPowersOfH} is $K(d_{w_{D'}},y_D,1)=\{ \kappa_0 \}$. Thus, 
by Proposition~\ref{prop:ApproximationStableEnvelopeModuloPowersOfH}, $\iota_{D'}^\ast(\widetilde{\mathrm{Stab}}_{\mathfrak C_-}(D))$ is congruent modulo $h^2$ to
\[
\frac{ (-1)^{\ell(y_D)+\ell(w_{D})} \cdot h \cdot
\Big(\prod_{\alpha \in L'_{w_{D'}}}\Psi_{\mathcal D}(\alpha+h) \Big)
\cdot \Big( \prod_{\kappa\in K_W(d_{w_{D'}})\setminus\{\kappa_0\}} \Psi_{\mathcal D}(\mathrm{wt}(\kappa))  \Big) }
{ \prod_{\beta \in R_{\textbf r}}\Psi_{\mathcal D}(w_{D'}.\beta)}.
\]
By Proposition~\ref{prop:SignSMandPermutations}, $(-1)^{\ell(w_D)+\ell(y_D)}=\operatorname{sgn}(D,D')$ which proves the corollary.
\end{proof}

\begin{proof}[Proof of Proposition~\ref{prop:ClosestNeighbors}]
By definition, we have
\begin{equation}\label{eq:ProofOfClosetsNeighbors1}
\frac{\iota_{D'}^\ast(\mathrm{Stab}_{\mathfrak C_-}(D))}{ e_{\mathbb T}(T_{D'}\mathcal C(\mathcal D)_{\mathfrak C_-}^-)}
=
\frac{\iota_{D'}^\ast(\mathrm{Stab}_{\mathfrak C_-}(D))}{\iota_{D'}^\ast(\mathrm{Stab}_{\mathfrak C_-}(D'))}
=
\frac{\iota_{D'}^\ast(\widetilde{\mathrm{Stab}}_{\mathfrak C_-}(D))}{\iota_{D'}^\ast(\widetilde{\mathrm{Stab}}_{\mathfrak C_-}(D'))}
.
\end{equation}
Proposition~\ref{prop:ApproximationStableEnvelopeModuloPowersOfH} gives that $\iota_{D'}^\ast(\widetilde{\mathrm{Stab}}_{\mathfrak C_-}(D'))$ is congruent modulo $h$ to
\begin{equation}\label{eq:ProofOfClosetsNeighbors2}
\frac{
\Big(\prod_{\alpha \in L'_{w_{D'}}}\Psi_{\mathcal D}(\alpha) \Big)
\cdot \Big( \prod_{\substack{\kappa\in K_W(d_{w_{D'}}) \\ } } \Psi_{\mathcal D}(\mathrm{wt}(\kappa))  \Big) }
{\prod_{\beta \in R_{\textbf r}}\Psi_{\mathcal D}(w_{D'}.\beta}.
\end{equation}
Combining~\eqref{eq:ProofOfClosetsNeighbors2} and Corollary~\ref{cor:ApproximationEMOfSimpleMove} then yields
\[
\eqref{eq:ProofOfClosetsNeighbors1}\equiv 
\frac{\operatorname{sgn}(D,D')\cdot h}{\Psi_{\mathcal D}(\mathrm{wt}(\kappa_0))}
\equiv
\frac{\operatorname{sgn}(D,D')\cdot h}{t_{j_1}-t_{j_2}}
\quad \mathrm{mod}\;h^2,
\]
which proves Proposition~\ref{prop:ClosestNeighbors}.
\end{proof}

\subsection{Chevalley--Monk formula for arbitrary chamber}

Employing Theorem~\ref{thm:ComparisonOfChambers}, generalizes the Chevalley--Monk formula for the antidominant chamber from Theorem~\ref{thm:CMAntidominantChamber} to any choice of chamber:

\begin{theorem}\label{thm:CMArbitraryChamber} Let $\mathfrak C=z^{-1}.\mathfrak C_-$ for $z\in S_N$, $D$ be a tie diagram of $\mathcal D$ and $i\in\{1,\ldots,M\}$. Then, the following identity holds in $H_{\mathbb T}^\ast(\mathcal C(\mathcal D))_{\mathrm{loc}}$:
\[
c_1(\xi_i)\cup \mathrm{Stab}_{\mathfrak C}(D)=\iota_{D}^\ast(c_1(\xi_i))\cdot \mathrm{Stab}_{\mathfrak C}(D) + \sum_{D'\in \mathrm{SM}_{D,z,i}} \operatorname{sgn}_z(D,D')h\cdot\mathrm{Stab}_{\mathfrak C}(D'),
\]
where
$
\mathrm{SM}_{D,z,i} = \{D'\in\mathrm{Tie}(\mathcal D)\mid z.D'\in \mathrm{SM}_{z.D,i} \}
$
and
$
 \operatorname{sgn}_z(D,D')= \operatorname{sgn}(z.D,z.D').
$
\end{theorem}

Let $\mathrm{SM}_{D,z}=\bigcup_{i=1}^M \mathrm{SM}_{D,z,i}$.
If $D'\in\mathrm{SM}_{D,z}$ then we say that $D'$ \textit{is obtained from $D$ via a $z$-twisted simple move}.

\begin{proof}[Proof of Theorem~\ref{thm:CMArbitraryChamber}] 
At first, note that \eqref{eq:RestrictionTautologicalBundles} implies 
\begin{equation}\label{eq:InnocentChernClassesTautologicalBundles}
\iota_T^\ast(\xi_i(\mathcal D))= w^{-1}.(\iota_{w.T}^\ast(\xi_i(w.\mathcal D))),\quad
\textup{for all $w\in S_N$ and $T\in \mathrm{Tie}(\mathcal D)$.}
\end{equation}
Employing~\eqref{eq:InnocentChernClassesTautologicalBundles} and Theorem~\ref{thm:ComparisonOfChambers} for a given $T\in\mathrm{Tie}(\mathcal D)$ yields
\begin{equation}\label{eq:ProofOfCMArbitraryChamberInnocent}
\iota_T^\ast(c_1(\xi(\mathcal D))\cup \widetilde{\mathrm{Stab}}_{\mathfrak C}(D))=z^{-1}.(\iota_{z.T}^\ast( c_1(\xi_i(z.\mathcal D))\cup \widetilde{\mathrm{Stab}}_{\mathfrak C_-}(z.D))).
\end{equation}
Then, Theorem~\ref{thm:CMAntidominantChamber} gives that \eqref{eq:ProofOfCMArbitraryChamberInnocent} is equal to
\begin{equation}\label{eq:ProofOfCMArbitraryChamberInnocen2}
%\eqref{eq:ProofOfCMArbitraryChamberInnocent} &= 
z^{-1}.\Big(
\iota_{z.T}^\ast\Big(
\iota_{z.D}^\ast(c_1(\xi_i(z.\mathcal D)))\cdot( \widetilde{\mathrm{Stab}}_{\mathfrak C_-}(z.D))
+ \sum_{D'\in \mathrm{SM}_{z.D,i}} \operatorname{sgn}(z.D,D')h\cdot\widetilde{\mathrm{Stab}}_{\mathfrak C_-}(D')
\Big)
\Big).
\end{equation}
Applying again \eqref{eq:InnocentChernClassesTautologicalBundles} and Theorem~\ref{thm:ComparisonOfChambers} then gives
\[
\eqref{eq:ProofOfCMArbitraryChamberInnocen2}
=
\iota_{T}^\ast\Big( \iota_D^\ast( c_1(\xi_i(\mathcal D)))\cdot \mathrm{Stab}_{\mathfrak C}(D) + \sum_{D'\in \mathrm{SM}_{D,z,i}} \operatorname{sgn}_z(D,D')h\cdot\mathrm{Stab}_{\mathfrak C}(D')\Big).
\]
Therefore, we conclude Theorem~\ref{thm:CMArbitraryChamber} by the localization theorem.
\end{proof}

\section{Chevalley--Monk formulas in the general case}

In the previous section, we proved Chevalley--Monk formulas for bow varieties with separated brane diagram (Theorem~\ref{thm:CMArbitraryChamber}). Via Hanany--Witten transition, we finally deduce Chevalley--Monk formulas for bow varieties corresponding to arbitrary choices of brane diagram and chamber.

\subsection{Simple moves for general brane diagrams}

Fix a brane diagram $\mathcal D$. First, we generalize the notion of (twisted) simple moves:

\begin{definition}
For $D\in\mathrm{Tie}(\mathcal D)$, we define the \textit{set of simple moves} $\mathrm{SM}_{D}$ as the set of all $D'\in\mathrm{Tie}(\mathcal D)$ such that there exist $1\le i_1< i_2\le M$ and $1\le j_1<j_2\le N$ satisfying
\begin{enumerate}[label=(\roman*)]
\item\label{item:GeneralSimpleMove1} $M(D)_{i_1,j_1}=M(D)_{i_2,j_2}=1$ and $M(D)_{i_1,j_2}=M(D)_{i_1,j_2}=0$,
\item\label{item:GeneralSimpleMove2} $M(D')_{i_1,j_1}=M(D')_{i_2,j_2}=0$ and $M(D')_{i_1,j_2}=M(D')_{i_1,j_2}=1$,
\item\label{item:GeneralSimpleMove3}
$M(D)_{l,k}=M(D')_{l,k}$ for all $(l,k))\notin \{(i_1,j_1),(i_2,j_1),(i_1,j_2),(i_2,j_2)\}$.
\end{enumerate}
If $D'\in \mathrm{SM}_{D}$ we say that \textit{$D'$ is obtained from $D$ via a simple move}. 
\end{definition}

Given additionally $i\in \{1,\ldots, M\}$, we define the \textit{set of simple move relative to $i$} $\mathrm{SM}_{D,i}$ as the set of all tie diagrams $D'$ of $\mathcal D$ such that there exists $1\le i_1\le M-i+1\le i_2\le M$ as well as $1\le j_1<j_2\le N$ satisfying \ref{item:GeneralSimpleMove1}-\ref{item:GeneralSimpleMove3}.

The graphical illustration of simple moves depends on the position of the separating line relative to the respective $2\times 2$ submatrix where the simple move is performed. The six possible cases are recorded in Figure~\ref{figure:IllustrationSMGeneralCase}.

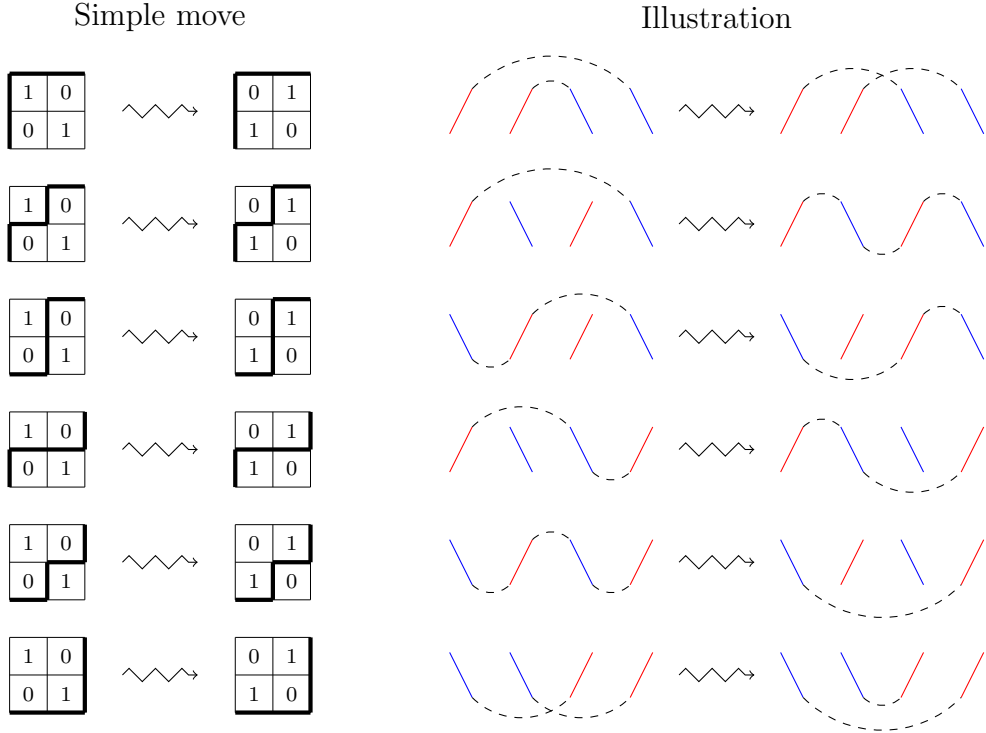
\begin{figure}
\begin{tikzpicture}
\coordinate (O) at (0,0);
\node at (O) {Simple move};

%First Simple Move
\coordinate (D11) at (-2,-0.75);
\coordinate[right=0.5 of D11] (D12);
\coordinate[right=1 of D11] (D13);
\coordinate[below=0.5 of D11] (D14);
\coordinate[below right=0.5 and 1 of D11] (D15);
\coordinate[below=1 of D11] (D16);
\coordinate[below right=1 and 0.5 of D11] (D17);
\coordinate[below right=1 and 1 of D11] (D18);

\coordinate[below right=0.25 and 0.25 of D11] (DE11);
\node at (DE11) {\tiny $1$};
\coordinate[below right=0.25 and 0.75 of D11] (DE12);
\node at (DE12) {\tiny $0$};
\coordinate[below right=0.75 and 0.25 of D11] (DE13);
\node at (DE13) {\tiny $0$};
\coordinate[below right=0.75 and 0.75 of D11] (DE14);
\node at (DE14) {\tiny $1$};

\draw[ultra thin] (D11) -- (D13);
\draw[ultra thin] (D14) -- (D15);
\draw[ultra thin] (D16) -- (D18);
\draw[ultra thin] (D11) -- (D16);
\draw[ultra thin] (D12) -- (D17);
\draw[ultra thin] (D13) -- (D18);

\draw[ultra thick] (D16) -- (D11);
\draw[ultra thick] (D11) -- (D13);

\coordinate[right = 0.5 of D15] (ADS1);
\coordinate[right = 1 of ADS1] (ADT1);
\draw[-to] decorate[decoration=zigzag] {(ADS1) -- (ADT1)};

\coordinate[above right = 0.5 and 0.5 of ADT1] (DP11);
\coordinate[right=0.5 of DP11] (DP12);
\coordinate[right=1 of DP11] (DP13);
\coordinate[below=0.5 of DP11] (DP14);
\coordinate[below right=0.5 and 1 of DP11] (DP15);
\coordinate[below=1 of DP11] (DP16);
\coordinate[below right=1 and 0.5 of DP11] (DP17);
\coordinate[below right=1 and 1 of DP11] (DP18);

\coordinate[below right=0.25 and 0.25 of DP11] (DPE11);
\node at (DPE11) {\tiny $0$};
\coordinate[below right=0.25 and 0.75 of DP11] (DPE12);
\node at (DPE12) {\tiny $1$};
\coordinate[below right=0.75 and 0.25 of DP11] (DPE13);
\node at (DPE13) {\tiny $1$};
\coordinate[below right=0.75 and 0.75 of DP11] (DPE14);
\node at (DPE14) {\tiny $0$};

\draw[ultra thin] (DP11) -- (DP13);
\draw[ultra thin] (DP14) -- (DP15);
\draw[ultra thin] (DP16) -- (DP18);
\draw[ultra thin] (DP11) -- (DP16);
\draw[ultra thin] (DP12) -- (DP17);
\draw[ultra thin] (DP13) -- (DP18);

\draw[ultra thick] (DP16) -- (DP11);
\draw[ultra thick] (DP11) -- (DP13);

%Second Simple Move
\coordinate[below=1.5 of D11] (D21);
\coordinate[right=0.5 of D21] (D22);
\coordinate[right=1 of D21] (D23);
\coordinate[below=0.5 of D21] (D24);
\coordinate[below right=0.5 and 1 of D21] (D25);
\coordinate[below=1 of D21] (D26);
\coordinate[below right=1 and 0.5 of D21] (D27);
\coordinate[below right=1 and 1 of D21] (D28);

\coordinate[below right=0.25 and 0.25 of D21] (DE21);
\node at (DE21) {\tiny $1$};
\coordinate[below right=0.25 and 0.75 of D21] (DE22);
\node at (DE22) {\tiny $0$};
\coordinate[below right=0.75 and 0.25 of D21] (DE23);
\node at (DE23) {\tiny $0$};
\coordinate[below right=0.75 and 0.75 of D21] (DE24);
\node at (DE24) {\tiny $1$};

\draw[ultra thin] (D21) -- (D23);
\draw[ultra thin] (D24) -- (D25);
\draw[ultra thin] (D26) -- (D28);
\draw[ultra thin] (D21) -- (D26);
\draw[ultra thin] (D22) -- (D27);
\draw[ultra thin] (D23) -- (D28);

\coordinate[right = 0.5 of D24] (Sep2);
\draw[ultra thick] (D26) -- (D24);
\draw[ultra thick] (D24) -- (Sep2);
\draw[ultra thick] (Sep2) -- (D22);
\draw[ultra thick] (D22) -- (D23);

\coordinate[right = 0.5 of D25] (ADS2);
\coordinate[right = 1 of ADS2] (ADT2);
\draw[-to] decorate [decoration=zigzag] {(ADS2) -- (ADT2)};

\coordinate[above right = 0.5 and 0.5 of ADT2] (DP21);
\coordinate[right=0.5 of DP21] (DP22);
\coordinate[right=1 of DP21] (DP23);
\coordinate[below=0.5 of DP21] (DP24);
\coordinate[below right=0.5 and 1 of DP21] (DP25);
\coordinate[below=1 of DP21] (DP26);
\coordinate[below right=1 and 0.5 of DP21] (DP27);
\coordinate[below right=1 and 1 of DP21] (DP28);

\coordinate[below right=0.25 and 0.25 of DP21] (DPE21);
\node at (DPE21) {\tiny $0$};
\coordinate[below right=0.25 and 0.75 of DP21] (DPE22);
\node at (DPE22) {\tiny $1$};
\coordinate[below right=0.75 and 0.25 of DP21] (DPE23);
\node at (DPE23) {\tiny $1$};
\coordinate[below right=0.75 and 0.75 of DP21] (DPE24);
\node at (DPE24) {\tiny $0$};

\draw[ultra thin] (DP21) -- (DP23);
\draw[ultra thin] (DP24) -- (DP25);
\draw[ultra thin] (DP26) -- (DP28);
\draw[ultra thin] (DP21) -- (DP26);
\draw[ultra thin] (DP22) -- (DP27);
\draw[ultra thin] (DP23) -- (DP28);

\coordinate[right = 0.5 of DP24] (SepP2);
\draw[ultra thick] (DP26) -- (DP24);
\draw[ultra thick] (DP24) -- (SepP2);
\draw[ultra thick] (SepP2) -- (DP22);
\draw[ultra thick] (DP22) -- (DP23);

%Third Simple Move
\coordinate[below=1.5 of D21] (D31);
\coordinate[right=0.5 of D31] (D32);
\coordinate[right=1 of D31] (D33);
\coordinate[below=0.5 of D31] (D34);
\coordinate[below right=0.5 and 1 of D31] (D35);
\coordinate[below=1 of D31] (D36);
\coordinate[below right=1 and 0.5 of D31] (D37);
\coordinate[below right=1 and 1 of D31] (D38);

\coordinate[below right=0.25 and 0.25 of D31] (DE31);
\node at (DE31) {\tiny $1$};
\coordinate[below right=0.25 and 0.75 of D31] (DE32);
\node at (DE32) {\tiny $0$};
\coordinate[below right=0.75 and 0.25 of D31] (DE33);
\node at (DE33) {\tiny $0$};
\coordinate[below right=0.75 and 0.75 of D31] (DE34);
\node at (DE34) {\tiny $1$};

\draw[ultra thin] (D31) -- (D33);
\draw[ultra thin] (D34) -- (D35);
\draw[ultra thin] (D36) -- (D38);
\draw[ultra thin] (D31) -- (D36);
\draw[ultra thin] (D32) -- (D37);
\draw[ultra thin] (D33) -- (D38);

\draw[ultra thick] (D36) -- (D37);
\draw[ultra thick] (D37) -- (D32);
\draw[ultra thick] (D32) -- (D33);

\coordinate[right = 0.5 of D35] (ADS3);
\coordinate[right = 1 of ADS3] (ADT3);
\draw[-to] decorate[decoration=zigzag] {(ADS3) -- (ADT3)};

\coordinate[above right = 0.5 and 0.5 of ADT3] (DP31);
\coordinate[right=0.5 of DP31] (DP32);
\coordinate[right=1 of DP31] (DP33);
\coordinate[below=0.5 of DP31] (DP34);
\coordinate[below right=0.5 and 1 of DP31] (DP35);
\coordinate[below=1 of DP31] (DP36);
\coordinate[below right=1 and 0.5 of DP31] (DP37);
\coordinate[below right=1 and 1 of DP31] (DP38);

\coordinate[below right=0.25 and 0.25 of DP31] (DPE31);
\node at (DPE31) {\tiny $0$};
\coordinate[below right=0.25 and 0.75 of DP31] (DPE32);
\node at (DPE32) {\tiny $1$};
\coordinate[below right=0.75 and 0.25 of DP31] (DPE33);
\node at (DPE33) {\tiny $1$};
\coordinate[below right=0.75 and 0.75 of DP31] (DPE34);
\node at (DPE34) {\tiny $0$};

\draw[ultra thin] (DP31) -- (DP33);
\draw[ultra thin] (DP34) -- (DP35);
\draw[ultra thin] (DP36) -- (DP38);
\draw[ultra thin] (DP31) -- (DP36);
\draw[ultra thin] (DP32) -- (DP37);
\draw[ultra thin] (DP33) -- (DP38);

\draw[ultra thick] (DP36) -- (DP37);
\draw[ultra thick] (DP37) -- (DP32);
\draw[ultra thick] (DP32) -- (DP33);

%Fourth Simple Move
\coordinate[below=1.5 of D31] (D41);
\coordinate[right=0.5 of D41] (D42);
\coordinate[right=1 of D41] (D43);
\coordinate[below=0.5 of D41] (D44);
\coordinate[below right=0.5 and 1 of D41] (D45);
\coordinate[below=1 of D41] (D46);
\coordinate[below right=1 and 0.5 of D41] (D47);
\coordinate[below right=1 and 1 of D41] (D48);

\coordinate[below right=0.25 and 0.25 of D41] (DE41);
\node at (DE41) {\tiny $1$};
\coordinate[below right=0.25 and 0.75 of D41] (DE42);
\node at (DE42) {\tiny $0$};
\coordinate[below right=0.75 and 0.25 of D41] (DE43);
\node at (DE43) {\tiny $0$};
\coordinate[below right=0.75 and 0.75 of D41] (DE44);
\node at (DE44) {\tiny $1$};

\draw[ultra thin] (D41) -- (D43);
\draw[ultra thin] (D44) -- (D45);
\draw[ultra thin] (D46) -- (D48);
\draw[ultra thin] (D41) -- (D46);
\draw[ultra thin] (D42) -- (D47);
\draw[ultra thin] (D43) -- (D48);

\draw[ultra thick] (D46) -- (D44);
\draw[ultra thick] (D44) -- (D45);
\draw[ultra thick] (D45) -- (D43);

\coordinate[right = 0.5 of D45] (ADS4);
\coordinate[right = 1 of ADS4] (ADT4);
\draw[-to] decorate[decoration=zigzag] {(ADS4) -- (ADT4)};

\coordinate[above right = 0.5 and 0.5 of ADT4] (DP41);
\coordinate[right=0.5 of DP41] (DP42);
\coordinate[right=1 of DP41] (DP43);
\coordinate[below=0.5 of DP41] (DP44);
\coordinate[below right=0.5 and 1 of DP41] (DP45);
\coordinate[below=1 of DP41] (DP46);
\coordinate[below right=1 and 0.5 of DP41] (DP47);
\coordinate[below right=1 and 1 of DP41] (DP48);

\coordinate[below right=0.25 and 0.25 of DP41] (DPE41);
\node at (DPE41) {\tiny $0$};
\coordinate[below right=0.25 and 0.75 of DP41] (DPE42);
\node at (DPE42) {\tiny $1$};
\coordinate[below right=0.75 and 0.25 of DP41] (DPE43);
\node at (DPE43) {\tiny $1$};
\coordinate[below right=0.75 and 0.75 of DP41] (DPE44);
\node at (DPE44) {\tiny $0$};

\draw[ultra thin] (DP41) -- (DP43);
\draw[ultra thin] (DP44) -- (DP45);
\draw[ultra thin] (DP46) -- (DP48);
\draw[ultra thin] (DP41) -- (DP46);
\draw[ultra thin] (DP42) -- (DP47);
\draw[ultra thin] (DP43) -- (DP48);

\draw[ultra thick] (DP46) -- (DP44);
\draw[ultra thick] (DP44) -- (DP45);
\draw[ultra thick] (DP45) -- (DP43);

%Fifth Simple Move
\coordinate[below=1.5 of D41] (D51);
\coordinate[right=0.5 of D51] (D52);
\coordinate[right=1 of D51] (D53);
\coordinate[below=0.5 of D51] (D54);
\coordinate[below right=0.5 and 1 of D51] (D55);
\coordinate[below=1 of D51] (D56);
\coordinate[below right=1 and 0.5 of D51] (D57);
\coordinate[below right=1 and 1 of D51] (D58);

\coordinate[below right=0.25 and 0.25 of D51] (DE51);
\node at (DE51) {\tiny $1$};
\coordinate[below right=0.25 and 0.75 of D51] (DE52);
\node at (DE52) {\tiny $0$};
\coordinate[below right=0.75 and 0.25 of D51] (DE53);
\node at (DE53) {\tiny $0$};
\coordinate[below right=0.75 and 0.75 of D51] (DE54);
\node at (DE54) {\tiny $1$};

\draw[ultra thin] (D51) -- (D53);
\draw[ultra thin] (D54) -- (D55);
\draw[ultra thin] (D56) -- (D58);
\draw[ultra thin] (D51) -- (D56);
\draw[ultra thin] (D52) -- (D57);
\draw[ultra thin] (D53) -- (D58);

\coordinate[right = 0.5 of D54] (Sep5);

\draw[ultra thick] (D56) -- (D57);
\draw[ultra thick] (D57) -- (Sep5);
\draw[ultra thick] (Sep5) -- (D55);
\draw[ultra thick] (D55) -- (D53);

\coordinate[right = 0.5 of D55] (ADS5);
\coordinate[right = 1 of ADS5] (ADT5);
\draw[-to] decorate[decoration=zigzag] {(ADS5) -- (ADT5)};

\coordinate[above right = 0.5 and 0.5 of ADT5] (DP51);
\coordinate[right=0.5 of DP51] (DP52);
\coordinate[right=1 of DP51] (DP53);
\coordinate[below=0.5 of DP51] (DP54);
\coordinate[below right=0.5 and 1 of DP51] (DP55);
\coordinate[below=1 of DP51] (DP56);
\coordinate[below right=1 and 0.5 of DP51] (DP57);
\coordinate[below right=1 and 1 of DP51] (DP58);

\coordinate[below right=0.25 and 0.25 of DP51] (DPE51);
\node at (DPE51) {\tiny $0$};
\coordinate[below right=0.25 and 0.75 of DP51] (DPE52);
\node at (DPE52) {\tiny $1$};
\coordinate[below right=0.75 and 0.25 of DP51] (DPE53);
\node at (DPE53) {\tiny $1$};
\coordinate[below right=0.75 and 0.75 of DP51] (DPE54);
\node at (DPE54) {\tiny $0$};

\draw[ultra thin] (DP51) -- (DP53);
\draw[ultra thin] (DP54) -- (DP55);
\draw[ultra thin] (DP56) -- (DP58);
\draw[ultra thin] (DP51) -- (DP56);
\draw[ultra thin] (DP52) -- (DP57);
\draw[ultra thin] (DP53) -- (DP58);

\coordinate[right = 0.5 of DP54] (SepP5);

\draw[ultra thick] (DP56) -- (DP57);
\draw[ultra thick] (DP57) -- (SepP5);
\draw[ultra thick] (SepP5) -- (DP55);
\draw[ultra thick] (DP55) -- (DP53);

%Sixth Simple Move
\coordinate[below=1.5 of D51] (D61);
\coordinate[right=0.5 of D61] (D62);
\coordinate[right=1 of D61] (D63);
\coordinate[below=0.5 of D61] (D64);
\coordinate[below right=0.5 and 1 of D61] (D65);
\coordinate[below=1 of D61] (D66);
\coordinate[below right=1 and 0.5 of D61] (D67);
\coordinate[below right=1 and 1 of D61] (D68);

\coordinate[below right=0.25 and 0.25 of D61] (DE61);
\node at (DE61) {\tiny $1$};
\coordinate[below right=0.25 and 0.75 of D61] (DE62);
\node at (DE62) {\tiny $0$};
\coordinate[below right=0.75 and 0.25 of D61] (DE63);
\node at (DE63) {\tiny $0$};
\coordinate[below right=0.75 and 0.75 of D61] (DE64);
\node at (DE64) {\tiny $1$};

\draw[ultra thin] (D61) -- (D63);
\draw[ultra thin] (D64) -- (D65);
\draw[ultra thin] (D66) -- (D68);
\draw[ultra thin] (D61) -- (D66);
\draw[ultra thin] (D62) -- (D67);
\draw[ultra thin] (D63) -- (D68);

\draw[ultra thick] (D66) -- (D68);
\draw[ultra thick] (D68) -- (D63);

\coordinate[right = 0.5 of D65] (ADS6);
\coordinate[right = 1 of ADS6] (ADT6);
\draw[-to] decorate[decoration=zigzag] {(ADS6) -- (ADT6)};

\coordinate[above right = 0.5 and 0.5 of ADT6] (DP61);
\coordinate[right=0.5 of DP61] (DP62);
\coordinate[right=1 of DP61] (DP63);
\coordinate[below=0.5 of DP61] (DP64);
\coordinate[below right=0.5 and 1 of DP61] (DP65);
\coordinate[below=1 of DP61] (DP66);
\coordinate[below right=1 and 0.5 of DP61] (DP67);
\coordinate[below right=1 and 1 of DP61] (DP68);

\coordinate[below right=0.25 and 0.25 of DP61] (DPE61);
\node at (DPE61) {\tiny $0$};
\coordinate[below right=0.25 and 0.75 of DP61] (DPE62);
\node at (DPE62) {\tiny $1$};
\coordinate[below right=0.75 and 0.25 of DP61] (DPE63);
\node at (DPE63) {\tiny $1$};
\coordinate[below right=0.75 and 0.75 of DP61] (DPE64);
\node at (DPE64) {\tiny $0$};

\draw[ultra thin] (DP61) -- (DP63);
\draw[ultra thin] (DP64) -- (DP65);
\draw[ultra thin] (DP66) -- (DP68);
\draw[ultra thin] (DP61) -- (DP66);
\draw[ultra thin] (DP62) -- (DP67);
\draw[ultra thin] (DP63) -- (DP68);

\draw[ultra thick] (DP66) -- (DP68);
\draw[ultra thick] (DP68) -- (DP63);

%First tie diagram
\coordinate[right = 2 of DP15] (T11);
\coordinate[right = 0.8 of T11] (T12);
\coordinate[right = 0.8 of T12] (T13);
\coordinate[right = 0.8 of T13] (T14);

\coordinate[below left = 0.3 and 0.15 of T11] (B11B);
\coordinate[above right = 0.3 and 0.15 of T11] (B11T);
\coordinate[below left = 0.3 and 0.15 of T12] (B12B);
\coordinate[above right = 0.3 and 0.15 of T12] (B12T);

\coordinate[below right = 0.3 and 0.15 of T13] (B13B);
\coordinate[above left = 0.3 and 0.15 of T13] (B13T);
\coordinate[below right = 0.3 and 0.15 of T14] (B14B);
\coordinate[above left = 0.3 and 0.15 of T14] (B14T);

\draw[red] (B11B) -- (B11T);
\draw[red] (B12B) -- (B12T);
\draw[blue] (B13B) -- (B13T);
\draw[blue] (B14B) -- (B14T);

\draw[dashed] (B12T) to[out=45, in = 135] (B13T);
\draw[dashed] (B11T) to[out=45, in = 135] (B14T);

\coordinate[right=0.5 of T14] (AT1S);
\coordinate[right=1 of AT1S] (AT1T);
\draw[-to] decorate[decoration=zigzag] {(AT1S) -- (AT1T)};

\coordinate[right = 0.5 of AT1T] (TP11);
\coordinate[right = 0.8 of TP11] (TP12);
\coordinate[right = 0.8 of TP12] (TP13);
\coordinate[right = 0.8 of TP13] (TP14);

\coordinate[below left = 0.3 and 0.15 of TP11] (BP11B);
\coordinate[above right = 0.3 and 0.15 of TP11] (BP11T);
\coordinate[below left = 0.3 and 0.15 of TP12] (BP12B);
\coordinate[above right = 0.3 and 0.15 of TP12] (BP12T);

\coordinate[below right = 0.3 and 0.15 of TP13] (BP13B);
\coordinate[above left = 0.3 and 0.15 of TP13] (BP13T);
\coordinate[below right = 0.3 and 0.15 of TP14] (BP14B);
\coordinate[above left = 0.3 and 0.15 of TP14] (BP14T);

\draw[red] (BP11B) -- (BP11T);
\draw[red] (BP12B) -- (BP12T);
\draw[blue] (BP13B) -- (BP13T);
\draw[blue] (BP14B) -- (BP14T);

\draw[dashed] (BP12T) to[out=45, in = 135] (BP14T);
\draw[dashed] (BP11T) to[out=45, in = 135] (BP13T);

\coordinate[above right = 1.25 and 0.5 of AT1S] (I);
\node at (I) {Illustration};

%Second tie diagram
\coordinate[right = 2 of DP25] (T21);
\coordinate[right = 0.8 of T21] (T22);
\coordinate[right = 0.8 of T22] (T23);
\coordinate[right = 0.8 of T23] (T24);

\coordinate[below left = 0.3 and 0.15 of T21] (B21B);
\coordinate[above right = 0.3 and 0.15 of T21] (B21T);
\coordinate[below right = 0.3 and 0.15 of T22] (B22B);
\coordinate[above left = 0.3 and 0.15 of T22] (B22T);

\coordinate[below left = 0.3 and 0.15 of T23] (B23B);
\coordinate[above right = 0.3 and 0.15 of T23] (B23T);
\coordinate[below right = 0.3 and 0.15 of T24] (B24B);
\coordinate[above left = 0.3 and 0.15 of T24] (B24T);

\draw[red] (B21B) -- (B21T);
\draw[blue] (B22B) -- (B22T);
\draw[red] (B23B) -- (B23T);
\draw[blue] (B24B) -- (B24T);

\draw[dashed] (B21T) to[out=45, in = 135] (B24T);

\coordinate[right=0.5 of T24] (AT2S);
\coordinate[right=1 of AT2S] (AT2T);
\draw[-to] decorate[decoration=zigzag] {(AT2S) -- (AT2T)};

\coordinate[right = 0.5 of AT2T] (TP21);
\coordinate[right = 0.8 of TP21] (TP22);
\coordinate[right = 0.8 of TP22] (TP23);
\coordinate[right = 0.8 of TP23] (TP24);

\coordinate[below left = 0.3 and 0.15 of TP21] (BP21B);
\coordinate[above right = 0.3 and 0.15 of TP21] (BP21T);
\coordinate[below right = 0.3 and 0.15 of TP22] (BP22B);
\coordinate[above left = 0.3 and 0.15 of TP22] (BP22T);

\coordinate[below left = 0.3 and 0.15 of TP23] (BP23B);
\coordinate[above right = 0.3 and 0.15 of TP23] (BP23T);
\coordinate[below right = 0.3 and 0.15 of TP24] (BP24B);
\coordinate[above left = 0.3 and 0.15 of TP24] (BP24T);

\draw[red] (BP21B) -- (BP21T);
\draw[blue] (BP22B) -- (BP22T);
\draw[red] (BP23B) -- (BP23T);
\draw[blue] (BP24B) -- (BP24T);

\draw[dashed] (BP21T) to[out=45, in = 135] (BP22T);
\draw[dashed] (BP22B) to[out=-45, in = -135] (BP23B);
\draw[dashed] (BP23T) to[out=45, in = 135] (BP24T);

%Third tie diagram
\coordinate[right = 2 of DP35] (T31);
\coordinate[right = 0.8 of T31] (T32);
\coordinate[right = 0.8 of T32] (T33);
\coordinate[right = 0.8 of T33] (T34);

\coordinate[below right = 0.3 and 0.15 of T31] (B31B);
\coordinate[above left = 0.3 and 0.15 of T31] (B31T);
\coordinate[below left = 0.3 and 0.15 of T32] (B32B);
\coordinate[above right = 0.3 and 0.15 of T32] (B32T);

\coordinate[below left = 0.3 and 0.15 of T33] (B33B);
\coordinate[above right = 0.3 and 0.15 of T33] (B33T);
\coordinate[below right = 0.3 and 0.15 of T34] (B34B);
\coordinate[above left = 0.3 and 0.15 of T34] (B34T);

\draw[blue] (B31B) -- (B31T);
\draw[red] (B32B) -- (B32T);
\draw[red] (B33B) -- (B33T);
\draw[blue] (B34B) -- (B34T);

\draw[dashed] (B32T) to[out=45, in = 135] (B34T);
\draw[dashed] (B31B) to[out=-45, in = -135] (B32B);

\coordinate[right=0.5 of T34] (AT3S);
\coordinate[right=1 of AT3S] (AT3T);
\draw[-to] decorate[decoration=zigzag] {(AT3S) -- (AT3T)};

\coordinate[right = 0.5 of AT3T] (TP31);
\coordinate[right = 0.8 of TP31] (TP32);
\coordinate[right = 0.8 of TP32] (TP33);
\coordinate[right = 0.8 of TP33] (TP34);

\coordinate[below right = 0.3 and 0.15 of TP31] (BP31B);
\coordinate[above left = 0.3 and 0.15 of TP31] (BP31T);
\coordinate[below left = 0.3 and 0.15 of TP32] (BP32B);
\coordinate[above right = 0.3 and 0.15 of TP32] (BP32T);

\coordinate[below left = 0.3 and 0.15 of TP33] (BP33B);
\coordinate[above right = 0.3 and 0.15 of TP33] (BP33T);
\coordinate[below right = 0.3 and 0.15 of TP34] (BP34B);
\coordinate[above left = 0.3 and 0.15 of TP34] (BP34T);

\draw[blue] (BP31B) -- (BP31T);
\draw[red] (BP32B) -- (BP32T);
\draw[red] (BP33B) -- (BP33T);
\draw[blue] (BP34B) -- (BP34T);

\draw[dashed] (BP31B) to[out=-45, in = -135] (BP33B);
\draw[dashed] (BP33T) to[out=45, in = 135] (BP34T);

%fourth tie diagram
\coordinate[right = 2 of DP45] (T41);
\coordinate[right = 0.8 of T41] (T42);
\coordinate[right = 0.8 of T42] (T43);
\coordinate[right = 0.8 of T43] (T44);

\coordinate[below left = 0.3 and 0.15 of T41] (B41B);
\coordinate[above right = 0.3 and 0.15 of T41] (B41T);
\coordinate[below right = 0.3 and 0.15 of T42] (B42B);
\coordinate[above left = 0.3 and 0.15 of T42] (B42T);

\coordinate[below right = 0.3 and 0.15 of T43] (B43B);
\coordinate[above left = 0.3 and 0.15 of T43] (B43T);
\coordinate[below left = 0.3 and 0.15 of T44] (B44B);
\coordinate[above right = 0.3 and 0.15 of T44] (B44T);

\draw[red] (B41B) -- (B41T);
\draw[blue] (B42B) -- (B42T);
\draw[blue] (B43B) -- (B43T);
\draw[red] (B44B) -- (B44T);

\draw[dashed] (B41T) to[out=45, in = 135] (B43T);
\draw[dashed] (B43B) to[out=-45, in = -135] (B44B);

\coordinate[right=0.5 of T44] (AT4S);
\coordinate[right=1 of AT4S] (AT4T);
\draw[-to] decorate[decoration=zigzag] {(AT4S) -- (AT4T)};

\coordinate[right = 0.5 of AT4T] (TP41);
\coordinate[right = 0.8 of TP41] (TP42);
\coordinate[right = 0.8 of TP42] (TP43);
\coordinate[right = 0.8 of TP43] (TP44);

\coordinate[below left = 0.3 and 0.15 of TP41] (BP41B);
\coordinate[above right = 0.3 and 0.15 of TP41] (BP41T);
\coordinate[below right = 0.3 and 0.15 of TP42] (BP42B);
\coordinate[above left = 0.3 and 0.15 of TP42] (BP42T);

\coordinate[below right = 0.3 and 0.15 of TP43] (BP43B);
\coordinate[above left = 0.3 and 0.15 of TP43] (BP43T);
\coordinate[below left = 0.3 and 0.15 of TP44] (BP44B);
\coordinate[above right = 0.3 and 0.15 of TP44] (BP44T);

\draw[red] (BP41B) -- (BP41T);
\draw[blue] (BP42B) -- (BP42T);
\draw[blue] (BP43B) -- (BP43T);
\draw[red] (BP44B) -- (BP44T);

\draw[dashed] (BP41T) to[out=45, in = 135] (BP42T);
\draw[dashed] (BP42B) to[out=-45, in = -135] (BP44B);

%Fifth tie diagram
\coordinate[right = 2 of DP55] (T51);
\coordinate[right = 0.8 of T51] (T52);
\coordinate[right = 0.8 of T52] (T53);
\coordinate[right = 0.8 of T53] (T54);

\coordinate[below right = 0.3 and 0.15 of T51] (B51B);
\coordinate[above left = 0.3 and 0.15 of T51] (B51T);
\coordinate[below left = 0.3 and 0.15 of T52] (B52B);
\coordinate[above right = 0.3 and 0.15 of T52] (B52T);

\coordinate[below right = 0.3 and 0.15 of T53] (B53B);
\coordinate[above left = 0.3 and 0.15 of T53] (B53T);
\coordinate[below left = 0.3 and 0.15 of T54] (B54B);
\coordinate[above right = 0.3 and 0.15 of T54] (B54T);

\draw[blue] (B51B) -- (B51T);
\draw[red] (B52B) -- (B52T);
\draw[blue] (B53B) -- (B53T);
\draw[red] (B54B) -- (B54T);

\draw[dashed] (B51B) to[out=-45, in = -135] (B52B);
\draw[dashed] (B52T) to[out=45, in = 135] (B53T);
\draw[dashed] (B53B) to[out=-45, in = -135] (B54B);

\coordinate[right=0.5 of T54] (AT5S);
\coordinate[right=1 of AT5S] (AT5T);
\draw[-to] decorate[decoration=zigzag] {(AT5S) -- (AT5T)};

\coordinate[right = 0.5 of AT5T] (TP51);
\coordinate[right = 0.8 of TP51] (TP52);
\coordinate[right = 0.8 of TP52] (TP53);
\coordinate[right = 0.8 of TP53] (TP54);

\coordinate[below right = 0.3 and 0.15 of TP51] (BP51B);
\coordinate[above left = 0.3 and 0.15 of TP51] (BP51T);
\coordinate[below left = 0.3 and 0.15 of TP52] (BP52B);
\coordinate[above right = 0.3 and 0.15 of TP52] (BP52T);

\coordinate[below right = 0.3 and 0.15 of TP53] (BP53B);
\coordinate[above left = 0.3 and 0.15 of TP53] (BP53T);
\coordinate[below left = 0.3 and 0.15 of TP54] (BP54B);
\coordinate[above right = 0.3 and 0.15 of TP54] (BP54T);

\draw[blue] (BP51B) -- (BP51T);
\draw[red] (BP52B) -- (BP52T);
\draw[blue] (BP53B) -- (BP53T);
\draw[red] (BP54B) -- (BP54T);

\draw[dashed] (BP51B) to[out=-45, in = -135] (BP54B);

%Sixth tie diagram
\coordinate[right = 2 of DP65] (T61);
\coordinate[right = 0.8 of T61] (T62);
\coordinate[right = 0.8 of T62] (T63);
\coordinate[right = 0.8 of T63] (T64);

\coordinate[below right = 0.3 and 0.15 of T61] (B61B);
\coordinate[above left = 0.3 and 0.15 of T61] (B61T);
\coordinate[below right = 0.3 and 0.15 of T62] (B62B);
\coordinate[above left = 0.3 and 0.15 of T62] (B62T);

\coordinate[below left = 0.3 and 0.15 of T63] (B63B);
\coordinate[above right = 0.3 and 0.15 of T63] (B63T);
\coordinate[below left = 0.3 and 0.15 of T64] (B64B);
\coordinate[above right = 0.3 and 0.15 of T64] (B64T);

\draw[blue] (B61B) -- (B61T);
\draw[blue] (B62B) -- (B62T);
\draw[red] (B63B) -- (B63T);
\draw[red] (B64B) -- (B64T);

\draw[dashed] (B61B) to[out=-45, in = -135] (B63B);
\draw[dashed] (B62B) to[out=-45, in = -135] (B64B);

\coordinate[right=0.5 of T64] (AT6S);
\coordinate[right=1 of AT6S] (AT6T);
\draw[-to] decorate[decoration=zigzag] {(AT6S) -- (AT6T)};

\coordinate[right = 0.5 of AT6T] (TP61);
\coordinate[right = 0.8 of TP61] (TP62);
\coordinate[right = 0.8 of TP62] (TP63);
\coordinate[right = 0.8 of TP63] (TP64);

\coordinate[below right = 0.3 and 0.15 of TP61] (BP61B);
\coordinate[above left = 0.3 and 0.15 of TP61] (BP61T);
\coordinate[below right = 0.3 and 0.15 of TP62] (BP62B);
\coordinate[above left = 0.3 and 0.15 of TP62] (BP62T);

\coordinate[below left = 0.3 and 0.15 of TP63] (BP63B);
\coordinate[above right = 0.3 and 0.15 of TP63] (BP63T);
\coordinate[below left = 0.3 and 0.15 of TP64] (BP64B);
\coordinate[above right = 0.3 and 0.15 of TP64] (BP64T);

\draw[blue] (BP61B) -- (BP61T);
\draw[blue] (BP62B) -- (BP62T);
\draw[red] (BP63B) -- (BP63T);
\draw[red] (BP64B) -- (BP64T);

\draw[dashed] (BP61B) to[out=-45, in = -135] (BP64B);
\draw[dashed] (BP62B) to[out=-45, in = -135] (BP63B);

\end{tikzpicture}
\caption{Illustration of simple moves for general brane diagrams.}\label{figure:IllustrationSMGeneralCase}
\end{figure}

\begin{example}\label{example:GeneralSimpleMoves}
Let $\mathcal D=0\textcolor{blue}{\backslash}
1
\textcolor{blue}{\backslash}
2
\textcolor{red}{\slash}
3
\textcolor{blue}{\backslash}
3
\textcolor{red}{\slash}
2
\textcolor{blue}{\backslash}
2
\textcolor{red}{\slash}
0$.
As tie diagram $D$ we choose
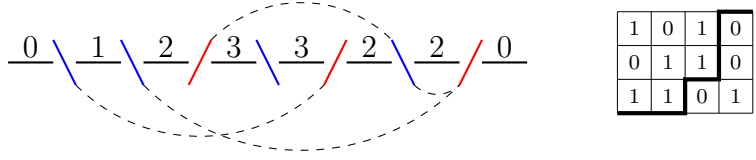
\begin{center}
\begin{tikzpicture}[scale=.3]
\draw[thick] (0,0)--(2,0);
\draw[thick] (3,0)--(5,0);
\draw[thick] (6,0)--(8,0);
\draw[thick] (9,0)--(11,0);
\draw[thick] (12,0)--(14,0);
\draw[thick] (15,0)--(17,0);
\draw[thick] (18,0)--(20,0);
\draw[thick] (21,0)--(23,0);

\draw [thick,red] (8,-1) --(9,1); 
\draw [thick,red] (14,-1) --(15,1); 
\draw [thick,red] (20,-1) --(21,1);

\draw [thick,blue] (3,-1) --(2,1); 
\draw [thick,blue] (6,-1) --(5,1); 
\draw [thick,blue] (12,-1) --(11,1); 
\draw [thick,blue] (18,-1) --(17,1);

\draw[dashed] (3,-1) to[out=-45, in=-135] (14,-1);
\draw[dashed] (6,-1) to[out=-45, in=-135] (20,-1);
\draw[dashed] (18,-1) to[out=-45, in=-135] (20,-1);
\draw[dashed] (9,1) to[out=45, in=135] (17,1);
\node at (1,0.7) {$0$};
\node at (4,0.7) {$1$};
\node at (7,0.7) {$2$};
\node at (10,0.7) {$3$};
\node at (13,0.7) {$3$};
\node at (16,0.7) {$2$};
\node at (19,0.7) {$2$};
\node at (22,0.7) {$0$};

\def\f{1.5};
\def\s{27};
\def\t{-2.25};

\draw[ultra thin] (0*\f+\s,0*\f+\t) -- (4*\f+\s,0*\f+\t);
\draw[ultra thin]  (0*\f+\s,1*\f+\t) -- (4*\f+\s,1*\f+\t);
\draw[ultra thin]  (0*\f+\s,2*\f+\t) -- (4*\f+\s,2*\f+\t);
\draw[ultra thin]  (0*\f+\s,3*\f+\t) -- (4*\f+\s,3*\f+\t);
\draw[ultra thin]  (0*\f+\s,0*\f+\t) -- (0*\f+\s,3*\f+\t);
\draw[ultra thin]  (1*\f+\s,0*\f+\t) -- (1*\f+\s,3*\f+\t);
\draw[ultra thin]  (2*\f+\s,0*\f+\t) -- (2*\f+\s,3*\f+\t);
\draw[ultra thin]  (3*\f+\s,0*\f+\t) -- (3*\f+\s,3*\f+\t);
\draw[ultra thin]  (4*\f+\s,0*\f+\t) -- (4*\f+\s,3*\f+\t);
\draw[ultra thick] (0*\f+\s,0*\f+\t) -- (2*\f+\s,0*\f+\t) -- (2*\f+\s,1*\f+\t) -- (3*\f+\s,1*\f+\t) -- (3*\f+\s,3*\f+\t) -- (4*\f+\s,3*\f+\t);
  
\node at (0.5*\f+\s,0.5*\f+\t) {\tiny $1$};
\node at (1.5*\f+\s,0.5*\f+\t) {\tiny $1$};
\node at (2.5*\f+\s,0.5*\f+\t) {\tiny $0$};
\node at (3.5*\f+\s,0.5*\f+\t) {\tiny $1$};

\node at (0.5*\f+\s,1.5*\f+\t) {\tiny $0$};
\node at (1.5*\f+\s,1.5*\f+\t) {\tiny $1$};
\node at (2.5*\f+\s,1.5*\f+\t) {\tiny $1$};
\node at (3.5*\f+\s,1.5*\f+\t) {\tiny $0$};

\node at (0.5*\f+\s,2.5*\f+\t) {\tiny $1$};
\node at (1.5*\f+\s,2.5*\f+\t) {\tiny $0$};
\node at (2.5*\f+\s,2.5*\f+\t) {\tiny $1$};
\node at (3.5*\f+\s,2.5*\f+\t) {\tiny $0$};
%\node at (-2.5,1.5) {$M(D)=$};
\end{tikzpicture}
\end{center}
The tie diagrams that are obtained from $D$ via a simple moves have the following binary contingency tables:
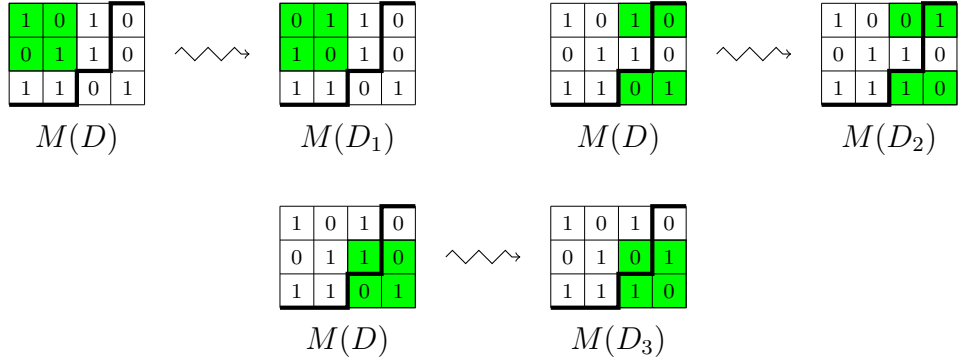
\begin{center}
\begin{tikzpicture}[scale=.3]
%First Simple Move

\def\f{1.5};
\def\s{0};
\def\t{0};

\draw [fill=ForestGreen] (0*\f+\s,1*\f+\t) rectangle (2*\f+\s,3*\f+\t);

\draw[ultra thin] (0*\f+\s,0*\f+\t) -- (4*\f+\s,0*\f+\t);
\draw[ultra thin]  (0*\f+\s,1*\f+\t) -- (4*\f+\s,1*\f+\t);
\draw[ultra thin]  (0*\f+\s,2*\f+\t) -- (4*\f+\s,2*\f+\t);
\draw[ultra thin]  (0*\f+\s,3*\f+\t) -- (4*\f+\s,3*\f+\t);
\draw[ultra thin]  (0*\f+\s,0*\f+\t) -- (0*\f+\s,3*\f+\t);
\draw[ultra thin]  (1*\f+\s,0*\f+\t) -- (1*\f+\s,3*\f+\t);
\draw[ultra thin]  (2*\f+\s,0*\f+\t) -- (2*\f+\s,3*\f+\t);
\draw[ultra thin]  (3*\f+\s,0*\f+\t) -- (3*\f+\s,3*\f+\t);
\draw[ultra thin]  (4*\f+\s,0*\f+\t) -- (4*\f+\s,3*\f+\t);
\draw[ultra thick] (0*\f+\s,0*\f+\t) -- (2*\f+\s,0*\f+\t) -- (2*\f+\s,1*\f+\t) -- (3*\f+\s,1*\f+\t) -- (3*\f+\s,3*\f+\t) -- (4*\f+\s,3*\f+\t);
  
\node at (0.5*\f+\s,0.5*\f+\t) {\tiny $1$};
\node at (1.5*\f+\s,0.5*\f+\t) {\tiny $1$};
\node at (2.5*\f+\s,0.5*\f+\t) {\tiny $0$};
\node at (3.5*\f+\s,0.5*\f+\t) {\tiny $1$};

\node at (0.5*\f+\s,1.5*\f+\t) {\tiny $0$};
\node at (1.5*\f+\s,1.5*\f+\t) {\tiny $1$};
\node at (2.5*\f+\s,1.5*\f+\t) {\tiny $1$};
\node at (3.5*\f+\s,1.5*\f+\t) {\tiny $0$};

\node at (0.5*\f+\s,2.5*\f+\t) {\tiny $1$};
\node at (1.5*\f+\s,2.5*\f+\t) {\tiny $0$};
\node at (2.5*\f+\s,2.5*\f+\t) {\tiny $1$};
\node at (3.5*\f+\s,2.5*\f+\t) {\tiny $0$};

\node at (2*\f+\s,-\f+\t ) {$M(D)$};

\draw[-to] decorate[decoration=zigzag] {(4.9*\f, 1.5*\f) -- (7.1*\f, 1.5*\f)};

\def\s{8*\f};
\def\t{0};

\draw [fill=ForestGreen] (0*\f+\s,1*\f+\t) rectangle (2*\f+\s,3*\f+\t);

\draw[ultra thin] (0*\f+\s,0*\f+\t) -- (4*\f+\s,0*\f+\t);
\draw[ultra thin]  (0*\f+\s,1*\f+\t) -- (4*\f+\s,1*\f+\t);
\draw[ultra thin]  (0*\f+\s,2*\f+\t) -- (4*\f+\s,2*\f+\t);
\draw[ultra thin]  (0*\f+\s,3*\f+\t) -- (4*\f+\s,3*\f+\t);
\draw[ultra thin]  (0*\f+\s,0*\f+\t) -- (0*\f+\s,3*\f+\t);
\draw[ultra thin]  (1*\f+\s,0*\f+\t) -- (1*\f+\s,3*\f+\t);
\draw[ultra thin]  (2*\f+\s,0*\f+\t) -- (2*\f+\s,3*\f+\t);
\draw[ultra thin]  (3*\f+\s,0*\f+\t) -- (3*\f+\s,3*\f+\t);
\draw[ultra thin]  (4*\f+\s,0*\f+\t) -- (4*\f+\s,3*\f+\t);
\draw[ultra thick] (0*\f+\s,0*\f+\t) -- (2*\f+\s,0*\f+\t) -- (2*\f+\s,1*\f+\t) -- (3*\f+\s,1*\f+\t) -- (3*\f+\s,3*\f+\t) -- (4*\f+\s,3*\f+\t);
  
\node at (0.5*\f+\s,0.5*\f+\t) {\tiny $1$};
\node at (1.5*\f+\s,0.5*\f+\t) {\tiny $1$};
\node at (2.5*\f+\s,0.5*\f+\t) {\tiny $0$};
\node at (3.5*\f+\s,0.5*\f+\t) {\tiny $1$};

\node at (0.5*\f+\s,1.5*\f+\t) {\tiny $1$};
\node at (1.5*\f+\s,1.5*\f+\t) {\tiny $0$};
\node at (2.5*\f+\s,1.5*\f+\t) {\tiny $1$};
\node at (3.5*\f+\s,1.5*\f+\t) {\tiny $0$};

\node at (0.5*\f+\s,2.5*\f+\t) {\tiny $0$};
\node at (1.5*\f+\s,2.5*\f+\t) {\tiny $1$};
\node at (2.5*\f+\s,2.5*\f+\t) {\tiny $1$};
\node at (3.5*\f+\s,2.5*\f+\t) {\tiny $0$};

\node at (2*\f+\s,-\f+\t ) {$M(D_1)$};
%Second Simple Move
\def\s{16*\f};
\def\t{0};

\draw [fill=ForestGreen] (2*\f+\s,2*\f+\t) rectangle (4*\f+\s,3*\f+\t);
\draw [fill=ForestGreen] (2*\f+\s,0*\f+\t) rectangle (4*\f+\s,1*\f+\t);

\draw[ultra thin] (0*\f+\s,0*\f+\t) -- (4*\f+\s,0*\f+\t);
\draw[ultra thin]  (0*\f+\s,1*\f+\t) -- (4*\f+\s,1*\f+\t);
\draw[ultra thin]  (0*\f+\s,2*\f+\t) -- (4*\f+\s,2*\f+\t);
\draw[ultra thin]  (0*\f+\s,3*\f+\t) -- (4*\f+\s,3*\f+\t);
\draw[ultra thin]  (0*\f+\s,0*\f+\t) -- (0*\f+\s,3*\f+\t);
\draw[ultra thin]  (1*\f+\s,0*\f+\t) -- (1*\f+\s,3*\f+\t);
\draw[ultra thin]  (2*\f+\s,0*\f+\t) -- (2*\f+\s,3*\f+\t);
\draw[ultra thin]  (3*\f+\s,0*\f+\t) -- (3*\f+\s,3*\f+\t);
\draw[ultra thin]  (4*\f+\s,0*\f+\t) -- (4*\f+\s,3*\f+\t);
\draw[ultra thick] (0*\f+\s,0*\f+\t) -- (2*\f+\s,0*\f+\t) -- (2*\f+\s,1*\f+\t) -- (3*\f+\s,1*\f+\t) -- (3*\f+\s,3*\f+\t) -- (4*\f+\s,3*\f+\t);
  
\node at (0.5*\f+\s,0.5*\f+\t) {\tiny $1$};
\node at (1.5*\f+\s,0.5*\f+\t) {\tiny $1$};
\node at (2.5*\f+\s,0.5*\f+\t) {\tiny $0$};
\node at (3.5*\f+\s,0.5*\f+\t) {\tiny $1$};

\node at (0.5*\f+\s,1.5*\f+\t) {\tiny $0$};
\node at (1.5*\f+\s,1.5*\f+\t) {\tiny $1$};
\node at (2.5*\f+\s,1.5*\f+\t) {\tiny $1$};
\node at (3.5*\f+\s,1.5*\f+\t) {\tiny $0$};

\node at (0.5*\f+\s,2.5*\f+\t) {\tiny $1$};
\node at (1.5*\f+\s,2.5*\f+\t) {\tiny $0$};
\node at (2.5*\f+\s,2.5*\f+\t) {\tiny $1$};
\node at (3.5*\f+\s,2.5*\f+\t) {\tiny $0$};

\node at (2*\f+\s,-\f+\t ) {$M(D)$};

\draw[-to] decorate[decoration=zigzag] {(4.9*\f+\s, 1.5*\f+\t) -- (7.1*\f+\s, 1.5*\f+\t)};

\def\s{8*\f+16*\f};
\def\t{0};

\draw [fill=ForestGreen] (2*\f+\s,2*\f+\t) rectangle (4*\f+\s,3*\f+\t);
\draw [fill=ForestGreen] (2*\f+\s,0*\f+\t) rectangle (4*\f+\s,1*\f+\t);

\draw[ultra thin] (0*\f+\s,0*\f+\t) -- (4*\f+\s,0*\f+\t);
\draw[ultra thin]  (0*\f+\s,1*\f+\t) -- (4*\f+\s,1*\f+\t);
\draw[ultra thin]  (0*\f+\s,2*\f+\t) -- (4*\f+\s,2*\f+\t);
\draw[ultra thin]  (0*\f+\s,3*\f+\t) -- (4*\f+\s,3*\f+\t);
\draw[ultra thin]  (0*\f+\s,0*\f+\t) -- (0*\f+\s,3*\f+\t);
\draw[ultra thin]  (1*\f+\s,0*\f+\t) -- (1*\f+\s,3*\f+\t);
\draw[ultra thin]  (2*\f+\s,0*\f+\t) -- (2*\f+\s,3*\f+\t);
\draw[ultra thin]  (3*\f+\s,0*\f+\t) -- (3*\f+\s,3*\f+\t);
\draw[ultra thin]  (4*\f+\s,0*\f+\t) -- (4*\f+\s,3*\f+\t);
\draw[ultra thick] (0*\f+\s,0*\f+\t) -- (2*\f+\s,0*\f+\t) -- (2*\f+\s,1*\f+\t) -- (3*\f+\s,1*\f+\t) -- (3*\f+\s,3*\f+\t) -- (4*\f+\s,3*\f+\t);
  
\node at (0.5*\f+\s,0.5*\f+\t) {\tiny $1$};
\node at (1.5*\f+\s,0.5*\f+\t) {\tiny $1$};
\node at (2.5*\f+\s,0.5*\f+\t) {\tiny $1$};
\node at (3.5*\f+\s,0.5*\f+\t) {\tiny $0$};

\node at (0.5*\f+\s,1.5*\f+\t) {\tiny $0$};
\node at (1.5*\f+\s,1.5*\f+\t) {\tiny $1$};
\node at (2.5*\f+\s,1.5*\f+\t) {\tiny $1$};
\node at (3.5*\f+\s,1.5*\f+\t) {\tiny $0$};

\node at (0.5*\f+\s,2.5*\f+\t) {\tiny $1$};
\node at (1.5*\f+\s,2.5*\f+\t) {\tiny $0$};
\node at (2.5*\f+\s,2.5*\f+\t) {\tiny $0$};
\node at (3.5*\f+\s,2.5*\f+\t) {\tiny $1$};

\node at (2*\f+\s,-\f+\t ) {$M(D_2)$};
%Third simple move
\def\s{8*\f};
\def\t{-6*\f};

\draw [fill=ForestGreen] (2*\f+\s,0*\f+\t) rectangle (4*\f+\s,2*\f+\t);

\draw[ultra thin] (0*\f+\s,0*\f+\t) -- (4*\f+\s,0*\f+\t);
\draw[ultra thin]  (0*\f+\s,1*\f+\t) -- (4*\f+\s,1*\f+\t);
\draw[ultra thin]  (0*\f+\s,2*\f+\t) -- (4*\f+\s,2*\f+\t);
\draw[ultra thin]  (0*\f+\s,3*\f+\t) -- (4*\f+\s,3*\f+\t);
\draw[ultra thin]  (0*\f+\s,0*\f+\t) -- (0*\f+\s,3*\f+\t);
\draw[ultra thin]  (1*\f+\s,0*\f+\t) -- (1*\f+\s,3*\f+\t);
\draw[ultra thin]  (2*\f+\s,0*\f+\t) -- (2*\f+\s,3*\f+\t);
\draw[ultra thin]  (3*\f+\s,0*\f+\t) -- (3*\f+\s,3*\f+\t);
\draw[ultra thin]  (4*\f+\s,0*\f+\t) -- (4*\f+\s,3*\f+\t);
\draw[ultra thick] (0*\f+\s,0*\f+\t) -- (2*\f+\s,0*\f+\t) -- (2*\f+\s,1*\f+\t) -- (3*\f+\s,1*\f+\t) -- (3*\f+\s,3*\f+\t) -- (4*\f+\s,3*\f+\t);
  
\node at (0.5*\f+\s,0.5*\f+\t) {\tiny $1$};
\node at (1.5*\f+\s,0.5*\f+\t) {\tiny $1$};
\node at (2.5*\f+\s,0.5*\f+\t) {\tiny $0$};
\node at (3.5*\f+\s,0.5*\f+\t) {\tiny $1$};

\node at (0.5*\f+\s,1.5*\f+\t) {\tiny $0$};
\node at (1.5*\f+\s,1.5*\f+\t) {\tiny $1$};
\node at (2.5*\f+\s,1.5*\f+\t) {\tiny $1$};
\node at (3.5*\f+\s,1.5*\f+\t) {\tiny $0$};

\node at (0.5*\f+\s,2.5*\f+\t) {\tiny $1$};
\node at (1.5*\f+\s,2.5*\f+\t) {\tiny $0$};
\node at (2.5*\f+\s,2.5*\f+\t) {\tiny $1$};
\node at (3.5*\f+\s,2.5*\f+\t) {\tiny $0$};

\node at (2*\f+\s,-\f+\t ) {$M(D)$};

\draw[-to] decorate[decoration=zigzag] {(4.9*\f+\s, 1.5*\f+\t) -- (7.1*\f+\s, 1.5*\f+\t)};

\def\s{8*\f+8*\f};
\def\t{-6*\f};

\draw [fill=ForestGreen] (2*\f+\s,0*\f+\t) rectangle (4*\f+\s,2*\f+\t);

\draw[ultra thin] (0*\f+\s,0*\f+\t) -- (4*\f+\s,0*\f+\t);
\draw[ultra thin]  (0*\f+\s,1*\f+\t) -- (4*\f+\s,1*\f+\t);
\draw[ultra thin]  (0*\f+\s,2*\f+\t) -- (4*\f+\s,2*\f+\t);
\draw[ultra thin]  (0*\f+\s,3*\f+\t) -- (4*\f+\s,3*\f+\t);
\draw[ultra thin]  (0*\f+\s,0*\f+\t) -- (0*\f+\s,3*\f+\t);
\draw[ultra thin]  (1*\f+\s,0*\f+\t) -- (1*\f+\s,3*\f+\t);
\draw[ultra thin]  (2*\f+\s,0*\f+\t) -- (2*\f+\s,3*\f+\t);
\draw[ultra thin]  (3*\f+\s,0*\f+\t) -- (3*\f+\s,3*\f+\t);
\draw[ultra thin]  (4*\f+\s,0*\f+\t) -- (4*\f+\s,3*\f+\t);
\draw[ultra thick] (0*\f+\s,0*\f+\t) -- (2*\f+\s,0*\f+\t) -- (2*\f+\s,1*\f+\t) -- (3*\f+\s,1*\f+\t) -- (3*\f+\s,3*\f+\t) -- (4*\f+\s,3*\f+\t);
  
\node at (0.5*\f+\s,0.5*\f+\t) {\tiny $1$};
\node at (1.5*\f+\s,0.5*\f+\t) {\tiny $1$};
\node at (2.5*\f+\s,0.5*\f+\t) {\tiny $1$};
\node at (3.5*\f+\s,0.5*\f+\t) {\tiny $0$};

\node at (0.5*\f+\s,1.5*\f+\t) {\tiny $0$};
\node at (1.5*\f+\s,1.5*\f+\t) {\tiny $1$};
\node at (2.5*\f+\s,1.5*\f+\t) {\tiny $0$};
\node at (3.5*\f+\s,1.5*\f+\t) {\tiny $1$};

\node at (0.5*\f+\s,2.5*\f+\t) {\tiny $1$};
\node at (1.5*\f+\s,2.5*\f+\t) {\tiny $0$};
\node at (2.5*\f+\s,2.5*\f+\t) {\tiny $1$};
\node at (3.5*\f+\s,2.5*\f+\t) {\tiny $0$};

\node at (2*\f+\s,-\f+\t ) {$M(D_3)$};
\end{tikzpicture}
\end{center} 
We highlighted the submatrices that are involved in the respective simple moves.
The corresponding illustrations of the simple moves in terms of tie diagrams are illustrated in Figure~\ref{fig:FigureExampleSimpleMoveGeneral}.
\end{example}

If $D'\in \mathrm{SM}_D$ then the \textit{sign of the simple move between $D$ and $D'$} is defined as
\[
\operatorname{sgn}(D,D')\coloneqq (-1)^{n_1+n_2},\quad\textup{where }n_1=\sum_{l=j_1+1}^{j_2-1} M(D)_{i_1,l},\;n_2=\sum_{l=j_1+1}^{j_2-1} M(D)_{i_2,l}.
\]
The notion of twisted simple moves also generalizes as expected: for $z\in S_N$ we set 
\[
\mathrm{SM}_{D,z}\coloneqq \{D'\in\mathrm{Tie}(\mathcal D)\mid z.D'\in \mathrm{SM}_{z.D} \}
\]
and
\[
\mathrm{SM}_{D,z,i} \coloneqq \{D'\in\mathrm{Tie}(\mathcal D)\mid z.D'\in \mathrm{SM}_{z.D,i} \},\quad
\textup{for $i=1,\ldots,M$.}
\]
If $D'\in \mathrm{SM}_{D,z}$, we say that $D'$ \textit{is obtained from $D$ via a $z$-twisted simple move}. The corresponding \textit{sign of the $z$-twisted simple move between $D$ and $D'$} is defined as
$
 \operatorname{sgn}_z(D,D')\coloneqq \operatorname{sgn}(z.D,z.D').
$

By Lemma~\ref{lemma:SImpleMovesBCT}, the definitions of ($z$-twisted) simple moves and the corresponding signs agree with the previous definitions for separated brane diagrams.

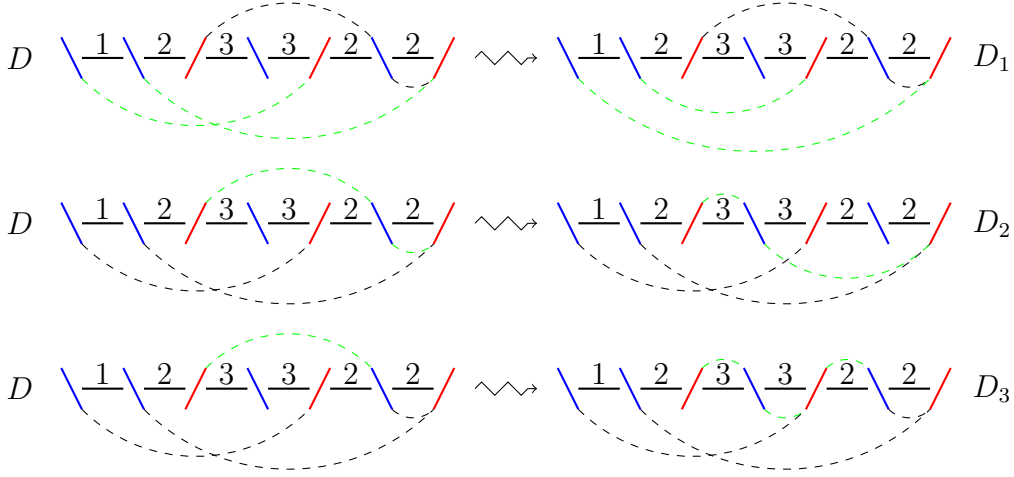
\begin{figure}
\begin{tikzpicture}[scale=.275]
\node at (0,0) {$D$};

\draw[thick] (3,0)--(5,0);
\draw[thick] (6,0)--(8,0);
\draw[thick] (9,0)--(11,0);
\draw[thick] (12,0)--(14,0);
\draw[thick] (15,0)--(17,0);
\draw[thick] (18,0)--(20,0);

\draw [thick,red] (8,-1) --(9,1); 
\draw [thick,red] (14,-1) --(15,1); 
\draw [thick,red] (20,-1) --(21,1);

\draw [thick,blue] (3,-1) --(2,1); 
\draw [thick,blue] (6,-1) --(5,1); 
\draw [thick,blue] (12,-1) --(11,1); 
\draw [thick,blue] (18,-1) --(17,1);
\draw[dashed] (18,-1) to[out=-45, in=-135] (20,-1);
\draw[dashed] (9,1) to[out=45, in=135] (17,1);

\draw[dashed, very thick, ForestGreen] (3,-1) to[out=-45, in=-135] (14,-1);
\draw[dashed, very thick, ForestGreen] (6,-1) to[out=-45, in=-135] (20,-1);
%\node at (1,0.7) {$0$};
\node at (4,0.7) {$1$};
\node at (7,0.7) {$2$};
\node at (10,0.7) {$3$};
\node at (13,0.7) {$3$};
\node at (16,0.7) {$2$};
\node at (19,0.7) {$2$};
%\node at (22,0.7) {$0$};
%
\draw[-to] decorate[decoration=zigzag] {(22,0) -- (25,0)}; 
\def \s{24};

\draw[thick] (3+\s,0)--(5+\s,0);
\draw[thick] (6+\s,0)--(8+\s,0);
\draw[thick] (9+\s,0)--(11+\s,0);
\draw[thick] (12+\s,0)--(14+\s,0);
\draw[thick] (15+\s,0)--(17+\s,0);
\draw[thick] (18+\s,0)--(20+\s,0);

\draw [thick,red] (8+\s,-1) --(9+\s,1); 
\draw [thick,red] (14+\s,-1) --(15+\s,1); 
\draw [thick,red] (20+\s,-1) --(21+\s,1);

\draw [thick,blue] (3+\s,-1) --(2+\s,1); 
\draw [thick,blue] (6+\s,-1) --(5+\s,1); 
\draw [thick,blue] (12+\s,-1) --(11+\s,1); 
\draw [thick,blue] (18+\s,-1) --(17+\s,1);

\draw[dashed] (18+\s,-1) to[out=-45, in=-135] (20+\s,-1);
\draw[dashed] (9+\s,1) to[out=45, in=135] (17+\s,1);

\draw[dashed, very thick, ForestGreen] (3+\s,-1) to[out=-45, in=-135] (20+\s,-1);
\draw[dashed, very thick, ForestGreen] (6+\s,-1) to[out=-45, in=-135] (14+\s,-1);

\node at (4+\s,0.7) {$1$};
\node at (7+\s,0.7) {$2$};
\node at (10+\s,0.7) {$3$};
\node at (13+\s,0.7) {$3$};
\node at (16+\s,0.7) {$2$};
\node at (19+\s,0.7) {$2$};

\node at (23+\s,0) {$D_1$};

\def\t{-16};
\node at (0,0+\t) {$D$};

\draw[thick] (3,0+\t)--(5,0+\t);
\draw[thick] (6,0+\t)--(8,0+\t);
\draw[thick] (9,0+\t)--(11,0+\t);
\draw[thick] (12,0+\t)--(14,0+\t);
\draw[thick] (15,0+\t)--(17,0+\t);
\draw[thick] (18,0+\t)--(20,0+\t);
%\draw[thick] (21,0)--(23,0);

\draw [thick,red] (8,-1+\t) --(9,1+\t); 
\draw [thick,red] (14,-1+\t) --(15,1+\t); 
\draw [thick,red] (20,-1+\t) --(21,1+\t);

\draw [thick,blue] (3,-1+\t) --(2,1+\t); 
\draw [thick,blue] (6,-1+\t) --(5,1+\t); 
\draw [thick,blue] (12,-1+\t) --(11,1+\t); 
\draw [thick,blue] (18,-1+\t) --(17,1+\t);

\draw[dashed] (18,-1+\t) to[out=-45, in=-135] (20,-1+\t);

\draw[dashed] (3,-1+\t) to[out=-45, in=-135] (14,-1+\t);
\draw[dashed] (6,-1+\t) to[out=-45, in=-135] (20,-1+\t);
\draw[dashed, very thick, ForestGreen] (9,1+\t) to[out=45, in=135] (17,1+\t);

\node at (4,0.7+\t) {$1$};
\node at (7,0.7+\t) {$2$};
\node at (10,0.7+\t) {$3$};
\node at (13,0.7+\t) {$3$};
\node at (16,0.7+\t) {$2$};
\node at (19,0.7+\t) {$2$};

\draw[-to] decorate[decoration=zigzag] {(22,0+\t) -- (25,0+\t)};

\draw[thick] (3+\s,0+\t)--(5+\s,0+\t);
\draw[thick] (6+\s,0+\t)--(8+\s,0+\t);
\draw[thick] (9+\s,0+\t)--(11+\s,0+\t);
\draw[thick] (12+\s,0+\t)--(14+\s,0+\t);
\draw[thick] (15+\s,0+\t)--(17+\s,0+\t);
\draw[thick] (18+\s,0+\t)--(20+\s,0+\t);

\draw [thick,red] (8+\s,-1+\t) --(9+\s,1+\t); 
\draw [thick,red] (14+\s,-1+\t) --(15+\s,1+\t); 
\draw [thick,red] (20+\s,-1+\t) --(21+\s,1+\t);

\draw [thick,blue] (3+\s,-1+\t) --(2+\s,1+\t); 
\draw [thick,blue] (6+\s,-1+\t) --(5+\s,1+\t); 
\draw [thick,blue] (12+\s,-1+\t) --(11+\s,1+\t); 
\draw [thick,blue] (18+\s,-1+\t) --(17+\s,1+\t);

\draw[dashed] (18+\s,-1+\t) to[out=-45, in=-135] (20+\s,-1+\t);

\draw[dashed] (3+\s,-1+\t) to[out=-45, in=-135] (14+\s,-1+\t);
\draw[dashed] (6+\s,-1+\t) to[out=-45, in=-135] (20+\s,-1+\t);
\draw[dashed, very thick, ForestGreen] (9+\s,1+\t) to[out=45, in=135] (11+\s,1+\t);
\draw[dashed, very thick, ForestGreen] (12+\s,-1+\t) to[out=-45, in=-135] (14+\s,-1+\t);
\draw[dashed, very thick, ForestGreen] (15+\s,1+\t) to[out=45, in=135] (17+\s,1+\t);

\node at (4+\s,0.7+\t) {$1$};
\node at (7+\s,0.7+\t) {$2$};
\node at (10+\s,0.7+\t) {$3$};
\node at (13+\s,0.7+\t) {$3$};
\node at (16+\s,0.7+\t) {$2$};
\node at (19+\s,0.7+\t) {$2$};

\node at (23+\s,0+\t) {$D_3$};

\def\t{-8};
\node at (0,0+\t) {$D$};

\draw[thick] (3,0+\t)--(5,0+\t);
\draw[thick] (6,0+\t)--(8,0+\t);
\draw[thick] (9,0+\t)--(11,0+\t);
\draw[thick] (12,0+\t)--(14,0+\t);
\draw[thick] (15,0+\t)--(17,0+\t);
\draw[thick] (18,0+\t)--(20,0+\t);

\draw [thick,red] (8,-1+\t) --(9,1+\t); 
\draw [thick,red] (14,-1+\t) --(15,1+\t); 
\draw [thick,red] (20,-1+\t) --(21,1+\t);

\draw [thick,blue] (3,-1+\t) --(2,1+\t); 
\draw [thick,blue] (6,-1+\t) --(5,1+\t); 
\draw [thick,blue] (12,-1+\t) --(11,1+\t); 
\draw [thick,blue] (18,-1+\t) --(17,1+\t);

\draw[dashed] (3,-1+\t) to[out=-45, in=-135] (14,-1+\t);
\draw[dashed] (6,-1+\t) to[out=-45, in=-135] (20,-1+\t);
\draw[dashed, very thick, ForestGreen] (9,1+\t) to[out=45, in=135] (17,1+\t);
\draw[dashed, very thick, ForestGreen] (18,-1+\t) to[out=-45, in=-135] (20,-1+\t);
\node at (4,0.7+\t) {$1$};
\node at (7,0.7+\t) {$2$};
\node at (10,0.7+\t) {$3$};
\node at (13,0.7+\t) {$3$};
\node at (16,0.7+\t) {$2$};
\node at (19,0.7+\t) {$2$};

\draw[-to] decorate[decoration=zigzag] {(22,0+\t) -- (25,0+\t)};

\draw[thick] (3+\s,0+\t)--(5+\s,0+\t);
\draw[thick] (6+\s,0+\t)--(8+\s,0+\t);
\draw[thick] (9+\s,0+\t)--(11+\s,0+\t);
\draw[thick] (12+\s,0+\t)--(14+\s,0+\t);
\draw[thick] (15+\s,0+\t)--(17+\s,0+\t);
\draw[thick] (18+\s,0+\t)--(20+\s,0+\t);

\draw [thick,red] (8+\s,-1+\t) --(9+\s,1+\t); 
\draw [thick,red] (14+\s,-1+\t) --(15+\s,1+\t); 
\draw [thick,red] (20+\s,-1+\t) --(21+\s,1+\t);

\draw [thick,blue] (3+\s,-1+\t) --(2+\s,1+\t); 
\draw [thick,blue] (6+\s,-1+\t) --(5+\s,1+\t); 
\draw [thick,blue] (12+\s,-1+\t) --(11+\s,1+\t); 
\draw [thick,blue] (18+\s,-1+\t) --(17+\s,1+\t);

\draw[dashed] (3+\s,-1+\t) to[out=-45, in=-135] (14+\s,-1+\t);
\draw[dashed] (6+\s,-1+\t) to[out=-45, in=-135] (20+\s,-1+\t);
\draw[dashed, very thick, ForestGreen] (9+\s,1+\t) to[out=45, in=135] (11+\s,1+\t);
\draw[dashed, very thick, ForestGreen] (12+\s,-1+\t) to[out=-45, in=-135] (20+\s,-1+\t);

\node at (4+\s,0.7+\t) {$1$};
\node at (7+\s,0.7+\t) {$2$};
\node at (10+\s,0.7+\t) {$3$};
\node at (13+\s,0.7+\t) {$3$};
\node at (16+\s,0.7+\t) {$2$};
\node at (19+\s,0.7+\t) {$2$};

\node at (23+\s,0+\t) {$D_2$};
\end{tikzpicture}
\caption{Example of simple moves for general brane diagrams.}\label{fig:FigureExampleSimpleMoveGeneral}
\end{figure}

Simple moves are well-behaved with respect to Hanany--Witten transition: let $\tilde D$ be the brane diagram obtained via Hanany--Witten transition from $\mathcal D$ by switching $U_{j_0}\in\mathrm b(\mathcal D)$ and $V_{i_0}\in\mathrm r(\mathcal D)$. Let $\Phi:\mathcal C(\mathcal D)\xrightarrow\sim \mathcal C(\tilde{\mathcal D})$ be the corresponding Hanany--Witten isomorphism (see Proposition~\ref{prop:HWTransition}) and let $\phi\colon\mathrm{Tie}(\mathcal D)\xrightarrow\sim \mathrm{Tie}(\tilde{\mathcal D})$ be the induced bijection.
\begin{lemma}\label{lemma:SimpleMovesAndHW}
Let $D$, $D'\in\mathrm{Tie}(\mathcal D)$, $z\in S_N$ and $i\in\{1,\ldots,M\}$. Then, we have $D'\in \mathrm{SM}_{D,i,z}$ if and only if $\phi(D')\in \mathrm{SM}_{\phi(D),i,z}$.
\end{lemma}
\begin{proof}
The proof is immediate from the fact that $M(D)=M(\phi(D))$ for all $D\in \mathrm{Tie}(\mathcal D)$, see Proposition~\ref{prop:HWFixedPointMatching}.
\end{proof}

\subsection{Chevalley--Monk formula in the general case} We finally formulate and prove Chevalley--Monk formulas for bow varieties corresponding to not-necessarily separated brane diagrams. 

We first set up some notation: given a brane diagram $\mathcal D$ and $i\in\{1,\ldots,M-1\}$ then we set
\[
I(\mathcal D,i)\coloneqq \{X\in \mathrm{h}(\mathcal D)\mid V_{i+1} \triangleleft X \triangleleft V_i \}.
\]
In addition, we set
\[
I(\mathcal D,0)\coloneqq \{X\in \mathrm{h}(\mathcal D)\mid V_1\triangleleft X \}\quad\textup{and}\quad
I(\mathcal D,M)\coloneqq \{X\in \mathrm{h}(\mathcal D)\mid X \triangleleft V_M \}.
\]
For instance, let $\mathcal D=0\textcolor{blue}{\backslash}
1
\textcolor{blue}{\backslash}
2
\textcolor{red}{\slash}
3
\textcolor{blue}{\backslash}
3
\textcolor{red}{\slash}
2
\textcolor{blue}{\backslash}
2
\textcolor{red}{\slash}
0$ as in Example~\ref{example:GeneralSimpleMoves}. Then, one can easily check that
$I(\mathcal D,0)=\{X_8\},$ 
$I(\mathcal D,1)=\{X_7,X_6\}$,
$I(\mathcal D,2)=\{X_5,X_4\}$, 
$I(\mathcal D,3)=\{X_3,X_2,X_1\}$. 

Now, we finally state the general Chevalley--Monk formula:

\begin{theorem}\label{thm:CMGeneralCase}
Let $\mathfrak C=z^{-1}.\mathfrak C_-$ for $z\in S_N$ and $i\in\{0,\ldots,M+1\}$. Then, we have the following identity in $H_{\mathbb T}^\ast(\mathcal C(\mathcal D))_{\mathrm{loc}}$:
\[
c_1(\xi_j)\cup \mathrm{Stab}_{\mathfrak C}(D)=\iota_{D}^\ast(c_1(\xi_j))\cdot \mathrm{Stab}_{\mathfrak C}(D) + \sum_{D'\in \mathrm{SM}_{D,z,i}} \operatorname{sgn}_z(D,D')h\cdot\mathrm{Stab}_{\mathfrak C}(D'),
\]
for all $X_j\in I(\mathcal D,i)$ and $D\in\mathrm{Tie}(\mathcal D)$.
\end{theorem}

For the proof, we use the following notation: given a $\mathbb T$-equivariant vector bundle $E$ on $\mathcal C(\mathcal D)$, we denote by $\mathrm{C}(\mathcal D,E)=\mathrm{C}(\mathcal D,E)_{D,D'}$ the matrix with entries in $H_{\mathbb T}^\ast(\operatorname{pt})_{\mathrm{loc}}$ corresponding to the operator of multiplication with $c_1(E)$ on $H_{\mathbb T}^{\ast}(\mathcal C(\mathcal D))_{\mathrm{loc}}$ with respect to the stable basis $(\mathrm{Stab}_{\mathfrak C}(D))_{D\in \mathrm{Tie}(\mathcal D)}$. 

We begin with the following lemma:

\begin{lemma}\label{lemma:HWMatrixInvariance}
Let $U_{j_0}\in\mathrm b(\mathcal D)$, $V_{i_0}\in\mathrm r(\mathcal D)$, $\tilde{\mathcal D}$, $\Phi$ and $\phi$ be as in Lemma~\ref{lemma:SimpleMovesAndHW}. Let $X_l=U_{j_0}^+$ and $D$, $D'\in \mathrm{Tie}(\mathcal D)$. Then, we have
\[
\varphi_{j_0}(\mathrm{C}(\tilde{\mathcal D},\tilde \xi_j)_{\phi(D),\phi(D')})=\mathrm{C}(\mathcal D,\xi_j)_{D,D'}, \quad \textit{for $j \ne l$}
\] 
and
\[ 
\varphi_{j_0}(\mathrm C(\tilde{\mathcal D},\tilde\xi_{l})_{\phi(D),\phi(D')})
=
\mathrm{C}(\mathcal D,\xi_{l+1})_{D,D'}+\mathrm C({\mathcal D},\xi_{l-1})_{D,D'}-\mathrm C({\mathcal D},\xi_{l})_{D,D'}+(t_{j_0}+h)\delta_{D,D'}.
\]
Here, $\tilde \xi_i$ is the tautological bundle over $\mathcal C(\tilde{\mathcal D})$ corresponding to $X_i$, $\delta_{D,D'}$ is the Kronecker delta and $\varphi_{j_0}:\mathbb Q[t_1,\ldots,t_N,h]\xrightarrow\sim \mathbb Q[t_1,\ldots,t_N,h]$ is the $\mathbb Q[h]$-algebra automorphism given by $t_{j_0}\mapsto t_{j_0}+h$ and $t_j\mapsto t_j$ for $j\ne j_0$.
\end{lemma}

\begin{proof}
Let $\Phi^\ast\colon H^\ast_{\mathbb T}(\mathcal C(\tilde{\mathcal D}))\xrightarrow\sim H^\ast_{\mathbb T}(\mathcal C(\mathcal D))$ be the induced ring isomorphism from $\Phi$. By Proposition~\ref{prop:MatchingStableEnvelopesHW}, we have $\Phi^\ast(\mathrm{Stab}_{\mathfrak C}(\phi(T)))=\mathrm{Stab}_{\mathfrak C}(T)$ for all $T\in \mathrm{Tie}(\mathcal D)$. Thus,
\[
\varphi_{j_0}(\mathrm{C}(\tilde{\mathcal D},\tilde{\xi}_i)_{\phi(D),\phi(D')}) = \mathrm{C}(\mathcal D,\Phi^\ast (\tilde \xi_i))_{D,D'}.
\] Hence, the lemma follows from Corollary~\ref{cor:MatchingOfTautologicalsHW}.
\end{proof}

\begin{proof}[Proof of Theorem~\ref{thm:CMGeneralCase}]
We prove the theorem via induction on the separatedness degree of $\mathcal D$. The case $\mathrm{sdeg}(\mathcal D)=0$ is exactly the statement of Theorem~\ref{thm:CMArbitraryChamber}, so let us assume  $\mathrm{sdeg}(\mathcal D)>0$. As in the proof of Theorem~\ref{thm:CMAntidominantChamber}, the support condition for stable basis elements directly implies
\[
\mathrm{C}(\mathcal D,\xi_j)_{D,D}=\iota_D^\ast (c_1(\xi_j)),\quad\textup{for all $D\in\mathrm{Tie}(\mathcal D)$}.
\]
Hence, it is left to show that $\mathrm{C}(\mathcal D,E)$ has the correct off-diagonal terms.
As $\mathrm{sdeg}(\mathcal D)>0$, there exist $U_{j_0}\in\mathrm b(\mathcal D)$ and $V_{i_0}\in\mathrm r(\mathcal D)$ as in Lemma~\ref{lemma:HWMatrixInvariance}. In the following, we use the notation from Lemma~\ref{lemma:HWMatrixInvariance}. Let $D$, $D'\in\mathrm{Tie}(\mathcal D)$ with $D\ne D'$. Assume first that $X_j\ne U_{j_0}^+$. Note that in this case $X_j\in I(\tilde{\mathcal D},i)$. Hence, the induction hypothesis gives
\[
\mathrm{C}(\tilde{\mathcal D},\tilde \xi_j)_{\phi(D),\phi(D')}
=
\begin{cases}
\operatorname{sgn}_z(\phi(D),\phi(D'))h &\textup{if $\phi(D')\in \mathrm{SM}_{\phi(D),i,z}$,}\\
0 &\textup{otherwise.}
\end{cases}
\]
Hence, Lemma~\ref{lemma:SimpleMovesAndHW} and Lemma~\ref{lemma:MatchingFunctionsProperties} imply
\[
\mathrm{C}({\mathcal D}, \xi_j)_{D,D'}
=
\begin{cases}
\operatorname{sgn}_z(D,D')h &\textup{if $D'\in \mathrm{SM}_{D,i,z}$,}\\
0 &\textup{otherwise.}
\end{cases}
\]
So $\mathrm{C}(\mathcal D,\xi_j)$ has the correct off-diagonal terms.
It remains to prove the case $X_j= U_{j_0}^+$. Note that in this case $i=i_0$. Since $X_{j+1}\in I(\mathcal D,i-1)$ and $X_{j}\in I(\tilde{\mathcal D},i-1)$, the induction hypothesis and the previous case imply $\varphi_{j_0}(\mathrm C(\tilde{\mathcal D},\tilde\xi_{j})_{\phi(D),\phi(D')})=\mathrm{C}({\mathcal D}, \xi_{j+1})_{D,D'}$. Therefore, Lemma~\ref{lemma:SimpleMovesAndHW} and Lemma~\ref{lemma:MatchingFunctionsProperties} again imply that $\mathrm{C}(\mathcal D,\xi_{j})$ and $\mathrm{C}(\mathcal D,\xi_{j-1})$ have identical off-diagonal entries. By the first case, the latter are given by
\[
\mathrm{C}({\mathcal D}, \xi_{j-1})_{D,D'}
=
\begin{cases}
\operatorname{sgn}_z(D,D')h &\textup{if $D'\in \mathrm{SM}_{D,i,z}$,}\\
0 &\textup{otherwise},
\end{cases}
\]
which completes the proof.
\end{proof}

	\bibliographystyle{alpha}
	\bibliography{BowCMv3ArXiv.bib}

\begin{thebibliography}{AMSS23}

\bibitem[AF23]{anderson2023equivariant}
David Anderson and William Fulton.
\newblock {\em Equivariant cohomology in algebraic geometry}, volume 210 of
  {\em Cambridge Studies in Advanced Mathematics}.
\newblock Cambridge University Press, 2023.

\bibitem[AMSS23]{aluffi2017shadows}
Paolo Aluffi, Leonardo Mihalcea, J{\"o}rg Sch{\"u}rmann, and Changjian Su.
\newblock Shadows of characteristic cycles, {Verma} modules, and positivity of
  {Chern}--{Schwartz}--{Mac}{Pherson} classes of {Schubert} cells.
\newblock {\em Duke Math. J.}, 172(17):3257--3320, 2023.

\bibitem[And12]{anderson2012introduction}
Dave Anderson.
\newblock Introduction to equivariant cohomology in algebraic geometry.
\newblock In {\em Contributions to algebraic geometry}, pages 71--92. European
  Mathematical Society, 2012.

\bibitem[AP93]{allday1993cohomological}
Christopher Allday and Volker Puppe.
\newblock {\em Cohomological methods in transformation groups}.
\newblock Cambridge University Press, 1993.

\bibitem[BR23]{botta2023mirror}
Tommaso~Maria Botta and Rich{\'a}rd Rim{\'a}nyi.
\newblock Bow varieties: stable envelopes and their 3d mirror symmetry.
\newblock {\em arXiv preprint
  \href{https://arxiv.org/abs/2308.07300}{arXiv:2308.07300}}, 2023.

\bibitem[CG97]{chriss1997representation}
Neil Chriss and Victor Ginzburg.
\newblock {\em Representation theory and complex geometry}.
\newblock Springer, 1997.

\bibitem[CG03]{costello2003hilbert}
Kevin Costello and Ian Grojnowski.
\newblock Hilbert schemes, {Hecke} algebras and the {Calogero}-{Sutherland}
  system.
\newblock {\em arXiv preprint
  \href{https://arxiv.org/abs/math/0310189}{arXiv:math/0310189}}, 2003.

\bibitem[Che94]{chevalley1994decompositions}
Claude Chevalley.
\newblock Sur les d{\'e}compositions cellulaires des espaces {G}/{B}.
\newblock {\em Proc. Sympos. Pure Math.}, 56(1):1--23, 1994.

\bibitem[Che09]{cherkis2009moduli}
Sergey Cherkis.
\newblock Moduli spaces of instantons on the {Taub}-{NUT} space.
\newblock {\em Commun. Math. Phys.}, 290:719--736, 2009.

\bibitem[Che10]{cherkis2010instantons}
Sergey Cherkis.
\newblock Instantons on the {Taub}-{NUT} space.
\newblock {\em Adv. Theor. Math. Phys.}, 14(2):609--642, 2010.

\bibitem[Che11]{cherkis2011instantons}
Sergey Cherkis.
\newblock Instantons on gravitons.
\newblock {\em Commun. Math. Phys.}, 306:449--483, 2011.

\bibitem[FRT88]{faddeev1988quantization}
Ludwig Faddeev, Nikolai Reshetikhin, and Leon Takhtajan.
\newblock Quantization of {Lie} groups and {Lie} algebras.
\newblock In {\em Algebraic analysis}, pages 129--139. Elsevier, 1988.

\bibitem[FS23]{foster2023tangent}
Alex Foster and Yiyan Shou.
\newblock Tangent weights and invariant curves in type {A} bow varieties.
\newblock {\em arXiv preprint
  \href{https://arxiv.org/abs/2310.04973}{arXiv:2310.04973}}, 2023.

\bibitem[GKM98]{goresky1998equivariant}
Mark Goresky, Robert Kottwitz, and Robert MacPherson.
\newblock Equivariant cohomology, {K}oszul duality, and the localization
  theorem.
\newblock {\em Invent. Math.}, 131(1):25--84, 1998.

\bibitem[GKS20]{gorbounov2020yang}
Vassily Gorbounov, Christian Korff, and Catharina Stroppel.
\newblock Yang--{Baxter} algebras, convolution algebras, and {Grassmannians}.
\newblock {\em Russ. Math. Sur.}, 75(5):791--842, 2020.

\bibitem[Hsi75]{hsiang1975cohomology}
Wu~Yi Hsiang.
\newblock {\em Cohomology theory of topological transformation groups}.
\newblock Springer, 1975.

\bibitem[Hum90]{humphreys1992reflection}
James Humphreys.
\newblock {\em Reflection groups and Coxeter groups}.
\newblock Number~29 in Cambridge Studies in Advanced Mathematics. Cambridge
  university press, 1990.

\bibitem[JK81]{james1981representation}
Gordon James and Adalbert Kerber.
\newblock {\em The representation theory of the symmetric group}, volume~16 of
  {\em Encyclopedia of Mathematics and Applications}.
\newblock Addison-Wesley, 1981.

\bibitem[KT03]{knutson2003puzzles}
Allen Knutson and Terence Tao.
\newblock Puzzles and (equivariant) cohomology of {Grassmannians}.
\newblock {\em Duke Math. J.}, 119(2):221--260, 2003.

\bibitem[MO19]{maulik2019quantum}
Davesh Maulik and Andrei Okounkov.
\newblock Quantum groups and quantum cohomology.
\newblock {\em Asterisque}, 408:1--225, 2019.

\bibitem[Mon59]{monk1959geometry}
David Monk.
\newblock The geometry of flag manifolds.
\newblock {\em Proc. London Math. Soc.}, 3(2):253--286, 1959.

\bibitem[Nak94]{nakajima1994instantons}
Hiraku Nakajima.
\newblock Instantons on {A}{L}{E} spaces, quiver varieties, and {Kac}-{Moody}
  algebras.
\newblock {\em Duke Math. J.}, 76(2):365--416, 1994.

\bibitem[Nak18]{nakajima2021geometric}
Hiraku Nakajima.
\newblock Towards geometric {Satake} correspondence for {Kac}--{Moody}
  algebras--{Cherkis} bow varieties and affine {Lie} algebras of type {$A$}.
\newblock {\em arXiv preprint
  \href{https://arxiv.org/abs/1810.04293}{arXiv:1810.04293}}, 2018.

\bibitem[NT17]{nakajima2017cherkis}
Hiraku Nakajima and Yuuya Takayama.
\newblock Cherkis bow varieties and {Coulomb} branches of quiver gauge theories
  of affine type {A}.
\newblock {\em Sel. Math., New Ser.}, 23(4):2553--2633, 2017.

\bibitem[OP10]{okounkov2010quantum}
Andrei Okounkov and Rahul Pandharipande.
\newblock Quantum cohomology of the {Hilbert} scheme of points in the plane.
\newblock {\em Invent. Math.}, 179(3):523--557, 2010.

\bibitem[RS24]{rimanyi2020bow}
Rich{\'a}rd Rim{\'a}nyi and Yiyan Shou.
\newblock Bow varieties---geometry, combinatorics, characteristic classes.
\newblock {\em Commun. Anal. Geom.}, 32(2):507--575, 2024.

\bibitem[RTV15]{rimanyi2015partial}
Rich{\'a}rd Rim{\'a}nyi, Vitaly Tarasov, and Alexander Varchenko.
\newblock Partial flag varieties, stable envelopes, and weight functions.
\newblock {\em Quantum topology}, 6(2):333--364, 2015.

\bibitem[Sho21]{shou2021bow}
Yiyan Shou.
\newblock {\em Bow Varieties---Geometry, Combinatorics, Characteristic
  Classes}.
\newblock PhD thesis, University of North Carolina at Chapel Hill, 2021.

\bibitem[Su16]{su2016equivariant}
Changjian Su.
\newblock Equivariant quantum cohomology of cotangent bundle of {G}/{P}.
\newblock {\em Adv. Math.}, 289:362--383, 2016.

\bibitem[Su17a]{su2017restriction}
Changjian Su.
\newblock Restriction formula for stable basis of the {Springer} resolution.
\newblock {\em Sel. Math.}, 23:497--518, 2017.

\bibitem[Su17b]{su2017stable}
Changjian Su.
\newblock {\em Stable basis and quantum cohomology of cotangent bundles of flag
  varieties}.
\newblock PhD thesis, Columbia University, 2017.

\bibitem[SW23]{expository2023orthogonality}
Catharina Stroppel and Till Wehrhan.
\newblock Existence and orthogonality for stable envelopes of bow varieties.
\newblock {\em arXiv preprint
  \href{https://arxiv.org/abs/2312.03144}{arxiv:2310.11235}}, 2023.

\bibitem[Tak16]{takayama2016nahm}
Yuuya Takayama.
\newblock Nahm's equations, quiver varieties and parabolic sheaves.
\newblock {\em Publ. Res. Inst. Math. Sci.}, 52(1):1--41, 2016.

\bibitem[tD87]{tomdieck1987transformation}
Tammo tom Dieck.
\newblock {\em Transformation Groups}.
\newblock de Gruyter, 1987.

\bibitem[TV00]{tarasov2000difference}
Vitaly Tarasov and Alexander Varchenko.
\newblock Difference equations compatible with trigonometric {K}{Z}
  differential equations.
\newblock {\em International Mathematics Research Notices}, 2000(15):801--829,
  2000.

\bibitem[TV05]{tarasov2005dynamical}
Vitaly Tarasov and Alexander Varchenko.
\newblock Dynamical differential equations compatible with rational q{K}{Z}
  equations.
\newblock {\em Lett. Math. Phys.}, 71(2):101--108, 2005.

\bibitem[TV14]{tarasov2014hypergeometric}
Vitaly Tarasov and Alexander Varchenko.
\newblock Hypergeometric solutions of the quantum differential equation of the
  cotangent bundle of a partial flag variety.
\newblock {\em Open Math.}, 12(5):694--710, 2014.

\bibitem[Weh24]{wehrhanphd}
Till Wehrhan.
\newblock {\em Combinatorial aspects of bow varieties}.
\newblock PhD thesis, Univerity of Bonn, 2024.

\end{thebibliography}
\end{document}